\documentclass[12pt]{amsart}
\usepackage{amsbsy}
\usepackage{amsmath}
\usepackage{amsthm}
\usepackage{amssymb}
\usepackage{amscd}
\usepackage{amsfonts}
\usepackage{epsfig}
\usepackage{xypic}
\usepackage[pdftex,bookmarks,colorlinks,breaklinks=true]{hyperref}
\usepackage{tikz-cd}           
\usepackage[T1]{fontenc}
\DeclareFontFamily{T1}{wncyr}{}
\DeclareFontShape{T1}{wncyr}{m}{n}{%
<5><6><7><8><9>gen*wncyr%
<10><10.95><12><14.4><17.28><20.74><24.88>wncyr10}{}

\newcommand{\sha}{\mathbin{\textup{\usefont{T1}{wncyr}{m}{n}\char120 }}}
\newtheorem{thm}{Theorem}[section]
\newtheorem{lem}[thm]{Lemma}
\newtheorem{prop}[thm]{Proposition}
\newtheorem{cor}[thm]{Corollary}
\newtheorem{defn}[thm]{Definition}

\theoremstyle{remark}
\newtheorem{rem}[thm]{Remark}
\newtheorem{ex}[thm]{Example}

\newtheorem{sch}{Step}

\def\gH{\mathfrak H}\def\gM{\mathfrak M}\def\gS{\mathfrak S}\def\gT{\mathfrak T}\def\gB{\mathfrak B}
\def\gK{\mathfrak K}\def\gL{\mathfrak L}
\def\gG{\mathfrak G}\def\gI{\mathfrak I}\def\gJ{\mathfrak J}\def\gV{\mathfrak V}\def\ga{\mathfrak a}
\def\gf{\mathfrak f}\def\gg{\mathfrak g}\def\gh{\mathfrak h}\def\ge{\mathfrak e}
\def\gP{\mathfrak P}\def\gp{\mathfrak p}\def\gr{\mathfrak r}\def\gq{\mathfrak q}\def\gn{\mathfrak n}\def\gm{\mathfrak m}

\def\bD{\mathbf D}\def\bP{\mathbf P}\def\bU{\mathbf U}\def\bC{\mathbf C}\def\bR{\mathbf R}\def\bT{\mathbf T}\def\bJ{\mathbf J}\def\Bb{\mathbf B}
\def\Mot{\mathbf{Mot}}\def\bj{\mathbf j}\def\bd{\mathbf d}\def\be{\mathbf e}\def\bh{\mathbf h}\def\jar{\overline{\mathbf j}}\def\bard{\overline{\mathbf d}}

\def\Ob{\mathtt{Ob}}\def\CAS{\mathit B}\def\alg{\mathit{Alg}}\def\ect{\mathit{Vec}}\def\yst{\mathit{Sys}}\def\mic{\mathit{MIC}}
\def\Endo{\mathit{End}}\def\Auto{\mathit{Aut}}\def\Homo{\mathit{Hom}}\def\bRep{\mathit{Rep}}\def\BI{\mathit B}\def\CI{\mathit C}\def\EI{\mathit E}\def\cry{\mathit{Isoc}}
\def\eF{\mathbf{F}\mathit{Mod}}\def\ef{\mathbf{V}\mathit{Mod}}\def\cris{\mathit{Cris}}\def\set{\mathit{Set}}

\def\A{\mathcal A}\def\Marb{\overline{\mathcal M}}
\def\Parb{\overline{\mathcal P}}\def\Barb{\overline{\mathcal B}}

\def\B{\mathcal B}\def\O{\mathcal O}\def\E{\mathcal E}\def\N{\mathcal N}
\def\D{\mathcal D}\def\P{\mathcal P}\def\F{\mathcal F}\def\G{\mathcal G}\def\I{\mathcal I}\def\H{\mathcal H}
\def\K{\mathcal K}\def\L{\mathcal L}\def\M{\mathcal M}\def\R{\mathcal R}\def\V{\mathcal V}\def\Z{\mathcal Z}
\def\X{\mathcal X}\def\T{\mathcal T}\def\W{\mathcal W}\def\U{\mathcal U}\def\S{\mathcal S}\def\Q{\mathcal Q}\def\ton{\mathit{Tors}}\def\con{\mathit{Conn}}

\def\Ubar{\overline U}\def\Abar{\overline A}
\def\Rbar{\overline R}\def\Cbar{\overline C}\def\Gbar{\overline{G}}
\def\Ybar{\overline{Y}}\def\Xbar{\overline{X}}\def\Pbar{\overline P}

\def\zbar{\overline z}
\def\xbar{\overline x}\def\Ibar{\overline{\mathcal I}}

\def\Gamba{\overline\Gamma}\def\gamba{\overline\gamma}\def\Bapsi{\overline\Psi}\def\Phbar{\overline\Phi}
\def\thevar{\overline\vartheta}\def\bachi{\overline\chi}
\def\bohr{\overline\rho}\def\nbar{\overline\nu}\def\zeba{\overline{\zeta}}\def\mub{\overline\mu}
\def\ebar{\overline{\eta}}\def\iotar{\overline\iota}\def\hbar{\overline h}

\def\Sb{\mathcal S^{basic}}
\def\essbar{\overline{\mathcal S}}\def\peebar{\overline{pr}}\def\barb{\overline b}
\def\val{\overline v}\def\du{\underline d}

\def\l{\mathbb L}\def\a{\mathbb A}\def\d{\mathbb D}\def\q{\mathbb Q}\def\f{\mathbb F}\def\c{\mathbb C}
\def\g{\mathbb G}\def\h{\mathbb H}\def\r{\mathbb R}\def\z{\mathbb Z}\def\n{\mathbb N}\def\s{\mathbb S}\def\t{\mathbb T}
\def\k{\mathbb K}\def\fc{{\mathbb F_p^{ac}}}

\def\zen{\operatorname Z}  \def\UL{\operatorname U}  \def\GU{\operatorname{GU}}  \def\Sp{\operatorname{Sp}}  \def\SO{\operatorname{SO}}
\def\GL{\operatorname{GL}}  \def\GSp{\operatorname{GSp}}      \def\sd{\operatorname{std}}
  \def\diag{\operatorname{diag}}  \def\Int{\operatorname{Int}}  \def\int{\operatorname{int}}
\def\Ad{\operatorname{Ad}}        
\def\Lie{\operatorname{Lie}}  \def\Res{\operatorname{Res}}  \def\tr{\operatorname{tr}}    
\def\fx{\operatorname{Flex}}  \def\Fx{\operatorname{flex}}  \def\sy{\operatorname{Fib}}    \def\End{\operatorname{End}}
\def\Aut{\operatorname{Aut}}  \def\Iso{\operatorname{Iso}}  \def\id{\operatorname{id}}  \def\Hom{\operatorname{Hom}}  \def\Mat{\operatorname{Mat}}
\def\Gal{\operatorname{Gal}}  \def\Pic{\operatorname{Pic}}  \def\Br{\operatorname{Br}}  \def\Spf{\operatorname{Spf}}  \def\Spec{\operatorname{Spec}}  
\def\Card{\operatorname{Card}}  \def\rk{\operatorname{rank}}  \def\ch{\operatorname{char}}  \def\rad{\operatorname{rad}}  \def\pr{\operatorname{pr}}  
        \def\cur{\operatorname{curv}} 
\def\rest{\operatorname{res}}  \def\core{\operatorname{cor}}

\title{PEL Moduli Spaces without $\c$-valued points} 

\author{Oliver B\"ultel}
\thanks{
14K10, 14G35 (MSC 2020), supported by the german research foundation}

\begin{document}
\maketitle

\begin{center}
To my parents, Heinrich and Karola
\end{center}

\begin{abstract}
We give several new moduli interpretations of the fibers of certain Shimura varieties over several prime numbers. As a consequence (of our theorem \ref{uniformizeIV}) one obtains that 
for every prescribed odd prime characteristic $p$ every bounded symmetric domain possesses quotients by arithmetic groups whose models have good reduction at a prime divisor of $p$.
\end{abstract}

\tableofcontents

\section{Introduction}

Let $(G,X)$ be a Shimura datum in the sense of all five axioms \cite[(2.1.1.1)-(2.1.1.5)]{deligne3} of Deligne. Fix any connected component $X^+\subset X$ and let $K$ be a 
compact open subgroup of $G(\a^\infty)$, where $\a^\infty$ is given by $\q\otimes\prod_{\ell}\z_\ell$, in which $\ell$ runs through the set of prime numbers. Following the ideas in, 
and using the notation of \cite[(2.1.2)]{deligne3}, we know that
\begin{eqnarray}
\label{bekanntV}
&&G(\q)\backslash(X\times G(\a^\infty)/K)={_KM}(G,X)=\\
\label{bekanntVI}
&&G(\q)_+\backslash(X^+\times G(\a^\infty)/K)=\coprod_g\Gamma_g\backslash X^+
\end{eqnarray} 
is a quasi-projective complex algebraic variety, where $\Gamma_g$ is a bunch of discrete subgroups lying in the connected component $G^{ad}(\r)^+$ of the neutral element in the 
Lie-group of $\r$-valued points on the adjoint group $G^{ad}$. Well-known mild conditions on $K$ may guarantee that each $\Gamma_g$ is torsion-free, so that \eqref{bekanntV} is smooth, 
and in the special case of an anisotropic $\q$-group $G^{ad}$ each $\Gamma_g$ even turns out to be cocompact, so that \eqref{bekanntV} is projective. Moreover, one knows, that there 
exists a canonical variation $\sy_{Hodge}(\rho)$ of pure Hodge structures over $_KM(G,X)$ for every $\q$-rational linear representation $\rho:G\rightarrow\GL(V/\q)$. In some cases 
these yield so-called ``moduli interpretations'', more specifically if there exists an injective map $\rho:G\rightarrow\GSp(2g)$ with $\rho(X)\subset\gh_g^\pm$, then $_KM(G,X)$ is a moduli 
space of $g$-dimensional abelian varieties with additional structure, i.e. equipped with additional Hodge cycles (on suitable powers). This result is not only theoretically significant, but it also 
has a practical meaning, because it is this particular class of Shimura varietes of Hodge type which are, at least in principle, amenable to the methods of arithmetic algebraic geometry. For 
instance, it is this way that Milne has reobtained Deligne's canonical models over the reflex field $E\subset\c$ in \cite{milne2}. Outside this class of Shimura varieties these methods fail, but 
the work of many people (e.g. \cite{kazhdan}, \cite{borovoi}) has culminated in more general results. Canonical models are finally shown to exist unconditionally in \cite{milne3}. However, 
the proof uses a very amazing reduction to the $\GL(2)$-cases, and again these are treated by means of abelian varieties whose endomorphism rings have a suitable structure.\\
However, it is also very desireable to control the integrality properties of the Shimura variety \eqref{bekanntV}. So let us fix an odd rational prime $p$ and write $\O_{E_\gp}$ for 
the complete valuation rings corresponding to prime divisors $\gp|p$. If $G$ is anisotropic Langlands conjectured the existence of a projective and smooth $\O_{E_\gp}$-model 
$_K\M_\gp$, provided only that the group $K$ can be written as $K^p\times K_p$, where $K_p$ is a hyperspecial subgroup of $G(\q_p)$, and $K^p$ is a sufficiently small 
compact open subgroup of $G(\a^{\infty,p})$, where $\a^{\infty,p}=\q\otimes\prod_{\ell\neq p}\z_\ell$. He also points out the necessity of characterizing $_K\M_\gp$ (be it 
projective or quasi-projective), which would make the model more canonical, so that the conjecture remains meaningful in the isotropic case too \cite[p.411, l.17-19]{langlands}. 
Such a characterization was suggested by Milne, however his early attempt (in \cite[Definition(2.5)]{milne4}) had to be modified (cf. \cite{vasiu2} and \cite[Definition(3.3)]{moonen}). 
Finally, integral canonical models were shown to exist for Shimura varieties of abelian type, see \cite{kisin1}, \cite{kisin} and the references therein.

\subsubsection*{Our results:} In this paper we focus on Langlands original conjecture for Shimura data  $(G,X)$ which do not allow any embeddings of $(G^{ad},X^{ad})$ 
into $(\Sp(2g)/\{\pm1\},\gh_g^{\pm1})$. Such non-preabelian Shimura data are rare, but they do exist. One of the outcomes of the present paper is the following result:

\begin{thm}
\label{uniformizationIII}
Suppose that  $(G,X)$ is a Shimura datum, such that $G^{der}$ is simply connected. Suppose that $f\geq1$ is an integer, and $p>2$ a prime, such that 
and $G$ splits over $K(\f_{p^f})$ (i.e. $\q_p[\exp\frac{2i\pi}{p^f-1}]$) while $G^{ad}$ is simple and quasisplit over $\q_p$. Suppose in addition that $G$ 
satisfies at least one of the following conditions:
\begin{itemize}
\item[(i)]
$G^{ad}\times_\q\r$ possesses more than three times as many compact simple real factors than non-compact ones and is of type $B_l$ or $C_l$.
\item[(ii)]
$G^{ad}\times_\q\r$ possesses more than four times as many compact simple real factors than non-compact ones and is of type $E_7$.
\end{itemize}
Fix an embedding $\iota:K(\f_{p^f})\rightarrow\c$. Then there exists a compact open subgroup $K_p\subset G(\q_p)$ and a scheme $\M$ with a continuous 
right $G(\a^{\infty,p})$-action over $W(\f_{p^f})$, such that $_{K_p}M(G,X)\cong\M\times_{W(\f_{p^f}),\iota}\c$ and every sufficiently small compact open 
subgroup $K^p\subset G(\a^{\infty,p})$ leads to a quotient $_{K^p}\M=\M/{K^p}$ which is a projective and smooth model of ${_{K^pK_p}M}(G,X)$. 
\end{thm}

Under the assumptions of the theorem $p$ is unramified and inert in the totally real number field $L^+$ which is needed in order to write $G^{der}$ as a restriction of scalars of some 
simply-connected absolutely simple $L^+$-group, and without any loss of generality this field is an extension of degree $f$ over $\q$. Observe also that the choice of $\iota$ induces a continuous embedding of $E_\gp$ into $K(\f_{p^f})$, where $\gp$ is a prime of $E$ over $p$, and notice that the degree of $E_\gp$ over $\q_p$ is at least five, in fact if $X^+$ is irreducible, then $E$ contains 
the whole splitting field $R^+$ of $L^+$, so that $E_\gp=K(\f_{p^f})$. The unramified inertness of $p$ in $L^+$ is equivalent to $\Gal(R^+/\q)=\{\id_{R^+},\vartheta,\dots,\vartheta^{f-1}\}\Gal(R^+/L^+)$, 
where $\vartheta\in\Gal(R^+/\q)$ stands for a Frobenius element of the $\q$-extension $R^+$ (N.B.: this is independent of choosing $L^+\hookrightarrow R^+$). Thus, on the one hand 
Chebotarev's theorem implies that under the assumptions of theorem \ref{uniformizationIII} we can apply it also to a set of positive density of other rational primes, but on the other hand 
our proof of theorem \ref{uniformizationIII} does not make it easy to compare the results for two different primes. Nevertheless, we do expect that our models $_{K^p}\M$ agree with 
the scalar extensions to $W(\f_{p^f})$ of the conjectured integral canonical models $_{K^pK_p}\M_\gp$ for all such $\gp$, we hope to come back to this problem in a future paper.  

\subsubsection*{Our Methods:} Very much unlike Kisin's and Vasiu's work, our construction of $_{K^p}\M$ sheds no light on its 
conjectural but natural interpretation as a moduli space of ``polarized motives with $G$-structure'', but instead relates its special fiber 
\begin{equation}
\label{modulo}
_{K^p}\Marb={_{K^p}\M}\times_{W(\f_{p^f})}\f_{p^f},
\end{equation}
to the subject of pathological additional structures on abelian varieties, which was pioneered e.g. in \cite[Chapter 9]{katz3}. More specifically, 
our roundabout proof of theorem \ref{uniformizationIII} is preceded by the construction of a family of intermediary moduli schemes
\begin{equation}
\label{magic}
_{K^p}\tilde M\rightarrow{_{K^p}\gM}\rightarrow{_{\tilde U}\S}\rightarrow\A_{g,n},
\end{equation}
whose definitions depend on many choices. At least in principle $_{K^p}\gM$ is an intersection of a carefully chosen PEL-type Shimura subvariety of $\A_{g,n}$ with 
the canonical integral model $_{\tilde U}\S$ of an auxiliary Shimura subdatum $(\tilde G,\tilde X)\hookrightarrow(\GSp(2g),\gh_g^\pm)$ of Hodge type. The classification 
of the endomorphism algebras of a point on a Shimura subvariety is a classical but still unexplored subject, cf. \cite{oort1}, \cite{moritaII}. The scheme $_{K^p}\tilde M$ plays 
the role of a provisional candidate for the $\f_{p^f}$-variety \eqref{modulo}, and the morphism on the left hand-side of  \eqref{magic} turns out to be radicial and universally 
closed. An analogous moduli problem can be put in the category of $p$-divisible groups, which is used to introduce a certain fpqc-stack $\gB$ together with a $p$-adically 
formally \'etale $1$-morphism $_{K^p}\gM\rightarrow\gB$, which is  easily defined by the 'passage to the underlying $p$-divisible group'. The number $g$ is quite large, 
in the special case of a group $G$ of type $E_7$ our construction uses $g=25.650f$, and due to the condition (ii) the number $f$ is at least $[E_\gp:\q_p]\geq6$.\\

Before we give further details we would like to remark that integral canonical models and their special fibers are expected to satisfy many other nice properties. For every $\q_p$-rational 
representation $\rho:G\times_\q\q_p\rightarrow\GL(V/\q_p)$, for instance, there should exist a $F$-isocrystal $\sy_{Cris}(\rho)$ (in the sense of e.g. \cite[VI.3.1.3]{rivano}) over 
$_{K^p}\Marb$, and every $K_p$-invariant lattice in $V$ should determine a natural non-degenerate $F$-crystal (in the sense of e.g. \cite[VI.3.1.1]{rivano}) in a suitable Tate twist of 
$\sy_{Cris}(\rho)$. At least for every closed point $x\in{_{K^p}\Marb}$ one expects that the completed local $W(\f_{p^f})$-algebra $\hat\O_{_{K^p}\Marb,x}$ is determined by such crystalline 
data, and should ideally prorepresent a functor of ``crystals with additional structure'', see e.g. \cite[Th\'eor\`eme(2.1.7), Th\'eor\`eme(2.1.14)]{deligne1} for the ordinary $\SO(19,2)$-case, 
but also \cite[3.2.7 Remarks 8b)]{vasiu2} and \cite[Proposition(4.9)]{moonen}, for the general case. In this optic it seems to be reasonable to try to construct $_{K^p}\M$ together with 
the most general crystalline objects that might exist over it. It is one of the standpoints of the present paper that such should be decoded in a $p$-adically formally \'etale $1$-morphism
\begin{equation}
\label{bekanntVII}
_{K^p}\hat\M\rightarrow\B(\gG_p,\mu_p),
\end{equation}
where the left hand-side is the $W(\f_{p^f})$-functor $R\mapsto\begin{cases}\emptyset&\q\otimes R\neq0\\{_{K^p}\M(R)}&\mbox{ otherwise }\end{cases}$ and the right 
hand-side is a certain fpqc-stack over $W(\f_{p^f})$, which should only depend on the reductive $\z_p$-model with $K_p=\gG_p(\z_p)$ and on the minuscule cocharacter 
$\mu_p$ which is associated with the Shimura-datum. The paper \cite{pappas} suggests a natural definition for $\B(\gG_p,\mu_p)$, which seems to work well at least for those 
pairs $(\gG_p,\mu_p)$ for which every simple factor of $G^{ad}\times_\q\q_p$ contains a simple factor of $G^{ad}\times_\q E_\gp^{ac}$ in which $\mu_p$ is trivial (N.B.: this implies 
that every simple factor of $G^{ad}$ contains a compact simple factor of $G^{ad}\times_\q\r$ and the only PEL-cases with this property are unitary groups). We next introduce a map 
\begin{equation}
\label{trick}
\fx:\B(\gG_p,\mu_p)\times_{W(\f_{p^f})}\f_{p^f}\rightarrow\gB\times_{W(\f_{p^f})}\f_{p^f}
\end{equation}
which is, roughly speaking a formal analog of the left hand-side morphism in \eqref{magic}. This adds a lot of content to the whole picture, and it gives 
us a clue to construct $_{K^p}\Marb$ as the largest $\f_{p^f}$-variety fitting into a $2$-commutative diagram
$$\begin{CD}
{\B(\gG_p,\mu_p)\times_{W(\f_{p^f})}\f_{p^f}}@<<<_{K^p}\Marb\\
@V{\fx}VV@VVV\\
\gB\times_{W(\f_{p^f})}\f_{p^f}@<<<_{K^p}\gM\times_{W(\f_{p^f})}\f_{p^f}
\end{CD}$$
with a formally \'etale map to $\B(\gG_p,\mu_p)\times_{W(\f_{p^f})}\f_{p^f}$ and a radicial map to $_{K^p}\gM\times_{W(\f_{p^f})}\f_{p^f}$. With an eye 
towards \eqref{bekanntVII} we move on to describe $_{K^p}\M$ as a lift of $_{K^p}\Marb$. In the holomorphic category there exists a period map from 
the universal covering space to the associated symmetric Hermitian domain of compact type. Our investigation of this map (and its image) follows 
\cite{varshavsky} very closely and completes the proof of theorem \ref{uniformizationIII}, see also \cite{nori}. In a future paper we hope to compute the group $K_p$.\\

The subsection \ref{twisted} computes the generic Newton-polygon of the universal equicharacteristic deformations over $k[[t_1,\dots,t_d]]$, and it applies this to obtain another 
technical result, which is later on needed for the determination of the monodromy of the complex analytic period map. Subsubsection \ref{steady} constructs a provisional version 
$_{K^p}\tilde M$ of the envisaged special fiber $_{K^p}\Marb$. It uses proposition \ref{properIII} together with subsubsection \ref{ready}. The passage to the actual special fiber and 
the construction of a lift $_{K^p}\M$ are explained in subsubsection \ref{go}. The former uses corollary \ref{smoothIV} together with section \ref{Frobenius}, and the latter is a simple 
consequence of our subsection \ref{meckerei} on deformation theory (in particular corollary \ref{lift} therein). The relation between $_{K^p}\Marb$ and $_{K^p}\gM$ is reminiscent 
of the work \cite{zink3}, which identifies the good reduction of several $\GL(2)$-cases, as strata in the special fiber of a Hilbert-Blumenthal-style moduli problem of bad reduction.\\ 

\subsubsection*{Acknowledgements:} It is a pleasure to thank Prof. M. Kisin for encouragement and for pointing out the references \cite{sansuc}, 
\cite{langlands} and \cite{prasad}. Moreover, I thank Prof. D. Blasius, Mounia Chkirda, Prof. L. Fargues, Claudia Glanemann, Prof. W. Kohnen, Prof. M. Levine, Prof. E. Mantovan, 
Andrea Miller, Prof. R. Noot, Prof. G. Pappas, Prof. N. Schappacher, Ismael Souderes and Roland Torka for interest, and Prof. C. Cheng, Prof. R. Coleman, Prof. W. Goldring, 
Prof. H. Hedayatzadeh Razavi, Prof. C. Huyghe, Prof. I. Ikeda, Prof. R. Kottwitz, Prof. V. Paskunas, Prof. R. Pink, Prof. R. Taylor and Prof. T. Zink for invitations to Nanjing, 
Berkeley, Stockholm, Tehran, Strasbourg, Istanbul, Chicago, Duisburg-Essen, Z\"urich, Harvard and Bielefeld, where some predecessors of this work could be discussed. 
I thank Prof. M. Rapoport for pointing out the reference \cite{zink3} and I thank Prof. J. Milne, Prof. G. Prasad and Prof. Y. Varshavsky for many explanations, and I owe 
many further thanks to the audience of my talk \cite{bultel3}, in which I presented a variant of the map \eqref{trick} for the orthogonal group with one compact and one 
non-compact factor i.e. $\SO(h-2,2)\times\SO(h)$. Finally, I thank the speakers and the audience of a study group on displays, which was organized by Prof. R. Pink 
at the ETH Z\"urich during spring term 2005. I thank the referee of this paper for having written a report with profoundly helpful suggestions. The research on this 
paper was begun in 2003 at Berkeley during the DFG funded scholarship BU-1382/1-1. 

\section{Preliminaries}
This section contains most basic notions and should be consulted only when needed.
\subsection{Functors}
\label{synthIV}
We let $\set$ be the category of sets, and we let $\alg_X$ be the usual category structure on the class of pairs consisting of a commutative ring $R$ together with a morphism from its spectrum 
to some fixed base scheme $X$. Any covariant functor from $\alg_X$ to $\set$ will simply be called a $X$-functor. Notice that every morphism $\phi$ from a (not necessarily affine) scheme $Y$ 
to the scheme $X$ gives rise to an $X$-functor, and that the set of $X$-morphisms between any two $X$-schemes can be recovered from the functorial transformations between their $X$-functors. 
In this optic every $X$-scheme ``is'' (i.e. represents) a $X$-functor. We use the term $X$-category synonymously with fibrations over $\alg_X^{op}$, and a $X$-groupoid is a $X$-category none 
of whose fibers have any non-invertible morphisms, and for an arbitrary $X$-category $\CI$ we write $\CI^*$ for the $X$-groupoid, all of whose fibers are obtained from $\CI(R)$ by discarding 
all non-invertible morphisms. A $X$-groupoid $\CI$ is called discrete if no object of $\CI(R)$ possesses non-trivial automorphisms (i.e. ``is'' a $X$-functor). Moreover, for any $\alg_Y$-object 
$R$ we let $R_{[\phi]}\in\Ob_{\alg_X}$ stand for the altered algebra-structure on the same underlying ring, and for any $X$-category $\CI$ we denote the pull-back along the morphism 
$\phi:X\rightarrow Y$ (i.e. $\CI\times_{\alg_X^{op},[\phi]^{op}}\alg_Y^{op}$) by: $\CI\times_{X,\phi}Y$, we suppress the ``$\phi$'' in the notation whenever the context allows to do that. For every $R\in\Ob_{\alg_X}$ we shall denote the fiber of a $X$-category $\CI$ over $R$ by $\CI(R)$, while $\CI_R$ should stand for the $\Spec R$-category $\CI\times_X\Spec R$. The same notation 
applies to any $X$-functor $\F$, we write $\F\times_{X,\phi}Y$ for the $Y$-functor defined by $\alg_Y\rightarrow\set;R\mapsto\F(R_{[\phi]})$, while $\F_R$ stands for $\F\times_X\Spec R$.
By a fpqc-covering we simply mean any faithfully flat morphism in the category $\alg_X^{op}$. By the big fpqc-site $X_{fpqc}$ we mean $\alg_X^{op}$ equipped with the Grothendieck 
topology induced by the above class of coverings. For instance, a fpqc-sheaf on $X$ is a $X$-functor $\F$ such that the sequence $\F(R)\rightarrow\F(S)\rightrightarrows\F(S\otimes_RS)$ 
is exact for all faithfully flat $X$-algebra morphisms $R\rightarrow S$. The notions of $X$-stacks, and stacks in groupoids over the site $X_{fpqc}$ will be used synonymously. 

\subsection{Witt-vectors and Greenberg Transforms}
\label{old}
By means of the usual addition and multiplication one can regard $\a^1=\Spec\z[x]$ as a ring scheme over $\z$. Now pick a prime number $p$ and recall the ring scheme of Witt vectors 
with respect to $p$: Its underlying scheme is $\Spec\z[x_0,x_1,\dots]$, on which there exists a unique ring scheme structure $W$ which makes each of the so-called ghost coordinates
$$w_k:\Spec\z[x_0,x_1,\dots]\rightarrow\a^1;(x_0,x_1,\dots)\mapsto\sum_{i=0}^kp^ix_i^{p^{k-i}}$$
into a homomorphism from the ring scheme $W$ to the ring scheme $\a^1$. The additive map $V:W\rightarrow W;(x_0,x_1,\dots)\mapsto(0,x_0,\dots)$ is called 
the Verschiebung, and there also exists a unique Frobenius homomorphism $F:W\rightarrow W$, such that $w_k\circ F=w_{k+1}$ for all $k\in\n_0$. Notice that 
$F(V(x))=px$ and $V(xF(y))=yV(x)$ hold for any two Witt-vectors $x$ and $y$. We will write $I_m\subset W$ for the closed subscheme defined by the equations 
$x_0=\dots=x_{m-1}=0$, and $I$ for $I_1$. One also writes $[x]$ for the Teichm\"uller lift, which is the Witt-vector $(x,0,\dots)$.\\
If $\CI$ is a $\Spec W(\f_{p^r})$-category, then we introduce further $\Spec W(\f_{p^r})$-categories 
$^F\CI$ and $^W\CI$ by defining their fibers over an arbitrary $W(\f_{p^r})$-algebra $R$ to be
\begin{eqnarray*}
&&^F\CI(R)=\CI(R_{[F]})\\
&&^W\CI(R)=\CI(W(R)),
\end{eqnarray*}
where the $W(\f_{p^r})$-algebra structure on $W(R)$ is induced by the diagonal $\Delta:W(\f_{p^r})\rightarrow W(W(\f_{p^r}))$. Likewise a 
fibered $\Spec W(\f_{p^r})$-functor $\rho:\CI\rightarrow\EI$ gives rise to fibered $\Spec W(\f_{p^r})$-functors $^F\rho:{^F\CI}\rightarrow{^F\EI}$, and 
$^W\rho:{^W\CI}\rightarrow{^W\EI}$. The $\Spec W(\f_{p^r})$-categories $^{FW}\CI$ and $^{WF}\CI$ are canonically $\Spec W(\f_{p^r})$-equivalent, since the diagram
$$\begin{CD}
W(W(\f_{p^r}))@>F>>W(W(\f_{p^r}))\\
@A{\Delta}AA@A{\Delta}AA\\
W(\f_{p^r})@>F>>W(\f_{p^r})\\
\end{CD}$$
is commutative. Notice, that there are canonical $\Spec W(\f_{p^r})$-functors 
\begin{eqnarray}
\label{FrobI}
&&F:{^W\CI}\rightarrow{^{WF}\CI}\\
\label{FrobII}
&&w_0:{^W\CI}\rightarrow\CI
\end{eqnarray}
induced by the $W(\f_{p^r})$-algebra homomorphisms $F:W(R)\rightarrow W(R)_{[F]}$ and $w_0:W(R)\rightarrow R$. Both $F$ and $W$ commute with base change along 
$W(\f_{p^r})\rightarrow W(\f_{p^f})$, where $f$ is a multiple of $r$, in particular it does not cause confusion to denote $^W(\CI_{W(\f_{p^f})})=(^W\CI)_{W(\f_{p^f})}$ by 
$^W\CI_{W(\f_{p^f})}$ while $^W\CI_{\f_{p^f}}$ is shorthand for nothing but $(^W\CI)_{\f_{p^f}}$. The following representability results are relevant in this paper:
\begin{itemize}
\item
If $\CI$ is representable by an affine $\Spec W(\f_{p^r})$-scheme, then so is $^W\CI$.
\item
The fiber of $^W\ton(\GL(n)_{W(\f_{p^r})})$ over $\Spec\f_{p^r}$ coincides with $\ton(^W\GL(n)_{\f_{p^r}})$.
\end{itemize}
The first statement is well-known from \cite[Proposition 29]{kreidl} (and the references therein) and the second results from Zink's 
theory of Witt descent (\cite[Proposition 33]{zink2}). The assignment $\CI\mapsto{^W\CI}$ seems to have originated in \cite{green}.

\subsection{Weil Restriction} 

Let $X$ be a $W(\f_{p^r})$-scheme, and suppose that $\Delta$ is an abstract group acting on $X$ from the right. By an equivariant right $\Delta$-action 
$\phi$ on some $X$-functor $P$ we mean a family of maps $\phi_R(g):P(R)\rightarrow P(R_{[g]})$ for every $g\in\Delta$ and $R\in\Ob_{\alg_X}$ which satisfy 
$\phi_{R_{[g]}}(h)\circ\phi_R(g)=\phi_R(gh)$ for every $h\in\Delta$ and are compatible with the restriction maps, in the sense that $|_{S_{[g]}}\circ\phi_R(g)=\phi_S(g)\circ|_S$ 
holds for every $R$-algebra $S$ (N.B.: If $P$ is representable by a scheme, this just means that $P$ and $X$ have compatible $\Delta$-actions, from the right). In addition, 
suppose that $\Delta$ acts from the right on some fpqc-sheaf of groups $\G$ over $\Spec W(\f_{p^r})$. We say that $P$ is a $\Delta$-equivariant formal principal homogeneous 
space for $\G$ over $X$, if we are given a $\Delta$-equivariant map of fpqc-sheaves $P\times_{W(\f_{p^r})}\G\rightarrow P$ endowing $P$ with the structure of a formal 
principal homogeneous space for $\G$ over $X$ in the usual sense (N.B.: If $P$ and $\G$ are representable by $W(\f_{p^r})$-schemes, this just means that the $\Delta$-action 
on the former extends to an action of the semidirect product $\G\rtimes\Delta$ such that every non-empty $P(S)$ becomes a principal homogeneous $\G(S)$-set).
We need some notation for Weil restriction: Consider an extension of finite fields $\f_{p^r}\subset\f_{p^s}$, and recall that any $W(\f_{p^s})$-functor $\F$ 
gives rise to its Weil restriction $\Res_{W(\f_{p^s})/W(\f_{p^r})}\F$, which is the $W(\f_{p^r})$-functor given by $R\mapsto\F(W(\f_{p^s})\otimes_{W(\f_{p^r})}R)$, 
and similarly any $W(\f_{p^s})$-functorial transformation $\rho:\E\rightarrow\F$ gives rise to a corresponding transformation 
$\Res_{W(\f_{p^s})/W(\f_{p^r})}\rho:\Res_{W(\f_{p^s})/W(\f_{p^r})}\E\rightarrow\Res_{W(\f_{p^s})/W(\f_{p^r})}\F$. Observe
\begin{eqnarray*}
&&{^F(\Res_{W(\f_{p^s})/W(\f_{p^r})}\F)}\stackrel{\cong}{\rightarrow}\Res_{W(\f_{p^s})/W(\f_{p^r})}{^F\F}\\
&&^W(\Res_{W(\f_{p^s})/W(\f_{p^r})}\F)\stackrel{\cong}{\rightarrow}\Res_{W(\f_{p^s})/W(\f_{p^r})}(^W\F),
\end{eqnarray*}
which is due to the natural isomorphisms
$$W(\f_{p^s})\otimes_{W(\f_{p^r})}R_{[F]}\stackrel{\cong}{\rightarrow}(W(\f_{p^s})\otimes_{W(\f_{p^r})}R)_{[F]};\,a\otimes x\mapsto F(a)\otimes x,$$
as well as $$W(\f_{p^s})\otimes_{W(\f_{p^r})}W(R)\stackrel{\cong}{\rightarrow}W(W(\f_{p^s})\otimes_{W(\f_{p^r})}R).$$
Let $\Delta$ be the cyclic group $r\z/s\z$. If $\F$ commutes with products, then there are canonical isomorphisms: 
\begin{equation*}
(\Res_{W(\f_{p^s})/W(\f_{p^r})}\F)\times_{W(\f_{p^r})}W(\f_{p^s})\cong\prod_{\sigma\in\Delta}{^{F^{-\sigma}}\F}
\end{equation*}
Composing $\times_{W(\f_{p^r})}W(\f_{p^s})$ and $\Res_{W(\f_{p^s})/W(\f_{p^r})}$ in the other order yields $\Delta$-equivariant $W(\f_{p^r})$-functors 
$\Res_{W(\f_{p^s})/W(\f_{p^r})}P_{W(\f_{p^s})}$ equipped with a transformation $P\rightarrow\Res_{W(\f_{p^s})/W(\f_{p^r})}P_{W(\f_{p^s})}$ from arbitrary $W(\f_{p^r})$-functors 
$P$. The $W(\f_{p^r})$-group $\G$ gives rise to a $\Delta$-equivariant $W(\f_{p^r})$-group $\Res_{W(\f_{p^s})/W(\f_{p^r})}\G_{W(\f_{p^s})}$ containing $\G$ as a closed subgroup.

\begin{lem}
\label{galoisI}
For every flat and affine $W(\f_{p^r})$-group $\G$, the functor $$P\mapsto\Res_{W(\f_{p^s})/W(\f_{p^r})}P_{W(\f_{p^s})}$$ defines an equivalence from the category of 
$\G$-torsors over $X$ to the category of $\Delta$-equivariant $\Res_{W(\f_{p^s})/W(\f_{p^r})}\G_{W(\f_{p^s})}$-torsors over $X$, where $\Delta$ acts trivially on $X$.
\end{lem}
\begin{proof}
This follows from $\ton(\G)$ being a $\Spec W(\f_{p^r})$-stack.
\end{proof}

\section{Windows and Displays with additional structure}
\label{bekanntVIII}

Let us fix some $f\in\n$ and a reductive $W(\f_{p^f})$-group $\G$. Recall that a cocharacter $\mu:\g_{m,W(\f_{p^f})}\rightarrow\G$ is called minuscule if all of its weights on 
$\Lie\G$ are contained in the set $\{-1,0,1\}$. For this class of cocharacters \cite[Proposition 3.1.2]{pappas} proves, that there is a unique $W(\f_{p^f})$-rational homomorphism 
\begin{equation*}
^W\G\times_{w_0,\G}U_{\mu^{-1}}^0=:\I^\mu\stackrel{\Phi^\mu}{\rightarrow}{^{WF}\G}
\end{equation*}
such that the image of $\Phi^\mu(g)$ in $\G(\q\otimes W(R))$ agrees with $^F(\mu(\frac1p)g\mu(p))$ for any $W(\f_{p^f})$-algebra $R$ and any $g\in\I^\mu(R)$, please see 
to \eqref{realIV} for the definition of the parabolic subgroup $U_{\mu^{-1}}^0\subset\G$. The $W(\f_{p^f})$-groups $\I^\mu\subset{^W\G}$ turn out to be affine, flat and formally 
smooth but are not of finite type. The present work requires a curious modification of the group $\I^\mu$ and the homomorphism $\Phi^\mu$, which we present in this section.

\begin{lem}
\label{val}
Let $R$ be a $\f_p$-algebra and consider the function:
$$\val_R(x):W(R)\rightarrow\n_0\cup\{\infty\};(x_0,x_1,\dots)\mapsto\inf\{n|x_n\neq0\}$$ 
For any $x,y\in W(R)$ one has $\val_R(x+y)\geq\min\{\val_R(x),\val_R(y)\}$ and $\val_R(xy)\geq\val_R(x)+\val_R(y)$ and $\val_R({^Vx})=1+\val_R(x)$.
\end{lem}

We will call $\val_R(x)$ the $V$-adic valuation of $x$. Notice that the corresponding chain of ideals $I_m(R)=\{x\in W(R)|m\leq\val_R(x)\}$ defines a filtration on $W(R)$ in the sense of 
part \ref{cocha} of the appendix. The following proposition elucidates the significance of the function $\val_R$, please see to part \ref{ballast} of the appendix for the definitions of the 
group $\hat U_{\mu^{-1}}^0(W(R),\val_R)$ and of the endomorphism $L_{\mu^{-1}}(p)$:

\begin{prop}
\label{twistI}
Let $\mu:\g_{m,W(\f_{p^f})}\rightarrow\G$ be a cocharacter all of whose weights are less than or equal to $h\geq1$, where $\G$ is a reductive group over 
$W(\f_{p^f})$. Let us write $\Ibar^\mu$ for the $\f_{p^f}$-scheme representing the closed subgroup functor of $^W\G_{\f_{p^f}}$ defined by 
$R\mapsto\hat U_{\mu^{-1}}^0(W(R),\val_R)$. Then there exists a unique homomorphism $\Phbar^{\mu,h}:\Ibar^\mu\rightarrow{^{WF^h}\G_{\f_{p^f}}}$, 
such that the image of $\Phbar^{\mu,h}(g)$ in $\G(\q\otimes W(R))$ agrees with $^{F^h}(\mu(\frac1p)g\mu(p))$ for any $\f_{p^f}$-algebra $R$ and any 
$g\in\Ibar^\mu(R)$. Moreover, the restriction of $\Phbar^{\mu,h}$ to the special fibre of $^WU_{\mu^{-1}}^0$ agrees with $F^h\circ{^WL_{\mu^{-1}}(p)}$.
\end{prop}
\begin{proof}
The asserted representability of the functor $\Ibar^\mu$ is an immediate consequences of theorem \ref{triangleI}, in fact the description given there shows that it is a reduced affine group 
scheme. It follows that the characterizing property of $\Phbar^{\mu,h}$ could just as well be phrased for reduced test $\f_{p^f}$-algebras, which already establishes its uniqueness, given that 
$W(R)$ is torsion free for reduced $R$. By the same token we may work over a finite extension $\f_{p^{e}}$ of $\f_{p^f}$ and thus find a split maximal torus $\T\subset\G_{W(\f_{p^e})}$ 
which contains the image of $\mu$. Our construction of $\Phbar^{\mu,h}$ follows the ideas of \cite[Proposition 3.1.2]{pappas}: Let us write $\U_1$, ..., $\U_d$ for the root spaces with 
positive $\mu$-weights, say $h_1\leq\dots\leq h_{d-1}\leq h_d\leq h$, so that $\Ibar_{\f_{p^{e}}}^\mu$ is scheme theoretically the product of $^WU_{\mu^{-1},\f_{p^{e}}}^0$ and the 
groups of $I_{h_i}$-valued sections of $\U_i$. The proof is completed by taking $\Phbar^{\mu,h}$ to be $F^{h-h_i} V^{-h_i}$ on each $\U_i(I_{h_i})\cong I_{h_i}$ (N.B.: $\U_i\cong\g_a$).
\end{proof}

To facilitate the notation we write $\I_0^\mu$ for the $W(\f_{p^f})$-group $U_{\mu^{-1}}^0$ and $\Ibar_0^\mu$ for its special fibre. Many groupoids in this paper arise from the following construction:

\begin{ex}
\label{diagram} 
Let $M$ be an abstract monoid, let $\Gamma$ be a subgroup thereof, and let $\phi:\Gamma\rightarrow M$ be a morphism of monoids. Following \cite{pappas} and \cite{hadi} 
we let $[M/_\phi\Gamma]$ be the category whose class of objects is given by $M$ while the set of morphisms between any two of whose elements $U'$ and $U$ is given by:
$$\Hom_{[M/_\phi\Gamma]}(U',U)=\{k\in\Gamma\mid\,U'=k^{-1}U\phi(k)\}$$
We will refer to $[M/_\phi\Gamma]$ as the groupoid corresponding to the diagram $M\supset\Gamma\stackrel{\phi}{\rightarrow}M$, note that there is a faithful functor
$[M/_\phi\Gamma]\rightarrow[\frac1{\Gamma}]$, where the latter denotes the groupoid with a single object whose group of automorphisms is $\Gamma$. Suppose, 
that $M'\supset\Gamma'\stackrel{\phi'}{\rightarrow}M'$ is another such diagram, then giving a commutative diagram 
$$\begin{CD}
[M/_\phi\Gamma]@>>>[M'/_{\phi'}\Gamma']\\
@VVV@VVV\\
[\frac1{\Gamma}]@>>>[\frac1{\Gamma'}]
\end{CD}$$
is equivalent to giving a map $m:M\rightarrow M'$ and a homomorphism $\gamma:\Gamma\rightarrow\Gamma'$, satisfying 
$m(k^{-1}U\phi(k))=\gamma(k)^{-1}m(U)\phi'(\gamma(k))$. Again, the functors between groupoids in this paper will usually be given in this language. 
Notice that the underlying map from $\Ob_{[M/_\phi\Gamma]}$ to the class of objects $\Ob_{[M'/_{\phi'}\Gamma']}$ is described by $U\mapsto m(U)$.
\end{ex}

\subsection{$\Phbar$-data}
\label{pivotalII}

In this subsection we fix a positive divisor $r\mid f$, for any non-empty set $\Sigma\subset\z/r\z$ and any $\sigma\in\z/r\z$ we will use the notations:
\begin{eqnarray*}
&&\bd_\Sigma^+(\sigma):=\min\{h\geq0|h+\sigma\in\Sigma\}\in\n_0\\
&&\bd_\Sigma(\sigma):=\sigma+\bd_\Sigma^+(\sigma)\in\Sigma\\
&&r_\Sigma(\sigma):=\Card(\bd_\Sigma^{-1}(\{\sigma\}))\in\n_0
\end{eqnarray*}
Notice that $\bd_\Sigma(1+\sigma)$ restricts to a cyclic permutation on $\Sigma$, which we denote by 
$$\varpi_\Sigma:\Sigma\rightarrow\Sigma;\,\sigma\mapsto\sigma+\varpi_\Sigma^+(\sigma),$$ 
where $\varpi_\Sigma^+(\sigma):=\bd_\Sigma^+(1+\sigma)+1$. Notice that $\varpi_\Sigma^{-1}(\sigma)=\sigma-r_\Sigma(\sigma)$. 
Also, for a reductive $W(\f_{p^r})$-group $\G$, we need to consider the group: 
$$\G^\Sigma:=\prod_{\sigma\in\Sigma}{^{F^{-\sigma}}\G_{W(\f_{p^f})}}$$ 
Finally, we let $\gg$ and $\gg^\Sigma$ be the respective Lie-algebras of $\G$ and $\G^\Sigma$.

\begin{defn}
\label{tragic}
Let $\G$ be a reductive $W(\f_{p^r})$-group and let $\Sigma$ be a non-empty subset of $\z/r\z$. The pair $(\G,\{\mu_\sigma\}_{\sigma\in\Sigma})$ is called a $\Phbar$-datum if each 
$\mu_\sigma:\g_m\rightarrow{^{F^{-\sigma}}\G_{W(\f_{p^f})}}$ is a cocharacter whose weights on $W(\f_{p^f})\otimes_{F^{-\sigma},W(\f_{p^r})}\gg$ are less than or equal to $r_\Sigma(\sigma)$.
If the weights of each $\mu_\sigma$ are less than or equal to $1$ (i.e. if each $\mu_\sigma$ is minuscule), than  $(\G,\{\mu_\sigma\}_{\sigma\in\Sigma})$ is called a $\Phi$-datum.
\end{defn}

For some fixed $W(\f_{p^f})$-rational $\Phbar$-datum $(\G,\{\mu_\sigma\}_{\sigma\in\Sigma})$ we need to introduce 
further notions and notations, which we will use in the whole paper without further notice. Consider the cocharacter
$$\mu_\Sigma:\g_{m,W(\f_{p^f})}\rightarrow\G^\Sigma,$$
of which the components are given by $\mu_\sigma$, for every $\sigma\in\Sigma$. Occasionally we need to work with the scalar restriction 
\begin{equation}
\label{eins}
\gG:=\Res_{W(\f_{p^r})/\z_p}\G,
\end{equation}
notice that $\G^\Sigma$ is canonically contained in $\gG_{W(\f_{p^f})}=\prod_{\sigma=0}^{r-1}{^{F^{-\sigma}}\G_{W(\f_{p^f})}}$. Observe 
that $\Ibar^{\mu_\Sigma}$ is equal to the product $\prod_{\sigma\in\Sigma}\Ibar^{\mu_\sigma}$, on the factors of which we defined homomorphisms
\begin{equation}
\label{cycleI}
\Phbar^{\mu_\sigma,r_\Sigma(\sigma)}:\Ibar^{\mu_\sigma}\rightarrow{^{WF^{r_\Sigma(\sigma)-\sigma}}\G_{\f_{p^f}}}.
\end{equation}
Since the composition with $\varpi_\Sigma$ yields an isomorphism:
\begin{equation}
\label{cycleII}
\prod_{\sigma\in\Sigma}{^{F^{r_\Sigma(\sigma)-\sigma}}\G}\stackrel{\cong}{\rightarrow}{\G^\Sigma}
\end{equation}
one obtains a homomorphism
\begin{equation*}
\Phbar:\Ibar^{\mu_\Sigma}\rightarrow{^W\G_{\f_{p^f}}^\Sigma},
\end{equation*}
by composing \eqref{cycleII} with the product of the homomorphisms \eqref{cycleI}, and we will refer to 
$${^W\G_{\f_{p^f}}^\Sigma}\hookleftarrow\Ibar^{\mu_\Sigma}\stackrel{\Phbar}{\rightarrow}{^W\G_{\f_{p^f}}^\Sigma}$$
as the reduced Frobenius map. If each $\mu_\sigma$ is minuscule we have additional homomorphisms
\begin{equation}
\label{cycleIV}
F^{r_\Sigma(\sigma)-1}\circ\Phi^{\mu_\sigma}:\I^{\mu_\sigma}\rightarrow{^{WF^{r_\Sigma(\sigma)-\sigma}}\G_{W(\f_{p^f})}}
\end{equation}
at our disposal, and again composing \eqref{cycleII} with the product of the homomorphisms \eqref{cycleIV} yields the homomorphism
\begin{equation*}
\Phi:\I^{\mu_\Sigma}\rightarrow{^W\G^\Sigma},
\end{equation*}
and we will refer to 
$${^W\G^\Sigma}\hookleftarrow\I^{\mu_\Sigma}\stackrel{\Phi}{\rightarrow}{^W\G^\Sigma}$$
as the non-reduced Frobenius map (N.B.: $\Ibar^{\mu_\Sigma}=\I_{\f_{p^f}}^{\mu_\Sigma}$).

\begin{defn}
\label{concept}
If ${^W\G_{\f_{p^f}}^\Sigma}\hookleftarrow\Ibar^{\mu_\Sigma}\stackrel{\Phbar}{\rightarrow}{^W\G_{\f_{p^f}}^\Sigma}$ is the reduced Frobenius map of some $W(\f_{p^f})$-rational 
$\Phbar$-datum $(\G,\{\mu_\sigma\}_{\sigma\in\Sigma})$, then we let $\Barb(\G,\{\mu_\sigma\}_{\sigma\in\Sigma})$ be the $\f_{p^f}$-stack rendering the diagram:
\begin{equation}
\label{poshI}
\begin{CD}
\ton(^W\G_{\f_{p^f}}^\Sigma)@<K<<\Barb(\G,\{\mu_\sigma\}_{\sigma\in\Sigma})\\
@V{\Delta_{\ton(^W\G_{\f_{p^f}}^\Sigma)}}VV@VqVV\\
\ton(^W\G_{\f_{p^f}}^\Sigma)\times_{\f_{p^f}}\ton(^W\G_{\f_{p^f}}^\Sigma)@<{\ton(\Phbar\times\id)\circ\Delta_{\ton(\Ibar^{\mu_\Sigma})}}<<\ton(\Ibar^{\mu_\Sigma})
\end{CD}
\end{equation}
$2$-cartesian. Moreover, if ${^W\G^\Sigma}\hookleftarrow\I^{\mu_\Sigma}\stackrel{\Phi}{\rightarrow}{^W\G^\Sigma}$ is the non-reduced Frobenius map of a 
$\Phi$-datum $(\G,\{\mu_\sigma\}_{\sigma\in\Sigma})$ we let $\B(\G,\{\mu_\sigma\}_{\sigma\in\Sigma})$ be the $W(\f_{p^f})$-stack rendering the diagram:
\begin{equation}
\label{poshII}
\begin{CD}
\ton(^W\G^\Sigma)@<K<<\B(\G,\{\mu_\sigma\}_{\sigma\in\Sigma})\\
@V{\Delta_{\ton(^W\G^\Sigma)}}VV@VqVV\\
\ton(^W\G^\Sigma)\times_{W(\f_{p^f})}\ton(^W\G^\Sigma)@<{\ton(\Phi\times\id)\circ\Delta_{\ton(\I^{\mu_\Sigma})}}<<\ton(\I^{\mu_\Sigma})
\end{CD}
\end{equation}
$2$-cartesian. 
\end{defn}

Since we have $\B(\G,\{\mu_\sigma\}_{\sigma\in\Sigma})_{\f_{p^f}}\cong\Barb(\G,\{\mu_\sigma\}_{\sigma\in\Sigma})$ for any $\Phi$-datum 
$(\G,\{\mu_\sigma\}_{\sigma\in\Sigma})$, the following convention turns out to be quite handy: By saying that $\P$ was a $(\G,\{\mu_\sigma\}_{\sigma\in\Sigma})$-display 
over a $W(\f_{p^f})$-scheme $X$ we mean that at least one of the following two statements holds:
\begin{itemize}
\item[(i)]
$p$ vanishes in $\O_X$, and $\P$ is a $1$-morphism from $X$ to $\Barb(\G,\{\mu_\sigma\}_{\sigma\in\Sigma})$, where $X$ has to be regarded as a $\f_{p^f}$-scheme.
\item[(ii)]
$(\G,\{\mu_\sigma\}_{\sigma\in\Sigma})$ is a $\Phi$-datum and $\P$ is a $1$-morphism from $X$ to $\B(\G,\{\mu_\sigma\}_{\sigma\in\Sigma})$.
\end{itemize}
In each of these two cases we wish to define the underlying locally trivial principal homogeneous space of $\P$ to be $q(\P)$, which is a $X$-valued point in one of 
$\ton(\Ibar^{\mu_\Sigma})$ or $\ton(\I^\mu)$. Let us denote the lowest truncation of $q(\P)$ by $q_0(\P)$, which is a $X$-valued point in $\ton(\I_0^{\mu_\Sigma})$. Finally 
let us say that $\P$ is a banal $(\G,\{\mu_\sigma\}_{\sigma\in\Sigma})$-display over $X$ if and only if $q(\P)$ possesses a global section. For every $W(\f_{p^f})$-algebra 
$R$ we let $\BI_R(\G,\{\mu_\sigma\}_{\sigma\in\Sigma})$ be the groupoid of banal $(\G,\{\mu_\sigma\}_{\sigma\in\Sigma})$-displays over $\Spec R$. Using the 
formalism of example \ref{diagram} this could also be described as $[{^W\G_{\f_{p^f}}^\Sigma}(R)/_{\Phbar_R}\Ibar^{\mu_\Sigma}(R)]$ or 
$[{^W\G_{W(\f_{p^f})}^\Sigma}(R)/_{\Phi_R}\I^{\mu_\Sigma}(R)]$ depending on whether we are in case (i) or (ii). Observe that the groupoid of banal 
$(\G,\{\mu_\sigma\}_{\sigma\in\Sigma})$-displays over an arbitrary $W(\f_{p^f})$-scheme $X$ is just $\BI_{\Gamma(X,\O_X)}(\G,\{\mu_\sigma\}_{\sigma\in\Sigma})$. 

\begin{rem}
Please see to \cite[Section 2]{langer2} and \cite[Subsection 5.3]{lau2} for more information on displays of not necessarily minuscule type. For every $W(\f_{p^f})$-rational $\Phbar$-datum 
$(\G,\{\mu_\sigma\}_{\sigma\in\Sigma})$, there seems to be a canonical morphism from the special fiber of the $W(\f_{p^f})$-stack of $\gG$-displays of type $\mu_\Sigma$ (in the sense 
of Lau) to our $\f_{p^f}$-stack $\Barb(\G,\{\mu_\sigma\}_{\sigma\in\Sigma})$. In the non-minuscule case the former seems to be different from the latter.
\end{rem}

\begin{rem}
\label{merke}
Suppose that $(\G,\{\mu_\sigma\}_{\sigma\in\Sigma})$ is a $W(\f_{p^f})$-rational $\Phi$-datum, so that the notion $(\G,\{\mu_\sigma\}_{\sigma\in\Sigma})$-display 
in our sense agrees with the notion of $(\gG,\mu_\Sigma)$-display in the sense of \cite[Definition 3.2.1]{pappas}, as $\mu_\Sigma$ is minuscule. 
In this case we will say that $X\stackrel{\P}{\rightarrow}\B(\G,\{\mu_\sigma\}_{\sigma\in\Sigma})$ is adjoint nilpotent if $p$ acts locally nilpotent on 
$\O_X$ and $\P$ is an adjoint nilpotent $(\gG,\mu_\Sigma)$-display in the sense of \cite[Definition 3.4.2]{pappas}. This condition defines a full substack 
$\B^{nil}(\G,\{\mu_\sigma\}_{\sigma\in\Sigma})\subset\B(\G,\{\mu_\sigma\}_{\sigma\in\Sigma})$. In the same vein we let
$$\BI_R^{nil}(\G,\{\mu_\sigma\}_{\sigma\in\Sigma})\subset\BI_R(\G,\{\mu_\sigma\}_{\sigma\in\Sigma})$$ 
be the full subcategory of adjoint nilpotent banal $(\G,\{\mu_\sigma\}_{\sigma\in\Sigma})$-displays over some $W(\f_{p^f})$-algebra $R$ with $pR\subset\sqrt{0_R}$. In the 
special case $\Sigma\neq\z/r\z$ every $(\G,\{\mu_\sigma\}_{\sigma\in\Sigma})$-display over a $W(\f_{p^f})$-scheme satisfying $X\times\q=\emptyset$ is adjoint nilpotent.
\end{rem}

\subsection{Functoriality}
\label{spass}
Fix $f$ and let $(\G,\{\mu_\sigma\}_{\sigma\in\Sigma})$ and $(\tilde\G,\{\tilde\mu_\sigma\}_{\sigma\in\tilde\Sigma})$ 
be two $\Phbar$-data, giving rise to reduced Frobenius maps $\Phbar:\Ibar^{\mu_\Sigma}\rightarrow{^W\G_{\f_{p^f}}^\Sigma}$ and 
$\tilde\Phbar:\Ibar^{\tilde\mu_{\tilde\Sigma}}\rightarrow{^W\tilde\G_{\f_{p^f}}^{\tilde\Sigma}}$. Suppose we are given a pair consisting of a group homomorphism 
$$\gamma:\Ibar^{\mu_\Sigma}\rightarrow\Ibar^{\tilde\mu_{\tilde\Sigma}}$$
together with a morphism
$$m:{^W\G}_{\f_{p^f}}^\Sigma\rightarrow{^W{\tilde\G_{\f_{p^f}}^{\tilde\Sigma}}},$$
that renders the diagram
$$\begin{CD}
{^W\G}_{\f_{p^f}}^\Sigma\times_{\f_{p^f}}\Ibar^{\mu_\Sigma}@>{\pr}>>{^W\G}_{\f_{p^f}}^\Sigma\\
@V{m\times\gamma}VV@VmVV\\
^W{\tilde\G_{\f_{p^f}}^{\tilde\Sigma}}\times_{\f_{p^f}}\Ibar^{\tilde\mu_{\tilde\Sigma}}@>{\tilde\pr}>>^W{\tilde\G_{\f_{p^f}}^{\tilde\Sigma}}
\end{CD}$$
commutative where $\pr$ (resp. $\tilde\pr$) is defined by sending $(U,h)$ to the element $h^{-1}U\Phbar(h)$ 
(resp. $h^{-1}U\tilde\Phbar(h)$). Then one obtains a canonical $2$-commutative diagram:
$$\begin{CD}
{^W\G_{\f_{p^f}}^\Sigma}@>>>\Barb(\G,\{\mu_\sigma\}_{\sigma\in\Sigma})@>q>>\ton(\Ibar^{\mu_\Sigma})\\
@V{m}VV@VVV@V{\ton(\gamma)}VV\\
{^W{\tilde\G_{\f_{p^f}}^{\tilde\Sigma}}}@>>>\Barb(\tilde\G,\{\tilde\mu_\sigma\}_{\sigma\in\tilde\Sigma})@>{\tilde q}>>\ton(\Ibar^{\tilde\mu_{\tilde\Sigma}})\\
\end{CD}$$
Many instances of this procedure will occur in this paper, here are some of them (in these examples we fix $r$ too):

\subsubsection{A Covariant Functoriality in $\G$}
\label{Geh}
Let $i:\G\rightarrow\tilde\G$ be a homomorphism of $W(\f_{p^r})$-groups, and assume 
\begin{eqnarray*}
&&\Sigma=\tilde\Sigma\\
&&\forall\sigma\in\Sigma:\,^{F^{-\sigma}}i\circ\mu_\sigma=\tilde\mu_\sigma,
\end{eqnarray*}
in the sequel we will refer to this as a morphism from $(\G,\{\mu_\sigma\}_{\sigma\in\Sigma})$ to $(\tilde\G,\{\tilde\mu_\sigma\}_{\sigma\in\Sigma})$. Consider the homomorphism 
$i^\Sigma:\G^\Sigma\rightarrow\tilde\G^\Sigma$ which is given by the cartesian product of the maps $^{F^{-\sigma}}i:{^{F^{-\sigma}}\G}\rightarrow{^{F^{-\sigma}}\tilde\G}$. 
It is easy to see that the homomorphism $^Wi_{\f_{p^f}}^\Sigma=:m$ from $^W\G_{\f_{p^f}}^\Sigma$ to $^W{\tilde\G_{\f_{p^f}}^\Sigma}$ restricts to a homomorphism 
$\gamma$ from $\Ibar^{\mu_\Sigma}$ to $\Ibar^{\tilde\mu_\Sigma}$. The canonical $1$-morphism which results from the pair $(\gamma,m)$ will be denoted by: 
$$\Barb(i):\Barb(\G,\{\mu_\sigma\}_{\sigma\in\Sigma})\rightarrow\Barb(\tilde\G,\{\tilde\mu_\sigma\}_{\sigma\in\Sigma}).$$

\subsubsection{Two Functorialities in $\{\mu_\sigma\}_{\sigma\in\Sigma}$}
\label{Mueh}
Consider a family $\{g_\sigma\}_{\sigma\in\Sigma}=g\in{\G^\Sigma}(W(\f_{p^f}))$, and assume 
\begin{eqnarray*}
&&\Sigma=\tilde\Sigma\\
&&\G=\tilde\G\\
&&\tilde\mu_\Sigma=\Int^{\G^\Sigma}(g/W(\f_{p^f}))\circ\mu_\Sigma. 
\end{eqnarray*}
Let $\gamma:\Ibar^{\mu_\Sigma}\rightarrow\Ibar^{\tilde\mu_\Sigma}$ be the homomorphism which is obtained by restriction of the inner 
automorphism $\Int^{^W\G^\Sigma}(g/\f_{p^f})$. Finally we describe a map $m$ from $^W\G_{\f_{p^f}}^\Sigma$ to itself as follows: 
$$\{U_\sigma\}_{\sigma\in\Sigma}\mapsto\{g_\sigma U_\sigma{^{F^{\varpi_\Sigma^+(\sigma)}}g_{\varpi_\Sigma(\sigma)}^{-1}}\}_{\sigma\in\Sigma}$$ 
The $1$-morphism which results from this pair $(\gamma,m)$ will be denoted by: 
$$\Barb(g):\Barb(\G,\{\mu_\sigma\}_{\sigma\in\Sigma})\rightarrow\Barb(\G,\{\tilde\mu_\sigma\}_{\sigma\in\Sigma})$$
One more notation seems to be practical: Suppose that $(\G,\{\mu_\sigma\}_{\sigma\in\Sigma})$ and $(\G,\{\tilde\mu_\sigma\}_{\sigma\in\Sigma})$ are $\Phbar$-data such 
that $\mu_\sigma=\alpha_\sigma\tilde\mu_\sigma$ holds for some family $\alpha_\sigma:\g_m\rightarrow{^{F^{-\sigma}}\zen_{W(\f_{p^f})}^{\G}}$ of weight $0$. Then we will write
\begin{equation}
\label{Kuh}
\id_{\alpha_\Sigma}:\Barb(\G,\{\mu_\sigma\}_{\sigma\in\Sigma})\stackrel{\cong}{\rightarrow}\Barb(\G,\{\tilde\mu_\sigma\}_{\sigma\in\Sigma})
\end{equation}
for the canonical isomorphism which is obtained by taking $\gamma=\id_{\Ibar^{\mu_\Sigma}}$ and $m=\id_{^W\G_{\f_{p^f}}^\Sigma}$.

\subsection{Weil restriction revisited}
For every $m\in\n_0$ and $\Sigma\subset\z/r\z$ we let $\Sigma^{(m)}\subset\z/mr\z$ be the inverse image of $\Sigma$ under the canonical projection 
$\z/mr\z\rightarrow\z/r\z$. Recall that we fixed a $W(\f_{p^f})$-rational $\Phbar$-datum $(\G,\{\mu_\sigma\}_{\sigma\in\Sigma})$ and let $s$ be a number 
with $r\mid s\mid f$. Observe that the $W(\f_{p^f})$-rational $\Phbar$-datum $(\G_{W(\f_{p^s})},\{\mu_\sigma\}_{\sigma\in\Sigma^{(\frac sr)}})$ is equipped 
with an action of the cyclic group $\Delta:=r\z/s\z$, in the sense of subsubsection \ref{Geh}. The following is a restatement of lemma \ref{galoisI}:

\begin{lem}
\label{galoisII}
For any $W(\f_{p^f})$-scheme $X$, there is a natural equivalence from the category of $(\G,\{\mu_\sigma\}_{\sigma\in\Sigma})$-displays over $X$ to the 
category of $\Delta$-equivariant $(\G_{W(\f_{p^s})},\{\mu_\sigma\}_{\sigma\in\Sigma^{(\frac sr)}})$-displays over $X$, where $\Delta$ acts trivially on $X$. 
\end{lem}

\subsection{Windows and Cartier's Diagonal}
\label{win}

Let us fix a $W(\f_{p^f})$-rational $\Phi$-datum $(\G,\{\mu_\sigma\}_{\sigma\in\Sigma})$. This subsection is a commentary on how to read and apply the results 
of \cite{hadi} and \cite{zink1} in the $(\G,\{\mu_\sigma\}_{\sigma\in\Sigma})$-language, which is used in the main part of this paper. Let $A$ be a torsionfree and 
$p$-adically separated and complete $W(\f_{p^f})$-algebra and let $\tau:A\rightarrow A$ be a lift of the Frobenius endomorphism. For every $p$-adically open pd-ideal 
$J\subset A$ we consider the product $\G^\Sigma(A)\supset\Gamma=\prod_{\sigma\in\Sigma}\Gamma_{\mu_\sigma^{-1}}^{A,J}$, of which the factors are defined as 
in \eqref{realI}. Let $\hat\Phi_A:\Gamma\rightarrow\G^\Sigma(A)$ be the composition of the cyclic permutation \eqref{cycleII} with the product of the homomorphisms:
\begin{equation}
\label{cycleIII}
\Gamma_{\mu_\sigma^{-1}}^{A,J}\rightarrow\G(A_{[F^{r_\Sigma(\sigma)-\sigma}]});\,g\mapsto\tau^{r_\Sigma(\sigma)}(\mu_\sigma(\frac1p)g\mu_\sigma(p))
\end{equation}
This map is well-defined, because 
\begin{itemize}
\item
$\tau(J)\subset pA$, so that $\tau(\Gamma_{\mu_\sigma^{-1}}^{A,J})\subset\Gamma_{\tau(\mu_\sigma^{-1})}^{A,pA}$
\item
$\Gamma_{\tau(\mu_\sigma^{-1})}^{A,pA}\subset\G(A_{[F^{1-\sigma}]})\cap\tau(\mu_\sigma(p))\G(A_{[F^{1-\sigma}]})\tau(\mu_\sigma(\frac1p))$ according to \cite[Corollary 2.22]{hadi}
\end{itemize}
Following the formalism of example \ref{diagram} we define $\hat\CAS_{A,J}(\G,\{\mu_\sigma\}_{\sigma\in\Sigma})$ to be $[\G^\Sigma(A)/_{\hat\Phi_A}\Gamma]$. Let us write 
$\hat\delta_\tau:A\rightarrow W(A)$ for Cartier's diagonal, which is the unique ring homomorphism satisfying $w_n\circ\hat\delta_\tau=\tau^n$. This gadget induces a functor
$$\hat\CAS_{A,J}(\G,\{\mu_\sigma\}_{\sigma\in\Sigma})\stackrel{\hat\delta_\tau}{\rightarrow}\BI_{A/J}(\G,\{\mu_\sigma\}_{\sigma\in\Sigma}),$$
and we let
$$\hat\CAS_{A,J}^{nil}(\G,\{\mu_\sigma\}_{\sigma\in\Sigma})\subset\hat\CAS_{A,J}(\G,\{\mu_\sigma\}_{\sigma\in\Sigma})$$
denote the full subcategory which is obtained as the inverse image of $\BI_{A/J}^{nil}(\G,\{\mu_\sigma\}_{\sigma\in\Sigma})$ 
via the aforementioned functor. In view of \cite{hadi} and the dictionary of remark \ref{merke} we have the following:

\begin{lem}
\label{winI}
If $(A,J,\tau)$ is as above, then Cartier's diagonal induces an equivalence:
$$\hat\CAS_{A,J}^{nil}(\G,\{\mu_\sigma\}_{\sigma\in\Sigma})\stackrel{\cong}{\rightarrow}\BI_{A/J}^{nil}(\G,\{\mu_\sigma\}_{\sigma\in\Sigma})$$
\end{lem}

\section{Connections}
\label{GaussManin}

Consider the field $K(k):=W(k)[\frac1p]$, where $k$ is a perfect field of characteristic $p$. We have to recall and apply some of the key-concepts of \cite{rapoport}: On the set of 
$K(k)$-valued points of a connected reductive algebraic $\q_p$-group $G$ one defines the important equivalence relation of $F$-conjugacy by requiring that $b'\sim b$ if and only if 
$b'=g^{-1}b{^Fg}$, for some $g\in G(K(k))$, we write $B_k(G):=G(K(k))/\sim$ for the set of $F$-conjugacy classes, and $B(G):=B_\fc(G)$. If $k$ is algebraically closed we have $B_k(G)=B(G)$ by 
\cite[Lemma 1.3]{rapoport}, and the same independence result is valid for the Newton points that we introduce next: Write $\bD$ for the pro-algebraic torus whose character group is $\q$ and put
$$\N(G):=(G(K(k))\backslash\Hom(\bD_{K(k)},G_{K(k)}))^{<F>},$$
where the left $G(K(k))$-action is defined by composition with interior automorphisms, and where the action of the infinite cyclic group 
$<F>$ is defined by $\nu\mapsto{{^F}\nu}$. Every element $b\in G(K(k))$ gives rise to an interesting representable $\Spec\q_p$-group
\begin{equation}
\label{thejack}
\Ob_{\alg_{\Spec\q_p}}\ni R\mapsto J_b(R)=\{g\in G(R\otimes_{\q_p}K(k))|(\id_R\otimes F)g=b^{-1}gb\},
\end{equation}
and to the slope homomorphism $\nu_b:\bD_{K(k)}\rightarrow G_{K(k)}$, for  which we refer the reader to loc.cit. Its formation is canonical in the sense that 
$\Int^G(g/K(k))^{-1}\circ\nu_b=\nu_{g^{-1}b{^Fg}}$, and the centralizer of $\nu_b$ is equal to $J_{b,K(k)}$. Consequently, we may introduce the Newton-map:
$$\nbar:B(G)\rightarrow\N(G);\barb\mapsto\nbar(\barb):=\overline{\nu_b}$$
Here notice that the fractionary cocharacters ${^F}\nu_b$ and $\nu_b$ are lying in the same conjugacy class. Let $\Omega$ be the Weyl group of a Levi section $T$ of a $\q_p^{ac}$-rational 
Borel group $B\subset G_{\q_p^{ac}}$. The group $\Omega\rtimes\Gal(\q_p^{ac}/\q_p)$ acts on the associated lattice of cocharacters $X_*(T)$ from the left, and there are natural inclusions:
$$\N(G)\subset(\Omega\backslash X_*(T)_\q)^{\Gal(\q_p^{ac}/\q_p)}\subset\Omega\backslash X_*(T)_\q$$
Moreover, the set $\Omega\backslash X_*(T)_\r$ is partially ordered in a very natural way, as one puts $\Omega x\prec\Omega x'$ if and only if $x$ lies in the convex hull of the 
$\Omega$-orbit of $x'$, cf. \cite[Lemma 2.2(i)]{rapoport}. The same notation is used for the partial orders on $\N(G)$ and $B(G)$ which are immediately inherited via the Newton map.

\subsection{$\mu$-Ordinarity}
\label{twisted}
From now on we suppose that $(\gG,\mu)$ is a $W(\f_{p^f})$-rational display datum with $G=\gG_{\q_p}$. Let $\P$ be a $(\gG,\mu)$-display over an extension $k\supset\f_{p^f}$, so 
that the banal display $\P_{k^{ac}}$ allows a representative $U\in\gG(W(k^{ac}))$, where $k^{ac}$ denotes an algebraic closure of $k$. Let us say that $\P$ is $\mu$-ordinary if and 
only if $\nu_{U{^F\mu(\frac1p)}}$ lies in the conjugacy class $\mub$ of the average  of $\mu$ over the group $\Gal(\q_p^{ac}/\q_p)$, the action of which factors through $\Gal(\fc/\f_p)$. 
Notice that this is equivalent to $\mub\prec\nu_{U{^F\mu(\frac1p)}}$, since Mazur's inequality (cf. \cite[Theorem 4.2(ii)]{rapoport}) implies that $\nu_{U{^F\mu(\frac1p)}}\prec\mub$ 
holds under no assumption on $U\in\gG(W(k^{ac}))$. For the purposes of this paper it is necessary to translate the outcome of \cite{wedhorn} to the fpqc-stack $\B(\gG,\mu)$:

\begin{prop}
\label{Hodgepoint}
For every $(\gG,\mu)$-display $\P$ over some field extension $k\supset\f_{p^f}$ there exists a finite separable field extension $l\supset k$ 
together with a $(\gG,\mu)$-display over $l[[t]]$ whose special fiber agrees with $\P_l$ while its geometric generic fiber is $\mu$-ordinary.
\end{prop} 
\begin{proof}
It does no harm to assume that $T$ is the generic fiber of the normalizer $\gT\subset\gG$ of some maximal 
$\z_p$-rational split torus of $\gG$ and that the Borel group $B$ is $\z_p$-rational. We prove the proposition in several steps:
\begin{sch}
There exists a $\mu$-ordinary $\f_{p^f}$-valued point $\P'$ of $\B(\gG,\mu)$: We may choose a cocharacter 
$\tilde\mu:\g_{m,K(\f_{p^f})}\rightarrow\gT_{K(\f_{p^f})}$ which is conjugated to $\mu^{-1}$ and dominant with respect to 
$B_{K(\f_{p^f})}$. By composing the $1$-morphisms of the subsubsections \ref{Geh} and \ref{Mueh} we obtain:
$$\B(\gT,\tilde\mu^{-1})\rightarrow\B(\gG,\tilde\mu^{-1})\rightarrow\B(\gG,\mu)$$
It is easy to see that this sends any $\f_{p^f}$-valued point of the left hand side to a $\mu$-ordinary object of the right hand side.
\end{sch}
\begin{sch}
For a suitable finite separable extension $l\supset k$, we claim that there exists some $(\gG,\mu)$-display $\tilde\P$ over some open neighbourhood $U\subset\Spec l[t]$ of 
$\{0,1\}$, that specializes to $\P_l$ and $\P'_l$ in the points $0$ and $1$: It is harmless to assume that $\gG$ splits over $W(l)$, so that the big cell of the Bruhat decomposition 
yields an open immersion $v:(\g_m^d\times\a^c)_{W(l)}\hookrightarrow\gG_{W(l)}$, where $d$ is the rank, and $d+c$ the relative dimension of $\gG_{W(l)}$. The map 
\begin{eqnarray*}
&u:(\g_m^{2d}\times\a^{2c})_{W(l)}\rightarrow\gG_{W(l)};\,(x_1,\dots,x_{2d+2c})\mapsto&\\
&v(x_1,\dots,x_d,x_{2d+1},\dots,x_{2d+c})v(x_{d+1},\dots,x_{2d},x_{2d+c+1},\dots,x_{2d+2c})&
\end{eqnarray*}
is smooth and surjective. Now choose representatives $U$ and $U'$ for $\P$ and $\P'$. Over a suitable finite separable field extension, these can be 
written as $U=u(x_1,\dots,x_{2d+2c})$, and $U'=u(x'_1,\dots,x'_{2d+2c})$, where $x_{2d+1},\dots,x_{2d+2c},x'_{2d+1},\dots,x'_{2d+2c}\in W(l)$ and 
$x_1,\dots,x_{2d},x'_1,\dots,x'_{2d}\in W(l)^\times$. Consider $z_i:=x_i+[t](x'_i-x_i)\in W(l[t])$ and observe that the $0$th ghost coordinate of $\prod_{i=1}^{2d}z_i$ 
is equal to the $\pmod p$ reduction $\hbar\in l[t]$ of the polynomial $\prod_{i=1}^{2d}(x_i+t(x'_i-x_i))$, so that $z_1,\dots,z_{2d}\in W(l[t]_{\hbar})^\times$ while 
$z_{2d+1},\dots,z_{2d+2c}\in W(l[t]_{\hbar})$. Let $\tilde\P$ be the banal  $(\gG,\mu)$-display over $\Spec l[t]_{\hbar}$ which is represented by the $\gG(W(l[t]_{\hbar}))$-element 
$u(z_1,\dots,z_{2d+2c})$.
\end{sch}
\begin{sch}
By Grothendieck's specialization theorem (cf. \cite[Theorem 3.6(ii)]{rapoport}) the generic fiber of $\tilde\P$ is $\mu$-ordinary as well.
\end{sch}
\end{proof}

\subsection{Automorphisms} 
\label{karre}
In this subsection we study a characteristic $0$ analog of the group \eqref{thejack}. Let $N$ be a discretely valued complete field 
containing $K(\f_{p^f})$, and let us assume that the residue field $l=\O_N/\gm_N$ is algebraically closed. By means of Witt's diagonal 
$\Delta:W(l)\rightarrow W(W(l))$ one can regard $W(\O_N)$ as an augmented $W(l)$-algebra, we have to start with the following observation:

\begin{lem}
\label{obstIII}
Consider some $\BI_l^{nil}(\gG,\mu)$-object represented by $O\in\gG(W(l))$. Let $U\in\gG(W(\O_N))$ be a lift of $O$ 
and consider the $F$-centralizers of the elements $O{^F\mu}(\frac1p)$ and $U{^F\mu}(\frac1p)$, i.e. the groups:
\begin{eqnarray*}
&&\{k\in\gG(W(l)[\frac1p])\mid O{^F\mu}(\frac1p)=k^{-1}O{^F(\mu(\frac1p)k)}\}=:\Aut^0(O)\\
&&\{\hat k\in\gG(W(\O_N)[\frac1p])\mid U{^F\mu}(\frac1p)=\hat k^{-1}U{^F(\mu(\frac1p)\hat k)}\}=:\hat\Aut^0(U)
\end{eqnarray*}
For every pair $(n_1,n_2)\in\n_0\times\z$, there exists an element $m$ of $\gG(N)$ such that the diagram 
$$\xymatrix{{\hat\Aut^0(U)} \ar[r]^{\mod\gm_N} \ar[d] &
{\Aut^0(O)} \ar[r] &
{\gG(K(l))} \ar[d]^{\int^{\gG(N)}(m)\circ F^{n_2}} \\
{\gG(W(\O_N)[\frac1p])} \ar[rr]^{w_{n_1}} && {\gG(N)}}$$
commutes.
\end{lem}
\begin{proof}
The content of the lemma does not depend on the particular pair $(n_1,n_2)$, due to $w_{n_1}(U{^F\mu}(\frac1p))=w_{n_1}(\hat k)^{-1}w_{n_1}(U{^F\mu}(\frac1p))w_{n_1+1}(\hat k)$ 
and a similar formula relating any $k\in\Aut^0(O)$ to its image under $F$. Furthermore if $U$ is replaced by $^{F^f}U$ the content of the corollary does not change either, because we have 
$^{F^f}\mu=\mu$. We conclude that we may assume $\Delta(O)\equiv U\mod p\O_N$, because their images under sufficiently large iterates of the Frobenius would satisfy this. According 
to \cite[Theorem 3.5.4]{pappas} there exists a unique $h\in\gG(W(p\O_N))$ with $\Delta(O{^F\mu}(\frac1p))=h^{-1}U{^F(\mu(\frac1p)h)}$. Consider some $\hat k\in\hat\Aut^0(U)$ 
with $\mod\gm_N$-reduction $k\in\Aut^0(O)$. According to the rigidity of quasi-isogenies we obtain $h\Delta(k)h^{-1}=\hat k$ from $h\Delta(k)h^{-1}\equiv\hat k\mod\gm_N$ 
(cf. \cite[Proposition 4.2.3(ii)]{hadi}). This implies that $^{F^n}k$ and $w_n(\hat k)$ are conjugated as homomorphisms on $\hat\Aut^0(U)$ and the lemma is proven.
\end{proof}

\begin{lem} 
\label{obstX}
Consider some $\BI_l^{nil}(\gG,\mu)$-object represented by $O\in\gG(W(l))$. Let $U\in\gG(W(\O_N))$ be a lift of $O$ and consider the group:
\begin{eqnarray*}
&&\Aut^0(U):=\hat\Aut^0(U)\cap U_{\mu^{-1}}^0(W(\O_N)[\frac1p])=\\
&&\{\hat k\in U_{\mu^{-1}}^0(W(\O_N)[\frac1p])\mid U{^F\mu}(\frac1p)=\hat k^{-1}U{^F(\mu(\frac1p)\hat k)}\}
\end{eqnarray*}
Then there exist
\begin{itemize}
\item
a $\q_p$-algebraic group $\hat\Gamma_U$ in which $\hat\Gamma_U(\q_p)$ is Zariski-dense.
\item
a $N$-rational inclusion $j_U:\hat\Gamma_{U,N}\hookrightarrow U_{\mu^{-1},N}^0$
\end{itemize}
such that $j_U$ maps $\hat\Gamma_U(\q_p)$ onto $w_0(\Aut^0(U))$. Moreover $j_U^{-1}(w_0(\Aut(U))$ is a compact 
open subgroup of $\hat\Gamma_U(\q_p)$, and neither $(\hat\Gamma_U,j_U)$ nor $\Aut(U)$ change upon enlarging $N$.
\end{lem}
\begin{proof}
The condition on the Zariski density is imposed only to achieve the uniqueness of $\hat\Gamma_U$, and can be ignored for proving 
its existence. Lemma \ref{obstIII} and the commutative diagram
$$\begin{CD}
\Aut^0(U)@>>>\hat\Aut^0(U)@>{\mod\gm_N}>>\Aut^0(O)\\
@V{w_0}VV@V{w_0}VV@VVV\\
U_{\mu^{-1}}^0(N)@>>>{\gG(N)}@<{\int^{\gG(N)}(m)}<<{\gG(K(l))}
\end{CD}$$
confirm that $j_U$ can be taken to be the restriction of $\Int^\gG(m/N)$ to $\hat\Gamma_U$ being the largest $\q_p$-rational subgroup of $J_{O{^F\mu(\frac1p)}}$ 
(cf. \eqref{thejack}) of which the base change to $N$ is contained in the $N$-rational subgroup $U_{\tilde\mu^{-1}}^0$, where $\tilde\mu=\Int^\gG(m^{-1}/N)\circ\mu$. 
The openness of $w_0(\Aut(U))$ in the image of $j_U$ (i.e. $w_0(\Aut^0(U))$) is a little bit subtle: Let us write $\Gamba\subset\gG(W(\O_N))$ for the subgroup 
of elements of which the $\pmod p$-reduction is contained in $\I^\mu(\O_N/p\O_N)$. According to the deformation theory \cite[Theorem 3.5.4]{pappas}, the group 
of automorphisms of the $\pmod p$-reduction of $U$ agrees with $\hat\Aut^0(U)\cap\Gamba$, in fact the latter coincides with the automorphisms of $U$ when 
regarded as a $(\gG,\mu)$-triple for the pd-ideal $p\O_N\subset\O_N$ (in the sense of \cite[Definition 4.26]{hadi}). By the same token, we have a cartesian diagram:
$$\begin{CD}
\Aut(U)@>>>\hat\Aut^0(U)\cap\Gamba\\
@V{w_0}VV@V{w_0}VV\\
U_{\mu^{-1}}^0(\O_N)@>>>{\gG(\O_N)}
\end{CD}$$
From the definition of $j_U$ it follows, that it is enough to check the $p$-adic openness of the natural image of $\hat\Aut^0(U)\cap\Gamba$ in 
$\Aut(O)$, which however follows from (possibly several applications of) the auxiliary lemma \ref{geeI} below. The independence of $N$ is clear.
\end{proof}

\begin{lem}
\label{geeI}
Let $\P$ be a $(\gG,\mu)$-display over a $W(\f_{p^f})$-algebra $R$ containing ideals $J\supset I$ satisfying $JI=pI=0_R$. Then, the 
$p$th power of any element in the kernel of $\Aut(\P_{R/I})\rightarrow\Aut(q_0(\P_{R/J}))$ allows a lift to an automorphism of $\P$.
\end{lem}
\begin{proof}
Let $k$ be an automorphism of $\P_{R/I}$ such that $q_0(k)$ is neutral $\mod J$. Observe that the map $\Aut(q(\P))\rightarrow\Aut(q(\P_{R/I}))$ is surjective, so let 
$\tilde k$ be a preimage of $q(k)$ under this map. We claim that $\tilde k^p$ is contained in the subgroup $\Aut(\P)\hookrightarrow\Aut(q(\P))$. The question being local, 
we may assume that $\P$ is banal, and a choice of global section of $q(\P)$ gives rise to a representative $U\in\gG(W(R))$ of $\P$, moreover the element $\tilde k$ satisfies 
$\tilde k^{-1}U\Phi^\mu(\tilde k)=NU$ for some $N\in\gG(W(I))$. Observe that Zink's description of $W(I)$ in terms of logarithmic ghost coordinates yields the mutual annihilation 
of the ideals $W(J)+I(R)$ and $W(I)$. It follows that the elements $\tilde k$ and $N$ commute with each other, so that we may infer $\tilde k^{-p}U\Phi^\mu(\tilde k^p)=N^pU=U$.
\end{proof}

\subsection{Deformations of adjoint nilpotent $(\gG,\mu)$-displays}
\label{meckerei}

Let $X$ be a $W(\f_{p^f})$-scheme such that $p$ acts Zariski-locally nilpotent on $\O_X$ and fix an adjoint nilpotent $(\gG,\mu)$-display $\Q$ over some closed subscheme
$X_0\hookrightarrow X$ the square of whose ideal sheaf vanishes. By slight abuse of notation we write $\O_{X_0}$ (resp. $\O_X$) for the fpqc sheaf on $X$ which attains 
the value $R_0$ (resp. $R$) on some $R\in\Ob_{\alg_X}$ where $\Spec R_0\hookrightarrow\Spec R$ is the closed immersion arising from $X_0\hookrightarrow X$ 
by change of base. We also write $\N$ for the kernel of the natural surjection from $\O_X$ to $\O_{X_0}$, notice that $\N$ is a sheaf of $\O_{X_0}$-modules. 
By a lift of $\Q$ to $R\in\Ob_{\alg_X}$ we mean a pair $(\P,\delta)$ where $\P$ is a $(\gG,\mu)$-display over $R$, and $\delta$ is an isomorphism from $\Q_{R_0}$ to 
$\P_{R_0}$. Isomorphisms between lifts are expected to preserve the respective $\delta$'s. In particular no lift has any automorphisms other than the identity, thanks to 
\cite[Corollary 3.5.6]{pappas}. Let $\D_{\Q,X}(R)$ be the set of isomorphism classes of lifts over $R$. In the same vein one proves that $\D_{\Q,X}$ is a fpqc-sheaf on $X$. 
Let $D_{\Q,X}=\Gamma(X,\D_{\Q,X})$ denote the set of isomorphism classes of global lifts. We also record the following (cf. \cite[Theorem 3.5.11, Corollary 3.5.12]{pappas}):  
 
\begin{thm}
\label{obstII}
Let $\Q$ and $X_0\hookrightarrow X$ be as in the beginning of this subsection.
\begin{itemize}
\item[(i)]
$\D_{\Q,X}(R)$ is non-empty for every $R\in\alg_X$
\item[(ii)]
$\D_{\Q,X}$ possesses the structure of a locally trivial principal homogeneous space for the fpqc-sheaf $T_\Q\otimes_{\O_{X_0}}\N=\Homo_{\O_{X_0}}(\check T_{\Q},\N)$, 
in particular $D_{\Q,X}$ is either empty or it is a principal homogeneous space for the group $\Gamma(X,T_\Q\otimes_{\O_{X_0}}\N)=\Hom_{\O_{X_0}}(\check T_{\Q},\N)$.
\end{itemize}
\end{thm}

We wish to introduce another definition: Fix some $1$-morphism 
$$X\stackrel{\P}{\rightarrow}\B(\gG,\mu),$$
and let $X'$ be the spectrum of the $\O_X$-algebra $\O_X\oplus\Omega_{X/W(\f_{p^f})}^1$. The inclusion $\O_X\subset\O_X\oplus\Omega_{X/W(\f_{p^f})}^1$ 
(resp. the map $x\mapsto x+d_{X/W(\f_{p^f})}(x)$) gives rise to a projection $\peebar_1$ (resp. $\peebar_2$) from $X'$ to $X$: Consequently there exists 
a unique $k_\P\in\Hom_{\O_X}(\check T_\P,\Omega_{X/W(\f_{p^f})}^1)$ which measures the difference between $\P\times_{X,\peebar_2}X'$ and 
$\P\times_{X,\peebar_1}X'$ when regarded as global lifts in $D_{\P,X'}$. We will call $k_\P$ the Kodaira-Spencer map of $\P$.

\begin{lem}
\label{smoothIII}
Fix a positive integer $\nu\in\n$ and let $\P$ be an adjoint nilpotent $(\gG,\mu)$-display over some $W_\nu(\f_{p^f})$-scheme $X$.
\begin{itemize}
\item[(i)]
If the classifying morphism of $\P$ renders $X$ relatively formally smooth over $\B(\G,\mu)_{W_\nu(\f_{p^f})}$, then $X$ is a formally smooth $W_\nu(\f_{p^f})$-scheme.
\item[(ii)]
The classifying morphism of $\P$ renders $X$ relatively formally \'etale over $\B(\gG,\mu)_{W_\nu(\f_{p^f})}$, 
if and only if $k_\P$ is a bijection and $X$ is formally smooth over $W_\nu(\f_{p^f})$.
\end{itemize}
\end{lem}
\begin{proof}
It is known that the diagonal of $\B(\gG,\mu)$ is affine (cf. \cite[Lemma 3.2.9b)]{pappas}), so that the $1$-morphism $X\stackrel{\P}{\rightarrow}\B(\gG,\mu)_{W_\nu(\f_{p^f})}$ 
is formally smooth if and only if it satisfies the infinitesimal lifting criterion, so let us begin with the proof of (i) and thus consider a commutative diagram
$$\begin{CD}
X@<<<\Spec R/\ga\\
@VVV@VVV\\
\Spec W_\nu(\f_{p^f})@<<<\Spec R
\end{CD},$$   
where $\ga\subset R$ is an ideal of vanishing square. According to part (i) of theorem \ref{obstII} we may choose 
a $(\gG,\mu)$-display $\P'$ over $R$, such that $\P'_{R/\ga}\cong\P_{R/\ga}$. We obtain a $2$-commutative diagram
$$\begin{CD}
X@<<<\Spec R/\ga\\
@V{\P}VV@VVV\\
\B(\G,\mu)_{W_\nu(\f_{p^f})}@<{\P'}<<\Spec R
\end{CD},$$
and the formal smoothness of $\P$ yields the requested lift $\Spec R\rightarrow X$. In order to prove (ii) we may assume that $X$ is formally smooth. 
Observe that the property of $\P$ being formally \'etale is local for the Zariski topology of $X$, as is the bijectivity of $k_\P$. It follows that we may 
even assume $X$ to be the spectrum of some formally smooth $W_\nu(\f_{p^f})$-algebra $A$. So let us consider a $2$-commutative diagram:
$$\begin{CD}
\Spec A@<<<\Spec R/\ga\\
@V{\P}VV@VVV\\
\B(\gG,\mu)_{W_\nu(\f_{p^f})}@<{\P'}<<\Spec R
\end{CD}$$
Choose a lift $\alpha_0:A\rightarrow R$ of the given $A$-structure on $R/\ga$, and observe that all other lifts are of the form $\alpha_t(x)=\alpha_0(x)+t\circ d_{A/W_\nu(\f_{p^f})}x$, 
where $t$ runs through $\Hom_A(\Omega_{A/W_\nu(\f_{p^f})}^1,\ga)$. Pulling-back $\P$ along the homomorphism $\alpha_t$ gives rise to a $(\gG,\mu)$-display $\P_t$ 
over $R$, moreover each of these is a lift of $\P'_{R/\ga}$ and it is easy to see that the difference between $\P_t$ and $\P_0$ is given by the canonical map:
\begin{equation}
\label{smoothV}
\Hom_A(\Omega_{A/W_\nu(\f_{p^f})}^1,\ga)
\rightarrow\Hom_A(\check T_\P,\ga);t\mapsto t\circ k_\P
\end{equation}
We deduce that $\P$ renders $\Spec A$ formally \'etale over $\B(\gG,\mu)_{W_\nu(\f_{p^f})}$ if and only if \eqref{smoothV} 
is an isomorphism for every $0=\ga^2\subset\ga\subset R$ over $A$, i.e. if and only if $k_\P$ is an isomorphism.
\end{proof}

\begin{cor}
\label{smoothIV}
Let $X$ be a $\f_{p^f}$-scheme whose reduced induced subscheme $X_{red}$ is of finite type over $\f_{p^f}$ and let $\P$ be an adjoint nilpotent $(\gG,\mu)$-display over 
$X$ whose classifying morphism renders $X$ formally \'etale over $\B(\gG,\mu)_{\f_{p^f}}$. Then $X_{red}$ is smooth over $\f_{p^f}$ and $k_{\P_{X_{red}}}$ is a bijection.
\end{cor}
\begin{proof}
Without loss of generality, we may assume that $X$ is affine, and we want to consider some maximal ideal $\gm$ of its coordinate ring $A:=\Gamma(X,\O_X)$. Notice that the 
proartinian completion $\hat A_\gm$ of the localization $A_\gm$ is isomorphic to an algebra of powerseries over the finite field $A/\gm$, since it represents the formal deformation 
functor of $\P_{A/\gm}$. We infer that $\hat A_\gm$ also agrees with the $\gm$-adic completion of $A_{\gm}/\sqrt{0_{A_\gm}}$, given that the former is already reduced (this means 
nothing but $\sqrt{0_{A_\gm}}\subset\bigcap_{\nu=1}^\infty\gm^\nu$). The smoothness of $X_{red}$ follows easily, as all of its completed local rings are regular while $\f_{p^f}$ 
is perfect. One also obtains an isomorphism between the $\gm$-adic completions of $\Omega_{A/\f_{p^f}}^1$ and $\Omega_{A_{red}/\f_{p^f}}^1$, observe the factorizations: 
$$\Omega_{A/\gm^{\nu+1}/\f_{p^f}}^1\rightarrow A/\gm^\nu\otimes_A\Omega_{A/\f_{p^f}}^1\rightarrow\Omega_{A/\gm^\nu/\f_{p^f}}^1$$
At last, the requested bijectivity of $k_{\P_{A_{red}}}$ may be deduced from $\Omega_{A_{red}/\f_{p^f}}^1$ being isomorphic to $A_{red}\otimes_A\Omega_{A/\f_{p^f}}^1$, 
which in turn can be deduced from considerations over completed local rings.
\end{proof}

\begin{cor}
\label{lift}
Let $\Parb$ be an adjoint nilpotent $(\gG,\mu)$-display over a smooth $\f_{p^f}$-variety $\Xbar$ of which the classifying morphism renders $\Xbar$ formally \'etale over $\B(\gG,\mu)_{\f_{p^f}}$. 
\begin{itemize}
\item[(i)]
For every $\nu\in\n$ there exists a $(\gG,\mu)$-display $\P$ over a smooth $W_\nu(\f_{p^f})$-scheme $\X$ together with isomorphisms 
\begin{eqnarray*}
&&\Xbar\cong\X\times_{W_\nu(\f_{p^f})}\f_{p^f}\\
&&\Parb\cong\P\times_{\X^{(\nu)}}\Xbar,
\end{eqnarray*}
moreover the quadruple consisting of $\X$, $\P$ together with the two isomorphisms above is unique up to a unique isomorphism. 
\item[(ii)]
In addition, suppose that $\Xbar$ is proper and that there exists a character $\chi:\I_0^\mu\rightarrow\g_{m,W(\f_{p^f})}$ such that 
$\omega_{q_0(\Parb)}(\chi)$ is an ample line bundle on $\Xbar$. Then $\Xbar$ allows a smooth and proper lift $\X\rightarrow\Spec W(\f_{p^f})$ such 
that $\Parb$ allows lifts to $(\gG,\mu)$-displays $\{\P^{(\nu)}\}_{\nu\in\n}$ over $\X\times_{W(\f_{p^f})}W_\nu(\f_{p^f})=\X^{(\nu)}$ for each $\nu\in\n$.
\end{itemize}
\end{cor}
\begin{proof}
We choose an affine open covering $\Xbar=\bigcup_l\Ubar_l$ in order to begin with the uniqueness assertion of (i): Let us fix $\nu$ and let us 
consider $(\gG,\mu)$-displays $\P$ and $\P'$ that lift $\Parb$ to certain smooth $W_\nu(\f_{p^f})$-schemes $\X$ and $\X'$ whose special fibers 
are $\Xbar$. Let $U_l=\Spec A_l$ be the affine open subscheme of $\X$ whose underlying point set agrees with $\Ubar_l$ and consider the diagram:
$$\begin{CD}
\Spec A_l/pA_l@>>>\X'\\
@VVV@VVV\\
\Spec A_l@>>>\B(\gG,\mu)_{W_\nu(\f_{p^f})}
\end{CD}$$
The formal \'etaleness of the vertical arrow on the right yields a unique morphism $\phi_l:\Spec A_l\rightarrow\X'$ lifting 
the open immersion $\Ubar_l\subset\Xbar$. In order to construct the requested isomorphism $\phi:\X\rightarrow\X'$ it remains to check 
$\phi_l\vert_{U_l\cap U_m}=\phi_m\vert_{U_l\cap U_m}$, which can be accomplished upon passage to an affine open covering of $U_l\cap U_m$ together 
with the aforementioned uniqueness of the morphisms $\phi_l$. Notice, that the invertibility of $\phi$ follows from swapping the roles of $\X$ and $\X'$.\\
The existence of $\X$ is evident in the affine case, while the existence of $\P$ follows from repeated applications of part (i) of theorem \ref{obstII}. 
In the general non-affine case the existence of $\X$ and $\P$ follows from an obvious patching, which uses the previously established uniqueness 
of $\X$ and $\P$. The part (ii) of our corollary follows immediately from part (i) together with \cite[Th\'eor\`eme (5.4.5)]{egaiii}, observe that the 
lift of the line bundle $\omega_{q_0(\Parb)}(\chi)$ is automatic, since we can work with the lowest truncations of the $(\gG,\mu)$-displays $\P^{(\nu)}$.
\end{proof}

\subsection{Crystals}

Recall that a triple $(R,\ga,\gamma)$ is called a pd-thickening (of $R/\ga$) if $\gamma$ is a divided-power structure on an ideal $\ga$ in a ring $R$, 
in which $p$ is nilpotent. Fix a $W(\f_{p^f})$-scheme $X$. Our concept of the big crystalline site is based on the stacks project, i.e. Definition 59.8.1 of 
\href{https://stacks.math.columbia.edu/tag/07IF}{Tag 07IF} (please consult \cite[1.1]{berthelot} for a slightly different setup). We define $(X/W(\f_{p^f}))_{cris}$ 
to be the opposite of the obvious category structure on the class of pd-thickenings $(R,\ga,\gamma)$ that are equipped with two additional data, namely
\begin{itemize}
\item
a morphism from $W(\f_{p^f})$ to $R$ and
\item
a morphism from $\Spec R/\ga$ to $X$, such that the diagram
$$\begin{CD}
\Spec R/\ga@>>>\Spec R\\
@VVV@VVV\\
X@>>>\Spec W(\f_{p^f})
\end{CD}$$
commutes.
\end{itemize}
We choose to endow $(X/W(\f_{p^f}))_{cris}$ with the crystalline fpqc-topology, which is induced from the big fpqc-site 
$(\Spec W(\f_{p^f}))_{fpqc}$, by means of the important functor (cf. \cite[Lemme 1.1.2]{berthelot}): 
\begin{equation}
\label{fpqcris}
(X/W(\f_{p^f}))_{cris}\rightarrow(\Spec W(\f_{p^f}))_{fpqc};(R,\ga,\gamma)\mapsto R
\end{equation}
This endows $(X/W(\f_{p^f}))_{cris}$ with a structure sheaf, thus turning it into a locally ringed site, but it also has the consequence 
that the composition of any sheaf $\F$ on the big fpqc site of $\Spec W(\f_{p^f})$ with the said functor gives automatically a sheaf on 
$(\Spec W(\f_{p^f})/W(\f_{p^f}))_{cris}$, and we will denote it by: $\F_{cris}$. Recall the concept of a crystal $\h$ over $(X/W(\f_{p^f}))_{cris}$ taking 
values in some fibration, say $\CI\stackrel{G}{\rightarrow}(\Spec W(\f_{p^f}))_{fpqc}$, being a functor from $(X/W(\f_{p^f}))_{cris}$ to $\CI$ such that
\begin{itemize}
\item[(i)]
$H:=G\circ\h$ is isomorphic to \eqref{fpqcris}, and
\item[(ii)]
the image under $\h$ of any $(X/W(\f_{p^f}))_{cris}$-morphism $w$ from $V$ to $U$ 
induces an isomorpism $\h(V)\stackrel{\cong}{\rightarrow}\h(U)\times_{H(U),H(w)}H(V)$.
\end{itemize}
The last condition is essential and means that $\h$ is a $(X/W(\f_{p^f}))_{cris}$-valued cartesian section (cf. \cite[Chapter III, Definition(3.6)]{messing}). Finally note that a crystal over 
$(X/W(\f_{p^f}))_{cris}$ in $\ton(\gG)$ is just a locally trivial principal homogeneous space for $\gG_{W(\f_{p^f}),cris}$ over the big crystalline site $(X/W(\f_{p^f}))_{cris}$. Sheaves on 
$(X/W(\f_{p^f}))_{cris}$ as well as crystals in $\CI$ thereon, can be pulled back to $(X'/W(\f_{p^f}))_{cris}$, where $X'$ is a $X$-scheme. In a somewhat different direction every sheaf $\E$ on 
$(X/W(\f_{p^f}))_{cris}$ has an inverse image $\F=i_{X/W(\f_{p^f})}^*(\E)$ on $X_{fpqc}$, defined by $\F(R):=\E(R,0,0)$ (see to \cite[1.1.4.a)]{berthelot} for this and \cite[1.1.4.b)]{berthelot} for its 
right-adjoint). The results of \cite{hadi} can be used to attach to an adjoint nilpotent $(\gG,\mu)$-display $\P$ over a $W(\f_{p^f})$-scheme $X$ its Witt-crystal $\k_\P$ over $X$, which is a crystal in 
$\ton({^W\gG})$: Simply lift $\P_{R/\ga}$ to a $(\gG,\mu)$-display $\tilde\P$ over $R$, and put $\k_\P(R,\ga,\gamma):=K(\tilde\P)$. Following \href{https://stacks.math.columbia.edu/tag/07J5}{Tag 07J5} 
we note that the locally trivial principal homogeneous $^W\gG_{W(\f_{p^f})}$-spaces $\k_\P(R,\ga,\gamma)$ are equipped with canonical integrable connections, which we sketch in the banal case: 
Let us write $O'\in\gG(W(R\oplus\Omega_R^1))$ for the image under $\peebar_2:R\rightarrow R\oplus\Omega_R^1$ of some $O\in\gG(W(R))$ satisfying the adjoint nilpotence condition. According 
to \cite[Proposition 4.6]{hadi} one can find an element $x_O\in\gG(W(\Omega_R^1))$ with 
$$x_O^{-1}O\Psi^\mu(x_O)=O',$$ 
where the map $\Psi^\mu$ is defined as the composition of $\Phi^\mu$ with the projector $W(\Omega_R^1)\rightarrow I(\Omega_R^1)$ which results from the canonical decomposition $W(\Omega_R^1)=\Omega_R^1\oplus I(\Omega_R^1)$. By the same token one sees that $x_O$ depends only on the $\mod\ga$-reduction of $O$ and thus describes the canonical integrable 
connection on $\k_\P(R,\ga,\gamma)$, where $\P$ is the banal $(\gG,\mu)$-display which is represented by the image of $O$ in $\gG(W(R/\ga))$. In the sequel we will call $x_O$ the 
Witt-connection of (the $\mod\ga$-reduction of) $O$. When thinking of $\gG(W(\Omega_R^1))$ as the additive group $W(\Omega_R^1)\otimes_{\z_p}\gg$, the equation above metamorphoses into: 
\begin{equation}
\label{synthII}
\Ad^\gG(O)\Psi^\mu(x_O)-x_O=O'O^{-1}
\end{equation} 
Similarily, if $U\in\gG(W(R))$ is an alternative representative, so that $O=k^{-1}U\Phi^\mu(k)$ holds for some element $k\in\I^\mu(R)$, then $\Ad^\gG(k)(x_O)$ is the sum of 
$x_U$ and $k'k^{-1}$ in $\gG(W(\Omega_R^1))$. Our next result presents an alternative approach to connections, which is more classical: Consider a Frobenius lift $\tau$ on a 
torsionfree, $\ga$-adically separated and complete $W(\f_{p^f})$-algebra $A$, where $\ga$ is an ideal containing some power of $p$. We consider the module of formal $m$-forms
$$\hat\Omega_A^m:=\lim_{\leftarrow}\bigwedge_{A/\ga^n}^m\Omega_{A/\ga^n}^1,$$ 
which is isomorphic to the $\ga$-adic completion of $\bigwedge_A^m\Omega_A^1$. Let us write $\tau_1$ (resp. $\tau_m$) for the unique $\ga$-adically 
continuous $\tau$-linear endomorphism on $\hat\Omega_A^1$ (resp. on $\hat\Omega_A^m$) given by $\tau_1(d_Ax):=d_A\frac{\tau(x)-x^p}{p}+x^{p-1}d_Ax$ 
(resp. $\tau_m(d_Ax_1\wedge\dots\wedge d_Ax_m):=\tau_1(d_Ax_1)\wedge\dots\wedge\tau_1(d_Ax_m)$). The Lie-algebra of $\gG$ is equipped with a 
decomposition $\bigoplus_{l\in\z}\gg_l=W(\f_{p^f})\otimes_{\z_p}\gg$, according to the weights of $\mu:\g_{m,W(\f_{p^f})}\rightarrow\gG_{W(\f_{p^f})}$. On the space 
$$\bigoplus_{m\geq1\geq l}\hat\Omega_A^m\otimes_{W(\f_{p^f})}\gg_l=\bigoplus_{m\geq1}\hat\Omega_A^m\otimes_{\z_p}\gg$$
of $\gg$-valued formal differential forms of positive degree we consider yet another endomorphism given by 
$\psi_A^\mu(\omega\otimes x)=p^{m-l}\tau_m(\omega)\otimes\tau(x)$ for all $\omega\in\hat\Omega_A^m$ and $x\in\gg_l$. In the 
following result $\eta^\gG$ denotes the canonical right-invariant $\gg$-valued Cartan-Maurer $1$-form on $\gG/\z_p$. Notice that its pull-back 
$$\eta^\gG\circ U\in\hat\Omega_A^1\otimes_{\z_p}\gg\cong\gG(\hat\Omega_A^1),$$
along some $A$-valued point $U$ of $\gG$ describes the quotient $U'U^{-1}$ as an element of the above additive group, where 
$U'\in\gG(A\oplus\hat\Omega_A^1)$ stands for the image of $U\in\gG(A)$ under the map $A\rightarrow A\oplus\hat\Omega_A^1;x\mapsto x+d_Ax$.%and then we have:

\begin{lem}
\label{synthVII}
If $U\in\gG(A)$ satisfies the $\mod\ga$-adjoint nilpotence condition, then there exists a unique element $D_U\in\hat\Omega_A^1\otimes_{\z_p}\gg$ with:
\begin{equation}
\label{synthI}
\Ad^\gG(U)(\psi_A^\mu(D_U))-D_U=\eta^\gG\circ U
\end{equation}
Moreover, let $n$ be a positive integer and let $O$ be the image of $\hat\delta_\tau(U)\in\gG(W(A))$ in $\gG(W(A/\ga^n))$, 
where $\hat\delta_\tau$ is as in subsection \ref{win}. Then $O$ represents an adjoint nilpotent banal $(\gG,\mu)$-display over $\Spec A/\ga^n$ 
whose Witt-connection $x_O$ is described by the element of $\gG(W(\Omega_{A/\ga^n}^1))=(\prod_{k=0}^\infty\Omega_{A/\ga^n,[w_k]}^1)\otimes_{\z_p}\gg$, 
of which the $k$th component is equal to the image of $\tau_1^k(D_U)$ in $\Omega_{A/\ga^n}^1\otimes_{\z_p}\gg$.
\end{lem}
\begin{proof}
The first assertion is analogous to \cite[Lemma 2.8]{habil}. The adjoint nilpotence condition implies that $D\mapsto\Ad^\gG(U)(\psi_A^\mu(D))-\eta^\gG\circ U$ 
is a contractive map for the $\ga$-adic topology, so there exists a unique fixedpoint $D_U$. Towards the second assertion we define a homomorphism
$$\hat\delta:A\oplus\hat\Omega_A^1\rightarrow W(A\oplus\hat\Omega_A^1)\cong W(A)\oplus\prod_{k=0}^\infty{\hat\Omega_{A,[w_k]}}^1$$ 
by letting $\hat\delta\vert_A$ be the Cartier homomorphism $\hat\delta_\tau$ and decreeing $\hat\delta(\omega):=(\omega,\tau_1(\omega),\tau_1^2(\omega),\dots)$ 
for any $\omega\in\hat\Omega_A^1$. The diagram
$$\begin{CD}
A\oplus\hat\Omega_A^1@>{\hat\delta}>>W(A\oplus\hat\Omega_A^1)\\
@A{\peebar_2}AA@A{W(\peebar_2)}AA\\
A@>{\hat\delta_\tau}>>W(A)
\end{CD}$$
is commutative and it follows that \eqref{synthI} implies \eqref{synthII}.
\end{proof}

By slight abuse of terminology we will call the said element $D_U$ of the previous lemma the Dieudonn\'e connection of $U$. 

\subsection{Monodromy}
\label{thedetailI}
We want to specialize to $A:=W(k_0)[[t_1,\dots,t_d]]$ and $\ga:=pA+\sum_{i=1}^dt_iA$, where the field $k_0$ is an algebraically 
closed or an algebraic extension of $\f_{p^f}$ and $d:=\rk_{W(\f_{p^f})}\gg_1$. Let us fix the Frobenius lift determined by $\tau(t_i):=t_i^p$. 
As in \cite{katz1} we need to invoke an important subalgebra of the power series algebra in $d$ indeterminates over $K(k_0)$:
\begin{equation*}
K(k_0)\{\{t_1,\dots,t_d\}\}:=\{\sum_{\underline n}a_{\underline n}\underline t^{\underline n}|\,|a_{\underline n}|_pC^{n_1+\dots+n_d}\to0\,\forall C<1\}
\end{equation*}
One has $A[\frac1p]\subset K(k_0)\{\{t_1,\dots,t_d\}\}\subset K(k_0)[[t_1,\dots,t_d]]$, and it is very straightforward to extend the endomorphism $\tau$ to each of these 
$K(k_0)$-algebras. We write $\Aut(A/W(k_0))$ for the $W(k_0)$-linear automorphisms of $A$. This group acts naturally on $K(k_0)\{\{t_1,\dots,t_d\}\}$ and on 
$\bigoplus_{i=1}^dK(k_0)\{\{t_1,\dots,t_d\}\}dt_i$ from the left (in fact it is well-known that the latter objects allow more canonical descriptions without explicit variables, 
e.g. as spaces of rigid functions or Kaehler differentials on the generic fiber of $\Spf A$). From this $\Aut(A/W(k_0))$-action one derives an important semi-direct product
$$\t:=\gG(K(k_0)\{\{t_1,\dots,t_d\}\})\rtimes\Aut(A/W(k_0)).$$ 
Any display $\P_0$ with $(\gG,\mu)$-structure over $\Spec k_0$ is automatically banal, and hence represented by some $U_0\in\gG(W(k_0))$. Choose a basis 
$\{\epsilon_1,\dots,\epsilon_d\}$ of $\gg_1$. Notice that the element $U_1:=\exp(\sum_{i=1}^dt_i\epsilon_i)U_0$ satisfies the $\mod\ga$-adjoint nilpotence 
condition, so that it possesses a well-defined Dieudonn\'e connection $D_1\in\hat\Omega_A^1\otimes_{\z_p}\gg$, in the sense of and 
according to lemma \ref{synthVII}. Our prime tool is the trick of Dwork, which provides us with an element 
$\Theta\in\gG(K(k_0)\{\{t_1,\dots,t_d\}\})$ satisfying:
\begin{eqnarray*}
&&\Theta(0,\dots,0)=1\\
&&\Theta^{-1}b_1\tau(\Theta)=b_0\\
&&\eta^\gG\circ\Theta=-D_1
\end{eqnarray*}
where $b_0=U_0{^F\mu(\frac 1p)}$ and $b_1=U_1{^F\mu(\frac 1p)}$ (N.B.: the last equation has to be interpreted in $\bigoplus_{i=1}^dK(k_0)\{\{t_1,\dots,t_d\}\}dt_i$). Notice 
that the subgroup $\Theta(\gG(K(k_0))\times\Aut(A/W(k_0)))\Theta^{-1}\subset\t$ consists of exactly all elements $(u,s)\in\t$ which are solutions to the differential equation:
\begin{equation}
\label{synthIII}
\Ad(u)s(D_1)-D_1=\eta^\gG\circ u,
\end{equation}
here the right-hand side is again the image of $\eta^\gG$ under the map 
$\Omega_\gG^1\rightarrow\bigoplus_{i=1}^dK(k_0)\{\{t_1,\dots,t_d\}\}dt_i$ induced by $u$, and the left-hand side uses the natural left action of $\Aut(A/W(k_0))$. Also, notice that the element $U_{uni}:=\hat\delta(U_1)\in\Ob_{\B(\gG,\mu)(A)}$ 
represents the universal formal mixed characteristic deformation $\P_{uni}$ 
of $\P_0$, which has a rich amount of symmetry: For every $\gamma\in\Aut(\P_0)$, there exists a unique pair $(h_\gamma,s_\gamma)\in\I^\mu(A)\rtimes\Aut(A/W(k_0))$ with:
\begin{itemize}
\item
$s_\gamma(U_{uni})=h_\gamma^{-1}U_{uni}\Phi^\mu(h_\gamma)$
\item
$h_\gamma$ is a lift of $\gamma\in\I^\mu(k_0)$
\end{itemize}
Passage to the lowest truncation $u_\gamma:=w_0(h_\gamma)$ yields elements $(u_\gamma,s_\gamma)\in\I_0^\mu(A)\rtimes\Aut(A/W(k_0))$ which are horizontal in the 
sense that \eqref{synthIII} holds (N.B.: Lemma \ref{synthVII} shows that $D_1$ depends only on $\P_{uni}$ and not on the choice of $\tau$). After these preparatory remarks 
we are able to state and prove the technical fact that $D_1$ tends to be complicated as can be, here $K(k_0)^{ac}$ (resp. $K(k_0)^{ac}\{\{t_1,\dots,t_d\}\}$, $A^{ac}$, 
$\hat\Omega_{A^{ac}/K(k_0)^{ac}}^1$, etc.) denotes an algebraic closure of $K(k_0)$ (resp. the $p$-adically incomplete tensor products 
$K(k_0)^{ac}\otimes_{K(k_0)}K(k_0)\{\{t_1,\dots,t_d\}\}$, $K(k_0)^{ac}\otimes_{K(k_0)}A$, $K(k_0)^{ac}\otimes_{W(k_0)}\hat\Omega_{A/W(\f_{p^f})}^1$, etc.).

\begin{lem}
\label{maximalholonomy}
Let $G^{spc}\triangleleft\gG_{\q_p}$ be the smallest $\q_p$-rational normal subgroup, such that no $\mu$-weight of $\Lie\gG_{\q_p}/\Lie G^{spc}$ 
is positive. Then $D_1\in\hat\Omega_A^1\otimes_{\z_p}\Lie(G^{spc})$ holds. Assume in addition that the following holds:
\begin{itemize}
\item[(i)]
The group $\gG$ is reductive and of adjoint type.
\item[(ii)]
The field $k_0$ is finite.
\item[(iii)]
There exists some $s\in\n$ with $b_0{^F}b_0\cdots{^{F^{s-1}}}b_0=1$
\end{itemize}
Then there do not exist any proper subgroups $H\subset G_{K(k_0)^{ac}}^{spc}$ for which there are elements $u\in\gG(A^{ac})$ with: 
$$\Ad^\gG(u)(D_1)-\eta^\gG\circ u\in\hat\Omega_A^1\otimes_{W(k_0)}\Lie H$$
\end{lem}
\begin{proof}
By the right-invariance of the Cartan-Maurer form one has: 
$$\eta^\gG\circ U_1=\eta^\gG\circ\exp(\sum_{i=1}^dt_i\epsilon_i)\in\hat\Omega_A^1\otimes_{\z_p}\Lie(G^{spc})$$ 
Hence the $p$-adically contractive map $D\mapsto\Ad^{\gG}(U_1)(\psi_A^\mu(D))-\eta^\gG\circ U_1$ preserves 
$\hat\Omega_A^1\otimes_{\z_p}\Lie(G^{spc})$, because $\Lie(G^{spc})$ is a Lie ideal.\\
Note that $K(k_0)$ will allow a sufficiently large Galois extension $N\subset K(k_0)^{ac}$ over which both of $H$ and $u$ are defined, furthermore we 
can pick an extension of the absolute Frobenius on $K(k_0)$ to $N$ (hence an extension of $\tau$ to $N\{\{t_1,\dots,t_d\}\}$). In order to simplify the 
notation we may also assume $k_0=\O_N/\gm_N$, so that $N$ is totally ramified over $K(k_0)$. Let $u_0\in\gG(N)$ be the evaluation of $u$ at the 
specific point $t_1=\dots=t_d=0$. Consider the elements $\Ad^\gG(u)(D_1)-\eta^\gG\circ u=:\tilde D\in\hat\Omega_A^1\otimes_{W(k_0)}\Lie H$, 
$\tilde b=ub_1\tau(u^{-1})$, and $\tilde\Theta:=u\Theta u_0^{-1}$. Since $\tilde\Theta$ is neutral at the origin and satisfies the differential equation 
$-\tilde D=\eta^\gG\circ\tilde\Theta$, we are allowed to deduce $\tilde\Theta\in H(N\{\{t_1,\dots,t_d\}\})$.\\ 
Observe $\tilde\Theta^{-1}\tilde b\tau(\tilde\Theta)=u_0b_0\tau(u_0^{-1})$, and if the assumption (iii) is satisfied by some positive multiple $s$ of the 
degree of $k_0$ we get: $\tilde\Theta\tau^s(\tilde\Theta^{-1})=\tilde b\tau(\tilde b)\cdots\tau^{s-1}(\tilde b)\in H(N\otimes_{W(k_0)}W(k_0)[[t_1,\dots,t_d]])$. 
Towards shifting our attention to the Witt ring we introduce the map $\hat\delta_N:=\id_N\otimes\hat\delta$ from $N\otimes_{W(k_0)}A$ to 
$N\otimes_{W(k_0)}W(A/pA)$, where $\hat\delta$ is Cartier's diagonal associated to the Frobenius lift $\tau$. If we let
\begin{eqnarray*}
&&\tilde b_{uni}:=\hat\delta_N(\tilde b)\\
&&\tilde h_\gamma:=\hat\delta_N(u)h_\gamma\hat\delta_N(s_\gamma(u^{-1}))\\
&&\tilde u_\gamma:=uu_\gamma s_\gamma(u^{-1})
\end{eqnarray*}
for each $\gamma\in\Aut(\P_0)$, we find: 
$$s_\gamma(\tilde b_{uni})=\tilde h_\gamma^{-1}\tilde b_{uni}F(\tilde h_\gamma)$$
The key to the proof is the slope homomorphism $\tilde\nu$ of $\tilde b_{uni}$. Since $H$ is $s_\gamma$-invariant, we infer that each 
$\tilde h_\gamma^{-1}\tilde\nu\tilde h_\gamma$ is a cocharacter of $H$. The $h_\gamma$'s are Zariski-dense by inspection of their special 
fibers, observe that $\Aut(\P_0)$ is a $p$-adically open subgroup in a certain $\q_p$-form of $\gG$ by \cite[Proposition 1.12]{rapoport}. 
It follows that $H$ contains a normal subgroup containing $\tilde\nu$, and proposition \ref{Hodgepoint} finishes the proof.
\end{proof}

\section{Auxiliary results on $\GL(n)$}

The units of $\Mat(n\times n,R)$ define the $\z$-group $\GL(n)$, on which we define the standard involution to be the automorphism
\begin{equation}
\label{involution}
\GL(n)\stackrel{\cong}{\rightarrow}\GL(n);\,\alpha_{i,j}\mapsto\check\alpha_{i,j},
\end{equation}
where $\check\alpha_{i,j}$ are the entries of the inverse to the $n\times n$-matrix having the element $\alpha_{n-j+1,n-i+1}$ 
in its $i$th row and $j$th column, for $i,j\in\{1,\dots,n\}$. The image of the injection
\begin{equation}
\delta:\g_m\hookrightarrow\GL(n);a\mapsto\left(\begin{matrix}a&\dots&0\\
\vdots&\ddots&\vdots\\0&\dots&a\end{matrix}\right).
\end{equation}
coincides with $\zen^{\GL(n)}$. By the $W(\f_{p^r})$-similarity group we mean the group representing the $W(\f_{p^r})$-functor:
\begin{equation}
\label{similarity}
\GU(n,R)=\{(m,A)|m\in R^\times,\,A\in\GL(n,R\otimes_{W(\f_{p^r})}W(\f_{p^{2r}})),\,\bar A=m\check A\},
\end{equation}
where the conjugation $\bar\empty$ stands for the $R$-linear extension of the $r$th iterate of the 
absolute Frobenius on $W(\f_{p^{2r}})$. Notice that $\zen^{\GU(n)}$ agrees with the image of the injection
\begin{equation}
\label{dilatation}
\zeta:\Res_{W(\f_{p^{2r}})/W(\f_{p^r})}\g_{m,W(\f_{p^{2r}})}\hookrightarrow\GU(n);a\mapsto(a\bar a,a),
\end{equation}
which we will call the dilatation homomorphism. The canonical surjection 
$$\chi:\GU(n)\twoheadrightarrow\g_{m,W(\f_{p^r})};(m,A)\mapsto m$$ 
is called the multiplier character, and the canonical homomorphism 
$$\rho:\GU(n)_{W(\f_{p^{2r}})}\rightarrow\GL(n)_{W(\f_{p^{2r}})};(m,A)\mapsto A$$
is called the tautological representation, notice that $\chi_{W(\f_{p^{2r}})}\oplus\rho$ defines an isomorphism from $\GU(n)_{W(\f_{p^{2r}})}$ 
to $\g_{m,W(\f_{p^{2r}})}\times\GL(n)_{W(\f_{p^{2r}})}$ and that $\chi\otimes\check\bohr\cong\rho$ (and $\bachi\cong\chi$) holds. For a 
$\Phbar$-datum of the form $(\GU(n),\{\upsilon_\sigma\}_{\sigma\in\Sigma})$, we will often have to consider two associated $\Phbar$-data, namely 
$(\GL(n)_{W(\f_{p^{2r}})},\{\upsilon_\sigma^{(2)}\}_{\sigma\in\Sigma^{(2)}})$ and $(\g_{m,W(\f_{p^r})},\{\upsilon_\sigma^{(1)}\}_{\sigma\in\Sigma})$, 
where $\upsilon_\sigma^{(2)}:=(^{F^{-\sigma}}\rho)\circ\upsilon_\sigma$ and $\upsilon_\sigma^{(1)}:=(^{F^{-\sigma}}\chi)\circ\upsilon_\sigma$. 
We call $(\GL(n)_{W(\f_{p^{2r}})},\{\upsilon_\sigma^{(2)}\}_{\sigma\in\Sigma^{(2)}})$ (resp. 
$(\g_{m,W(\f_{p^r})},\{\upsilon_\sigma^{(1)}\}_{\sigma\in\Sigma})$) the tautological linear (resp. 
multiplicative) $\Phbar$-datum of $(\GU(n),\{\upsilon_\sigma\}_{\sigma\in\Sigma})$, notice that giving 
a $W(\f_{p^f})$-rational $\Phbar$-datum for the standard unitary $W(\f_{p^r})$-group $\GU(n)$ is actually 
equivalent to giving $(\GL(n)_{W(\f_{p^{2r}})},\{\upsilon_\sigma^{(2)}\}_{\sigma\in\Sigma^{(2)}})$, subject to 
the condition that $\frac{\check\upsilon_{\sigma+r}^{(2)}}{\upsilon_\sigma^{(2)}}$ be a homothety (namely 
$\upsilon_\sigma^{(1)}$). Here are some $\Phbar$-data, which will acquire a special importance in this paper:

\begin{defn}
\label{standard}
Let $r$ be a divisor of $f$. 
\begin{itemize}
\item[(i)]
The pair $(\GL(n)_{W(\f_{p^r})},\{\upsilon_\sigma\}_{\sigma\in\Sigma})$ is called a standard linear (resp. multiplicative) $\Phbar$-datum 
if the weights of $\upsilon_\sigma$ are contained in the interval $[0,r_\Sigma(\sigma)]$ (resp. are equal to $r_\Sigma(\sigma)$). 
\item[(ii)]
Two pairs $(\GL(n)_{W(\f_{p^r})},\{\upsilon_\sigma\}_{\sigma\in\Sigma})$ and 
$(\GL(n)_{W(\f_{p^r})},\{\upsilon_\sigma^t\}_{\sigma\in\Sigma})$ of standard linear $\Phbar$-data are Cartier duals of each other if and only if
$$\upsilon_\sigma^t=\delta_\sigma\check\upsilon_\sigma$$ 
holds for all $\sigma\in\Sigma$, where $(\g_{m,W(\f_{p^r})},\{\delta_\sigma\}_{\sigma\in\Sigma})$ is the standard multiplicative $\Phbar$-datum.
\item[(iii)]
The pair $(\GU(n),\{\upsilon_\sigma\}_{\sigma\in\Sigma})$ is called a standard unitary $\Phbar$-datum if $2r$ is a divisor of $f$, and the 
aforementioned tautological linear and multiplicative $\Phbar$-data $(\GL(n)_{W(\f_{p^{2r}})},\{\upsilon_\sigma^{(2)}\}_{\sigma\in\Sigma^{(2)}})$ 
and $(\g_{m,W(\f_{p^r})},\{\upsilon_\sigma^{(1)}\}_{\sigma\in\Sigma})$ are standard linear and multiplicative ones. 
\end{itemize}
\end{defn}

\subsection{The functor $\fx^{\bj,\{\upsilon_\sigma\}_{\sigma\in\Sigma}}$}
\label{weirdVI}

The purpose of this subsection is to give a more conceptual approach to \cite[Subsubsection 2.2.1]{habil}, by using the 
fpqc-stacks $\Barb(\GL(n)_{W(\f_{p^r})},\{\upsilon_\sigma\}_{\sigma\in\Sigma})$ of our definition \ref{concept}.

\begin{thm}
Fix integers $l_1\geq\dots\geq l_n$, let $h:=l_1-l_n$. Consider the cocharacter $\upsilon:\g_m\rightarrow\GL(n)$ given by:
$$z\mapsto\left(\begin{matrix}z^{l_1}&\dots&0\\\vdots&\ddots&\vdots\\0&\dots&z^{l_n}\end{matrix}\right)$$
Then the group of $R$-valued points of $\Ibar^\upsilon$ consists of all matrices $A=(a_{i,j})_{1\leq i,j\leq n}$ with $a_{i,j}\in I_{\max\{l_i-l_j,0\}}(R)$, and for 
every such $A\in\Ibar^\upsilon(R)=\hat U_{\upsilon^{-1}}^0(W(R),\val_R)$ one has $\Phbar^{\upsilon,h}(A)=B=(b_{i,j})_{1\leq i,j\leq n}\in\GL(n,W(R))$ with 
$$b_{i,j}=\begin{cases}{^{F^h}a_{i,j}}p^{l_j-l_i}&l_i\leq l_j\\
{^{F^{h+l_j-l_i}V^{l_j-l_i}}a_{i,j}}&\mbox{ otherwise}\end{cases}.$$
Furthermore, if 
$$\upsilon':\g_m\rightarrow\GL(n);z\mapsto\left(\begin{matrix}z^{l'_1}&\dots&0\\\vdots&\ddots&\vdots\\0&\dots&z^{l'_n}\end{matrix}\right)$$
is another cocharacter with $l'_1\geq\dots\geq l'_n$ and $h':=l'_1-l'_n$, then $\Phbar^{\upsilon,h}(\Ibar^{\upsilon\upsilon'}(R))\subset\Ibar^{\upsilon'}(R)$ 
and $\Phbar^{\upsilon',h'}(\Phbar^{\upsilon,h}(A))=\Phbar^{\upsilon\upsilon',h+h'}(A)$ holds for every $A\in\Ibar^{\upsilon\upsilon'}(R)$.
\end{thm}
\begin{proof}
One can reduce all statements to the special case of reduced base $\f_p$-algebras $R$, in which case they 
follow from $^{F^h}(\upsilon(\frac 1p)A\upsilon(p))=\Phbar^{\upsilon,h}(A)$, which holds in $\GL(n,W(R)[\frac1p])$.
\end{proof}

\subsubsection{Definition and Schematicness of $\fx^{\bj,\{\upsilon_\sigma\}_{\sigma\in\Sigma}}$}
\label{Universum}
Let $r$ be a divisor of $f$, and let $(\GL(n)_{W(\f_{p^r})},\{\upsilon_\sigma\}_{\sigma\in\Sigma})$ be a 
standard linear $\Phbar$-datum. Let $\bj:\z/r\z\rightarrow\n_0$ be a function satisfying the following:
\begin{itemize}
\item[(i)]
For every $\sigma\in\z/r\z$ one has $0\leq\bj(\sigma)\leq r_\Sigma(\bd_\Sigma(\sigma))-1$. 
\item[(ii)]
The assignment $\z/r\z\ni\sigma\mapsto(\bd_\Sigma(\sigma),\bj(\sigma))$ is injective.
\end{itemize}
Define new cocharacters, by means of the operator $H_0$, as introduced in part \ref{cocha} of the appendix:
\begin{equation}
\label{soul}
H_0(\frac{^{F^{\bd_\Sigma^+(\sigma)}}\upsilon_{\bd_\Sigma(\sigma)}}{\delta^{\bj(\sigma)}})=:\tilde\upsilon_\sigma:\g_{m,W(\f_{p^f})}\rightarrow\GL(n)_{W(\f_{p^f})}
\end{equation}
Note that $(\GL(n)_{W(\f_{p^r})},\{\tilde\upsilon_\sigma\}_{\sigma\in\z/r\z})$ is a standard linear $\Phi$-datum, as the weights of $\tilde\upsilon_\omega$ are contained in $\{0,1\}$, while 
\begin{equation}
\label{heart}
\upsilon_\sigma=\prod_{\bd_\Sigma(\omega)=\sigma}{^{F^{-\bd_\Sigma^+(\omega)}}\tilde\upsilon_\omega}
\end{equation}
holds. In this section we are going to establish a certain $2$-commutative diagram, namely:
$$\begin{CD}
^W\GL(n)_{W(\f_{p^f})}^\Sigma @>m>>^W\GL(n)_{W(\f_{p^f})}^r\\
@VVV@VVV\\
\Barb(\GL(n)_{W(\f_{p^r})},\{\upsilon_\sigma\}_{\sigma\in\Sigma})
@>{\fx^{\bj,\{\upsilon_\sigma\}_{\sigma\in\Sigma}}}>>\Barb(\GL(n)_{W(\f_{p^r})},\{\tilde\upsilon_\sigma\}_{\sigma\in\z/r\z})\\
@V{q}VV@V{\tilde q}VV\\
\ton(\Ibar^{\upsilon_\Sigma})@>{\ton(\gamma)}>>\ton(\Ibar^{\tilde\upsilon})
\end{CD}$$
We follow subsection \ref{spass}, starting with a description of $\gamma:\Ibar^{\upsilon_\Sigma}\rightarrow\Ibar^{\tilde\upsilon}$, to this end consider the cocharacters
\begin{equation*}
\prod_{\omega=\sigma+1}^{\bd_\Sigma(\sigma)}{^{F^{-\bd_\Sigma^+(\omega)}}\tilde\upsilon_\omega}=:\Upsilon_\sigma:\g_{m,W(\f_{p^f})}\rightarrow\GL(n)_{W(\f_{p^f})},
\end{equation*}
and let $\gamma:\prod_{\sigma\in\Sigma}\Ibar^{\upsilon_\sigma}\rightarrow\prod_{\omega=0}^{r-1}\Ibar^{\tilde\upsilon_\omega}$ be the homomorphism whose $\omega$th 
coordinate is given by the formula $\gamma_\omega(\{k_\sigma\}_{\sigma\in\Sigma})=\Phbar^{\Upsilon_\omega,\bd_\Sigma^+(\omega)}(k_{\bd_\Sigma(\omega)})$. 
Let $m:{^W\GL(n)_{W(\f_{p^f})}^\Sigma}\rightarrow{^W\GL(n)_{W(\f_{p^f})}^r}$ be the inclusion.

\begin{prop}
\label{new}
The $1$-morphism $\fx^{\bj,\{\upsilon_\sigma\}_{\sigma\in\Sigma}}$ is schematic, quasicompact and separated.
\end{prop}
\begin{proof}
Notice that for every $\omega\in\Sigma$ we have $\bd_\Sigma(\omega)=\omega$, $\Upsilon_\omega=1$ and a commutative diagram:
$$\begin{CD}
\Ibar^{\upsilon_\Sigma}@>\gamma>>\Ibar^{\tilde\upsilon}\\
@AAA@VVV\\
\Ibar^{\upsilon_\omega}@>\id>>\Ibar^{\tilde\upsilon_\omega}
\end{CD}$$
It follows that $\gamma:\Ibar^{\upsilon_\Sigma}\rightarrow\Ibar^{\tilde\upsilon}$ is a closed immersion, moreover the quotient of $\Ibar^{\tilde\upsilon}$ by the image of $\gamma$ is 
representable, simply because it can be regarded as a locally trivial principal homogeneous space under the group $\prod_{\omega\notin\Sigma}\Ibar^{\tilde\upsilon_\omega}$ over the ground 
$\f_{p^f}$-scheme $\prod_{\omega\in\Sigma}\Ibar^{\tilde\upsilon_\omega}/\Ibar^{\upsilon_\omega}$. This proves that $\ton(\gamma)$ is schematic, quasicompact and separated. Using the ideas of 
\cite[Lemma 3.2.9a)]{pappas} the same holds for the map from $\Barb(\GL(n)_{W(\f_{p^r})},\{\upsilon_\sigma\}_{\sigma\in\Sigma})$ to $\ton(\Ibar^{\tilde\upsilon})$ which is obtained by composing 
the canonical projection $\tilde q:\Barb(\GL(n)_{W(\f_{p^r})},\{\tilde\upsilon_\sigma\}_{\sigma\in\z/r\z})\rightarrow\ton(\Ibar^{\tilde\upsilon})$ with $\fx^{\bj,\{\upsilon_\sigma\}_{\sigma\in\Sigma}}$. 
To finish the proof one only needs to observe that the diagonal of the former is again schematic, quasicompact and separated (in fact it is a closed immersion).
\end{proof}

\begin{rem}
\label{weirdVIII}
The earlier defined $\id_{\tilde\upsilon^{-1}\upsilon_\Omega}$ (cf. \eqref{Kuh}) can be recovered as any one of our functors $\fx^{\bj,\{\upsilon_\omega\}_{\omega\in\Omega}}$, 
provided only that $\upsilon_\omega$ lies in $\zen^{\GL(n)}$ for every $\omega\in\Omega$. Thus we have 
$\id_{\delta^{-1}\delta_\Omega}\cong\fx^{\bj,\{\delta_\omega\}_{\omega\in\Omega}}$ in the standard multiplicative case, where the choice of $\bj$ is arbitrary. 
\end{rem}

\subsubsection{Composition with the modular character} 
\label{Haar}

Later on we need to consider the character $\bachi_\sigma^\bj$ arising from precomposing $\gamma$ with the truncation to 
$\Ibar_0^{\tilde\upsilon}$ followed by the character $\prod_{\omega=0}^{r-1}\chi_{\tilde\upsilon_\omega^{-1}}^{\GL(n)}(-1,\sd)$, i.e.:
$$\begin{CD}
\Ibar_0^{\tilde\upsilon}@>{\chi_{\tilde\upsilon^{-1}}^{\GL(n)^r}(-1,\sd^r)}>>
\g_{m,\f_{p^f}}\\
@AAA@A{\bachi_\sigma^\bj}AA\\
\Ibar^{\tilde\upsilon}@<{\gamma|_{\Ibar^{\upsilon_\sigma}}}<<\Ibar^{\upsilon_\sigma}
\end{CD},$$
please see to part \ref{ballast} of the appendix for the definition of $\chi_{\tilde\upsilon_\omega^{-1}}^{\GL(n)}(-1,\sd)$. 
Furthermore, it turns out that $\bachi_\sigma^\bj$ factors through the $\pmod p$-reduction of a certain character 
\begin{equation*}
\prod_{l\in\z}\chi_{\upsilon_\sigma^{-1}}^{\GL(n)}(-l,\sd)^{d_{\sigma,l}}=\chi_\sigma^\bj:\I_0^{\upsilon_\sigma}\rightarrow\g_{m,W(\f_{p^f})},
\end{equation*}
in which the exponents are given by the nifty formulae: 
\begin{equation*}
d_{\sigma,l}:=\sum_{\bd_\Sigma(\omega)=\sigma,\,\bj(\omega)<l}p^{\bd_\Sigma^+(\omega)}
\end{equation*}
for every $\sigma\in\Sigma$.

\subsubsection{Compatibility with direct sums} 

If $\upsilon_\sigma=\left(\begin{matrix}\upsilon_\sigma^{(1)}&0\\0&\upsilon_\sigma^{(2)}\end{matrix}\right)$ is a decomposition into matrix blocks of 
size $n^{(1)}\times n^{(1)},\dots,n^{(2)}\times n^{(2)}$, with each $\upsilon_\sigma^{(i)}$ being a cocharacter of $\GL(n^{(i)})_{W(\f_{p^f})}$, then so is 
$\tilde\upsilon_\sigma=\left(\begin{matrix}\tilde\upsilon_\sigma^{(1)}&0\\0&\tilde\upsilon_\sigma^{(2)}\end{matrix}\right)$, with each $\tilde\upsilon_\sigma^{(i)}$ being 
the cocharacter of $\GL(n^{(i)})_{W(\f_{p^f})}$ gotten from the family $\upsilon_\sigma^{(i)}$ by using \eqref{soul}. We have a canonical $2$-commutative diagram
$$\begin{CD}
\Barb(\GL(n)_{W(\f_{p^r})},\{\upsilon_\sigma\}_{\sigma\in\Sigma})@<<<\prod_{i=1}^2\Barb(\GL(n^{(i)})_{W(\f_{p^r})},\{\upsilon_\sigma^{(i)}\}_{\sigma\in\Sigma})\\
@V{\fx^{\bj,\{\upsilon_\sigma\}_{\sigma\in\Sigma}}}VV@V{\prod_{i=1}^2\fx^{\bj,\{\upsilon_\sigma^{(i)}\}_{\sigma\in\Sigma}}}VV\\
\Barb(\GL(n)_{W(\f_{p^r})},\{\tilde\upsilon_\sigma\}_{\sigma\in\z/r\z})@<<<\prod_{i=1}^2\Barb(\GL(n^{(i)})_{W(\f_{p^r})},\{\tilde\upsilon_\sigma^{(i)}\}_{\sigma\in\z/r\z})
\end{CD},$$
where the horizontal arrows are induced from morphisms of $\Phbar$-data
\begin{eqnarray*}
&&((\GL(n^{(1)})\times\GL(n^{(2)}))_{W(\f_{p^r})},\{\upsilon_\sigma\}_{\sigma\in\Sigma})\rightarrow(\GL(n)_{W(\f_{p^r})},\{\upsilon_\sigma\}_{\sigma\in\Sigma})\\ 
&&((\GL(n^{(1)})\times\GL(n^{(2)}))_{W(\f_{p^r})},\{\tilde\upsilon_\sigma\}_{\sigma\in\z/r\z})\rightarrow(\GL(n)_{W(\f_{p^r})},\{\tilde\upsilon_\sigma\}_{\sigma\in\z/r\z}).
\end{eqnarray*}

\subsubsection{Compatibility with Cartier-Duality} 
\label{dualityII}

Consider the homomorphism $i:\g_m\times\GL(n)\rightarrow\GL(n);\,(m,A)\mapsto m\check A$, and let $(\g_{m,W(\f_{p^r})},\{\delta_\sigma\}_{\sigma\in\Sigma})$ 
be a standard multiplicative $\Phbar$-datum. Notice that $i$ induces maps between two pairs of $\Phbar$-data, namely:
\begin{eqnarray*}
&&((\g_m\times\GL(n))_{W(\f_{p^r})},\{(\delta_\sigma,\upsilon_\sigma)\}_{\sigma\in\Sigma})\rightarrow(\GL(n)_{W(\f_{p^r})},\{\upsilon_\sigma^t\}_{\sigma\in\Sigma})\\
&&((\g_m\times\GL(n))_{W(\f_{p^r})},\{(\delta,\tilde\upsilon_\sigma)\}_{\sigma\in\z/r\z})\rightarrow(\GL(n)_{W(\f_{p^r})},\{\tilde\upsilon_\sigma^t\}_{\sigma\in\z/r\z})
\end{eqnarray*}
where $\{\upsilon_\sigma^t\}_{\sigma\in\Sigma}$ and $\{\tilde\upsilon_\sigma^t\}_{\sigma\in\z/r\z}$ are the duals of $\{\upsilon_\sigma\}_{\sigma\in\Sigma}$ 
and $\{\tilde\upsilon_\sigma\}_{\sigma\in\z/r\z}$ in the sense of part (ii) of definition \ref{standard}. Consider the function 
$$\check\bj(\sigma):=r_\Sigma(\bd_\Sigma(\sigma))-\bj(\sigma)-1.$$ 
Since $\tilde\upsilon_\sigma$ is the cocharacter of $\GL(n)_{W(\f_{p^f})}$ obtained from the family $\upsilon_\sigma$ 
by using \eqref{soul}, we find that each $\tilde\upsilon_\sigma^t$ is obtained similarly from the family $\upsilon_\sigma^t$, 
provided that one uses $\check\bj$ instead of $\bj$. We have a canonical $2$-commutative diagram
$$\begin{CD}
\Barb(\GL(n)_{W(\f_{p^r})},\{\upsilon_\sigma^t\}_{\sigma\in\Sigma})@<<<\Barb((\g_m\times\GL(n))_{W(\f_{p^r})},\{(\delta_\sigma,\upsilon_\sigma)\}_{\sigma\in\Sigma})\\
@V{\fx^{\check\bj,\{\upsilon_\sigma^t\}_{\sigma\in\Sigma}}}VV@V{\fx^{\{\delta_\sigma\}_{\sigma\in\Sigma}}}\times\fx^{\bj,\{\upsilon_\sigma\}_{\sigma\in\Sigma}}VV\\
\Barb(\GL(n)_{W(\f_{p^r})},\{\tilde\upsilon_\sigma^t\}_{\sigma\in\z/r\z})@<<<\Barb((\g_m\times\GL(n))_{W(\f_{p^r})},\{(\delta,\tilde\upsilon_\sigma)\}_{\sigma\in\z/r\z})
\end{CD},$$
where the horizontal arrows are the ones that are induced from $i$. At last we wish to explain the unitary version of $\fx^{\bj,\{\upsilon_\sigma\}_{\sigma\in\Sigma}}$, so 
suppose that $2r$ is a divisor of $f$, that $(\GU(n),\{\upsilon_\sigma\}_{\sigma\in\Sigma})$ is a standard unitary $\Phbar$-datum, and that $\bj$ is a $\n_0$-valued function on 
$\z/2r\z$ with $\bj(r+\sigma)=r_\Sigma(\bd_\Sigma(\sigma))-\bj(\sigma)-1$ (satisfying (i) and (ii)). Whence it follows that there is a canonical $2$-commutative diagram
$$\begin{CD}
\Barb(\GU(n),\{\upsilon_\sigma\}_{\sigma\in\Sigma})
@>{\chi_{W(\f_{p^{2r}})}\oplus\rho}>>\Barb((\g_m\times\GL(n))_{W(\f_{p^{2r}})},\{(\upsilon^{(1)}_\sigma,\upsilon_\sigma^{(2)})\}_{\sigma\in\Sigma^{(2)}})\\
@V{\fx^{\bj,\{\upsilon_\sigma\}_{\sigma\in\Sigma}}}VV@V{\fx^{\bj,\{(\upsilon_\sigma^{(1)},\upsilon_\sigma^{(2)})\}_{\sigma\in\Sigma}}}VV\\
\Barb(\GU(n),\{\tilde\upsilon_\sigma\}_{\sigma\in\z/r\z})
@>{\chi_{W(\f_{p^{2r}})}\oplus\rho}>>\Barb((\g_m\times\GL(n))_{W(\f_{p^{2r}})},\{(\tilde\upsilon^{(1)}_\sigma,\tilde\upsilon_\sigma^{(2)})\}_{\sigma\in\z/2r\z})
\end{CD},$$
according to lemma \ref{galoisII}. 

\subsection{Skew-Hermitian variants}

We wish to introduce a couple of auxiliary notions: In the following $C^+$ stands for $W(\f_{p^r})$ and $C$ stands for one of the two $C^+$-algebras 
$W(\f_{p^{2r}})$ or $W(\f_{p^r})\oplus W(\f_{p^r})$ and we write $\bar\empty:C\rightarrow C$ for the involution given by $x\mapsto\tr_{C/C^+}(x)-x$. 
By a $C$-skew-Hermitian display of height $n$ over $R$ we mean a structure consisting of the following data:

\begin{itemize}
\item
a display $(M,N,F,V^{-1})$ over $R$ together with an action of the $\z_p$-algebra $C$, such that the rank of $M$ is equal to $2rn$,
\item
a multiplicative display $(K,I(R)K,F,V^{-1})$ together with an action of the $\z_p$-algebra $C^+$ such that $K$ is invertible, when regarded as a module over $C^+\otimes_{\z_p}W(R)$,
\item
a perfect skew-Hermitian pairing
$$-\check\Bapsi=\Psi:M\stackrel{\cong}{\rightarrow}\Hom_{C\otimes_{\z_p}W(R)}(\bar M,C\otimes_{C^+}K)$$
and such that $\Psi(N,N)\subset I(R)K$ and $\Psi(V^{-1}(x),V^{-1}(y))=V^{-1}(\Psi(x,y))$ holds for all $x,y\in N$.
\end{itemize}

Let us write $C_{an}$ for the set of $\z_p$-linear homomorphism $\iota:C\rightarrow W(\f_{p^f})$, where $f$ is a fixed multiple 
of $r$ or $2r$ depending on whether or not $C=C^+\oplus C^+$. Following \cite[Subsection 2.3]{pappas} we notice:

\begin{thm}
\label{exampleV}
Consider a perfect skew-Hermitian pairing 
$$-\check\Bapsi=\Psi:\V\stackrel{\cong}{\rightarrow}\check{\bar\V}$$ 
on some free $C$-module $\V$ of finite rank $n$ and let $\upsilon$ be a $W(\f_{p^f})$-rational cocharacter of the reductive $\z_p$-group 
$\Res_{C^+/\z_p}\GU(\V/C,\Psi)$ of which the $\iota$-component of its image in the $\z_p$-group $\Res_{C/\z_p}\GL(\V/C)$ lies in the conjugacy class of
\begin{equation}
\label{exampleVIII}
\g_m\rightarrow\GL(n);z\mapsto\diag(\overbrace{z,\dots,z}^{d_\iota},\overbrace{1,\dots,1}^{n-d_\iota}),
\end{equation}
for some family $\du=\{d_\iota\}_{\iota\in C_{an}}$ of non-negative integers satisfying $d_{\iotar}=n-d_\iota$. Then $(\Res_{C^+/\z_p}\GU(\V/C,\Psi),\upsilon)$
is a $W(\f_{p^f})$-rational display datum and the fiber of the $W(\f_{p^f})$-stack $\B(\Res_{C^+/\z_p}\GU(\V/C,\Psi),\upsilon)$ over a 
$p$-adically separated and complete $W(\f_{p^f})$-algebra $R$ is canonically equivalent to the groupoid of $C$-skew-Hermitian displays 
$(M,N,F,V^{-1})$ of height $n$ over $R$ such that the $\iota$-eigenspace of $M/N$ is a projective $R$-module of rank $d_\iota$.
\end{thm}

We note in passing that $d_\iota\in\{0,n\}$ for at least one $\iota\in C_{an}$ implies the adjoint nilpotence of all $(\Res_{C^+/\z_p}\GU(\V/C,\Psi),\upsilon)$-displays! 
Let us fix $r\in\n$ and $f\in\n$ and consider an index set $\pi$ of finite and odd cardinality together a family of skew-Hermitian 
display data $\{(\V_i,\Psi_i,\upsilon_i)\}_{i\in\pi}$, whose dimension invariants $\du_i=\{d_{i,\iota}\}_{\iota\in C_{an}}$ satisfy

\begin{eqnarray}
\label{exampleIII}
&&\Card(\{i\in\pi\mid\,d_{i,\iota}=n_i\})\geq\frac{\Card(\pi)-1}2\\
\label{exampleIV}
&&\Card(\{i\in\pi\mid\,d_{i,\iota}=0\})\geq\frac{\Card(\pi)-1}2
\end{eqnarray}

for every $\iota\in C_{an}$. We write $\Psi^\pi$ for the natural skew-Hermitian pairing on the $C$-linear tensor product $\V^\pi$ of the $C$-modules $\V_i$ and we
write $\upsilon^\pi$ for the $W(\f_{p^r})$-rational cocharacter of $\Res_{C^+/\z_p}\GU(\V^\pi/C,\Psi^\pi)$ arising from the tensor product of the $\upsilon_i$, divided 
by the central cocharacter $\delta^{\frac{\Card(\pi)-1}{2}}$. The validity of \eqref{exampleIII} and \eqref{exampleIV} implies that $(\V^\pi,\Psi^\pi,\upsilon^\pi)$ is a 
$W(\f_{p^f})$-rational skew-Hermitian display datum, of which the dimension invariants $\du_\pi=\{d_{\pi,\iota}\}_{\iota\in C_{an}}$ satisfy
\begin{equation}
\label{exampleI}
\sum_{\Card(I)=\frac{\Card(\pi)+1}2}\prod_{i\in I}d_{i,\iota}\prod_{i\notin I}(n_i-d_{i,\iota})=d_{\pi,\iota}
\end{equation}
while $n=\prod_{i\in\pi}n_i$. Moreover, the tensor product induces a natural morphism of display data
\begin{equation*}
g^\pi:(\prod_{i\in\pi}\Res_{C^+/\z_p}\GU(\V_i/C,\Psi_i),\upsilon_i)\rightarrow(\Res_{C^+/\z_p}\GU(\V^\pi/C,\Psi^\pi),\delta^{\frac{\Card(\pi)-1}{2}}\upsilon^\pi),
\end{equation*}
thus inducing a $1$-morphism:
\begin{eqnarray}
\label{exampleVII}
&&\id_{\delta^{\frac{\Card(\pi)-1}{2}}}\circ\B(g^\pi):\B(\prod_{i\in\pi}\Res_{C^+/\z_p}\GU(\V_i/C,\Psi_i),\upsilon_i)\\
\label{exampleXI}
&&\rightarrow\B(\Res_{C^+/\z_p}\GU(\V^\pi/C,\Psi^\pi),\upsilon^\pi)
\end{eqnarray}

\begin{defn}
\label{exampleX}
A  $C$-multi-skew-Hermitian display datum over $W(\f_{p^f})$ consists of a family $\{(\V_i,\Psi_i,\upsilon_i)\}_{i\in\Lambda}$ as in theorem 
\ref{exampleV} together with a family of involutive $C$-algebras $\R_\pi$, where $\pi$ runs through a subset $\Pi$ of the power set of $\Lambda$ each 
of whose elements have odd cardinality and satisfy \eqref{exampleIII} and \eqref{exampleIV} (where $\du_i$ are the dimension invariants of $\upsilon_i$ 
as in \eqref{exampleVIII}). We write $\gB^{(\{(\V_i,\Psi_i,\upsilon_i)\}_{i\in\Lambda},\{\R_\pi\}_{\pi\in\Pi})}$ for the fpqc-stack over $W(\f_{p^f})$ assigning to 
any $W(\f_{p^f})$-algebra $R$ with $pR\subset\sqrt{0_R}$ the groupoid consisting of a $(\prod_{i\in\Lambda}\Res_{C^+/\z_p}\GU(\V_i/C,\Psi_i),\upsilon_i)$-display 
$\S$ over $R$ together with $C$-linear and $*$ preserving actions $s_\pi$ of $\R_\pi$ on the canonical $C$-skew-Hermitian $R$-displays $(M^\pi,N^\pi,F,V^{-1})$ 
over $R$ arising from \eqref{exampleVII} followed by an application of theorem \ref{exampleV} to \eqref{exampleXI}. If, in addition 
\begin{equation}
\label{exampleVI}
\exists\,\iota\in C_{an}\,\forall i\in\Lambda:d_{i,\iota}\in\{0,n_i\}
\end{equation}
holds, then the family $(\{(\V_i,\Psi_i,\upsilon_i)\}_{i\in\Lambda},\{\R_\pi\}_{\pi\in\Pi})$ is called multicompact.
\end{defn}

\subsubsection{Deformations with constant Newton polygon}
Let $k$ be an algebraically closed field of characteristic $p$ and fix $r\in\n$. Let $\P=(\{M_\sigma\}_{\sigma\in\z/r\z},\{N_\sigma\}_{\sigma\in\z/r\z},F,V^{-1})$ be an adjoint 
nilpotent $\z/r\z$-graded display over $k$ and let $\{\Psi_\sigma\}_{\sigma\in\z/r\z}$ be a skew-Hermitian structure thereon (N.B.: In this subsubsection it does no harm to take 
$K_\sigma=W(k)$, so that $\Psi_\sigma$ is an isomorphism from $M_\sigma$ to $\check M_{\sigma+\frac r2}$). Let us write $s_1\geq\dots\geq s_n$ for the $V$-slopes of $M$, in 
the sense of \cite{zink4}, notice that $1-s_{n-i+1}=s_i\in]0,1[$. Let us write $\q\otimes M_\sigma=\bigoplus_{l\in\q}M_{l,\sigma}$ for the $V$-slope decomposition of $\q\otimes M$. 
One knows that for some integer $0<s\equiv0\pmod r$ and suitable elements $e_{i,\sigma}\in\q\otimes M_\sigma$ the following assertions hold for all $\sigma\in\{0,\dots,\frac r2-1\}$:
\begin{itemize}
\item
$e_{1,\sigma},\dots,e_{n,\sigma}$ is a $K(k)$-basis for $\q\otimes M_\sigma$
\item
$F^s(e_{i,\sigma})=p^{(1-s_i)s}e_{i,\sigma}$
\item
$\Psi_\sigma(e_{i,\sigma},e_{j,\sigma+\frac r2})=\delta_{i+j,n+1}$
\end{itemize}
With respect to a choice of bases as above, let $A_\sigma\in\Mat(n\times n,K(k))$ be the matrix describing the homomorphism $V^\sharp$ from 
$\bigoplus_{i=1}^nK(k)e_{i,\sigma}=\q\otimes M_\sigma$ to $\bigoplus_{i=1}^nK(k)\otimes e_{i,\sigma+1}=K(k)\otimes_{F,W(k)}M_{\sigma+1}$, 
due to the slope decomposition it has the shape of a diagonal block matrix and satisfies:
$${^{F^{s-1}}A_{s-1}}\dots{^FA_1}A_0=\diag(p^{s_1s},\dots,p^{s_ns})$$
From now on we assume that $M$ has at least two different slopes, say $s_m>s_{m+1}$ for some fixed index $m$, and we want to pay attention to the two corresponding submatrices
$B_\sigma\in\Mat(b\times b,K(k))$ and $D_\sigma\in\Mat(d\times d,K(k))$ as $\sigma$ varies. Consider the $W(k)[[t]]$-module $W(k)[[t]]\otimes_{W(k)}M=:\tilde M$ and the frame 
$(W(k)[[t]],pW(k)[[t]],\tau)$ where $\tau$ is defined by $\tau(t)=t^p$. We will use the technique of Norman to define (a family of) $\z/r\z$-graded skew-Hermitian $W(k)[[t]]$-windows whose 
underlying $W(k)[[t]]$-module is $\tilde M$, and whose special fiber is $M$. Let $H=(H_0,\dots,H_{\frac r2-1})$ be a $\frac r2$-tuple of $b\times d$-matrices over $K(k)$ and proceed further 
as follows: Define a $\q\otimes W(k)[[t]]$-valued element $U_\sigma$ of $\GL(M_\sigma/W(k))$ by decreeing $U_\sigma-\id_{\q\otimes\tilde M_\sigma}$ to be an upper triangular block matrix 
with a single non-zero block, namely $tH_\sigma$ (being the $b\times d$-block lying above the $m+1$st row and to the right of the $m$th column). For the time being we fix $H$ and declare 
$U_\sigma$ to be $\check U_{\sigma-\frac r2}$ for all $\sigma\in\{\frac r2,\dots,r-1\}$. Note that $U_\sigma\equiv\id_{\q\otimes\tilde M_\sigma}\mod t$. If $H\in\Mat(b\times d,K(k))^\frac r2$ 
is $p$-adically sufficiently small, then $U_\sigma$ is a $W(k)[[t]]$-valued element of $\GL(M_\sigma/W(k))$, so that precomposing $F$ and $V^{-1}$ with $U_\sigma$ yields some 
$\z/r\z$-graded skew-Hermitian $W(k)[[t]]$-window $_H\tilde M$, whose special fiber is $M$. Also, note that the corresponding matrix for its Verschiebung $V^\sharp|_{\tilde M_\sigma}$ from 
$_H\tilde M_\sigma$ to $W(R)\otimes_{F,W(R)}{_H\tilde M_{\sigma+1}}$ is $A_\sigma U_\sigma^{-1}=\tilde A_\sigma\in\Mat(n\times n,W(k)[[t]])$, of which the restriction to the subquotient
\begin{equation}
 \label{immernochnicht}
(M_\sigma\cap\bigoplus_{l\geq s_{m+1}}M_{l,\sigma}))/(M_\sigma\cap\bigoplus_{l>s_m}M_{l,\sigma}),
\end{equation}
gives rise to some block matrix $\left(\begin{matrix}B_\sigma&C_\sigma\\0&D_\sigma\end{matrix}\right)$, and 
$C_\sigma=-tB_\sigma H_\sigma$ holds for $\sigma\in\{0,\dots,\frac r2-1\}$. Furthermore, the restriction of the block matrix
$${^{F^{s-1}}\tilde A_{s-1}}\dots{^F\tilde A_1}\tilde A_0=:\tilde A$$
to \eqref{immernochnicht} is given by
$$\left(\begin{matrix}p^{s_ms}&C\\0&p^{s_{m+1}s}\end{matrix}\right),$$ 
where $C={^{F^{s-1}}B_{s-1}}\dots{^FB_1}C_0+\dots+{^{F^{s-1}}C_{s-1}}{^{F^{s-2}}D_{s-2}}\dots D_0$. Next we study whether $_H\tilde M$ is an isotrivial deformation 
of $M$, so suppose that some $\z/r\z$-graded isogeny $h_\sigma:\q\otimes{_H\tilde M_\sigma}\rightarrow\q\otimes{_0\tilde M_\sigma}$ lifts $\id_{\q\otimes M}$. Again 
we want to look at the matrix of $h_\sigma$: Due to slope reasons this is an upper triangular block matrix, and due to the rigidity of unit root crystals, its diagonal entries 
are equal to $1$ (cf. \cite[Lemma 42]{zink2}). Furthermore, the $\z/r\z$-graded homomorphism that is induced on \eqref{immernochnicht} is again described by a family 
of block matrices $\left(\begin{matrix}1&k_\sigma\\0&1\end{matrix}\right)$ for certain $k_\sigma\in\q\otimes\Mat(b\times d,tW(k)[[t]])$ and an inspection of $h_0$ shows: 
$$p^{(s_m-s_{m+1})s}k_0-{^{F^s}k_0}=\frac C{p^{s_{m+1}s}}$$ 
This equation has no solution unless $H_0=\dots=H_{\frac r2-1}=0$, in which case it has a unique solution being $k_0=0$.

\begin{lem}
\label{isogdefo}
Let $\gamma:\P_1\rightarrow\P_2$ be a $\z/r\z$-graded isogeny between two $\z/r\z$-graded displays (possibly equipped with skew-Hermitian 
structures) over an algebraically closed field $k$. If $\P_1$ (and hence $\P_2$) possesses more than one Newton slope, then there exists 
at least one non-isotrivial deformation $\tilde\gamma:\tilde\P_1\rightarrow\tilde\P_2$ with a constant Newton-polygon over $k[[t]]$.
\end{lem}
\begin{proof}
We begin with choosing a $p$-adically sufficiently small $0\neq H$, which gives rise to a deformation of $\P_1$, simply by following the algorithm explained above. Next, we choose $c\in\n_0$ such 
that $p^c\gamma^{-1}$ constitutes a $\z/r\z$-graded isogeny from $\P_2$ to $\P_1$. Observe that $p^cH$ gives rise to another non-isotrivial deformation $\tilde\P_1$ of $\P_1$, which however can be manufactured into a corresponding deformation $\tilde\P_2$ of $\P_2$, namely by composing $V^\sharp$ with the inverse of the deformation matrix $U=(\id_{\P_2}+\gamma\circ tH\circ(p^c\gamma^{-1}))$.
\end{proof}

\section{Realizations of displays with additional structure}
\label{Klamauk}

Let us complement the proposition \ref{twistI} with the following result, whose proof follows along identical lines:

\begin{lem}
\label{twistII}
Let $\upsilon:\g_{m,W(\f_{p^f})}\rightarrow\G$ be a cocharacter all of whose weights are less than or equal to $h\geq1$, where $\G$ is a reductive group over 
$W(\f_{p^f})$. Let $\rho:\G\rightarrow\GL(n)_{W(\f_{p^f})}$ be a representation, such that $\rho\circ\upsilon$ has no positive weight, so that there exists an effective 
cocharacter $\beta:\a_{W(\f_{p^f})}^1\rightarrow\Mat(n\times n)_{W(\f_{p^f})}$ with $\beta|_{\g_{m,W(\f_{p^f})}}=\rho\circ\upsilon^{-1}$. Then, one has an equality 
$$F^h\circ(\beta(p){^W\rho|_{\Ibar^\upsilon}})=(^{WF^h}\rho\circ\Phbar^{\upsilon,h}){^{F^h}\beta(p)}$$
of functions from $\Ibar^\upsilon$ to $^W\Mat(n\times n)_{\f_{p^f}}$. 
\end{lem}

The previous result is going to enter into the construction of certain realization functors, which 
we define in this section. We want to begin with a short description of their target categories:

\begin{defn}
\label{moreflexIV}
Let $A$ be a $p$-adically separated and complete ring, and let $\tau:A\rightarrow A$ be an endomorphism satisfying $\tau(x)\equiv x^p\pmod p$ for all $x\in A$.
\begin{itemize}
\item
By a $F$-module (resp. $V$-module) over $(A,\tau)$ we mean a pair $(M,F^\sharp)$ (resp. $(M,V^\sharp)$) which consists of a finitely generated and projective 
$A$-module $M$ together with an $A$-linear map $F^\sharp:A\otimes_{\tau,A}M\rightarrow M$ (resp. $V^\sharp:M\rightarrow A\otimes_{\tau,A}M$). For every $0\geq n\in\z$ 
(resp. $n\in\n_0$) we let $A(n)$ be the $F$-module (resp. $V$-module) over $(A,\tau)$ which is given by the pair $(A,p^{-n}\id_A)$ (resp. by $(A,p^n\id_A)$).
\item
By a $\tau$-crystal over $(A,\tau)$, we mean a triple $(M,F^\sharp,V^\sharp)$ consisting of a finitely generated and projective $A$-module $M$, together with a pair of 
mutually invers $A$-linear bijections $F^\sharp:A[\frac1p]\otimes_{\tau,A}M\rightarrow\q\otimes M$ and $V^\sharp:\q\otimes M\rightarrow A[\frac1p]\otimes_{\tau,A}M$
\item
The category $\eF_{A,\tau}$ (resp. $\ef_{A,\tau}$) consisting of all $F$-modules over $(A,\tau)$ (resp. $V$-modules over 
$(A,\tau)$) is defined by decreeing its morphisms to be the class of $A$-linear maps $h:N\rightarrow M$ rendering the diagram
$$\begin{CD}
A\otimes_{\tau,A}M@>{F_M^\sharp}>>M\\
@A{\id_A\otimes h}AA@AhAA\\
A\otimes_{\tau,A}N@>{F_N^\sharp}>>N\\
\end{CD}\qquad(\mbox{resp. }\,\,\begin{CD}
M@>{V_M^\sharp}>>A\otimes_{\tau,A}M\\
@AhAA@A{\id_A\otimes h}AA\\
N@>{V_N^\sharp}>>A\otimes_{\tau,A}N\\
\end{CD})$$
commutative, for some objects $(M,F_M^\sharp)$ and $(N,F_N^\sharp)$ of $\eF_{A,\tau}$ (resp. $(M,V_M^\sharp)$ and $(N,V_N^\sharp)$ 
of $\ef_{A,\tau}$).The category $\cris_{A,\tau}$ consisting of all $\tau$-crystals over $(A,\tau)$ is defined by decreeing its morphisms to 
be the class of $A$-linear maps $h:N\rightarrow M$ such that any of the above diagrams commutes upon tensorization with $\q$.
\end{itemize}
\end{defn}

It is clear that all of $\eF_{A,\tau}$, $\ef_{A,\tau}$ and $\cris_{A,\tau}$ are $\z_p$-linear additive and Karoubian categories. Also, notice that there exists a biadditive, natural, associative, commutative 
and unital $\otimes$-structure on each of them. Observe that $(\cris_{A,\tau},\otimes)$ is rigid, while $(\eF_{A,\tau},\otimes)$ and $(\ef_{A,\tau},\otimes)$ are not. Instead, passage to the dual 
underlying module yields a natural anti-equivalence between $(\eF_{A,\tau},\otimes)$ and $(\ef_{A,\tau},\otimes)$. Finally notice that there exist natural faithful forgetful $\ect(\Spec A)$-valued 
$\otimes$-functors on $\eF_{A,\tau}$, $\ef_{A,\tau}$ and $\cris_{A,\tau}$, and it will not cause confusion to denote all of them by $\omega^A$. For any $r\in\n$ we denote by 
$\cris_{A,\tau}^{W(\f_{p^r})}$ (resp. $\ef_{A,\tau}^{W(\f_{p^r})}$ or $\eF_{A,\tau}^{W(\f_{p^r})}$) the additive and  Karoubian $W(\f_{p^r})$-linear $\otimes$-category consisting of $\tau$-crystals (resp. 
$V$- or $F$-modules) over $(A,\tau)$ which are endowed with a $W(\f_{p^r})$-action. In this paper the rings $A$ will usually have the structure of a $W(\f_{p^f})$-algebra, where $f\in\n$, in which case 
we want to write the composition $\eF_{A,\tau}^{W(\f_{p^r})}\rightarrow\eF_{A,\tau}\stackrel{\omega^A}{\rightarrow}\ect(\Spec A)$ (which is faithful but not a $\otimes$-functor) as a direct sum 
$\bigoplus_{\sigma\in\z/r\z}\omega_\sigma^A$, provided that $r$ divides $f$. Each of these (non-faithful $\otimes$-functors) $\omega_\sigma^A(M)$ arise as the largest subspace of $\omega^A(M)$ 
on which the $W(\f_{p^r})$-operation agrees with the scalar multiplication composed with the map $W(\f_{p^r})\stackrel{F^{-\sigma}}{\rightarrow}W(\f_{p^r})$. A particular class of 
$\eF_{A,\tau}^{W(\f_{p^r})}$-objects arises as follows: Pick a non-negative integer $n$. We want to regard $M_\sigma:=\Mat(n\times 1,A)$ as a $W(\f_{p^r})\otimes_{\z_p}A$-module, by letting $A$ 
act according to the obvious multiplication, while letting $W(\f_{p^r})$ act via the embedding $W(\f_{p^r})\rightarrow A; a\mapsto F^{-\sigma}(a)$. Let $M$ be the $W(\f_{p^r})\otimes_{\z_p}A$-module 
$\bigoplus_{\sigma\in\z/r\z}M_\sigma$. Now, pick arbitrary $B_\sigma\in\Mat(n\times n,A)$, and consider the map
$$A\otimes_{\tau,A}M_{\sigma+1}\rightarrow M_\sigma;a\otimes x\mapsto aB_\sigma\tau(x)$$ 
(this means: compose the effect of plain matrix multiplication with the map obtained by applying the absolute Frobenius to each of the entries in the $n\times1$-matrix $x$). 
The sum of these maps defines a $F$-module structure on $M$. It is clear that a $F$-module with $W(\f_{p^r})$-operation over $A$ arises in this way if and only if all eigenspaces of the 
$W(\f_{p^r})$-operation are free $A$-modules of the same rank. If $M_\sigma'$ denotes another such $F$-module with $W(\f_{p^r})$-operation over $A$, that is gotten in the same way from a 
second bunch of $n\times n$-matrices $B_\sigma'$, having again all of its entries in $A$, then it is easy to see that the isomorphisms $M\rightarrow M'$ are given by families $h_\sigma\in\GL(n,A)$ with 
$B_\sigma=h_\sigma^{-1}B_\sigma'\tau(h_{\sigma+1})$. So the groupoid of $F$-modules with $W(\f_{p^r})$-operation over $A$, all of whose eigenspaces are free of rank $n$ is equivalently given by
$$[\Mat(n\times n,A)^r/_\phi\GL(n,A)^r],$$
(in the sense of example \ref{diagram}) where $\phi$ is the map $(U_0,U_1,\dots,U_{r-1})\mapsto(\tau(U_1),\dots,\tau(U_{r-1}),\tau(U_0))$. In the same 
vein, the groupoid of $\tau$-crystals with $W(\f_{p^r})$-operation over $A$, all of whose eigenspaces are free of rank $n$ is equivalently given by
\begin{equation*}
[\GL(n,A[\frac1p])^r/_\phi\GL(n,A)^r],
\end{equation*}
where $\phi$ is defined in the same way. If $F$ stands for the absolute Frobenius on  the Witt ring over a commutative $\f_{p^f}$-algebra $R$, then we 
denote $\cris_{W(R),F}^{W(\f_{p^r})}$ (resp. $\ef_{W(R),F}^{W(\f_{p^r})}$ or $\eF_{W(R),F}^{W(\f_{p^r})}$) by $\cris^{W(\f_{p^r})}(R)$ (resp. $\ef^{W(\f_{p^r})}(R)$ 
or $\eF^{W(\f_{p^r})}(R)$). As $R$ varies within the category of $\f_{p^f}$-algebras the formation of $\cris^{W(\f_{p^r})}(R)$ (resp. $\ef^{W(\f_{p^r})}(R)$ or 
$\eF^{W(\f_{p^r})}(R)$) builds up natural $\Spec\f_{p^f}$-fibered $W(\f_{p^r})$-linear $\Spec\f_{p^f}$-$\otimes$-categories $\cris^{W(\f_{p^r})}$ (resp. $\ef^{W(\f_{p^r})}$ or 
$\eF^{W(\f_{p^r})}$), in the sense of \cite[I.4.5.5]{rivano}, and again, there are natural forgetful $\Spec\f_{p^f}$-fibered $^W\ect$-valued fiber functors defined on each of 
$\cris^{W(\f_{p^r})}$, $\ef^{W(\f_{p^r})}$ and $\eF^{W(\f_{p^r})}$, and it will not cause confusion to denote all of them by $\omega_\sigma$, for every $\sigma\in\z/r\z$.

\begin{prop}
\label{real}
Let $(\G/W(\f_{p^r}),\{\upsilon_\sigma\}_{\sigma\in\Sigma})$ be a $W(\f_{p^f})$-rational $\Phbar$-datum. Let $\rho:\G\rightarrow\GL(n)_{W(\f_{p^r})}$ 
be a representation and assume that all weights of all of the cocharacters $\rho\circ\upsilon_\sigma^{-1}$ are non-negative 
numbers, so that they extend to effective cocharacters $\beta_\sigma:\a_{W(\f_{p^f})}^1\rightarrow\Mat(n\times n)_{W(\f_{p^f})}$. 
Let $\gamma:\Ibar^{\upsilon_\Sigma}\rightarrow{^W\GL(n)_{\f_{p^f}}^r}$ be the function whose $\omega$th component is given by
\begin{equation*}
\gamma_\omega:\{k_\sigma\}_{\sigma\in\Sigma}\mapsto(^{WF^{-\omega}}\rho)(^{F^{\bd_\Sigma^+(\omega)}}k_{\bd_\Sigma(\omega)}),
\end{equation*}
for any $\omega\in\z/r\z$, and let $m:{^W\G_{\f_{p^f}}^\Sigma}\rightarrow{^W\Mat(n\times n)_{\f_{p^f}}^r}$ be the function whose $\omega$th component is given by
\begin{equation*}
m_\omega:\{U_\sigma\}_{\sigma\in\Sigma}\mapsto
\begin{cases}(^{WF^{-\omega}}\rho)(U_\omega){^{F^{\varpi_\Sigma^+(\omega)}}\beta_{\varpi_\Sigma(\omega)}}(p)&\omega\in\Sigma\\
1&\mbox{ otherwise}\end{cases}
\end{equation*}
for any $\omega\in\z/r\z$. Then there exists a unique $\Spec\f_{p^f}$-fibered functor 
$$\sy^-(\rho):\Barb(\G,\{\upsilon_\sigma\}_{\sigma\in\Sigma})\rightarrow\eF^{W(\f_{p^r})}$$ 
of which the restriction to the groupoid $\BI_R(\G,\{\upsilon_\sigma\}_{\sigma\in\Sigma})$ is given by the 
pair of functions $(\gamma_R,m_R)$ in the sense of example \ref{diagram}, for any $\f_{p^f}$-algebra $R$.
\end{prop}
\begin{proof}
We choose a monotone bijection $\z\rightarrow\Sigma^{(0)};j\mapsto\sigma_j$. In view of 
$\gamma_\omega=F\circ\gamma_{\omega+1}$ for all $\omega\in[\sigma_j+1,\sigma_{j+1}-1]$, all we have to do is prove that 
$\gamma_{\sigma_j}(k)^{-1}m_{\sigma_j}(U){^{F^{\sigma_{j+1}-\sigma_j}}(\gamma_{\sigma_{j+1}}(k))}=m_{\sigma_j}(k^{-1}U\Phbar(k))$ 
holds for elements $k$ of $\Ibar^{\upsilon_\Sigma}$ and $U$ of $^W\G_{\f_{p^f}}^\Sigma$. This follows, if we apply the lemma 
\ref{twistII} to $\beta_{\sigma_{j+1}}$. The reduction to the banal situation is achieved by Witt descent (\cite[Proposition 33]{zink2}).
\end{proof}

\subsection{Two Variants}
\label{realII}
We also have to work with the following convention: If the dual $\check\rho$, rather than $\rho$ satisfies the assumptions of the previous proposition, then we will write 
\begin{equation*}
\sy^+(\rho):\Barb(\G,\{\upsilon_\sigma\}_{\sigma\in\Sigma})\rightarrow\ef^{W(\f_{p^r})}
\end{equation*}
for the composition of the canonical contravariant functor from $\eF^{W(\f_{p^r})}$ to $\ef^{W(\f_{p^r})}$, which is defined by passage to the dual object, the covariant functor 
$\sy^-(\check\rho):\Barb(\G,\{\upsilon_\sigma\}_{\sigma\in\Sigma})\rightarrow\eF^{W(\f_{p^r})}$ and the canonical self-antiequivalence of $\Barb(\G,\{\upsilon_\sigma\}_{\sigma\in\Sigma})$, which is 
defined by reversing the isomorphisms while being the identity on the objects. Suppose that $(\G,\{\upsilon_\sigma\}_{\sigma\in\Sigma})$ is a $\Phi$-datum and that $(A,J,\tau)$ is a frame, where 
$A$ is a $W(\f_{p^f})$-algebra. In this scenario we associate to an arbitrary representation $\rho\in\Ob_{\bRep_0(\G)}$ a covariant functor
\begin{equation}
\label{realIII}
\sy_{A,\tau}(\rho):\hat\CAS_{A,J}(\G,\{\upsilon_\sigma\}_{\sigma\in\Sigma})\rightarrow\cris_{A,\tau}^{W(\f_{p^r})},
\end{equation}
by using the same formulae for $\gamma$ and $m$. 

\subsection{Compatibility with $\omega_\sigma$ and $\otimes$}

For every $\sigma\in\Sigma$ there is a natural commutative diagram:
$$\begin{CD}
\Barb(\G,\{\upsilon_\sigma\}_{\sigma\in\Sigma})@>\sy^-(\rho)>>\eF^{W(\f_{p^r})}\\
@VVV@V{\omega_\sigma}VV\\
\ton(^{WF^{-\sigma}}\G_{\f_{p^f}})@>{^{WF^{-\sigma}}\rho}>>^W\ect\\
\end{CD}$$
and similar ones for $\sy^+$ and $\sy$. Also, there are natural isomorphisms
\begin{equation*}
\sy^-(\rho\otimes_{W(\f_{p^r})}\rho',\P)\cong\sy^-(\rho,\P)\otimes_{W(\f_{p^r})\otimes_{\z_p}W(R)}\sy^-(\rho',\P),
\end{equation*} 
whenever one of the two sides (hence both of them) are well-defined, and similarly for $\sy^+$ and $\sy$.

\subsection{Compatibility with $\fx^{\bj,\{\upsilon_\sigma\}_{\sigma\in\Sigma}}$}
\label{rectify} 

\begin{prop}
\label{klein}
There exists a canonical family of natural transformations:
$$\Fx^{\bj,\{\upsilon_\sigma\}_{\sigma\in\Sigma}}:\sy^+(\sd)\circ\fx^{\bj,\{\upsilon_\sigma\}_{\sigma\in\Sigma}}\rightarrow\sy^+(\sd)$$ 
(indexed by the set of all standard linear $\Phbar$-data $(\GL(n)_{W(\f_{p^r})},\{\upsilon_\sigma\}_{\sigma\in\Sigma})$ together with a function $\bj$ as in subsection 
\ref{weirdVI}) such that the following properties hold, for any $(\GL(n)_{W(\f_{p^r})},\{\upsilon_\sigma\}_{\sigma\in\Sigma})$-display $\P$ over any $\f_{p^f}$-algebra:
\begin{itemize}
\item
Whenever $\P=\P^{(1)}\times\P^{(2)}$ holds for $(\GL(n^{(i)})_{W(\f_{p^r})},\{\upsilon_\sigma^{(i)}\}_{\sigma\in\Sigma})$-displays $\P^{(i)}$, where 
$\left(\begin{matrix}\upsilon_\sigma^{(1)}&0\\
0&\upsilon_\sigma^{(2)}\end{matrix}\right)$ is a matrix block decomposition of $\upsilon_\sigma$, then 
$$\begin{CD}
\sy^+(\sd,\tilde\P)@<<<\sy^+(\sd,\tilde\P^{(1)})\oplus\sy^+(\sd,\tilde\P^{(2)})\\
@V{\Fx_{\P}^{\bj,\{\upsilon_\sigma\}_{\sigma\in\Sigma}}}VV
@V{\Fx_{\P^{(1)}}^{\bj,\{\upsilon_\sigma^{(1)}\}_{\sigma\in\Sigma}}\oplus\Fx_{\P^{(2)}}^{\bj,\{\upsilon_\sigma^{(2)}\}_{\sigma\in\Sigma}}}VV\\
\sy^+(\sd,\P)@<<<\sy^+(\sd,\P^{(1)})\oplus\sy^+(\sd,\P^{(2)})
\end{CD}$$
commutes, where $\tilde\P=\fx^{\bj,\{\upsilon_\sigma\}_{\sigma\in\Sigma}}(\P)$ and $\tilde\P^{(i)}=\fx^{\bj,\{\upsilon_\sigma^{(i)}\}_{\sigma\in\Sigma}}(\P^{(i)})$ (for $i\in\{1,2\}$).
\item
For all standard multiplicative $\Phbar$-data $(\g_{m,W(\f_{p^r})},\{\delta_\sigma\}_{\sigma\in\Sigma})$ and all 
$\Barb(\g_{m,W(\f_{p^r})},\{\delta_\sigma\}_{\sigma\in\Sigma})$-objects $\K$ the map $\Fx_\K^{\bj,\{\delta_\sigma\}_{\sigma\in\Sigma}}$ 
is independent of the choice of $\bj$ (and we will suppress it in the notation).
\item
Whenever $\{\delta_\sigma\}_{\sigma\in\Sigma}$ and $\K$ are as above and $\P'$ is the $\K$-dual of $\P$, then the diagram
$$\begin{CD}
\sy^+(\sd,\tilde\K)@<<<\sy^+(\sd,\tilde\P)\otimes_{W(\f_{p^r})}\sy^+(\sd,\tilde\P')\\
@V{\Fx_{\K}^{\{\delta_\sigma\}_{\sigma\in\Sigma}}}VV
@V{\Fx_{\P}^{\bj,\{\upsilon_\sigma\}_{\sigma\in\Sigma}}\otimes\Fx_{\P'}^{\check\bj,\{\upsilon_\sigma'\}_{\sigma\in\Sigma}}}VV\\
\sy^+(\sd,\K)@<<<\sy^+(\sd,\P)\otimes_{W(\f_{p^r})}\sy^+(\sd,\P')
\end{CD}$$
commutes, where $\tilde\K=\fx^{\{\delta_\sigma\}_{\sigma\in\Sigma}}(\K)$ and 
$\tilde\P'=\fx^{\check\bj,\{\upsilon_\sigma'\}_{\sigma\in\Sigma}}(\P')$ (for $\upsilon_\sigma'=\delta_\sigma\check\upsilon_\sigma$).
\end{itemize}
Finally, the image of $\Fx^{\bj,\{\upsilon_\sigma\}_{\sigma\in\Sigma}}$ under $\omega_\sigma$ is an isomorphism from 
$\omega_\sigma\circ\sy^+(\sd)\circ\fx^{\bj,\{\upsilon_\sigma\}_{\sigma\in\Sigma}}$ to the functor $\omega_\sigma\circ\sy^+(\sd)$ for every $\sigma\in\Sigma$.
\end{prop}
\begin{proof}
It is clear that it is enough to do the banal case, so let $U\in\G^\Sigma(W(R))$ stand for a $(\GL(n)_{W(\f_{p^r})},\{\upsilon_\sigma\}_{\sigma\in\Sigma})$-display 
$\P$ over an affine $\f_{p^f}$-scheme $\Spec R$, and let $\tilde\P$ be $\fx^{\bj,\{\upsilon_\sigma\}_{\sigma\in\Sigma}}(\P)$. Again we choose a monotone bijection 
$\z\rightarrow\Sigma^{(0)};j\mapsto\sigma_j$, so that $U$ can be written as a $z$-tuple $(U_0,\dots,U_{z-1})$ with $U_j\in{^{WF^{-\sigma_j}}\G(R)}$. We switch to the dual 
situation $\sy^-(\check\sd)\rightarrow\sy^-(\check\sd)\circ\fx^{\bj,\{\upsilon_\sigma\}_{\sigma\in\Sigma}}$, and we write $\beta_\sigma$ and $\tilde\beta_\sigma$ for the two 
$r$-tuples of  effective cocharacters that come up when applying the proposition \ref{real} to $\P$ and $\tilde\P$. It does no harm to think of $\check\sd$ as the standard involution 
\eqref{involution}, in particular one sees that the cocharacters $\tilde\beta_\sigma$ arise from the cocharacters $\beta_\sigma$ by the procedure described in \eqref{soul}. Finally, 
let $B_\sigma$ and $\tilde B_\sigma$ be the two $r$-tuples of matrices that come up when applying the proposition \ref{real} to $\P$ and $\tilde\P$. The former looks like: 
$$B_\sigma=\dots,\check U_{j-1}(^{F^{\sigma_j-\sigma_{j-1}}}\beta_{\sigma_j})(p),\dots,1,\check U_j(^{F^{\sigma_{j+1}-\sigma_j}}\beta_{\sigma_{j+1}})(p),\dots,$$
where the ``$1$'' is in the $\sigma_j-1$th position, and the latter looks like:
$$\tilde B_\sigma=\dots,\check U_{j-1}(^F\tilde\beta_{\sigma_{j-1}+1})(p),\dots,(^F\tilde\beta_{\sigma_j})(p),\check U_j(^F\tilde\beta_{\sigma_j+1})(p),\dots.$$
Let us define $r$-tuples of matrices by $k_\sigma=\prod_{\omega=\sigma+1}^{\sigma_j}(^{F^{\omega-\sigma}}\tilde\beta_\omega)(p)$ whenever 
$\sigma_{j-1}+1\leq\sigma\leq\sigma_j$, and notice that $\beta_{\sigma_j}=\prod_{\omega=\sigma_{j-1}+1}^{\sigma_j}(^{F^{\omega-\sigma_j}}\tilde\beta_\omega)$
holds, because of \eqref{heart}. It follows that $k_\sigma B_\sigma=\tilde B_\sigma{^Fk_{\sigma+1}}$ is true and we are done, the bijectivity 
of $\omega_{\sigma_j}(\Fx^{\bj,\{\upsilon_\sigma\}_{\sigma\in\Sigma}})$ follows from $k_{\sigma_j}=1$.
\end{proof}

\subsection{Faithfulness of some realizations}
\label{faith}

\begin{defn}
\label{exampleIX}
Let $B$ be a discrete valuation ring of mixed characteristic or a field of characteristic $0$, let $\G$ be a reductive group scheme over $B$ and let $C$ be a finite and \'etale 
$B$-algebra of degree $2$. Fix a subset $\Pi$ of the power set of some index set $\Lambda$ where $\Card(\Lambda)<\infty$ and $\Card(\pi)\equiv1\pmod2$ for all $\pi\in\Pi$. By 
a $C$-multi-unitary collection for $\G/B$ we mean a pair $\bC$ of families $\{(\V_i/C,\Psi_i,\rho_i)\}_{i\in\Lambda}$ and $\{(\R_\pi,\iota_\pi)\}_{\pi\in\Pi}$ with the following properties:
\begin{itemize}
\item[(i)]
Each $\R_\pi$ is a free $C$-algebra of finite rank which is equipped with an involution fitting into the diagram
$$\begin{CD}
\R^{op}@>*>>\R\\
@AAA@AAA\\
C@>{\bar\empty}>>C
\end{CD},$$
where the involution $\bar\empty$ is given by: $\xbar:=\tr_{C/C^+}(x)-x$
\item[(ii)]
Each $-\check\Bapsi_i=\Psi_i:\V_i\stackrel{\cong}{\rightarrow}\check{\overline\V_i}$ is a perfect skew-Hermitian pairing on some free $C$-module $\V_i$ of finite rank.
\item[(iii)]
Each $\rho_i$ is a $C$-unitary representation, i.e. a homomorphism from $\G$ to $\GU(\V_i/C,\Psi_i)$ and each $\iota_\pi$ 
is a $C$-linear $*$-preserving homomorphism from $\R_\pi$ to $\End_\G(\V^\pi)$ where $\V^\pi$ stands for the $C$-module 
$\bigotimes_{i\in\pi}\V_i$ (which is equipped with self-explanatory $\G$-action $\rho^\pi$ and perfect skew-Hermitian pairing $\Psi^\pi$)
\item[(iv)]
The product 
$$\rho_\Lambda:=\prod_{i\in\Lambda}\rho_i:\G\rightarrow\prod_{i\in\Lambda}\GU(\V_i/C,\Psi_i)$$
is a closed immersion.
\item[(v)]
The (pointwise) stabilizer of $\bigcup_{\pi\in\Pi}\R_\pi$ in the generic fiber of $\prod_{i\in\Lambda}\GU(\V_i/C,\Psi_i)$ is contained in the image of $\rho_\Lambda$.
\end{itemize}
\end{defn}

\begin{rem}
\label{hyperspecial}
For any $(\{(\V_i/C,\Psi_i,\rho_i)\}_{i\in\Lambda},\{(\R_\pi,\iota_\pi)\}_{\pi\in\Pi})$, as above we would like to point out the $1$-morphism
\begin{equation}
\label{exampleXII}
\ton(\G)\rightarrow\ect^\bC
\end{equation}
where the groupoid of $R$-valued sections of the fpqc-stack $\ect^\bC$ consists of (the natural groupoid structure on the class of) pairs $(\T,\{t_\pi\}_{\pi\in\Pi})$, where $\T$ is a 
$\prod_{i\in\Lambda}\GU(\V_i/C,\Psi_i)$-torsor over $R$ and each $t_\pi$ is a $C$-linear $*$-preserving homomorphism $t_\pi:\R_\pi\rightarrow\End_R(\omega_\T(\rho^\pi))$.
If one assumes $\frac12\in B$ and that $\G$ is reductive, then $\rho_\Lambda$ is a closed immersion if and only if its generic fiber has that property (\cite[Corollary 1.3]{prasad}).
\end{rem}

In the above scenario we write $\chi_i:\G\rightarrow\g_{m,B}$ for the composition of the multiplier character of $\GU(\V_i/C,\Psi_i)$ with $\rho_i$. It is easy to see that 
$\G^1:=\bigcap_{i\in\Lambda}\ker(\chi_i)$ is a smooth $B$-group, in fact we will frequently use that $\Res_{C/B}\g_{m,C}^\Lambda$ is canonically immersed into the center 
of $\G$, while $\g_{m,B}^\Lambda$ intersects $\G^1$ in $\{\pm1\}^\Lambda$, so that $\G$ is canonically isomorphic to $(\g_{m,B}^\Lambda\times_B\G^1)/\{\pm1\}^\Lambda$. 
Furthermore, the properties (iii), (iv) and (v) imply that $\G$ (resp. $\G_C^1$) agrees with the Zariski-closure of its generic fiber in $\prod_{i\in\Lambda}\GU(\V_i/C,\Psi_i)$ 
(resp. in $\prod_{i\in\Lambda}\GL(\V_i/C)$), which in turn implies that one can pin down a constant $c$ such that $\G$ (resp. $\G_C^1$) is the stabilizer in 
$\prod_{i\in\Lambda}\GU(\V_i/C,\Psi_i)$ (resp. in $\prod_{i\in\Lambda}\GL(\V_i/C)$) of $\O_{Z_c}=\End_G(\bigotimes_{i\in\Lambda}\V_i^{\otimes c})$, cf. \cite[Proposition(1.3.2)]{kisin}. 
We focus on the case $B=W(\f_{p^r})$ with $C$ being one of $W(\f_{p^{2r}})$ or $W(\f_{p^r})\oplus W(\f_{p^r})$.

\subsection{Definition of $\fx^\bJ$}
\label{exampleXIV}
From now on, and for the rest of this section we fix a $W(\f_{p^f})$-rational $\Phi$-datum $(\G,\{\upsilon_\sigma\}_{\sigma\in\Sigma})$, 
choose a monotone bijection $\z\rightarrow\Sigma^{(0)};j\mapsto\sigma_j$, and let $C$ be as in the end of the previous subsection.

\begin{defn}
Let $\{\bj_i\}_{i\in\Lambda}$ be a family of functions from $\z$ to $\n_0$, such that each $\bj_i$ maps every interval $[\sigma_{j-1}+1,\sigma_j]$ onto $[0,\sigma_j-\sigma_{j-1}-1]$ 
and satisfies $\bj_i(r+\sigma)=\begin{cases}\bj_i(\sigma)&C=W(\f_{p^r})\oplus W(\f_{p^r})\\r_\Sigma(\bd_\Sigma(\sigma))-\bj_i(\sigma)-1&C=W(\f_{p^{2r}})\end{cases}$ 
for every $\sigma\in\z/r\z$ (cf. subsubsection \ref{Universum}). We say that 
$$\bJ=(\{(\V_i/C,\Psi_i,\rho_i,\bj_i)\}_{i\in\Lambda},\{(\R_\pi,\iota_\pi)\}_{\pi\in\Pi}),$$ 
is a gauged $C$-multi-unitary collection for $(\G,\{\upsilon_\sigma\}_{\sigma\in\Sigma})$ if the following additional requirements are fulfilled:
\begin{itemize}
\item
$\bC=(\{(\V_i/C,\Psi_i,\rho_i)\}_{i\in\Lambda},\{(\R_\pi,\iota_\pi)\}_{\pi\in\Pi})$ is a $C$-multi-unitary collection for the group $\G$ in the 
sense of definition \ref{exampleIX} and each $\rho_i$ induces a morphism from $(\G,\{\upsilon_\sigma\}_{\sigma\in\Sigma})$ to a certain 
standard unitary $\Phbar$-datum $(\GU(\V_i/C,\Psi_i),\{\upsilon_{i,\sigma}\}_{\sigma\in\Sigma})$ in the sense of definition \ref{standard}.
\item
$\bR=(\{(\V_i/C,\Psi_i,\tilde\upsilon_i)\}_{i\in\Lambda},\{\R_\pi\}_{\pi\in\Pi})$ is a multicompact $C$-multi-skew Hermitian display 
datum in the sense of definition \ref{exampleX}, where $\tilde\upsilon_i$ is derived from $\upsilon_i$ and $\bj_i$ via \eqref{soul}, 
\item
If $R_\Sigma:=\max\{r_\Sigma(\sigma)|\sigma\in\z/r\z\}$, then one has $\R_\pi\subset C+p^{(R_\Sigma-1)\Card(\pi)+1}\End_C(\V^\pi)$.
\end{itemize}
\end{defn}

We continue to denote $\Res_{W(\f_{p^r})/\z_p}\G$ by $\gG$. The purpose of introducing gauged $C$-multi-unitary collections lies in a certain $2$-commutative diagram:
$$\begin{CD}
\gB^{\bR}@>>>\prod_{i\in\Lambda}\B(\GU(\V_i/C,\Psi_i),\{\tilde\upsilon_{i,\sigma}\}_{\sigma\in\z/r\z})\\
@A{\fx^\bJ}AA@A{\prod_{i\in\Lambda}\fx^{\bj_i,\{\upsilon_{i,\sigma}\}_{\sigma\in\Sigma}}}AA\\
{\Barb(\G,\{\upsilon_\sigma\}_{\sigma\in\Sigma})}@>{\Barb(\rho_\Lambda)}>>\prod_{i\in\Lambda}\Barb(\GU(\V_i/C,\Psi_i),\{\upsilon_{i,\sigma}\}_{\sigma\in\Sigma})
\end{CD},$$
where $\gB^{\bR}$ is the $W(\f_{p^f})$-stack which was introduced in definition \ref{exampleX}. Here is the construction of 
$\fx^\bJ$: We start out from $\Barb(\rho_i,\P)=:\P_i\in\Ob_{\Barb(\GU(\V_i/C,\Psi_i),\{\upsilon_{i,\sigma}\}_{\sigma\in\Sigma})(R)}$ and 
$\Barb(\chi_i,\P)=:\K_i\in\Ob_{\Barb(\g_{m,W(\f_{p^r})},\{\delta_\sigma\}_{\sigma\in\Sigma})(R)}$, which are gotten from some fixed 
$\Spec R\stackrel{\P}{\rightarrow}\Barb(\G,\{\upsilon_\sigma\}_{\sigma\in\Sigma})$, by extensions of its structure group. On the one hand each of the 
representations $\rho_i$ and $\chi_i$ gives rise to a graded realization within the categories $\ef^{W(\f_{p^{2r}})}(R)$ or $\ef^{W(\f_{p^r})}(R)$, together 
with the canonical $\sy^+(\chi_i,\P)$-valued sesquilinear  perfect pairing on $\sy^+(\rho_i,\P)$, which is induced from the isomorphism 
$\chi_i\otimes\check\bohr_i\cong\rho_i$. On the other hand, we may look at the regularizations 
\begin{eqnarray*}
&&\tilde\P_i:=\fx^{\bj_i,\{\upsilon_{i,\sigma}\}_{\sigma\in\Sigma}}(\P_i)\\
&&\tilde\K_i:=\fx^{\{\delta_\sigma\}_{\sigma\in\Sigma}}(\K_i),
\end{eqnarray*}
which are displays in the usual sense whose underlying graded modules are given by $\sy^+(\sd,\tilde\P_i)$ 
and $\sy^+(\sd,\tilde\K_i)$. As we have seen already the results are related by specific isogenies
\begin{eqnarray*}
&&\Fx_{\P_i}^{\bj_i,\{\upsilon_{i,\sigma}\}_{\sigma\in\Sigma}}:\sy^+(\sd,\tilde\P_i)\rightarrow\sy^+(\rho_i,\P)\\
&&\Fx_{\K_i}^{\{\delta_\sigma\}_{\sigma\in\Sigma}}:\sy^+(\sd,\tilde\K_i)\rightarrow\sy^+(\chi_i,\P),
\end{eqnarray*}
preserving the $V$-actions, the gradations, and the canonical sesquilinear perfect pairings, which are defined in both of the two scenarios. It is easy to see that 
there exists a map from $\sy^+(\chi_i,\P)$ to $\sy^+(\sd,\tilde\K_i)$ of which the composition with $\Fx_{\K_i}^{\{\delta_\sigma\}_{\sigma\in\Sigma}}$  in any order 
is equal to the multiplication by $p^{R_\Sigma-1}$. Whence one obtains a map from $\sy^+(\rho_i,\P)$ back to $\sy^+(\sd,\tilde\P_i)$ of which the composition 
and precomposition with $\Fx_{\P_i}^{\bj_i,\{\upsilon_{i,\sigma}\}_{\sigma\in\Sigma}}$ yields the $p^{R_\Sigma-1}$th multiples of the identities of $\sy^+(\rho_i,\P)$ 
and $\sy^+(\sd,\tilde\P_i)$. The same reasoning can be applied to tensor products of these maps giving rise to gradation and $V$-preserving maps: 
$$\sy^+(\sd,\tilde\P^\pi)(\frac{\Card(\pi)-1}2)\rightleftarrows\sy^+(\rho^\pi,\P)$$
whose products are $p^{(R_\Sigma-1)\Card(\pi)}$, where $\tilde\P^\pi$ is defined to be the restricted 
tensor product $\dot\bigotimes_{i\in\pi}\tilde\P_i$ and $\rho^\pi:=\bigotimes_{i\in\pi}\rho_i$. In total there is an action of 
$p^{(R_\Sigma-1)\Card(\pi)}\End_\G(\V^\pi)$ on $\sy^+(\sd,\tilde\P^\pi)$ and hence there is an action of $\R_\pi$ on $\tilde\P^\pi$, since we have the following:

\begin{lem}
Let $\P=(M,N,F,V^{-1})$ be a display over a $\f_p$-algebra $R$, and suppose that $h$ is a $W(R)$-linear endomorphism of $M$, that renders at least one of the diagrams
$$\begin{CD}
W(R)\otimes_{F,W(R)}M@>{F^\sharp}>>M\\
@A{\id_{W(R)}\otimes h}AA@AhAA\\
W(R)\otimes_{F,W(R)}M@>{F^\sharp}>>M\\
\end{CD}
\qquad\text{or  }\,\,\,\,\,\,\,
\begin{CD}
M@>{V^\sharp}>>W(R)\otimes_{F,W(R)}M\\
@AhAA@A{\id_{W(R)}\otimes h}AA\\
M@>{V^\sharp}>>W(R)\otimes_{F,W(R)}M\\
\end{CD}$$
commutative. Then the $W(R)$-linear map $s:M\rightarrow M;x\mapsto ph(x)$ is an endomorphism of $\P$.
\end{lem}
\begin{proof}
By duality we are allowed to assume that $h$ commutes with $F$, it is clear that $s$ preserves 
$N$, and the proof is finished by $V^{-1}(sx)=F(hx)=hF(x)=sV^{-1}(x)$ for every $x\in N$.
\end{proof}

\subsection{Local properties of $\fx^\bJ$}

\begin{lem}
\label{faithfulIII}
The $1$-morphism $\fx^\bJ$ is schematic, quasicompact and radicial for every gauged $C$-multi-unitary collection $\bJ$. 
\end{lem}
\begin{proof}
Notice, that the upper horizontal $1$-morphism in the diagram at the beginning of subsection \ref{exampleXIV} is separated. So in order to obtain the quasicompact schematicness 
$\fx^\bJ$ it remains to prove the same property for $\Barb(\rho_\Lambda)$, since we have already shown proposition \ref{new}. In fact, this map is affine, which can be deduced from 
the well-known affinity of the quotient of the product $\prod_{i\in\Lambda}\GU(\V_i/C,\Psi_i)$ by the reductive subgroup $\G$ together with the methods of \cite[Lemma 3.2.9a)]{pappas}. 
The radiciality of $\fx^\bJ$ is equivalent to the full faithfulness of the induced functors between the groupoids of $k$-valued sections, where $k$ is an arbitrary algebraically 
closed field. Let $\cry_k$ be the $\q_p$-linear tannakian category of $F$-isocrystals over $k$ and let us write $\Bb_k(\gG)$ for (the natural groupoid structure on) the class of 
$F$-isocrystals with $\gG$-structure over $k$ (i.e. $\q_p$-linear exact and faithful rigid $\otimes$-functors from $\bRep_0(\gG)$ to $\cry_k$). Let $\Bb_k^{\bR}$ denote (the 
natural groupoid structure on) the class of pairs $(M,\{m_\pi\}_{\pi\in\Pi})$ where $M$ is an object of $\Bb_k(\prod_{i\in\Lambda}\Res_{W(\f_{p^r})/\z_p}\GU(\V_i/C,\psi_i))$ 
while each $m_\pi$ is a $C$-linear $*$-preserving homomorphism from $\R_\pi$ to $M(\End_C(\V^\pi))$, which is a $C$-algebra object in $\cry_k$. The $W(\f_{p^f})$-algebra 
$W(\f_{p^f})\otimes_{F^{-\sigma},W(\f_{p^r})}C$ shall be denoted by $C_\sigma$. Consider the $C_\sigma$-multi-unitary collection $\bC_\sigma$ which is obtained by base 
change from $\bC$ via $F^{-\sigma}:W(\f_{p^r})\rightarrow W(\f_{p^f})$, where $\sigma$ runs through $\z/r\z$. According to \eqref{exampleXII} there is an associated morphism:
\begin{equation}
\label{exampleXIII}
\ton({^{F^{-\sigma}}\G_{W(\f_{p^f})}})\rightarrow\ect^{\bC_\sigma}
\end{equation} 
Moreover, consider the (partial) products $\gV:=\prod_{\sigma\in\z/r\z}\ect^{\bC_\sigma}$ (resp. 
$\gV^\Sigma:=\prod_{\sigma\in\Sigma}\ect^{\bC_\sigma}$) and the natural $2$-commutative diagram:
\[
\begin{tikzcd}
{\BI_k(\G,\{\upsilon_\sigma\}_{\sigma\in\Sigma})}\ar[dddd]\ar[rrr,"\fx^\bJ"]\arrow[dr]&&&{\gB^{\bR}(k)}\arrow[dl]\ar[dd]\\
&{\Bb_k(\gG)}\ar[r]\ar[d]&{\Bb_k^{\bR}}\ar[d]\\
&{\ton(\gG)(K(k))}\ar[d]\ar[r]&{\gV(K(k))}\ar[d]&{\gV(W(k))}\ar[l]\ar[dd]\\
&{\ton(\G^\Sigma)(K(k))}\ar[r]&{\gV^\Sigma(K(k))}\\
{\ton(\G^\Sigma)(W(k))}\arrow[ur]\ar[rrr]&&&{\gV^\Sigma(W(k))}\arrow[ul]
\end{tikzcd}
\]
The requested full faithfulness of the upper horizontal arrow follows from the following observations:
\begin{itemize}
\item
The natural functor $\Bb_k(\gG)\rightarrow\Bb_k^{\bR}$ is fully faithful, because the generic fiber of $\gG$ is the stabilizer of $\bigcup_{\pi\in\Pi}\R_\pi$.
\item
The $1$-morphism $\ton(\G^\Sigma)(W(k))\rightarrow\gV^\Sigma(W(k))$ is fully faithful, because $\G$ is a closed subgroup of $\prod_{i\in\Lambda}\GU(\V_i/C,\Psi_i)$.
%\item The $1$-morphism $\ton(\G^\Sigma)(K(k))\rightarrow\gV^\Sigma(K(k))$ is faithful.
\item
The $1$-morphism 
$$\BI_k(\G,\{\upsilon_\sigma\}_{\sigma\in\Sigma})\rightarrow\Bb_k(\gG)\times_{\ton(\G^\Sigma)(K(k))}\ton(\G^\Sigma)(W(k))$$ 
is fully faithful, which is due to: 
$$\G^\Sigma(W(k))\cap\upsilon_\Sigma(p)\G^\Sigma(W(k))\upsilon_\Sigma(\frac1p)=\I^{\upsilon_\Sigma}(k)$$
\end{itemize}

\end{proof}

\subsection{Finiteness properties of $\fx^\bJ$}

\begin{lem}
\label{properIV}
Fix a $\f_{p^f}$-scheme $X$, and a $X$-scheme $Y$. Assume that 
$$\begin{CD}
Y@>>>\Barb(\G,\{\upsilon_\sigma\}_{\sigma\in\Sigma})\\
@VVV@V{\fx^\bJ}VV\\
X@>{\S}>>\gB^\bR
\end{CD},$$
is a $2$-commutative diagram.
\begin{itemize}
\item[(i)]
Assume in addition, that $Y$ is reduced. Then there is at most one such $2$-commutative diagram (i.e. there exist unique isomorphisms to any other ones) 
\item[(ii)]
Assume in addition, that $Y$ is faithfully flat and quasicompact over $X$. If $Y\times_XY$ is reduced, then $\S$ factors through $\Barb(\G,\{\upsilon_\sigma\}_{\sigma\in\Sigma})$.
\end{itemize}
\end{lem}
\begin{proof}
If $Z\rightarrow X$ is a radicial morphisms of schemes, then its relative diagonal induces an isomorphism from $Z_{red}$ to $(Z\times_XZ)_{red}$, 
as $\Delta_{Z/X}:Z\rightarrow Z\times_XZ$ is a surjection. Therefore part (i) follows easily and part (ii) follows from descent theory and part (i).
\end{proof}

Our next result in this subsection deals with fields:

\begin{prop}
\label{properI}
Consider an extension $\f_{p^f}\subset k$ of fields with finite degree of imperfection, and fix a $1$-morphism $\Spec k\stackrel{\S}{\rightarrow}\gB^\bR$. If the $k$-scheme 
$$\Barb(\G,\{\upsilon_\sigma\}_{\sigma\in\Sigma})\times_{\fx^\bJ,\gB^\bR,\S}\Spec k$$ 
is non-empty, then it possesses a point over a finite extension of $k$.
\end{prop}
\begin{proof}
By lemma \ref{properIV}, it is clear that we may assume $k$ to be separably closed, and due to the same reason we are allowed to assume that there exists 
a display $\P^\circ$ with $\Barb(\G,\{\upsilon_\sigma\}_{\sigma\in\Sigma})$-structure over the algebraic closure $k^{ac}$, together with a $2$-isomorphism: 
$$\zeta^\circ:\S_{k^{ac}}\stackrel{\cong}{\rightarrow}\fx^\bJ(\P^\circ)$$
By lemma \ref{galoisII} we may assume $C=W(\f_{p^r})\oplus W(\f_{p^r})$, in which case a simple application of Morita 
equivalence tells us that our ``unitary'' groups $\GU(\V_i/C,\Psi_i)$ are isomorphic to $\g_m\times\GL(n_i)$ and that our "unitary" representations read 
$\rho_i\oplus\chi_i\otimes_{W(\f_{p^r})}\check\rho_i$ for some $W(\f_{p^r})$-rational homomorphisms $\rho_i:\G\rightarrow\GL(n_i)_{W(\f_{p^r})}$ and 
characters $\chi_i:\G\rightarrow\g_{m,W(\f_{p^r})}$. In the 
proof we will work directly with $\rho_i$ and $\chi_i$ rather than $\rho_i\oplus\chi_i\otimes_{W(\f_{p^r})}\check\rho_i$. Likewise, a $k$-valued 
point of $\gB^\bR$ is determined by a family $s_\pi$ of $\R_\pi$-operations on the family of restricted tensor products 
$\S^\pi:=\dot\bigotimes_{i\in\pi}\S_i$ coming from some other family $\S_i$ of displays with 
$(\GL(n_i)_{W(\f_{p^r})},\{\tilde\upsilon_{i,\sigma}\}_{\sigma\in\z/r\z})$-structure over $k$, together with some multiplicative family of no significance. In this 
optic $\zeta^\circ$ is given by a family of $\{s_\pi\}_{\pi\in\Pi}$-preserving isomorphisms
$$\zeta_i^\circ:\S_{i,k^{ac}}\stackrel{\cong}{\rightarrow}\tilde\P_i^\circ:=\fx^{\bj_i,\{\upsilon_{i,\sigma}\}_{\sigma\in\Sigma}}(\P_i^\circ),$$ 
where $\P_i^\circ$ stands for the image of $\P^\circ$ under the canonical $1$-morphism
$$\Barb(\rho_i):\Barb(\G,\{\upsilon_\sigma\}_{\sigma\in\Sigma})\rightarrow\Barb(\GL(n_i)_{W(\f_{p^r})},\{\upsilon_{i,\sigma}\}_{\sigma\in\Sigma}).$$
Recall that $\sy^+(\sd,\P_i^\circ)=\sy^+(\rho_i,\P^\circ)=:N_i^\circ=\bigoplus_{\sigma\in\z/r\z}N_{i,\sigma}^\circ$ is a $\z/r\z$-graded $V$-module over $W(k^{ac})$ 
canonically associated to $\P^\circ$, and in the category of $\z/r\z$-graded $V$-modules over $W(k^{ac})$ we have a canonical $\z/r\z$-graded isogeny: 
$$\Fx_{\P_i^\circ}^{\bj_i,\{\upsilon_{i,\sigma}\}_{\sigma\in\Sigma}}:\sy^+(\sd,\tilde\P_i^\circ)\rightarrow N_i^\circ$$
Let $\tau$ be a choice of Frobenius lift on a choice of Cohen ring $A$ for the (in general non-perfect) 
field $k$, so that $(A,pA,\tau)$ is a frame over $W(\f_{p^f})$. We have a canonical commutative diagram:
$$\begin{CD}
W(\f_{p^f})@>>>A@>>>W(k)@>>>W(k^{ac})\\
@A{F}AA@A{\tau}AA@A{F}AA@A{F}AA\\
W(\f_{p^f})@>>>A@>>>W(k)@>>>W(k^{ac})
\end{CD}$$
Let $M_i$ denote the $\z/r\z$-graded $(A,pA,\tau)$-windows to $\S_i$. Composition produces further isogenies
$$W(k^{ac})\otimes_AM_i\stackrel{u_i^\circ}{\rightarrow}N_i^\circ,$$ 
paving the way for a transport of structure, that turns the natural action $\iota$ of $\O_{Z_c}$ on the $\cris^{W(\f_{p^r})}(k^{ac})$-object 
$\sy^+((\bigotimes_{i\in\Lambda}\rho_i)^{\otimes_{W(\f_{p^r})}c},\P^\circ)=:N^\circ$, whose $\sigma$-eigenspaces are 
$N_\sigma^\circ=(\bigotimes_{i\in\Lambda}N_{i,\sigma}^\circ)^{\otimes_{W(k^{ac})}c}$, into its analog $s:Z_c\rightarrow\q\otimes\End_{W(\f_{p^r})}(W(k^{ac})\otimes_AM)$ 
for the $\cris_{A,\tau}^{W(\f_{p^r})}$-object $M=\bigoplus_{\sigma\in\z/r\z}M_\sigma$ defined by the formula 
$$M_\sigma:=(\bigotimes_{i\in\Lambda}M_{i,\sigma})^{\otimes_Ac}$$ 
(within which $M_i$ is to be interpreted in $\cris_{A,\tau}^{W(\f_{p^r})}$). We are going to use, and will now have to check, that $s$ yields an action on $\q\otimes M$, 
rather than $K(k^{ac})\otimes_AM$: To this end we write $H_\sigma/A[\frac 1p]$ for the common stabilizer of the family of $\R_\pi$-actions $s_\pi$, this is a smooth and 
affine $A[\frac 1p]$-group contained in the product of the groups $\GL(\q\otimes M_i/A[\frac 1p])$. If $\Z_\sigma$ denotes the $A$-algebra $\End_{H_\sigma}(M_\sigma)$ 
of $H_\sigma$-invariant, $A$-linear endomorphisms of $M_\sigma$, then $\bigoplus_{\sigma\in\z/r\z}\Z_\sigma$ forms a $\z/r\z$-graded $\tau$-crystal over $A$, and the 
composition of endomorphisms makes a $\cris_{A,\tau}^{W(\f_{p^r})}$-morphism. In contrast, $\bigoplus_{\sigma\in\z/r\z}\End_{^{F^{-\sigma}}\G_{W(k^{ac})}}(N_\sigma^\circ)$ 
is a $\cris^{W(\f_{p^r})}(k^{ac})$-object of slope $0$, in fact it is equal to $W(k^{ac})\otimes_{\z_p}\O_{Z_c}$ with Frobenius acting in the obvious way. Comparing these, 
we find the existence of at least one Frobenius-invariant $A$-lattice in $\q\otimes\bigoplus_{\sigma\in\z/r\z}\Z_\sigma$, leading the latter to be generated by its skeleton 
$S:=(\q\otimes\bigoplus_{\sigma\in\z/r\z}\Z_\sigma)^{F=\id}$ (N.B.: the residue field of $A$ is separably closed). This shows that the image of $s$ is contained in 
$\q\otimes\End_{W(\f_{p^r})}(M)$, as the skeleton of $K(k^{ac})\otimes_A\bigoplus_{\sigma\in\z/r\z}\Z_\sigma$ is clearly also equal to $S$.\\ 
Notice, that every element $\{U_\sigma\}_{\sigma\in\Sigma}=U\in\G^\Sigma(W(k^{ac}))$ representing the object $\P^\circ$ (i.e. a choice of banalization) will induce specific 
isomorphisms $W(k^{ac})^{n_i}\cong N_{i,\sigma}^\circ$, and due to the algebraic closedness of $k^{ac}$, there exist a lot of banalizations. Moreover observe, that the images 
of $A^{n_i}$ under those very isomorphisms provide us with non-canonical $A$-module structures $N_{i,\sigma}\subset N_{i,\sigma}^\circ$ on the $W(k^{ac})$-modules 
$N_{i,\sigma}^\circ$, which are canonically associated to $\P^\circ$. Also notice, that a switch between two different banalizations, has the effect of altering the artificial 
$A$-structures on the canonical $W(k^{ac})$-modules $N_{i,\sigma}^\circ$ according to the formulae in proposition \ref{real}. When looking more specifically at 
$\sigma\in\Sigma$ a switch from, say $U$, to an alternative representative given by, say $O:=h^{-1}U\Phbar(h)$, has the effect of switching from $N_{i,\sigma}$ to 
$(^{F^{-\sigma}}\rho_i)(h_\sigma)^{-1}N_{i,\sigma}$, where $h=\{h_\sigma\}_{\sigma\in\Sigma}$. For the time being we do fix $U$ and the artificial $A$-structures 
$N_{i,\sigma}$ going with it.\\
The isogeny $u_i^\circ$ induces isomorphisms $u_{i,\sigma}^\circ:W(k^{ac})\otimes_AM_{i,\sigma}\stackrel{\cong}{\rightarrow}N_{i,\sigma}^\circ$ for all $\sigma\in\Sigma$ 
being due to an analogous property of $\Fx_{\P_i^\circ}^{\bj_i,\{\upsilon_{i,\sigma}\}_{\sigma\in\Sigma}}$ (cf. proposition \ref{klein}). This demonstrates already, that the 
$Z_c$-action $s$ on the isocrystal $\q\otimes M$ can be restricted to an integral action of $\O_{Z_c}$ on the $A$-module $M_\sigma$, provided that $\sigma$ is contained in 
$\Sigma$. We will shortly derive an even stronger result, but before we do that we have to introduce a certain principal homogeneous space $\gS_\sigma$ for $^{F^{-\sigma}}\G$ 
over $A$, for every $\sigma\in\Sigma$: If $Q$ is a $A$-algebra, then we let $\gS_\sigma(Q)$ be the set of families $\{q_i\}_{i\in\Lambda}$ with the following two properties: 
\begin{itemize}
\item
Each $q_i$ is a $Q$-linear isomorphism from $Q\otimes_AM_{i,\sigma}$ to $Q\otimes_AN_{i,\sigma}$.
\item
The isomorphism $(\bigotimes_{i\in\Lambda}q_i)^{\otimes c}=q:Q\otimes_AM_\sigma\rightarrow Q\otimes_AN_\sigma$ is $Q\otimes_{W(\f_{p^r})}\O_{Z_c}$-linear
\end{itemize}
Clearly, one has $\{u_{i,\sigma}^\circ\}_{i\in\Lambda}\in\gS_\sigma(W(k^{ac}))$ for every $\sigma\in\Sigma$. Since $W(k^{ac})$ is faithfully flat over $A$, we get the local triviality of 
the $\gS_\sigma$'s. As $A$ is strictly henselian we get the existence of global sections, i.e. some $A\otimes_{W(\f_{p^r})}\O_{Z_c}$-preserving family of $A$-linear isomorphisms
$$u_{i,\sigma}:M_{i,\sigma}\stackrel{\cong}{\rightarrow}N_{i,\sigma},$$
for each $\sigma\in\Sigma$. As $i$ varies the family $u_{i,\sigma}^\circ\circ u_{i,\sigma}^{-1}$ constitutes a $W(k^{ac})$-valued point of $^{F^{-\sigma}}\G$. One 
can choose a subfield $l\subset k^{ac}$ of finite degree over $k$, together with elements $\{h_\sigma^\circ\}_{\sigma\in\Sigma}=h^\circ\in\G^\Sigma(W(l))$ such that 
$$u_{i,\sigma}^\circ\circ u_{i,\sigma}^{-1}\equiv(^{F^{-\sigma}}\rho_i)(h_\sigma^\circ)\pmod p$$ 
and let $h_\sigma\in{^{F^{-\sigma}}\G(pW(k^{ac}))}$ be the elements with 
$(^{F^{-\sigma}}\rho_i)(h_\sigma^{-1}\circ h_\sigma^\circ)=u_{i,\sigma}^\circ\circ u_{i,\sigma}^{-1}$. Observe that the diagonal map in the diagram 
$$\begin{CD}
W(k^{ac})\otimes_AN_{i,\sigma}
@>{^{F^{-\sigma}}\rho_i(h_\sigma)}>>W(k^{ac})\otimes_AN_{i,\sigma}\\
@A{u_{i,\sigma}^\circ}AA@A{^{F^{-\sigma}}\rho_i(h_\sigma^\circ)}AA\\
M_{i,\sigma}@>{u_{i,\sigma}}>>N_{i,\sigma}
\end{CD}$$
carries $M_{i,\sigma}$ to $W(l)\otimes_AN_{i,\sigma}$, furthermore by the aforementioned remark on banalizations, we may assume $h_\sigma=1$ without loss 
of generality. We deduce that $u_{i,\sigma}^\circ$ carries $M_{i,\sigma}$ to $W(l)\otimes_AN_{i,\sigma}$, so that the vertical arrows in the commutative diagram
$$\begin{CD}
W(k^{ac})\otimes_AN_{i,\sigma}@>{(^{F^{\varpi_\Sigma^+(\sigma)}}\upsilon_{\varpi_\Sigma(\sigma)})(p)(^{F^{-\sigma}}\rho_i)(U_\sigma^{-1})}>>
W(k^{ac})\otimes_{\tau^{\varpi_\Sigma^+(\sigma)},A}N_{i,\varpi_\Sigma(\sigma)}\\
@A{u_{i,\sigma}^\circ}AA@A{u_{i,\varpi_\Sigma(\sigma)}^\circ}AA\\
W(l)\otimes_AM_{i,\sigma}@>{(V^\sharp)^{\varpi_\Sigma^+(\sigma)}}>>W(l)\otimes_{\tau^{\varpi_\Sigma^+(\sigma)},A}M_{i,\varpi_\Sigma(\sigma)}
\end{CD}$$
preserve the $W(l)$-structure, forcing its upper horizontal arrow to preserve that $W(l)$-structure too. For elements $\sigma\in\Sigma$ 
and for every $i$ it follows, that $(^{F^{-\sigma}}\rho_i)(U_\sigma^{-1})\in\GL(n_i,W(\sqrt[p^m]l))$ holds for a sufficiently large integer $m$, this is 
because of $W(k^{ac})\cap\frac{W(l)}{p^m}=W(\sqrt[p^m]l)$. Due to property (iv) of definition \ref{exampleIX} we finally deduce $U_\sigma\in(^{F^{-\sigma}}\G)(W(\sqrt[p^m]l))$, i.e. 
$\{U_\sigma\}_{\sigma\in\Sigma}$ represents indeed a display $\P$ with $(\G,\{\upsilon_\sigma\}_{\sigma\in\Sigma})$-structure over the field extension $\sqrt[p^m]l$ of $k$. 
The same type of argument settles the $\sqrt[p^m]l$-rationality of the $\zeta_i^\circ$'s, thus culminating in a $\sqrt[p^m]l$-valued point of 
$$\Barb(\G,\{\upsilon_\sigma\}_{\sigma\in\Sigma})\times_{\fx^\bJ,\gB^\bR,\S}\Spec k.$$ 
\end{proof}

\begin{prop}
\label{properII}
Let $\bJ$ be a gauged $C$-multi-unitary collection for a $W(\f_{p^f})$-rational $\Phi$-datum $(\G,\{\upsilon_\sigma\}_{\sigma\in\Sigma})$. 
Suppose that the field $k$ is an algebraically closed extension of $\f_{p^f}$, and fix a $1$-morphism: 
$$\Spec k[[t]]\stackrel{\S}{\rightarrow}\gB^\bR$$
The $k[[t]]$-scheme $\Barb(\G,\{\upsilon_\sigma\}_{\sigma\in\Sigma})\times_{\fx^\bJ,\gB_{\f_{p^f}}^\bR,\S}\Spec k[[t]]$ 
possesses a section if and only if there exists a $2$-commutative diagram:
$$\begin{CD}
\Spec k((t))@>>>\Barb(\G,\{\upsilon_\sigma\}_{\sigma\in\Sigma})\\
@VVV@V{\fx^\bJ}VV\\
\Spec k[[t]]@>{\S}>>\gB^\bR
\end{CD}$$
\end{prop}
\begin{proof}
Let $\P^\circ$ be a display with $(\G,\{\upsilon_\sigma\}_{\sigma\in\Sigma})$-structure over $k((t))$, and let 
$$\zeta^\circ:\S_{k((t))}\stackrel{\cong}{\rightarrow}\fx^\bJ(\P^\circ)$$
be a $2$-isomorphism. Just as in the proof of the previous proposition we may and we will assume that $C=W(\f_{p^r})\oplus W(\f_{p^r})$, 
furthermore adopting the previous notation, we have to start out from a family $(\{\S_i\}_{i\in\Lambda},\{s_\pi\}_{\pi\in\Pi})$ 
of  $\R_\pi$-operations $s_\pi$ on a family of displays $\S_i$ with $(\GL(n_i)_{W(\f_{p^r})},\{\tilde\upsilon_{i,\sigma}\}_{\sigma\in\z/r\z})$-structure 
over $k[[t]]$ (along with an insignificant multiplicative family). In the same vein $\zeta^\circ$ is given by a family:
$$\zeta_i^\circ:\S_{i,k((t))}\stackrel{\cong}{\rightarrow}\tilde\P_i^\circ:=\fx^{\bj_i,\{\upsilon_{i,\sigma}\}_{\sigma\in\Sigma}}(\P_i^\circ),$$ 
where $\P_i^\circ$ stands for the image of $\P^\circ$ under the canonical $1$-morphism
$$\Barb(\rho_i):\Barb(\G,\{\upsilon_\sigma\}_{\sigma\in\Sigma})\rightarrow\Barb(\GL(n_i)_{W(\f_{p^r})},\{\upsilon_{i,\sigma}\}_{\sigma\in\Sigma}).$$
We set up Frobenius lifts on both $A:=W(k)[[t]]$ and the $p$-adic completion $A_{\{t\}}$ of $A[\frac 1t]$, by decreeing $t\mapsto t^p$, and it does not cause 
confusion to denote both of these lifts by $\tau$. Note that $(A,pA,\tau)$ is a frame over $W(\f_{p^f})$ and dito for $(A_{\{t\}},pA_{\{t\}},\tau)$. In view of 
part (ii) of lemma \ref{properIV} we may assume the banality of $\P^\circ$, so that it is induced by a unique isomorphism class in the category 
$\hat\CAS_{A_{\{t\}},pA_{\{t\}}}(\G,\{\upsilon_\sigma\}_{\sigma\in\Sigma})$ of $(A_{\{t\}},pA_{\{t\}},\tau)$-windows with 
$(\G,\{\upsilon_\sigma\}_{\sigma\in\Sigma})$-structure, by lemma \ref{winI}. Its image under the functor 
$\sy_{A_{\{t\}},\tau}(\rho_i)$ (cf. \eqref{realIII}) shall be denoted by $\bigoplus_{\sigma\in\z/r\z}N_{i,\sigma}^\circ=N_i^\circ\in\Ob_{\ef_{A_{\{t\}},\tau}^{W(\f_{p^r})}}$. 
Any specific choice of $\{U_\sigma\}_{\sigma\in\Sigma}=U\in\G^\Sigma(A_{\{t\}})$ which represents the window for $\P^\circ$ determines coordinate systems 
$A_{\{t\}}^{n_i}\cong N_{i,\sigma}^\circ$, and we denote the images of $A^{n_i}$ under our (``non-canonical'') coordinate systems by: $N_{i,\sigma}\subset N_{i,\sigma}^\circ$. 
Notice that the isomorphism class of $N_i^\circ$ depends merely on $\P^\circ$ while different representatives give rise to different $A$-structures on $N_{i,\sigma}$, 
and at least for the elements of $\Sigma$, one can describe the transition from the representative $U$ to an alternative, say $O:=h^{-1}U\hat\Phi_{A_{\{t\}}}(h)$ by 
replacing $N_{i,\sigma}$ with $(^{F^{-\sigma}}\rho_i)(h_\sigma)^{-1}N_{i,\sigma}$. Let $\bigoplus_{\sigma\in\z/r\z}M_{i,\sigma}=M_i$ denote the $\z/r\z$-graded 
$(A,pA,\tau)$-windows to $\S_i$, and notice that just as in the proof of proposition \ref{properI} there are canonical $\cris_{A_{\{t\}},\tau}^{W(\f_{p^r})}$-isogenies
$$u_i^\circ:A_{\{t\}}\otimes_AM_i\rightarrow N_i^\circ,$$
induced by our $(A_{\{t\}},pA_{\{t\}},\tau)$-window to $\P_i^\circ$ and proposition \ref{klein}. By construction, their $\sigma$th components $u_{i,\sigma}^\circ$ 
are isomorphisms between the $A_{\{t\}}$-modules $A_{\{t\}}\otimes_AM_{i,\sigma}$ and $N_{i,\sigma}^\circ$, for $\sigma\in\Sigma$ only. Let $M$ be the 
$\z/r\z$-graded $\tau$-crystal over $A$ given by the formula $M_\sigma:=(\bigotimes_{i\in\Lambda}M_{i,\sigma})^{\otimes_Ac}$ for every $\sigma\in\z/r\z$. 
As soon as having noticed, that transport of structure establishes an action $s:Z_c\rightarrow\q\otimes\End_{W(\f_{p^r})}(A_{\{t\}}\otimes_AM)$ satisfying:
\begin{itemize}
\item[(i)]
By restriction $s$ induces an action of $\O_{Z_c}$ on the $A_{\{t\}}$-module $A_{\{t\}}\otimes_AM_\sigma$, for all $\sigma\in\Sigma$.
\item[(ii)]
The image of $s$ is contained in $\q\otimes\End_{W(\f_{p^r})}(M)$ 
\item[(iii)]
By restriction $s$ induces an action of $\O_{Z_c}$ on the $A$-module $M_\sigma$, for all $\sigma\in\Sigma$.
\end{itemize}
it is about time to invoke the principal homogeneous space $\gS_\sigma$ for $^{F^{-\sigma}}\G$ over $A$, of which the $Q$-valued points are the $\O_{Z_c}$-preserving families of $Q$-linear isomorphisms $Q\otimes_AM_{i,\sigma}\stackrel{\cong}{\rightarrow}Q\otimes_AN_{i,\sigma}$, for each $\sigma\in\Sigma$ of course. Here are the reasons for (i),(ii), and (iii): (i) follows from the aforementioned bijectivity of $u_{i,\sigma}^\circ$, (ii) is an application of the results of \cite{meszin}, \cite{deJo}, and (iii) follows from (i) together with (ii). Over each of $K(k)\{\{t\}\}$ and $A_{\{t\}}$ there exist rational points of $\gS_\sigma$, the former is due to Dwork's trick and the latter is due to the existence of the family $\{u_{i,\sigma}^\circ\}_{i\in\Lambda}$. Following the ideas in \cite{kisin} we see 
that we obtain the local triviality of $\gS_\sigma|_{\Spec A-\Spec k}$, and hence the global triviality of $\gS_\sigma$ by \cite{sansuc} together with the algebraic closedness of $k$: Just as in the proof of proposition \ref{properI}, any global section $\{u_{i,\sigma}\}_{i\in\Lambda}$ of $\gS_\sigma$ can be used for a careful adjustment of $N_{i,\sigma}$ in order to achieve the integrality of $u_{i,\sigma}^\circ$ for each $\sigma\in\Sigma$. Finally, we get $\{U_\sigma\}_{\sigma\in\Sigma}=U\in\G^\Sigma(A)$ from property (iv) of definition \ref{exampleIX} and we are done, notice $A_{\{t\}}\cap\frac{A}{p}=A$.
\end{proof}

\begin{prop}
\label{properIII}
Let $\bJ$ be a gauged $C$-multi-unitary collection for a $W(\f_{p^f})$-rational $\Phi$-datum $(\G,\{\upsilon_\sigma\}_{\sigma\in\Sigma})$. Let 
$X\rightarrow\gB_{\f_{p^f}}^\bR$ be a $1$-morphism, where $X$ is a noetherian $\f_{p^f}$-scheme. Then there exists a $X$-scheme $Y$ and a $2$-cartesian diagram
$$\begin{CD}
Y@>>>X\\
@VVV@VVV\\
\Barb(\G,\{\upsilon_\sigma\}_{\sigma\in\Sigma})@>{\fx^\bJ}>>\gB_{\f_{p^f}}^\bR
\end{CD},$$
furthermore, $Y$ is radicial, affine and integral over $X$. Assume in addition that:
\begin{itemize}
\item
$X$ is a Nagata scheme
\item
The degree of imperfection of the residue fields of each maximal point of $X$ is finite.
\end{itemize}
Then $Y_{red}$ is finite over $X$. 
\end{prop}
\begin{proof}
Since all assertions are local on $X$ we may assume $X=\Spec A$, recall that we proved that $\fx^\bJ$ is schematic, quasicompact and separated, hence the $2$-cartesian diagram 
for some quasicompact and separated $A$-scheme $Y$. First of all we note that $Y$ is closed over $\Spec A$, simply because the proposition \ref{properII} gives us the applicable criterion 
\cite[Proposition 7.2.1]{egaii} for closedness (N.B.: $Y$ is quasicompact and $A$ is noetherian). Without loss of generality we may assume the surjectivity of $Y\rightarrow\Spec A$, otherwise there 
was an ideal $\sqrt I=I\subset A$ such that $\Spec A/I$ was the image of $Y$ and it did no harm to replace $A$ by $A/I$. We continue with the affineness: Let $U_i\subset Y$ be an open and affine 
covering of $Y$, and let $V_i$ be the image of $U_i$ in $\Spec A$. Since bijectivity and closedness implies the openness of a morphism we know that the $V_i$'s form an open covering of $\Spec A$, 
so let us choose a refinement consisting of sets of the form $D(f_i)$, for certain elements $f_i\in A$. Let us define the set of sections $g_i\in\Gamma(Y,\O_Y)$ as the pull-backs of $f_i$, and 
let us note that $Y_{g_i}$ is clearly an open covering of $Y$. Moreover, every $Y_{g_i}$ has to be affine: Just choose some $j$ with $D(f_i)\subset V_j$, so that $Y_{g_i}\subset U_j$ and 
$Y_{g_i}=\Spec\Gamma(U_j,\O_Y)_{g_i|_{U_j}}$ as desired. Hence $Y=\Spec B$, and the integrality of $B$ over $A$ is clear, in view of the entirety criterion \cite[Corollaire 7.1.8]{egaii}.\\
Notice that the finiteness assertion was already shown for the special case of a field (of finite degree of imperfection). In order to prove the asserted greater generality we assume without loss of 
generality that $A$ is reduced (so that the map from $A$ to $B$ is injective). Let $S$ be the set of non-zerodivisors of $A$. The elements of $S$ are still non-zerodivisors of $B$, here notice, that a universal homeomorphism maps points of height $0$ to points of height $0$. Now, since the fiber of $\fx^\bJ$ over $\Spec A$ is $\Spec B$, its fiber over $\Spec S^{-1}A$ is clearly $\Spec S^{-1}B$ 
(being the fiber of $\Spec B\rightarrow\Spec A$ over $\Spec S^{-1}A$). However $S^{-1}A$ is a finite product of fields, say $\prod_{i=1}^nk_i$, so that $S^{-1}B_{red}=\prod_{i=1}^nl_i$, where 
$l_i$ is a (purely inseparable and) finite field extension of $k_i$. It follows that $B_{red}$ is contained in the integral closure of $A$ in $\prod_{i=1}^nl_i$, and this is finite over the Nagata ring $A$.
\end{proof}

\section{On $\fx^{\bd^+}$}
\label{Frobenius}

Let us say that a function $\bd^+:\z/r\z\rightarrow\n_0$ is a $\pmod r$-multidegree if $\bd^+(\omega)\leq\bd^+(\omega+1)+1$ holds for every $\omega\in\z/r\z$. In this case the function
\begin{equation*}
\bd:\z/r\z\rightarrow\z/r\z;\,\omega\mapsto\omega+\bd^+(\omega)
\end{equation*}
induces a bijection between the following subsets of $\z/r\z$:
\begin{eqnarray}
\label{shiftII}
&&\Omega=\{\omega\in\z/r\z|\bd^+(\omega)\leq\bd^+(\omega+1)\}\\
\label{shiftI}
&&\Sigma=\bd(\z/r\z)
\end{eqnarray}
More specifically, let $\sigma_1<\dots<\sigma_z<\sigma_1+r$ be the elements of a set of representatives for the $\pmod r$ congruence classes in $\Sigma$, 
where $z=\Card(\Sigma)$. Then $\Omega$ consists of the $\pmod r$-congruence classes of the elements $\omega_1<\dots<\omega_z<\omega_1+r$, where 
$\omega_j:=\max\{\omega|\,\omega+\bd^+(\omega)\leq\sigma_j\}$. In this scenario we would like to discuss a certain type of functoriality, which is expressed by 
the $2$-commutative diagram:
$$\begin{CD}
\G^\Sigma@>>>\Barb(\G,\{\mu_\sigma\}_{\sigma\in\Sigma})@>>>\ton(\Ibar^{\mu_\Sigma})\\
@VmVV@V{\fx^{\bd^+}}VV@V{\ton(\gamma)}VV\\
\G^{\Omega}@>>>\Barb(\G,\{\tilde\mu_\sigma\}_{\sigma\in\Omega})@>>>\ton(\Ibar^{\tilde\mu_{\Omega}})\\
\end{CD},$$
where $\G$ is a reductive group scheme over $W(\f_{p^r})$ while $(\G,\{\mu_\sigma\}_{\sigma\in\Sigma})$ and 
$(\G,\{\tilde\mu_\omega\}_{\omega\in\Omega})$ are $\Phi$-data satisfying: 
$$^{F^{\bd^+(\omega)}}\mu_{\bd(\omega)}=\tilde\mu_\omega$$
We define this by writing down a pair $(\gamma,m)$ as follows: 
\begin{eqnarray*}
&&m:\{U_\sigma\}_{\sigma\in\Sigma}\mapsto\{{^{F^{\bd^+(\omega)}}U_{\bd(\omega)}}\}_{\omega\in\Omega}\\
&&\gamma:\{k_\sigma\}_{\sigma\in\Sigma}\mapsto\{{^{F^{\bd^+(\omega)}}k_{\bd(\omega)}}\}_{\omega\in\Omega}
\end{eqnarray*}
Recall that to any $(\G,\{\mu_\sigma\}_{\sigma\in\Sigma})$-display $\P$ over some $\f_{p^f}$-algebra $B$, we associated a vector bundle $T_{\P}$, which possesses 
a natural decomposition $T_{\P}=\bigoplus_{\sigma\in\Sigma}T_{\P,\sigma}$. The differential of $\fx^{\bd^+}$ induces a map from $T_\P$ to $T_{\fx^{\bd^+}(\P)}$, 
in view of our deformation theory. The following infinitesimal property is crucial: For each $\sigma\in\Sigma$ the differential of $\fx^{\bd^+}$ kills $T_{\P,\sigma}$ unless 
$\bd^+(\sigma)=0$, in which case it induces a $B$-linear isomorphism $T_{\P,\sigma}\stackrel{\cong}{\rightarrow}T_{\fx^{\bd^+}(\P),\sigma}$, here notice that $\bd^+(\sigma)=0$ 
implies $\sigma\in\Omega$ and $\bd(\sigma)=\sigma$. In the following lemma $\tilde\bd^+$ denotes the function $\bd_{\Omega}^+$, which is defined as in subsection \ref{pivotalII}.

\begin{lem}
\label{thedetailII}
Let $\P$ be a $(\G,\{\mu_\sigma\}_{\sigma\in\Sigma})$-display over a finitely generated $K$-algebra $B$, where the field $K$ is an algebraically closed or an algebraic extension of 
$\f_{p^f}$. Assume that for each maximal ideal $\gn\subset B$ the $\gn$-adic completion $\P_{\hat B_\gn}$ agrees with the universal formal equicharacteristic deformation of its special 
fiber $\P_{B/\gn}$. Suppose that our $\pmod r$-multidegree satisfies $\bd^+\leq1+\tilde\bd^+$ and consider the $K$-subalgebra $C$, which is defined to be to be the kernel of the derivation: 
$$B\stackrel{d_{B/K}}{\rightarrow}\Omega_{B/K}^1\stackrel{\kappa_\P^{-1}}{\rightarrow}\bigoplus_{\sigma\in\Sigma}\check T_{\P,\sigma}
\twoheadrightarrow\bigoplus_{\sigma\in\Sigma\backslash\Omega}\check T_{\P,\sigma}.$$
Then there exists a $(\G,\{\tilde\mu_\omega\}_{\omega\in\Omega})$-display $\tilde\P$ with over $C$ together with an isomorphism 
$\fx^{\bd^+}(\P)\stackrel{\cong}{\rightarrow}\tilde\P_B$, furthermore $\tilde\P$ is unique up to a unique isomorphism and each 
localization $\tilde\P_{\hat C_\gm}$ agrees with the universal formal equicharacteristic deformation of its special fiber $\tilde\P_{C/\gm}$.
\end{lem}
\begin{proof}
Observe that the universality of the completions $\tilde\P_{\hat C_\gm}$ implies the smoothness of the $K$-algebra $C$, so that the map from $C$ to $B$ is not only finite and 
radicial, but also faithfully flat by \cite[Proposition (6.1.5)]{egaiv}. The uniqueness assertion follows easily from descent theory and rigidity, since another display $\tilde\P'$ 
with the requested properties would give rise to isomorphisms $\tilde\P'_B\cong\fx^{\bd^+}(\P)\cong\tilde\P_B$, while the kernel $\ga$ of the codiagonal from $B\otimes_CB$ 
to $B$ is nilpotent. For the rest of the proof, it does no harm to identify the sets of maximal ideals of $B$ and $C$. Moreover, the explicit description of universal formal 
deformation spaces as given in \cite[Subsubsection 3.5.9]{pappas} or in our subsection \ref{thedetailI} reveals that each $\fx^{\bd^+}(\P_{\hat B_\gn})$ descends uniquely 
to some $(\G,\{\tilde\mu_\omega\}_{\omega\in\Omega})$-display $\tilde\P_\gn$ over $\hat C_\gn$, which is again a universal formal equicharacteristic deformation of 
its special fiber. Consider the product $(\G,\{\tilde\mu_\omega\}_{\omega\in\Omega})$-display $\hat\P:=\prod_\gn\tilde\P_\gn$ over the proartinian completion 
$\hat C:=\prod_\gn\hat C_\gn$. At last we construct the display $\tilde\P$ with descent theory, where the descent datum $\phi:\pr_1^*\P\rightarrow\pr_2^*\P$ arises from 
applying \cite[Lemma 3.7.6]{pappas} to the extension $B\otimes_CB\subset\hat B\otimes_{\hat C}\hat B$, here notice again, that the kernel $\ga$ of the codiagonal is 
nilpotent, while the existence of $\hat\P$ grants a descent datum, say $\hat\phi$ from $(\pr_1^*\P)_{\hat B\otimes_{\hat C}\hat B}$ to $(\pr_2^*\P)_{\hat B\otimes_{\hat C}\hat B}$.
\end{proof}

Let $B$ be a commutative ring satisfying $pB=0$. For every non-negative integer $m$ we will write $\sqrt[p^m]B$ for the $B$-algebra whose underlying ring is $B$, 
and whose $B$-algebra structure is defined by the homomorphism $x\mapsto x^{p^m}$, the $m$th iterate of the absolute Frobenius, notice that their direct limit
$$\lim_\rightarrow\sqrt[p^m]B=:\sqrt[p^\infty]B$$
is the usual perfection of $B/\sqrt{0_B}$. For an element $x$ of $B$ we use the notation $\sqrt[p^m]x\in\sqrt[p^m]B$ to express the alteration of $B$-algebra 
structure on the same ring (the $p^m$-th power of $\sqrt[p^m]x$ agrees with the image of $x$ under the structural morphism $B\rightarrow\sqrt[p^m]B$).

\begin{thm}
\label{fake}
Let $\tilde\P$ be a display with $(\G,\{\tilde\mu_\sigma\}_{\sigma\in\Omega})$-structure over a finitely generated $K$-algebra $A$, where the field $K$ is 
an algebraically closed or an algebraic extension of $\f_{p^f}$. Assume that for each maximal ideal $\gm\subset A$ the $\gm$-adic completion $\tilde\P_{\hat A_\gm}$ 
agrees with the universal formal equicharacteristic deformation of its special fiber $\tilde\P_{A/\gm}$, for every maximal ideal $\gm\subset A$. Then there exists a 
display $\P$ with $(\G,\{\mu_\sigma\}_{\sigma\in\Sigma})$-structure over a finite $A$-subalgebra $B\subset\sqrt[p^\infty]A$, with the following additional properties:
\begin{itemize}
\item[(i)]
$\fx^{\bd^+}(\P)\cong\tilde\P_B$
\item[(ii)]
The $\gn$-adic completion $\P_{\hat B_\gn}$ agrees with the universal formal equicharacteristic 
deformation of its special fiber $\P_{B/\gn}$, for every maximal ideal $\gn\subset B$.
\end{itemize}
\end{thm}
\begin{proof}
Again, observe that the $K$-algebras $A$ and $B$ are automatically smooth, and that the morphism $A\rightarrow B$ is 
automatically faithfully flat, moreover $\P$ is unique (up to a unique isomorphism, once the isomorphism in (i) is fixed).\\
We will induct on $\bd^+$ in order to construct $\P/B$ in stages, with each stage being such that $\bd^+\leq1+\tilde\bd^+$, 
so that one can find another $\pmod r$-multidegree $\be^+\leq1+\tilde\be^+$, such that the composition
$$\Barb(\G,\{\mu_\sigma\}_{\sigma\in\Sigma})\stackrel{\fx^{\bd^+}}{\rightarrow}
\Barb(\G,\{\tilde\mu_\sigma\}_{\sigma\in\Omega})\stackrel{\fx^{\be^+}}{\rightarrow}\Barb(\G,\{{^F\mu}_{\sigma+1}\}_{\sigma\in{\Sigma-1}})$$
agrees with the Frobenius. In this situation we may apply lemma \ref{thedetailII} to find that the display $\fx^{\be^+}(\tilde\P)$ descends to some $\P'$ satisfying (ii) over a 
$K$-subalgebra $C$ of $A$, which contains $A^p$. The construction of $\P$ from $\P'$ is achieved by a change of base along $C\stackrel{\sqrt[p]\empty}{\rightarrow}\sqrt[p]C=:B$.
\end{proof}

\subsection{On the isogeny classes of $\fx^{\bj,\{\upsilon_\omega\}_{\omega\in\Omega}}$ and $\fx^{\bd^+}$}
\label{weirdoIII}

For any reductive $\z_p$-group $\gG$ we let $L\gG$ be the $\f_p$-functor defined by $R\mapsto\gG(\q\otimes W(R))$, which is in fact a 
fpqc-scheaf of groups on $\Spec\f_p$, in the sense of subsection \ref{synthIV}. We let $\B^0(\gG)$ be the $\f_p$-stack rendering the diagram
\begin{equation}
\label{poshIII}
\begin{CD}
{\ton(L\gG)}@<<<{\B^0(\gG)}\\
@V{\Delta_{\ton(L\gG)}}VV@VqVV\\
{\ton(L\gG)\times_{\f_p}\ton(L\gG)}@<{\ton(F\times\id)\circ\Delta_{\ton(L\gG)}}<<{\ton(L\gG)}
\end{CD}
\end{equation}
$2$-cartesian. A $1$-morphism from a $\f_p$-scheme $X$ to $\B^0(\gG)$ is called a $\gG$-isodisplay over $X$, and it is 
called banal if its image in $\ton(L\gG)$ is a trivial $L\gG$-torsor over $X$. The groupoid of banal $\gG$-isodisplays over an 
$\f_p$-algebra $R$ is canonically equivalent to $[\gG(\q\otimes W(R))/_F\gG(\q\otimes W(R))]$. For any reductive $W(\f_{p^r})$-group 
$\G$ and any $W(\f_{p^f})$-rational $\Phbar$-datum $(\G,\{\mu_\sigma\}_{\sigma\in\Sigma})$ we wish to define a $1$-morphism
\begin{equation}
\label{poshIV}
\bh_{\mu_\Sigma}^0:\Barb(\G,\{\mu_\sigma\}_{\sigma\in\Sigma})\rightarrow\B^0(\Res_{W(\f_{p^r})/\z_p}\G)
\end{equation}
by decreeing its impact on banal objects as follows: Every $\G^\Sigma(W(R))$-element $U=\{U_\sigma\}_{\sigma\in\Sigma}$, which represents 
a banal $(\G,\{\mu_\sigma\}_{\sigma\in\Sigma})$-display over an $\f_{p^f}$-algebra $R$, is sent to the $\G(K(\f_{p^r})\otimes_{\z_p}W(R))$-element 
$m(U):=\mu_\Sigma(\frac1p)U$. Moreover, every isomorphism $\{h_\sigma\}_{\sigma\in\Sigma}=h\in\Ibar^{\mu_\Sigma}=\prod_{\sigma\in\Sigma}\Ibar^{\mu_\sigma}$ 
from $h^{-1}U\Phbar(h)=U'$ to $U$ is sent to the $\G(K(\f_{p^r})\otimes_{\z_p}W(R))$-element $\gamma(h):=\{k_\sigma\}_{\sigma\in\z/r\z}$, where
$$^{F^{\bd^+(\sigma)}}(\mu_{\bd_\Sigma(\sigma)}(\frac1p)h_{\bd_\Sigma(\sigma)}\mu_{\bd_\Sigma(\sigma)}(p))=k_\sigma$$
here notice that $\prod_{\sigma=0}^{r-1}{^{F^{-\sigma}}\G_{W(\f_{p^f})}}=(\Res_{W(\f_{p^r})/\z_p}\G)_{W(\f_{p^f})}$.

\begin{prop}
\label{weirdIV}
Fix positive integers $r\mid f$, a standard linear $\Phbar$-datum $(\GL(n)_{W(\f_{p^r})},\{\upsilon_\omega\}_{\omega\in\Omega})$, in the sense of part (i) of 
definition \ref{standard}, and a function $\bj:\z/r\z\rightarrow\n_0$ satisfying the conditions (i) and (ii) of subsubsection \ref{Universum}. There exists a $2$-isomorphism: 
\begin{equation}
\label{weirdV}
J:\bh_{\upsilon_\Omega}^0\stackrel{\cong}{\rightarrow}\bh_{\tilde\upsilon}^0\circ\fx^{\bj,\{\upsilon_\omega\}_{\omega\in\Omega}}
\end{equation}
where $\tilde\upsilon$ is defined as in \eqref{soul}. Moreover, in the standard multiplicative case, the component of \eqref{weirdV} 
at some $\omega\in[\omega_{j-1}+1,\omega_j]$ is given by multiplication with $p^{\omega-\omega_{{j-1}}-1}$, please recall 
$\fx^{\bj,\{\upsilon_\omega\}_{\omega\in\Omega}}\cong\id_{\delta^{-1}\upsilon_\Omega}$ in that case (cf. remark \ref{weirdVIII}).
\end{prop}
\begin{proof}
Let $\{U_\omega\}_{\omega\in\Omega}$ represent a $\Barb(\GL(n)_{W(\f_{p^r})},\{\upsilon_\omega\}_{\omega\in\Omega})$-object $\P$, 
and let $(B_0,\dots,B_{r-1})$ and $(\tilde B_0,\dots,\tilde B_{r-1})$ be the $r$-tuples of matrices that result when applying the two functors 
$\bh_{\upsilon_\Omega}^0$ and $\bh_{\tilde\upsilon}^0\circ\fx^{\bj,\{\upsilon_\omega\}_{\omega\in\Omega}}$ to $\P$. The former looks like
\begin{eqnarray*}
&&(\dots,B_{\omega_{j-1}},B_{\omega_{j-1}+1},\dots,B_{\omega_j},\dots)=\\
&&(\dots,\upsilon_{\omega_{j-1}}(\frac1p)U_{\omega_{j-1}},1,\dots,\upsilon_{\omega_j}(\frac1p)U_{\omega_j},\dots),
\end{eqnarray*}
where the ``$1$'' stands in the $\omega_{j-1}+1$st position, and the latter looks like:
\begin{eqnarray*}
&&(\dots,\tilde B_{\omega_{j-1}},\tilde B_{\omega_{j-1}+1},\dots,\tilde B_{\omega_j},\dots)=\\
&&(\dots,\tilde\upsilon_{\omega_{j-1}}(\frac1p)U_{\omega_{j-1}},\tilde\upsilon_{\omega_{j-1}+1}(\frac1p),\dots,\tilde\upsilon_{\omega_j}(\frac1p)U_{\omega_j},\dots)
\end{eqnarray*}
Let us define a sequence of matrices by 
$k_\omega=\prod_{\sigma=\omega_{j-1}+1}^{\omega-1}{^{F^{\sigma-\omega}}\tilde\upsilon_\sigma(p)}$ for $\omega\in[\omega_{j-1}+1,\omega_j]$, 
and notice that $\upsilon_{\omega_j}=\prod_{\omega=\omega_{j-1}+1}^{\omega_j}{^{F^{\omega-\omega_j}}\tilde\upsilon_\omega}$ 
holds for all $j$, so that $k_\omega B_\omega=\tilde B_\omega{^Fk_{\omega+1}}$ is true for every $\omega\in\z/r\z$ and we are done.
\end{proof}

Our next result concerns a similar property of the functor
$$\fx^{\bd^+}:\Barb(\G,\{\mu_\sigma\}_{\sigma\in\Sigma})\rightarrow\Barb(\G,\{\tilde\mu_\omega\}_{\omega\in\Omega}),$$
where we require that:
\begin{itemize}
\item
$\G$ is a reductive $W(\f_{p^r})$-group and $(\G,\{\mu_\sigma\}_{\sigma\in\Sigma})$ and $(\G,\{\tilde\mu_\omega\}_{\omega\in\Omega})$ are $W(\f_{p^f})$-rational $\Phi$-data
\item
$\bd^+:\z/r\z\rightarrow\n_0$ is a $\pmod r$-multidegree, giving rise to a function $\bd:\z/r\z\rightarrow\z/r\z;\,\sigma\mapsto\sigma+\bd^+(\sigma)$ whose image is $\Sigma$.
\item
$\Omega=\{\omega\in\z/r\z|\bd^+(\omega)\leq\bd^+(\omega+1)\}$ and $^{F^{\bd^+(\omega)}}\mu_{\bd(\omega)}=\tilde\mu_\omega$ holds for any $\omega\in\Omega$
\end{itemize}
Observe that $\Omega$ consists of the $\pmod r$-congruence classes of the elements $\omega_j:=\max\{\sigma|\sigma+\bd^+(\sigma)\leq\sigma_j\}$.

\begin{prop}
\label{weirdIII}
There exists a $2$-isomorphism: 
\begin{equation}
\label{weirdoI}
D:\bh_{\tilde\mu_\Omega}^0\circ\fx^{\bd^+}\stackrel{\cong}{\rightarrow}\bh_{\mu_\Sigma}^0
\end{equation}
Moreover, in the special case $\G=\g_{m,W(\f_{p^r})}$ and $\tilde\mu_\omega=\delta_\omega$, the component of \eqref{weirdoI} at some 
$\omega\in[\omega_{j-1}+1,\omega_j]\cap[\sigma_{k-1}+1,\sigma_k]$ is given by multiplication with $p^{\omega_{k-1}-\omega_{j-1}}$, please recall 
$\fx^{\bd^+}\cong\id_{\delta_\Omega^{-1}\mu_\Sigma}$ in the case at hand.
\end{prop}
\begin{proof}
We would like to begin with the obviously $2$-commutative diagram
$$\begin{CD}
\B^0(\Res_{W(\f_{p^r})/\z_p}\G)@>{\fx^{0,\bd^+}}>>\B^0(\Res_{W(\f_{p^r})/\z_p}\G)\\
@A{\bh_{\mu_\Sigma}^0}AA@A{\bh_{\tilde\mu_\Omega}^0}AA\\
\Barb(\G,\{\mu_\sigma\}_{\sigma\in\Sigma})@>{\fx^{\bd^+}}>>\Barb(\G,\{\tilde\mu_\omega\}_{\omega\in\Omega})
\end{CD},$$
where $\fx^{0,\bd^+}$ is defined by sending banal objects represented by families $\{b_\sigma\}_{\sigma\in\z/r\z}=b\in(\Res_{W(\f_{p^r})/\z_p}\G)(W(R)[\frac1p])$ to the families
$$\{\prod_{i=\bd^+(\omega)}^{\bd^+(\omega+1)}{^{F^i}b_{\omega+i}}\}_{\omega\in\z/r\z}=:\tilde b,$$
between two of which the morphisms represented by $\{g_\sigma\}_{\sigma\in\z/r\z}$ get sent to the families: $$\{{^{F^{\bd^+(\omega)}}g}_{\bd(\omega)}\}_{\omega\in\z/r\z}$$
It remains to check $\fx^{0,\bd^+}\cong\id_{\B^0(\Res_{W(\f_{p^r})/\z_p}\G)}$, indeed the formula
$$h_\omega:=\prod_{i=0}^{\bd^+(\omega)-1}{^{F^i}b_{\omega+i}}$$
defines a functorial homomorphism $\{h_\omega\}_{\omega\in\z/r\z}=h$ between the banal objects represented by $\tilde b$ and $b$.
\end{proof}

Now fix a diagram 
$$\begin{CD}
\G@>{\rho}>>\GU(n)\\
@AAA@A{\zeta}AA\\
\zen^\G@<i<<\Res_{W(\f_{p^{2r}})/W(\f_{p^r})}\g_m
\end{CD},$$
where $\GU(n)$ is the $W(\f_{p^r})$-similarity group as defined in \eqref{similarity} and $\zeta$ its dilatation homomorphism \eqref{dilatation}. 
In addition, we require that $\rho$ induces a morphism (cf. subsubsection \ref{Geh}) from $(\G,\{\tilde\mu_\omega\}_{\omega\in\Omega})$ to a 
$W(\f_{p^f})$-rational $\Phbar$-datum $(\GU(n)_{W(\f_{p^r})},\{\upsilon_\omega\}_{\omega\in\Omega})$, which we assume to be standard unitary 
in the sense of part (iii) of definition \ref{standard}. Moreover, we fix a function $\bj:\z/2r\z\rightarrow\n_0$ satisfying the conditions (i) and (ii) of 
subsubsection \ref{Universum} along with $\bj(r+\omega)=r_\Omega(\bd_\Omega(\omega))-\bj(\omega)-1$, as in subsubsection \ref{dualityII}.\\
At last we also fix a $W(\f_{p^f})$-rational cocharacter $\alpha$ of $\Res_{W(\f_{p^{2r}})/\z_p}\g_m$ of which the $\sigma$-components have weights, say $a_\sigma$. We require that:
\begin{eqnarray*}
&&\forall\sigma\in\z/2r\z:\,1-\Card(\bd^{-1}(\{\sigma\})=a_{\sigma+r}+a_\sigma\\
&&\Card(\{1,\dots,r\}\backslash\bd^{-1}(\{1,\dots,r\}))\equiv\sum_{\sigma=1}^ra_\sigma\pmod 2
\end{eqnarray*}

\begin{prop}
\label{weirdI}
There exists a $2$-commutative diagram
\[
\begin{tikzcd}
\B^0(\Res_{W(\f_{p^r})/\z_p}\G)\ar[rr,"\B^0(\chi\circ\rho)"]\arrow[dr,"\B^0(\rho)"]&&\B^0(\Res_{W(\f_{p^r})/\z_p}\g_m)\\
&\B^0(\Res_{W(\f_{p^r})/\z_p}\GU(n))\arrow[ur,"\B^0(\chi)"]\\
&\Barb(\Res_{W(\f_{p^r})/\z_p}\GU(n),\tilde\upsilon)\ar[u,"\bh_{\tilde\upsilon}^0"]\ar[dr,"\Barb(\chi)"]\\
\Barb(\Res_{W(\f_{p^r})/\z_p}\G,\alpha\mu_\Sigma)\ar[uuu,"\bh_{\alpha\mu_\Sigma}^0"]\ar[ur,"\fx_\alpha^{\bd^+,\bj}(\rho)"]\ar[rr,"\Barb(\chi\circ\rho)"]
&&\Barb(\Res_{W(\f_{p^r})/\z_p}\g_m,\delta)\ar[uuu,"\bh_\delta^0"]
\end{tikzcd},
\]
where $\fx_\alpha^{\bd^+,\bj}(\rho):=\fx^{\bd^+,\bj}(\rho)\circ\id_{\alpha}$ while $\fx^{\bd^+,\bj}(\rho)$ stands for the composition of the following morphisms:
\begin{eqnarray*}
&&\Barb(\Res_{W(\f_{p^r})/\z_p}\G,\mu_\Sigma)\stackrel{\cong}{\rightarrow}\Barb(\G,\{\mu_\sigma\}_{\sigma\in\Sigma})
\stackrel{\fx^{\bd^+}}{\rightarrow}\Barb(\G,\{\tilde\mu_\omega\}_{\omega\in\Omega})\\
&&\stackrel{\Barb(\rho)}{\rightarrow}\Barb(\GU(n),\{\upsilon_\omega\}_{\omega\in\Omega})
\stackrel{\fx^{\bj,\{\upsilon_\omega\}_{\omega\in\Omega}}}{\rightarrow}\Barb(\Res_{W(\f_{p^r})/\z_p}\GU(n),\tilde\upsilon)
\end{eqnarray*}
\end{prop}
\begin{proof}
Establishing the $2$-commutativity of the two triangular (resp. trapezoidal) subdiagrams requires certain $2$-isomorphisms $a$ and $d$ (resp. $b$ and $c$), 
labelled from top to bottom (resp. left to right). Here we use the convention that $a$, $b$, $c$ and $d$ rotate counter clockwise between the $1$-morphisms of 
the above diagram, when viewed from their sources. Given the prismatic shape of this diagram, each choice of $a$, $b$, $c$ and $d$ yields a $2$-isomorphism
\begin{equation}
\label{weirdVII}
e=(a.\bh_{\mu_\Sigma\alpha}^0)\circ(\B^0(\chi).b)\circ(c.\fx_\alpha^{\bd^+,\bj}(\rho))\circ(\bh_\delta^0.d)
\end{equation}
rendering the outer rectangular subdiagram $2$-commutative too, where a "$.$" indicates the left and right actions of the $1$-morphisms on the 
$2$-ones. Notice that $e$ is a $2$-isomorphism between $\bh_\delta^0\circ\Barb(\chi\circ\rho)$ and $\B^0(\chi\circ\rho)\circ\bh_{\mu_\Sigma\alpha}^0$, 
which are $1$-morphisms from $\Barb(\Res_{W(\f_{p^r})/\z_p}\G,\mu_\Sigma\alpha)$ to $\B^0(\Res_{W(\f_{p^r})/\z_p}\g_m)$. Observe that $a$, 
$c$ and $e$ are already fixed by canonical conventions, and that $d$ may be obtained from the previously established commutative diagram:
$$\begin{CD}
\Barb(\Res_{W(\f_{p^r})/\z_p}\GU(n),\tilde\upsilon)@>{\Barb(\chi)}>>\Barb(\Res_{W(\f_{p^r})/\z_p}\g_m,\delta)\\
@A{\fx^{\bd^+,\bj}(\rho)}AA@A{\id_{\delta^{-1}\kappa_\Sigma}}AA\\
\Barb(\Res_{W(\f_{p^r})/\z_p}\G,\mu_\Sigma)@>{\Barb(\chi\circ\rho)}>>\Barb(\Res_{W(\f_{p^r})/\z_p}\g_m,\kappa_\Sigma)
\end{CD}$$
It remains to solve \eqref{weirdVII}, regarded as an equation in the variable $b$! Composition of \eqref{weirdoI} with the inverse of \eqref{weirdV} yields a certain $2$-isomorphism 
$$\tilde b:\bh_{\tilde\upsilon}^0\circ\fx^{\bd^+,\bj}(\rho)\rightarrow\B^0(\rho)\circ\bh_{\mu_\Sigma}^0,$$ 
from which we immediately obtain another $2$-isomorphism $\B^0(\chi).\tilde b$ from 
$\bh_\delta^0\circ\Barb(\chi\circ\rho)\circ\id_{\alpha^{-1}}=\bh_\delta^0\circ\id_{\delta^{-1}\kappa_\Sigma}\circ\Barb(\chi\circ\rho)$ to $\bh_{\kappa_\Sigma}^0\circ\Barb(\chi\circ\rho)$, 
which are functors from $\Barb(\Res_{W(\f_{p^r})/\z_p}\G,\mu_\Sigma)$ to $\B^0(\Res_{W(\f_{p^r})/\z_p}\g_m)$. Now, let us write $g_\sigma=\sigma-\omega_{k-1}-1$ 
for all $\sigma\in[\sigma_{k-1}+1,\sigma_k]$, notice that $g_{\sigma+r}=g_\sigma$ and $g_{\sigma+1}-g_\sigma=1-\Card(\bd^{-1}(\{\sigma\})=a_{\sigma+r}+a_\sigma$. 
The additional assertions of the propositions \ref{weirdIV} and \ref{weirdIII} allow to rewrite $\B^0(\chi).\tilde b$ as $\gamma^{-1}.\Barb(\chi\circ\rho)$, where 
$\gamma$ is the $2$-isomorphism from $\bh_{\kappa_\Sigma}^0$ to $\bh_\delta^0\circ\id_{\delta^{-1}\kappa_\Sigma}$ of which the $\sigma$-component is given by 
$p^{g_\sigma}$. Furthermore, the parity condition implies the existence of a family $\{l_\sigma\}_{\sigma\in\z/2r\z}\in\z^{2r}$ with $g_\sigma=l_{\sigma+r}+l_\sigma$ 
and $a_\sigma=l_{\sigma+1}-l_\sigma$ and we let $\lambda:\bh_{\mu_\Sigma}^0\rightarrow\bh_{\mu_\Sigma\alpha}^0\circ\id_{\alpha^{-1}}$ be the 
$2$-isomorphism whose $\sigma$-components are given by $p^{l_\sigma}$ (N.B.: source and target of $\lambda$ are regarded as functors from 
$\Barb(\Res_{W(\f_{p^r})/\z_p}\G,\mu_\Sigma)$ to $\B^0(\Res_{W(\f_{p^r})/\z_p}\G)$ and recall that $\zen^\G$ contains $\Res_{W(\f_{p^{2r}})/W(\f_{p^r})}\g_m$). 
Also, notice that $\B^0(\chi\circ\rho)\circ\bh_{\mu_\Sigma}^0\cong\bh_{\kappa_\Sigma}^0\circ\Barb(\chi\circ\rho)$, at last we define $b$ to be:
$$(\B^0(\rho).\lambda\circ\tilde b).\id_\alpha:\bh_{\tilde\upsilon}^0\circ\fx_\alpha^{\bd^+,\bj}(\rho)\rightarrow\B^0(\rho)\circ\bh_{\mu_\Sigma\alpha}^0$$
In view of $\B^0(\chi\circ\rho).\lambda=\gamma.\Barb(\chi\circ\rho)$ this is what we wanted.
\end{proof}

\begin{cor}
\label{weirdII}
Notice that $\mu_\Sigma\alpha$ factors through the subgroup $\gG:=(\g_m\times\Res_{W(\f_{p^r})/\z_p}\G^1)/\{\pm1\}$, where $\G^1:=\ker(\chi\circ\rho)$, so that 
proposition \ref{weirdI} induces a $1$-morphism: $$\Barb(\gG,\mu_\Sigma\alpha)\rightarrow\Barb((\g_m\times\Res_{W(\f_{p^r})/\z_p}\UL)/\{\pm1\},\tilde\upsilon)$$
\end{cor}

\section{Moduli spaces of abelian varieties with additional structure}
\label{twistingV}
Let $L$ be a totally imaginary quadratic extension of a totally real number field $L^+$ and write $*$ for the non-trivial element of $\Gal(L/L^+)$. By a skew-Hermitian $L$-vector space we mean a finite-dimensional $L$-vector space $V$ together with a $L$-linear isomorphism $\Psi:V\rightarrow\check V_{[*]}$ satisfying $-\Psi(y,x)=\Psi(x,y)^*$ for all $x,y\in V$. Let us fix such a pair $(V,\Psi)$. A place $v$ of $L$ is called inert if it is fixed by $*$, and for each of these one obtains a skew-Hermitian $L_v$-vector space $(L_v\otimes_LV,\Psi_v)$, by localisation. Let $\UL(V/L,\Psi)$ represent the group functor
$$R\mapsto\{g\in\GL_{L\otimes_{L^+}R}(V\otimes_{L^+}R)|\Psi(gx,gy)=\Psi(x,y)\,\forall x,y\in V\otimes_{L^+}R\},$$
and occasionally we have to work with
$$(\g_m\times\Res_{L^+/\q}\UL(V/L,\Psi))/\{\pm1\}=\GU(V/L,\tr_{L/\q}\Psi),$$
i.e. the $\q$-subgroup of $\GL(V/\q)$ which is generated by its center and $\Res_{L^+/\q}\UL(V/L,\Psi)$. In this section we need a few Hodge-theoretic preliminaries, recall that the real algebraic group 
$\s:=\Res_{\c/\r}\g_m$ is called the Deligne torus. Let us also write $\s^1$ for the kernel of the norm $\s\rightarrow\g_{m,\r}$. Notice that 
$\s=(\g_{m,\r}\times_\r\s^1)/\{\pm1\}$ holds, and we also want to choose $\sqrt{-1}\in\c$ once and for all. A $\q$-Hodge structure of weight $-1$ on $V$ is called skew-Hermitian if and only if
\begin{itemize}
\item
the associated homomorphism $h:\s\rightarrow\GL(\r\otimes V/\r)$ factors through $\GU(V/L,\tr_{L/\q}\Psi)_\r$, and
\item
the symmetric form $(\tr_{L/\q}\Psi)(h(\sqrt{-1})x,y)$ is positive definite on $\r\otimes V$.
\end{itemize} 
Fix a triple $(V,\Psi,h)$ as above, and note that there are Hodge decompositions $V_\iota=\bigoplus_{p+q=-1}V_\iota^{p,q}$, where 
$V_\iota=\c\otimes_{\iota,L}V$ stands for the eigenspace as $\iota$ is running through the set $L_{an}=\Spec L(\c)$ of embeddings of $L$ into $\c$, 
observe that $L_{an}$ carries a natural left $\Gal(R/\q)$-action commuting with the complex conjugation, which could be viewed as acting from the right, 
the subfield $R$ stands for the normal closure. Let us denote the Hodge numbers $\dim_\c V_\iota^{p,q}$ by $h_\iota^{p,q}$. The $\c$-vector spaces 
$V_\iota$ carry natural skew-Hermitian forms $\Psi_\iota$, obtained by extension of scalars. Notice that $\overline V_\iota^{p,q}=V_{\iota\circ*}^{q,p}$, 
so that $h_\iota^{p,q}=h_{\iota\circ*}^{q,p}$. One more piece of terminology will prove useful: Consider a reductive $\q$-group which can be written in the form
\begin{equation}
\label{restrictI}
G=(\g_m\times\Res_{L^+/\q}G^1)/\{\pm1\}
\end{equation}
for some connected $L^+$-group $G^1$, which is assumed to be an inner form of a totally compact form and some $\q$-rational normal subgroup $\{\pm1\}\triangleleft\g_m\times\Res_{L^+/\q}G^1$, 
which is assumed to be of order $2$ and not contained in $\g_m$ nor $\Res_{L^+/\q}G^1$. We say that $(G,h)$ is a Shimura datum with coefficients in $L^+$, provided that 
$h:\s\rightarrow G_\r$ satisfies the usual axioms (cf. \cite{deligne4}) while it restricts to the identity on $\g_{m,\r}\subset\s$, in terms of \eqref{restrictI}. Recall that the inverse of
$h\vert_{\g_{m,\r}}:\g_m\rightarrow G$ is usually called the weight-cocharacter $w$ and observe that in the said scenario $\{1,w(-1)\}$ agrees with the intersection of $\g_{m,\q}$ 
with $\Res_{L^+/\q}G^1$ in $G$. As usual we write $\mu:\g_{m,\c}\rightarrow G_\c$ for the restriction of the complexification of $h$ to the first factor in the canonical decomposition 
$\s_\c\cong\g_{m,\c}^2$. Notice that $G_\r$ is contained in the product of the groups $(\g_{m,\r}\times_\r G_\iota^1)/\{\pm1\}$ where $G_\iota^1:=G^1\times_{L^+,\iota}\r$ as $\iota$ runs 
through $L_{an}^+$. Accordingly we let $\mu_\iota:\g_{m,\c}\rightarrow(\g_{m,\c}\times_\c G_{\iota,\c}^1)/\{\pm1\}$ be obtained as the image of $\mu$ in the factor that corresponds to $\iota$. 
Let $E\subset\c$ be the smallest field over which the conjugacy class of $\mu$ is defined, observe that the conjugacy class of each $\mu_\iota$ is defined over some subfield of the composite 
$ER$. At last, suppose that $(V,\Psi)$ is a skew-Hermitian $L$-vector space. Then we will say that a $L^+$-homomorphism $\rho:G^1\rightarrow\UL(V/L,\Psi)$ is a unitary representation of type 
\begin{equation}
\label{restrictII}
\{(-b_\iota,b_\iota-1),\dots,(-a_\iota,a_\iota-1)\},
\end{equation}
if the following holds:
\begin{itemize}
\item[(U1)]
$\rho(w(-1))=-\id_V$
\item[(U2)] 
If $\varrho:G\rightarrow\GL(V/\q)$ denotes the extension of $\Res_{L^+/\q}\rho$ that restricts to the identity on the subgroup $\g_m$, 
then $(V,\Psi,\varrho_\r\circ h)$ is a skew-Hermitian Hodge structure of type \eqref{restrictII}.
\end{itemize}
Here, it is understood that $\{a_\iota\}_{\iota\in L_{an}}$ is some family of integers satisfying $a_\iota\leq b_\iota:=1-a_{\iota\circ*}$. We define the 
character $\chi_\rho$ of a unitary representation $\rho$ to be the composition $G_L^1\stackrel{\rho}{\rightarrow}\GL(V/L)\stackrel{\tr}{\rightarrow}\a_L^1$, 
it is a $L$-rational class function satisfying $\chi_\rho^*(\gamma)=\chi_\rho(\gamma^{-1})$.

\subsection{$\vartheta$-gauged representations}
\label{twistingVI}
In order to describe an important operation on the set of isometry classes of skew-Hermitian Hodge structures, we need to introduce certain combinatorial data: Fix an element 
$\vartheta\in\Gal(R/\q)$. Minimal non-empty $\vartheta$-invariant subsets of $L_{an}$ are called cycles. A cycle is called inert if it is invariant under composition with $*$, and 
otherwise it is called split. A set $S\subset L_{an}$ is called a semi-cycle if and only if the sets $S\circ*$ and $S$ are disjoint, and their union $S\circ*\cup S$ is a (necessarily inert) 
cycle. A function $\bd^+:L_{an}\rightarrow\z$ with $\bd^+(\vartheta\circ\iota)\leq\bd^+(\iota)+1$ and $\bd^+(\iota\circ*)=\bd^+(\iota)\geq0$ will be called a $\vartheta$-multidegree.

\begin{defn}
\label{vqflex}
Consider a family of integers $\{a_\iota\}_{\iota\in L_{an}}$ satisfying $a_\iota\leq b_\iota:=1-a_{\iota\circ*}$ and let $\bd:L_{an}\rightarrow L_{an}$ 
be the function defined by $\iota\mapsto\vartheta^{-\bd^+(\iota)}\circ\iota$, for some fixed $\vartheta$-multidegree $\bd^+$.
\begin{itemize}
\item
We say that a function $\bj:L_{an}\rightarrow\z$ is a $\vartheta$-gauge of type \eqref{restrictII} if the following properties hold:
\begin{itemize}
\item[(G1)]
For each $\iota\in L_{an}$ one has $\bj(\iota\circ*)=-\bj(\iota)$.
\item[(G2)]
For every $\iota\in L_{an}$ and $l\in[a_\iota,b_\iota-1]$, there exists a unique $\kappa\in L_{an}$, with $\bd(\kappa)=\iota$ and $\bj(\kappa)=l$. 
\end{itemize}
\item
We will say that the family $\{a_\iota\}_{\iota\in L_{an}}$ is normalized if the following properties hold:
\begin{itemize}
\item[(N1)]
For every $\iota\in L_{an}$ one has: $b_\iota-a_\iota=|\bd^{-1}(\iota)|$
\item[(N2)]
For every semi-cycle $S\subset L_{an}$ one has:
$$\Card(\{\kappa\in S|\bd(\kappa)\notin S\})\equiv\sum_{\iota\in S}a_\iota\pmod2$$
\item[(N3)]    
For every cycle $\Theta\subset L_{an}$ one has: $0=\sum_{\iota\in\Theta}a_\iota$
\end{itemize}
\end{itemize}
\end{defn}

Suppose that the condition (N1) holds. Then (N3) holds if and only if it holds for all split cycles, while (N2) holds if and only if it holds 
for one arbitrary choice of semi-cycle $S$, within each inert cycle. The remarkable parity condition (N2) already entered into a $p$-adic 
consideration (namely in proposition \ref{weirdI}), and now it is going to enter into a real analytic consideration in the proof of following.

\begin{lem}
\label{twistingVII}
Fix integers $a_\iota\leq b_\iota=1-a_{\iota\circ*}$, and a pair $(\bd^+,\bj)$ satisfying the conditions (N1), (N2), (N3), (G1) and (G2) in definition 
\ref{vqflex}. Consider a skew-Hermitian Hodge structure $(V,\Psi,h)$, such that the Hodge decomposition of $V_\iota$ is of type contained in 
\eqref{restrictII} and let $h_\iota^{p,q}$ be its Hodge numbers. Then there exists a skew-Hermitian Hodge structure $(\tilde V,\tilde\Psi,\tilde h)$ such that:
\begin{itemize}
\item[(i)]
For every finite inert place $v$ of $L$ there exists a $L_v$-linear isometry from $(L_v\otimes_LV,\Psi_v)$ 
to the skew-Hermitian $L_v$-vector space $(L_v\otimes_L\tilde V,\tilde\Psi_v)$.
\item[(ii)]
The Hodge numbers of $(\tilde V,\tilde\Psi,\tilde h)$ are given by 
\begin{equation}
\label{twistingIV}
\tilde h_\kappa^{\tilde p,\tilde q}=
\begin{cases}\sum_{p<-\bj(\kappa)}h_{\bd(\kappa)}^{p,q}&(\tilde p,\tilde q)=(-1,0)\\
\sum_{p\geq-\bj(\kappa)}h_{\bd(\kappa)}^{p,q}&(\tilde p,\tilde q)=(0,-1)\\
0&(\tilde p,\tilde q)\notin\{(-1,0),(0,-1)\}\end{cases}
\end{equation}
for every $\kappa\in L_{an}$.
\end{itemize}
\end{lem}
\begin{proof}
Recall that the signatures (resp. discriminants) of the envisaged forms $\sqrt{-1}\tilde\Psi_\kappa$ have to be equal to 
$\tilde h_\kappa^{-1,0}-\tilde h_\kappa^{0,-1}$ (resp. $(-1)^{\tilde h_\kappa^{0,-1}}$), while the signatures (resp. discriminants) of the given 
forms $\sqrt{-1}\Psi_\iota$ are equal to $\sum_{p+q=-1}(-1)^qh_\iota^{p,q}$ (resp. $\prod_{p+q=-1}(-1)^{qh_\iota^{p,q}}$). Choose a disjoint union 
$L_{an}=S\circ*\cup S$. In fact the only issue is the existence of a skew-Hermitian space $(\tilde V,\tilde\Psi)$, and all we have to do is check the congruence: 
$$\sum_{\iota\in S}\sum_{p+q=-1}qh_\iota^{p,q}\equiv\sum_{\kappa\in S}\tilde h_\kappa^{0,-1}\pmod2$$ 
It is easy to see that we have
$$-a_{\iota\circ*}\dim_LV=\sum_{p+q=-1}qh_\iota^{p,q}+\sum_{\bd(\kappa)=\iota}\tilde h_\kappa^{0,-1},$$ 
for every $\iota$. Summation over all $\iota\in S$ yields
$$(-\sum_{\iota\circ*\in S}a_\iota)\dim_LV=\sum_{\iota\in S}\sum_{p+q=-1}qh_\iota^{p,q}+\sum_{\kappa\in T}\tilde h_\kappa^{0,-1},$$
where $T:=\bd^{-1}(S)$. When calculating $\pmod2$ we find that the left-hand side agrees with $\Card(S-T)\dim_LV$, according to the conditions (N2) and (N3). It remains to show that 
$\sum_{\kappa\in S}\tilde h_\kappa^{0,-1}+\sum_{\kappa\in T}\tilde h_\kappa^{0,-1}$ agrees with $\Card(S-T)\dim_LV$  too. In these sums one can ignore $S\cap T$ 
and $L_{an}-(S\cup T)$, and the contribution from each element $\kappa$ in the difference set is precisely $\tilde h_\kappa^{0,-1}+\tilde h_{\kappa\circ*}^{0,-1}=\dim_LV$.
\end{proof}

Fix integers $a_\iota\leq b_\iota=1-a_{\iota\circ*}$ and a $\vartheta$-multidegree $\bd^+$, and let $(V,\rho,\Psi)$ (resp. $\bj$) be a unitary representation (resp. $\vartheta$-gauge) 
of type \eqref{restrictII}. For any $\iota\in L_{an}-\bd(L_{an})$, the conditions (U2) and (G2) imply that $(-a_\iota,a_\iota-1)=(-b_\iota,b_\iota-1)$ is the sole bi-weight occuring 
in the Hodge decomposition of $V_\iota$, while for general elements $\iota\in\bd(L_{an})$ the endpoints $a_\iota\leq b_\iota$ for which the conditions (U2) and (G2) hold 
simultaneously, are certainly not always unique. Let us say that $(V,\rho,\Psi,\bj)$ is a $\vartheta$-gauged $L$-unitary representation if (U1), (U2), (G1) and (G2) hold 
for a family of integers $a_\iota\leq1-a_{\iota\circ*}=:b_\iota$ satisfying the condition (N1). In this case the intervals $[a_\iota,b_\iota]$ are maximal in the sense that 
$a_\iota=\min\{\bj(\kappa)|\bd(\kappa)=\iota\}$ and $b_\iota-1=\max\{\bj(\kappa)|\bd(\kappa)=\iota\}$ holds for all $\iota\in\bd(L_{an})$, and in particular the family 
$\{a_\iota\}_{\iota\in L^{an}}$ is already uniquely determined (namely by the Hodge numbers of $(V,h)$ and the pair $(\bd^+,\bj)$ together with $\vartheta$).

\begin{rem}
\label{tricky}
From now onwards we always assume that condition (N1) is fulfilled. By the preceding comments this will allow the convention 
of saying ``$(V,\rho,\Psi,\bj)$ satisfies (N2) or/and (N3)'' if and only if $\{a_\iota\}_{\iota\in L_{an}}$ satisfies (N2) or/and (N3).
\end{rem}

\subsection{Construction of $_{K^p}\M_{\bT,\gp}$}
\label{special}

The input datum for poly-unitary moduli problems consists of a fixed element $\vartheta\in\Gal(R/\q)$ together with the following:

\begin{defn}
\label{twistingI}
Let $(G,h)$ be a Shimura datum with coefficients in $L^+$, as in the beginning of this section. A pair of families 
$$\bP=(\{(V_i,\Psi_i,\rho_i,\tilde V_i,\tilde\Psi_i,\tilde h_i,\varepsilon_i^\infty)\}_{i\in\Lambda},\{(R_\pi,\iota_\pi)\}_{\pi\in\Pi}),$$
is called a $\vartheta$-poly-$L$-unitary Shimura datum for $(G,h)$ if it enjoys the following properties:
\begin{itemize}
\item[(P0)]
The index set $\Lambda$ has finite cardinality, for each of its elements $i$ the triple $(\tilde V_i,\tilde\Psi_i,\tilde h_i)$ is a 
skew-Hermitian Hodge structure, the triple $(V_i,\rho_i,\Psi_i)$ is a $L$-unitary representation of $(G,h)$ and $\varepsilon_i^\infty$ is a 
$\a^\infty\otimes L$-linear isometry from $\a^\infty\otimes V_i$ to $\a^\infty\otimes\tilde V_i$. Moreover, every $\vartheta$-cycle contains 
at least one element $\iota$ for which all but at most one Hodge numbers $\tilde h_{i,\iota}^{p,q}$ of any $(\tilde V_i,\tilde h_i)$ vanish.
\item[(P1)]
Each element of the set $\Pi$ is a subset $\pi\subset\Lambda$ of odd cardinality, such that the triples $(\tilde V^\pi,\tilde\Psi^\pi,\tilde h^\pi)$ are skew-Hermitian Hodge 
structures of type $\{(-1,0),(0,-1)\}$, where $\tilde h^\pi$ (resp. $\tilde\Psi^\pi$) denote the canonical Hodge (resp. skew-Hermitian) structure on the tensor product: 
\begin{equation*}
\tilde V^\pi:=(\bigotimes_{i\in\pi}\tilde V_i)(\frac{1-\Card(\pi)}2),
\end{equation*}
which is formed in the $L$-linear $\otimes$-category of skew-Hermitian Hodge structures with coefficients in $L$, cf. \cite[section 3.1]{habil}. 
\item[(P2)]
$\{R_\pi\}_{\pi\in\Pi}$ is a family of $L$-algebras with positive involution of the second kind, and $(\{(V_i,\Psi_i,\rho_i)\}_{i\in\Lambda},\{(R_\pi,\iota_\pi)\}_{\pi\in\Pi})$ 
is a $L$-multi-unitary collection for the $L^+$-group $(\g_{m,L^+}^\Lambda\times_{L^+}G^1)/\{\pm1\}^\Lambda$ in the sense of definition \ref{exampleIX}. 
Furthermore, all singletons are elements of $\Pi$, and the $L$-algebra of $L$-linear endomorphisms of the $L$-unitary representation $(V_i,\rho_i,\Psi_i)$ 
is equal to the image of $\iota_{\{i\}}$ for every $i\in\Lambda$.
\end{itemize}
We say that $\bT=(\bd^+,\{(V_i,\Psi_i,\rho_i,\bj_i)\}_{i\in\Lambda},\{(R_\pi,\iota_\pi)\}_{\pi\in\Pi})$ is a $L$-metaunitary datum for $(G,h)$ if there exists a 
family $\{(\tilde V_i,\tilde\Psi_i,\tilde h_i,\varepsilon_i^\infty)\}_{i\in\Lambda}$, for which (P0)-(P2) together with the following three additional requirements hold:
\begin{itemize}
\item[(P3)]
$\bd^+$ is a $\vartheta$-multidegree, each quadruple $(V_i,\Psi_i,\rho_i,\bj_i)$ is a $\vartheta$-gauged $L$-unitary representation 
in the  sense that the conditions (U1), (U2), (G1), (G2), (N1), (N2) and (N3) of the previous subsection \ref{twistingVI} hold.
\item[(P4)]
If $\varrho_i$ stands for the extension of $\Res_{L^+/\q}\rho_i$, as in (U2), then the Hodge numbers of $(\tilde V_i,\tilde\Psi_i,\tilde h_i)$ are obtained from the 
Hodge numbers of $(V_i,\Psi_i,\varrho_{i,\r}\circ h)$ by means of the formula \eqref{twistingIV}, when using the $\vartheta$-gauge $\bj_i$ together with $\bd^+$.
\end{itemize}
\end{defn}

From now on, we suppose that $\vartheta$ fixes some maximal ideal $\gr\subset\O_R$ and induces the absolute Frobenius on $\O_R/\gr$, of which the 
characteristic is also assumed to be an odd prime $p$ which is unramified in $\O_L$. We will say that $\bP$ (resp. $\bT$) is unramified at $p$ if and only if 
$(\{(V_i,\Psi_i,\rho_i)\}_{i\in\Lambda},\{(R_\pi,\iota_\pi)\}_{\pi\in\Pi})$ is unramified at $p$ in the sense of part \ref{hilfssaetze} of the appendix. From now on we do 
fix choices of $(\tilde V_i,\tilde\Psi_i,\tilde h_i,\varepsilon_i^\infty)$. By formation of $\bigotimes_{i\in\pi}$ and $\bigoplus_{i\in\Lambda}$ one obtains further isometries 
\begin{eqnarray*}
&&\a^\infty\otimes V^\pi\stackrel{\varepsilon^{\pi,\infty}}{\rightarrow}\a^\infty\otimes\tilde V^\pi\\
&&\a^\infty\otimes V\stackrel{\varepsilon^\infty}{\rightarrow}\a^\infty\otimes\tilde V,
\end{eqnarray*}
when using the notations of (P1) along with $\tilde V:=\bigoplus_{i\in\Lambda}\tilde V_i$ and $V:=\bigoplus_{i\in\Lambda}V_i$. 

\subsubsection{First Step}
\label{ready}
We need to associate a couple of further structures to an unramified $\vartheta$-poly-$L$-unitary Shimura datum $\bP$, which we assume 
to be equipped with self-dual $\z_{(p)}\otimes\O_L$-lattices $\gV_{i,p}\subset V_i$ giving rise to a hyperspecial compact open subgroup
\begin{equation}
\label{twistingIII}
\bigcap_{i\in\Lambda}\varrho_i^{-1}(U_{i,p})=:U_p\subset G(\q_p),
\end{equation}
where the compact open subgroup $U_{i,p}\subset\GU(V_i/L,\tr_{L/\q}\Psi_i)(\q_p)$ is the stabilizer of $\gV_{i,p}$, as in \eqref{twistingII} (cf. lemma \ref{polyV}). Let us write 
$\tilde X^\pi$ for the $\tilde G^\pi(\r)$-conjugacy class of $\tilde h^\pi$, where $\tilde G^\pi:=\GU(\tilde V^\pi/L,\tr_{L/\q}\Psi^\pi)$. This sets up a family of canonical PEL-type Shimura 
data $(\tilde G^\pi,\tilde X^\pi)$. Moreover, regarding $\tilde V$ as a skew-Hermitian module over the $*$-algebra $L^\Lambda=\underbrace{L\oplus\dots\oplus L}_\Lambda$ 
yields yet another PEL-type Shimura datum $(\tilde G,\tilde X)$, where $\tilde G$ is the $\q$-group of $L^\Lambda$-linear similitudes of $\tilde V$. For later 
reference we put $\ge_i\subset\O_L^\Lambda$ for the ideal generated by the idempotent $(1,\dots,1,0,1,\dots,1)$ with the ``$0$'' in the $i$ th position, and we 
let $\ge_\iota\subset\O_R\otimes\O_L$ be the kernel of $\O_R\otimes\O_L\stackrel{\id_R\otimes\iota}{\rightarrow}\O_R$. Finally observe that there exist canonical 
morphisms of Shimura data $g^\pi:(\tilde G,\tilde X)\rightarrow(\tilde G^\pi,\tilde X^\pi)$, and hence canonical $\tilde G(\a^\infty)$-equivariant morphisms of Shimura varieties:
\begin{equation}
\label{bekanntI}
M(\tilde G,\tilde X)\stackrel{g^\pi}{\rightarrow}M(\tilde G^\pi,\tilde X^\pi)
\end{equation}
We need to collect further facts on integrality: Any compact open subgroup $K^p\subset G(\a^{\infty,p})$ can be written as an inverse image of some compact open subgroup 
$\tilde U^p\subset\tilde G(\a^{\infty,p})$, via the canonical embedding $\tilde\varrho^{\infty,p}:G_{\a^{\infty,p}}\rightarrow \tilde G_{\a^{\infty,p}}$, which we obtain immediately from 
$\varepsilon^{\infty,p}$. The previously introduced $\varepsilon_p^\pi$'s and $\varepsilon_p$ yield specific hyperspecial subgroups 
$\tilde U_p^\pi\subset\tilde G^\pi(\q_p)$ and $\tilde U_p\subset\tilde G(\q_p)$ by ``transport of structure'' (as we have already fixed our integral structure 
$\{\gV_{i,p}\}_{i\in\Lambda}$). Observe that $\tilde U_p^\pi$ contains $g^\pi(\tilde U_p)$, and let $\tilde U^\pi$ be any compact open subgroup of $\tilde G^\pi(\a^\infty)$ 
containing $g^\pi(\tilde U^p)\tilde U_p^\pi$. Let $\tilde U\subset\tilde G(\a^\infty)$ be the product of $\tilde U^p$ and $\tilde U_p$, so that \eqref{bekanntI} induces a morphism from 
$_{\tilde U}M(\tilde G,\tilde X)$ to $_{\tilde U^\pi}M(\tilde G^\pi,\tilde X^\pi)$. Let $_{\tilde U}\S/\O_{R_\gr}$ and $_{\tilde U^\pi}\S^\pi/\O_{R_\gr}$ be the usual moduli interpretations 
for these unitary group Shimura varieties, which are smooth and proper over $\O_{R_\gr}$ according to \cite{kottwitz1} and \cite{morita} (the latter only because of the last sentence 
in (P0)). We write $_{\tilde U}\essbar/\O_R/\gr$ and $_{\tilde U^\pi}\essbar^\pi/\O_R/\gr$ for their respective special fibers. We fix a set $S_p$ of extensions to $L$ of the primes of 
$L^+$ over $p$, and for each $\gq\in S_p$ we let $r_\gq$ be the degree of $\gq^+:=\gq\cap\O_{L^+}$, and we fix an embedding $\iota_\gq:L\hookrightarrow R$ with 
$\iota_\gq(\gq)\subset\gr$, so that: 
$$\vartheta^{r_\gq}\circ\iota_\gq=\begin{cases}\iota_\gq\circ*&\gq^*=\gq\\\iota_\gq&\text{otherwise}\end{cases}$$
At least up to a $\z_{(p)}$-isogeny we have a universal abelian scheme over $_{\tilde U}\S$ at our 
disposal, and picking some representative $Y$ allows us to introduce the important line bundles
\begin{equation}
\label{bundleI}
\det((\Lie Y[\ge_i])[\ge_{\vartheta^{-\sigma}\circ\iota_\gq}])^{-1}=:\L_{i,\gq,\sigma}\in\Pic(_{\tilde U}\S),
\end{equation} 
whose product $\L:=\bigotimes_{i\in\Lambda}\bigotimes_{\gq\in S_p}\bigotimes_{\sigma=0}^{r_\gq-1}\L_{i,\gq,\sigma}$ is ample, according to \cite{mori}. Following 
\cite[Theorem 4.8]{habil}, we write $Y^\pi$ for the pull-back of the universal abelian scheme on $_{\tilde U^\pi}\S^\pi$ to $_{\tilde U}\S$ by means of the canonical extension
\begin{equation}
\label{bekanntII}
{_{\tilde U}\S}\stackrel{g^\pi}{\rightarrow}{_{\tilde U^\pi}\S^\pi},
\end{equation}
of \eqref{bekanntI}. We tacitly omit the mentioning of level structures, but do notice that $\z_{(p)}\otimes\End_L(Y_S^\pi)$ is well-defined for every $S$-valued 
point of $_{\tilde U}\S$. Let $\P_\gq^\pi$ be the possibly skew-Hermitian, graded $_{\tilde U}\essbar$-display to $Y^\pi[\gq^\infty]\times_{\O_{R_\gr}}\O_R/\gr$, here 
note that the methods of loc.cit. are applicable only because the $p$-rank of the $\mod\gr$-reductions of at least one of $Y^\pi[\gq^\infty]$ or $Y^\pi[\gq^{*\infty}]$ 
vanishes for every $\gq\in S_p$, we denote $Y[\ge_i]=Y^{\{i\}}=:Y_i$ and $\P_\gq^{\{i\}}=:\P_{i,\gq}$. Moreover, there exist canonical comparison isomorphisms:
\begin{equation}
\label{Preisvergleich}
m_\gq^\pi:\dot\bigotimes_{i\in\pi}\P_{i,\gq}\rightarrow\P_\gq^\pi
\end{equation}
of possibly skew-Hermitian, graded $_{\tilde U}\essbar$-displays. Finally, the reduced induced subscheme structure on the Zariski closed subset of 
$_{\tilde U}\essbar$-points all of whose $\P_{i,\gq}$'s are isoclinal is denoted by $_{\tilde U}\Sb$. Now and again we need to invoke the projective limit 
$_{\tilde U_p}\S=\lim_{\tilde U^p\to1}{_{\tilde U}\S}$, which is a scheme with a right $\tilde G(\a^{\infty,p})$-action. At last, we need to introduce a family of certain orders  
$$\iota_\pi^{-1}(\z_{(p)}\otimes\O_L+\gf_\pi\End_{\O_L}(\gV_p^\pi))=:\R_{\gf,\pi}\subset R_\pi,$$
associated to ideals $\gf_\pi\subset\z_{(p)}\otimes\O_{L^+}$, where the self-dual $\z_{(p)}\otimes\O_L$-lattice 
$\gV_p^\pi$ stands for the $\O_L$-linear tensor product of the lattices $\gV_{i,p}\subset V_i$.\\ 
For a family of ideals $\gf_\pi\subset\z_{(p)}\otimes\O_{L^+}$ we let the $\O_{R_\gr}$-scheme $_{U^p}\gM_{\bP,\gr}^{\{\gf_\pi\}_{\pi\in\Pi}}$ represent the functor 
that sends a connected pointed base scheme $(S,s_0)$ to the set of $4+\Card(\Pi)$-tuples $(Y,\lambda,\iota,\ebar,\{y_\pi\}_{\pi\in\Pi})$ with the following properties:
\begin{itemize}
\item[(i)]
$(Y,\lambda,\iota,\ebar)$ is a $\z_{(p)}$-isogeny class of: homogeneously $p$-principally polarized abelian $S$-schemes $(Y,\lambda)$ together with a 
$*$-invariant action $\iota:\O_L^\Lambda\rightarrow\z_{(p)}\otimes\End(Y)$ satisfying the determinant condition with respect to the skew-Hermitian 
$L^\Lambda$-module $\tilde V$, and a $\pi_1^{\mathaccent 19 et}(S,s_0)$-invariant $U^p$-orbit $\ebar$ of $\O_L^\Lambda$-linear similitudes 
$$\bigoplus_i\eta_i=\eta:\a^{\infty,p}\otimes V\stackrel{\cong}{\rightarrow}H_1^{\mathaccent19 et}(Y_{s_0},\a^{\infty,p}).$$
\item[(ii)]
$y_\pi:\R_{\gf_\pi,\pi}\rightarrow\z_{(p)}\otimes\End_L(Y_S^\pi)$ is a $\O_L$-linear $*$-preserving 
homomorphism such that $\ebar$ contains at least one element $\eta$ rendering the diagrams
$$\begin{CD}
\a^{\infty,p}\otimes\bigotimes_{i\in\pi}\End_L(V_i)@<<<\R_{\gf_\pi,\pi}\\
@VVV@V{y_\pi}VV\\
\End(\bigotimes_{i\in\pi}H_1^{\mathaccent19 et}(Y_{i,s_0},\a^{\infty,p})(\frac{1-\Card(\pi)}2))@<{H_1^{\mathaccent19 et}}<<\End_L^0(Y_{s_0}^\pi)
\end{CD}$$
commutative, simultaneously for all $\pi\in\Pi$ where the vertical map on the left is given by:
$$\bigotimes_{i\in\pi}\phi_i\mapsto\bigotimes_{i\in\pi}\eta_i\circ\phi_i\circ\eta_i^{-1},$$
for some family of endomorphisms $\phi_i\in\a^{\infty,p}\otimes\End_L(V_i)$.
\end{itemize}

The projective limit $$\gM_{\bP,\gr}^{\{\gf_\pi\}_{\pi\in\Pi}}=\lim_{K^p\to1}{_{K^p}\gM_{\bP,\gr}^{\{\gf_\pi\}_{\pi\in\Pi}}}$$ is equipped with a right 
$G(\a^{\infty,p})$-action. Drawing on the above canonical embedding $\varepsilon^{\infty,p}$, there is a $G(\a^{\infty,p})$-equivariant morphism 
\begin{equation}
\label{forget}
\gM_{\bP,\gr}^{\{\gf_\pi\}_{\pi\in\Pi}}\rightarrow{_{\tilde U_p}\S},
\end{equation}
which at the finite levels recovers the tautological forgetful morphisms $(Y,\lambda,\iota,\ebar,\{y_\pi\}_{\pi\in\Pi})\mapsto(Y,\lambda,\iota,\ebar)$,
from $_{K^p}\gM_{\bP,\gr}^{\{\gf_\pi\}_{\pi\in\Pi}}$ to $_{\tilde U}\S$. 

\begin{rem}
\label{twistingIX}
These notations are justified as the above moduli problem does not depend on the choice of $\tilde U^p$, but only on its inverse 
image $K^p$. Notice that the introduction of the map \eqref{forget} does require the isomorphisms $\varepsilon_i^{\infty,p}$.
\end{rem}

Finally, the product of the line bundles \eqref{bundleI} remains ample, when pulled back to each finite layer $_{K^p}\gM_{\bP,\gr}^{\{\gf_\pi\}_{\pi\in\Pi}}$,
and these schemes are proper. This follows from the finiteness of the aforementioned forgetful maps, as the corresponding facts hold for 
$_{\tilde U}\S$ by \cite{mori} and \cite{morita}. The idea of $L$-poly-unitary moduli problems was implicitely present in \cite[Definition 5.3]{habil}.

\subsubsection{Second Step}
\label{steady}
Once and for all we fix a $\vartheta$-multidegree $\bd^+:L_{an}\rightarrow\n_0$ and a tuple 
$$(\{(V_i,\rho_i,\Psi_i,\bj_i,\tilde V_i,\tilde\Psi_i,\tilde h_i,\varepsilon_i^\infty)\}_{i\in\Lambda},\{(R_\pi,\iota_\pi)\}_{\pi\in\Pi}),$$ 
that satisfies (P0)-(P4) together with an integral structure $\{\gV_{i,p}\}_{i\in\Lambda}$ and a family of auxiliary ideals $\gf_\pi\subset\z_{(p)}\otimes\O_{L^+}$. Observe that the $L^+$-group 
$\Res_{L/L^+}\g_{m,L}^\Lambda$ is canonically contained in $(\g_{m,L^+}^\Lambda\times_{L^+}G^1)/\{\pm1\}^\Lambda$, and we shall use this scenario to introduce a cocharacter
\begin{equation*}
\alpha:\g_{m,\c}\rightarrow(\Res_{L/\q}\g_{m,L}^\Lambda)_\c\cong\g_{m,\c}^{\Lambda\times L_{an}} 
\end{equation*}
by decreeing its $(i,\iota)$-component to be given by $\alpha_{i,\iota}(z)=z^{a_{i,\iota}}$. Later on we need 
a family of $\pmod{r_\gq}$-multidegrees related to our $\vartheta$-multidegree by means of the formula 
$$\bd_\gq^+(\sigma):=\bd^+(\vartheta^{-\sigma}\circ\iota_\gq),$$ 
and we also put $r_\gq(\sigma):=\Card(\bd^{-1}(\{\vartheta^{-\sigma}\circ\iota_\gq\}))=\Card(\bd_\gq^{-1}(\{\sigma\}))$, where 
$\bd_\gq(\sigma)=\sigma+\bd_\gq^+(\sigma)$, so that $\bd(\vartheta^{-\sigma}\circ\iota_\gq)=\vartheta^{-\bd_\gq(\sigma)}\circ\iota_\gq$. Choose a prime $\gp\in\Spec\O_{ER}$ lying 
above $\gr\in\Spec\O_R$, let $f:=[\O_{ER}/\gp:\f_p]$ be its degree, and let us write $\f_{p^r}\subset\O_{ER}/\gp$ for the subfield of cardinality $p^r$, for any positive divisor $r\mid f$. Moreover, for arbitrary $\gq\in S_p$ we let $\G_\gq^1$ be the reductive $W(\f_{p^{r_\gq}})$-model of $G_\gq^1=G^1\times_{L^+,\iota_\gq}K(\f_{p^{r_\gq}})$ 
which is determined by the hyperspecial subgroup $U_\gq^1:=U_p\cap G^1(L_\gq^+)$ (cf. \eqref{twistingIII}). Let us explain how to recover the cocharacter 
$\alpha$ in these local settings: For every $i\in\Lambda$ and $\gq^*=\gq\in S_p$ (resp. $\gq^*\neq\gq\in S_p$), there are canonical inclusions 
$$\frac1{^{F^{r_\gq}}\zeta_{i,\gq}^1}=\zeta_{i,\gq}^1:\g_{m,W(\f_{p^{2r_\gq}})}\hookrightarrow\G_{\gq,W(\f_{p^{2r_\gq}})}^1,$$
(resp. $\zeta_{i,\gq}^1:\g_{m,W(\f_{p^{r_\gq}})}^\Lambda\hookrightarrow\G_\gq^1$) which are obtained by composing the respective 
scalar extension of the dilatation homomorphisms (cf. \eqref{dilatation}) with $\g_m\rightarrow\g_m^2;z\mapsto(z,\frac1z)$. Putting 
$a_{i,\gq,\sigma}:=a_{i,\vartheta^{-\sigma}\circ\iota_\gq}$, and $b_{i,\gq,\sigma}:=b_{i,\vartheta^{-\sigma}\circ\iota_\gq}$ determines cocharacters:
\begin{eqnarray*}
&&\alpha_{\gq,\sigma}:\g_{m,W(\f_{p^f})}\rightarrow(\g_{m,W(\f_{p^f})}\times_{W(\f_{p^f})}{^{F^{-\sigma}}\G_{\gq,W(\f_{p^f})}^1})/\{\pm1\};\\
&&z\mapsto\pm(z^{\frac{1-r_\gq(\sigma)}2},\prod_{i\in\Lambda}{^{F^{-\sigma}}\zeta_{i,\gq}^1}(z^{\frac{a_{i,\gq,\sigma}+b_{i,\gq,\sigma}-1}2}))
\end{eqnarray*}
For any $\gq\in S_p$ and any $\sigma\in\z/r_\gq\z$ we require the cocharacters 
$$\mu_{\gq,\sigma}:\g_{m,W(\f_{p^f})}\rightarrow(\g_{m,W(\f_{p^f})}\times_{W(\f_{p^f})}\G_{\gq,W(\f_{p^f})}^1)/\{\pm1\}$$
to lie in the conjugacy class of the cocharacters $\mu_{\vartheta^{-\sigma}\circ\iota_\gq}$, as introduced at the beginning of this section. 
Starting out from $\bd_\gq^+$ we form subsets $\Sigma_\gq$ and $\Omega_\gq$ of $\z/r_\gq\z$ according to the formulae \eqref{shiftI} and 
\eqref{shiftII}. Notice, that $\mu_{\gq,\sigma}=\alpha_{\gq,\sigma}$ for all $\sigma\notin\Sigma_\gq$. The $\Phi$-datum we wish to work with is:
\begin{equation}
\label{Fifi}
(\G_\gq,\{\upsilon_{\gq,\omega}\}_{\omega\in\Omega_\gq}),
\end{equation}
where $(\g_{m,W(\f_{p^{r_\gq}})}^\Lambda\times_{W(\f_{p^{r_\gq}})}\G_\gq^1)/\{\pm1\}^\Lambda=:\G_\gq$, and 
$^{F^{\bd_\gq^+(\omega)}}(\frac{\mu_{\gq,\bd_\gq(\omega)}}{\alpha_{\gq,\bd_\gq(\omega)}})=:\upsilon_{\gq,\omega}$ 
for all $\omega\in\Omega_\gq$. Consider the rings 
$$W(\f_{r_\gq})\otimes_{\iota_\gq,\O_{L^+}}\O_L=C_\gq=\begin{cases}W(\f_{p^{2r_\gq}})&\gq^*=\gq\\W(\f_{p^{r_\gq}})\oplus W(\f_{p^{r_\gq}})&\mbox{ otherwise}\end{cases}.$$
Next we will construct a family of gauged $C_\gq$-multi-unitary collections 
\begin{eqnarray*}
&&\bJ_\gq(\{\gf_\pi\}_{\pi\in\Pi})=\\
&&(\{(\V_{i,\gq},\Psi_{i,\gq},\rho_{i,\gq},\bj_{i,\gq})\}_{i\in\Lambda},\{(W(\f_{p^{r_\gq}})\otimes_{\iota_\gq,\O_{L^+}}\R_{\gf_\pi,\pi},\iota_{\pi,\gq})\}_{\pi\in\Pi})
\end{eqnarray*}
for each of the aforementioned $W(\f_{p^f})$-rational $\Phi$-data \eqref{Fifi}. Let $\Psi_{i,\gq}(x,y)=-\Psi_{i,\gq}(y,x)^*$ denote the perfect pairings on 
the selfdual $C_\gq$-lattices $\V_{i,\gq}:=W(\f_{r_\gq})\otimes_{\iota_\gq,\O_{L^+}}\gV_{i,p}$ and ditto for $\Psi_\gq^\pi(x,y)=-\Psi_\gq^\pi(y,x)^*$ on the 
$C_\gq$-linear tensor products $\V_\gq^\pi:=\bigotimes_{i\in\pi}\V_{i,\gq}$. The relation between the local gauges and the global ones is given by
\begin{equation}
\label{anfang}
\bj_{i,\gq}(\sigma)=\bj_i(\vartheta^{-\sigma}\circ\iota_\gq)-a_{i,\gq,\bd_\gq(\sigma)},
\end{equation} 
notice that $a_{i,\gq,\bd_\gq(\sigma)}=a_{i,\bd(\vartheta^{-\sigma}\circ\iota_\gq)}$. We are now in a position to appeal to proposition 
\ref{properIII} in order to introduce the provisional formally smooth $\f_{p^f}$-scheme $_{K^p}\tilde M_{\bT,\gp}$ rendering the diagram
\begin{equation}
\label{smoothVII}
\begin{CD}\prod_{\gq\in S_p}\Barb(\G_\gq,\{\upsilon_{\gq,\omega}\}_{\omega\in\Omega_\gq})@<{\prod_{\gq\in S_p}\tilde\P_\gq}<<{_{U^p}\tilde M_{\bT,\gp}}\\
@V{\prod_{\gq\in S_p}\fx^{\bJ_\gq(\{\gf_\pi\}_{\pi\in\Pi})}}VV@VVV\\
\prod_{\gq\in S_p}\gB_{\f_{p^f}}^{\bR_\gq(\{\gf_\pi\}_{\pi\in\Pi})}@<<<{_{K^p}\gM_{\bP,\gr}^{\{\gf_\pi\}_{\pi\in\Pi}}}
\end{CD}
\end{equation}
$2$-cartesian, as the Serre-Tate theorem (see e.g. \cite{katz2}) grants $_{K^p}\gM_{\bP,\gr}^{\{\gf_\pi\}_{\pi\in\Pi}}$ to be formally 
\'etale over $\prod_{\gq\in S_p}\gB_{\f_{p^f}}^{\bR_\gq(\{\gf_\pi\}_{\pi\in\Pi})}$. We take $\bR_\gq(\{\gf_\pi\}_{\pi\in\Pi})$ to be
$$(\{(\GU(\V_{i,\gq}/C_\gq,\Psi_{i,\gq}),\{\tilde\upsilon_{i,\gq,\sigma}\}_{\sigma\in\z/r_\gq\z}\}_{i\in\Lambda},
\{(W(\f_{p^{r_\gq}})\otimes_{\iota_\gq,\O_{L^+}}\R_{\gf_\pi,\pi},*)\}_{\pi\in\Pi}),$$ 
where $(\GU(\V_{i,\gq}/C_\gq,\Psi_{i,\gq}),\{\tilde\upsilon_{i,\gq,\sigma}\}_{\sigma\in\z/r_\gq\z})$ is the standard $\Phi$-datum arising from the standard 
$\Phbar$-datum $(\GU(\V_{i,\gq}/C_\gq,\Psi_{i,\gq}),\{(^{F^{-\omega}}\rho_{i,\gq})\circ\upsilon_{\gq,\omega}\}_{\omega\in\Omega_\gq})$ by plugging the function 
\eqref{anfang} into our formalism \eqref{heart} of subsubsection \ref{Universum}. The above $2$-cartesian diagram \eqref{smoothVII} tells us, that the $1$-morphism 
$$_{K^p}\tilde M_{\bT,\gp}\rightarrow\prod_{\gq\in S_p}\Barb(\G_\gq,\{\upsilon_{\gq,\omega}\}_{\omega\in\Omega_\gq})$$ 
is formally \'etale, and this property is shared by the limit $\tilde M_{\bT,\gp}=\lim_{K^p\to1}{_{K^p}\tilde M_{\bT,\gp}}$. In addition the $G(\a^{\infty,p})$-equivariance 
of the natural $1$-morphism from $\gM_{\bP,\gr}^{\{\gf_\pi\}_{\pi\in\Pi}}$ to $\prod_{\gq\in S_p}\gB^{\bR_\gq(\{\gf_\pi\}_{\pi\in\Pi})}$ is inherited by the 
$(\G_\gq,\{\upsilon_{\gq,\omega}\}_{\omega\in\Omega_\gq})$-displays $\tilde\P_\gq$ over the $G(\a^{\infty,p})$-scheme $\tilde M_{\bT,\gp}$.

\subsubsection{Third Step}
\label{go}
The previous two steps culminated in the construction of the provisional $\f_{p^f}$-schemes $_{K^p}\tilde M_{\bT,\gp}$. Let us sum up what this means:
\begin{itemize}
\item
Proposition \ref{properIII} implies, that the reduced induced subscheme structure on $_{K^p}\tilde M_{\bT,\gp}$ is finite over $_{K^p}\gM_{\bP,\gr}^{\{\gf_\pi\}_{\pi\in\Pi}}$, 
in particular the product of the line bundles \eqref{bundleI} remains ample, when pulled back to $_{K^p}\tilde M_{\bT,\gp}$ (cf. \cite[Chapter III, Exercise 5.9]{hartshorne}). 
\item
By corollary \ref{smoothIV} $_{K^p}\tilde M_{\bT,\gp,red}$ is smooth over $\f_{p^f}$, and for every closed point $x\in{_{K^p}\tilde M_{\bT,\gp}}$ 
the complete local ring $\hat\O_{_{K^p}\tilde M_{\bT,\gp},x}$ prorepresents the universal equicharacteristic deformation of 
the fiber of $\prod_{\gq\in S_p}\tilde\P_\gq$ over $x$ (N.B.: $_{K^p}\tilde M_{\bT,\gp}$ might be non-noetherian).
\end{itemize}
Theorem \ref{fake} yields the $2$-commutative diagram
$$\begin{CD}
\prod_{\gq\in S_p}\Barb(\G_\gq,\{\frac{\mu_{\gq,\sigma}}{\alpha_{\gq,\sigma}}\}_{\sigma\in\Sigma_\gq})@<{\prod_{\gq\in S_p}\Parb_\gq}<<{_{K^p}\Marb_{\bT,\gp}}\\
@V{\prod_{\gq\in S_p}\fx^{\bd_\gq^+}}VV@VVV\\
\prod_{\gq\in S_p}\Barb(\G_\gq,\{\upsilon_{\gq,\omega}\}_{\omega\in\Omega_\gq})@<{\prod_{\gq\in S_p}\tilde\P_\gq}<<{_{K^p}\tilde M_{\bT,\gp,red}}
\end{CD},$$
which becomes $2$-cartesian upon evaluation on algebraically closed fields, as does the diagram
$$\begin{CD}
\Barb(\gG_p,\mu_p)@<\Parb<<{_{K^p}\Marb_{\bT,\gp}}\\
@VVV@VVV\\
\prod_{\gq\in S_p}\gB_{\f_{p^f}}^{\bR_\gq(\{\gf_\pi\}_{\pi\in\Pi})}@<<<{_{K^p}\gM_{\bP,\gr,\f_{p^f}}^{\{\gf_\pi\}_{\pi\in\Pi}}}
\end{CD},$$
which can be obtained from corollary \ref{weirdII} and the canonical morphism 
$_{K^p}\gM_{\bP,\gr,\f_{p^f}}^{\{\gf_\pi\}_{\pi\in\Pi}}\rightarrow\Spec W(\f_{p^f})\rightarrow\Barb(\g_{m,\z_p},\delta)$ together with the diagram
$$\begin{CD}
\prod_{\gq\in S_p}\Barb(\G_\gq,\{\frac{\mu_{\gq,\sigma}}{\alpha_{\gq,\sigma}}\}_{\sigma\in\Sigma_\gq})@<<<\Barb(\gG_p,\mu_p)\\
@VVV@VVV\\
\prod_{\gq\in S_p}\Barb(\g_{m,W(\f_{r_\gq})},\{\delta\}_{\sigma\in\z/r_\gq\z})@<<<\Barb(\g_{m,\z_p},\delta)
\end{CD},$$
which is $2$-cartesian too. In view of corollary \ref{lift} to describe a lift $_{K^p}\M_{\bT,\gp}$ of $_{K^p}\Marb_{\bT,\gp}$ going with it, all we have to do is define certain characters: 
$\chi_{i,\gq,\sigma}:\I_0^{\mu_{\gq,\sigma}}\rightarrow\g_{m,W(\f_{p^f})}$ for $\gq\in S_p$, $\sigma\in[0,r_\gq-1]$ and $i\in\Lambda$, such that $\omega_{q_0(\Parb_\gq)}(\chi_{i,\gq,\sigma})$ 
agrees with the line bundle $\L_{i,\gq,\sigma}$, which is given by the formula \eqref{bundleI}. We let $\chi_{i,\gq,\sigma}$ be the $\I_0^{\mu_{\gq,\sigma}}$-character defined by 
$\prod_{l\in\z}\chi_{\mu_{\gq,\sigma}^{-1}}(-l,{^{F^{-\sigma}}\rho_{i,\gq}})^{d_{i,\gq,\sigma,l}}$ where 
$$d_{i,\gq,\sigma,l}:=\sum_{\bd_\gq(\omega)=\sigma,\bj_{i,\gq}(\omega)<l}p^{\bd_\gq^+(\omega)}.$$
At last we apply the said corollary to the $\I_0^{\mu_p}$-character defined by
$$\prod_{i\in\Lambda}\prod_{\gq\in S_p}\prod_{\sigma=0}^{r_\gq-1}\chi_{i,\gq,\sigma}=\chi:\I_0^{\mu_p}\rightarrow\g_{m,W(\f_{p^f})},$$ 
and we are done because of the subsubsection \ref{Haar}, here notice, that we have a canonical inclusion:
$$\I_0^{\mu_p}\hookrightarrow\prod_{\gq\in S_p}\prod_{\sigma=0}^{r_\gq-1}\I_0^{\mu_{\gq,\sigma}}$$ 
Observe the $2$-commutativity of the diagram:
$$\begin{CD}
\B(\gG_p,\mu_p)@<\P^{(\nu)}<<{_{K^p}\M_{\bT,\gp}^{(\nu)}}\\
@VVV@VVV\\
\prod_{\gq\in S_p}\gB_{\f_{p^f}}^{\bR_\gq(\{\gf_\pi\}_{\pi\in\Pi})}@<<<{_{K^p}\gM_{\bP,\gr,\f_{p^f}}^{\{\gf_\pi\}_{\pi\in\Pi}}}
\end{CD}$$
Accordingly we write $\P^{(\nu)}$ for the canonical $(\gG_p,\mu_p)$-display over $_{K^p}\M_{\bT,\gp}^{(\nu)}={_{K^p}\M_{\bT,\gp,W_\nu(\f_{p^f})}}$, and let 
$_{K^p}\F_{\bT,\gp}^{(\nu)}\in\ton(\I_0^{\mu_p})(_{K^p}\M_{\bT,\gp}^{(\nu)})$ be its lowest truncation. An analogous truncation of Witt-connections yields canonical elements 
$$\nabla_{\bT,\gp}^{(\nu)}\in \con_{W_\nu(\f_{p^f})}({_{U^p}\F_{\bT,\gp}^{(\nu)}}\times^{\I_0^{\mu_p}}\gG_p/{_{K^p}\M_{\bT,\gp}^{(\nu)}}),$$ 
whose curvatures vanish. Let $({_{K^p}\F_{\bT,\gp}},\nabla_{\bT,\gp})$ be the limit of the projective system $({_{K^p}\F_{\bT,\gp}^{(\nu)}},\nabla_{\bT,\gp}^{(\nu)})$ as 
$\nu\to\infty$. In all of this the omission of level structures indicates passage to the limit, i.e. $\M_{\bT,\gp}:=\lim_{K^p\to1}{_{K^p}\M_{\bT,\gp}}$, note that $\F_{\bT,\gp}$ 
is a $G(\a^{\infty,p})$-equivariant locally trivial principal homogeneous space for $\I_0^{\mu_p}$ over $\M_{\bT,\gp}$ and $G(\a^{\infty,p})$ fixes $\nabla_{\bT,\gp}$.

\subsection{Set-theoretic properties}
\label{overFq}
In this subsection we study quasi-isogenies: Let us write $\gamma:Y'\dashrightarrow Y$ for an element of $\Hom^0(Y',Y)$, whenever $Y$ and $Y'$ 
are abelian schemes over some base $S$. The smallest integer $n$ such that $\gamma\in\frac 1n\Hom(Y',Y)$ is called the denominator of $\gamma$. 
If $Y$ and $Y'$ are equipped with polarizations $\lambda:Y\rightarrow\check Y$ and $\lambda':Y'\rightarrow\check Y'$ we say that $\gamma$ is a 
quasi-isogeny provided that the pull-back of $\lambda$ along say $n\gamma:Y'\rightarrow Y$ agrees with $\lambda'$ up to a multiple. We say that 
$\gamma$ is a $p'$-quasi-isogeny if the denominators of $\gamma$ and $\gamma^{-1}$ are coprime to $p$. Let $\gamma:Y_1\dashrightarrow Y_2$ and 
$\gamma':Y'_1\dashrightarrow Y'_2$ be quasi-isogenies over an algebraically closed field $k$. We say that they possess a common generization if there exists 
\begin{itemize}
\item
two $k$-valued points $\xi$ and $\xi'$ on some irreducible normal $k$-variety $S$,
\item
a quasi-isogeny $\gamma'':Y''_1\dashrightarrow Y''_2$ over $S$
\end{itemize}
such that $\gamma''\times_{S,\xi}k\cong\gamma$ and $\gamma''\times_{S,\xi'}k\cong\gamma'$ holds. If $\ch(k)=p$ we write $\d(Y)$ for the covariant Grothendieck-Messing 
crystalline Dieudonn\'e theory of an abelian variety $Y$ over $k$, and we let $\d^0(Y)$ be the (non-effective) $F$-isocrystal whose underlying $K(k)$-space is $\q\otimes\d(Y)$, 
and whose Frobenius operator $K(k)\otimes_{F,K(k)}\d^0(Y)\rightarrow\d^0(Y)$ is the inverse of $\d(F_{Y/k})$, where $F_{Y/k}:Y\rightarrow Y\times_{k,F}k;y\mapsto y^p$ is the 
relative Frobenius. Suppose that $k$ is an algebraically closed extension of $\O_R/\gr$. Note that to any $k$-valued point $\xi$ of $_{\tilde U}\essbar$ one can associate an isomorphism
$$m_{\xi,\mathaccent19 et}^\pi:\bigotimes_{\a^{\infty,p}\otimes L}H_1^{\mathaccent19 et}(Y_{i,\xi},\a^{\infty,p})
\stackrel{\cong}{\rightarrow}H_1^{\mathaccent19 et}(Y_\xi^\pi,\a^{\infty,p})(\frac{\Card(\pi)-1}2),$$ 
and as a consequence of Deligne's theory of absolute Hodge cycles, \cite{deligne5} this assignment is functorial in $k$ 
(Sketch: Once one chooses a $\c$-lift $\kappa$ of $\xi$ one obtains an $\O_L$-linear isomorphism between the Hodge structures 
$\bigotimes_{\O_L}H_1(Y_{i,\kappa}(\c),\z)$ and $H_1(Y_\kappa^\pi(\c),\z)(\frac{\Card(\pi)-1}2)$ and $Y_\xi^\pi$ is the reduction of $Y_\kappa^\pi$, 
see \cite[Section 4.4]{habil} for some more details). The special fiber of \eqref{Preisvergleich} over $\xi$ sets up a further canonical isomorphism:
$$m_{\xi,cris}^\pi:\bigotimes_{K(k)\otimes L}\d(Y_{i,\xi})\stackrel{\cong}{\rightarrow}\d(Y_\xi^\pi)(\frac{\Card(\pi)-1}2)$$ 
We need two more functoriality properties which are slightly less immediate consequences of Deligne's theory of absolute Hodge cycles:

\begin{thm}
\label{bekanntIII}
Let $\xi$ and $\pi$ be as above, then the involution preserving algebra maps
$$op_{\xi,\mathaccent 19 et}^{\pi}:\bigotimes_{\a^{\infty,p}\otimes L}\End_{\a^{\infty,p}\otimes L}(H_1^{\mathaccent 19 et}(Y_{i,\xi},\a^{\infty,p}))
\rightarrow\End_{\a^{\infty,p}\otimes L}(H_1^{\mathaccent 19 et}(Y_\xi^\pi,\a^{\infty,p}))$$
and
$$op_{\xi,cris}^\pi:\bigotimes_{K(k)\otimes L}\End_{K(k)\otimes L}(\d^0(Y_{i,\xi}))\rightarrow\End_{K(k)\otimes L}(\d^0(Y_\xi^\pi))$$
that are defined by $(\dots,f_i,\dots)\mapsto\bigotimes_{i\in\pi}f_i$ with the help of the above comparison isomorphisms, send 
$\bigotimes_L\End_L^0(Y_{i,\xi})$ into $\End_L^0(Y_\xi^\pi)$, moreover the two maps thus induced are equal.
\end{thm}

\begin{thm}
\label{bekanntIV}
To any two points $\xi_1,\xi_2\in{_{\tilde U}\essbar(k)}$, and to every $\O_L^\Lambda$-linear quasi-isogeny $\gamma:Y_{\xi_1}\dashrightarrow Y_{\xi_2}$ 
there is a canonical quasi-isogeny $g^\pi(\gamma):Y_{\xi_1}^\pi\dashrightarrow Y_{\xi_2}^\pi$, such that application of $\d^0$ and 
$H_1^{\mathaccent 19 et}(\dots,\a^{\infty,p})$ recovers the usual tensor-products (with coefficients in $L$).
\end{thm}

\begin{rem}
None of the proofs for these two results are easy, however their analogs in characteristic zero do follow directly from the definition of 
$Y^\pi$. In particular, one knows already that the theorem \ref{bekanntIV} holds for $p'$-quasi-isogenies, because one can lift them.
\end{rem}

\begin{proof}
We begin with the proof of theorem \ref{bekanntIII} by following \cite[proof of Proposition 5.1]{habil} very closely: In order to check the assertion for an 
arbitrary family of endomorphisms $f_i\in\z_{(p)}\otimes\End_L(Y_{i,\xi})$ it suffices to restrict to the eigenspaces under $*$, i.e. $f_i^*=(-1)^{c_i}f_i$ for some 
$c_i\in\{0,1\}$. Now pick an auxiliary element with $-d^*=d\in\O_L\backslash\{0\}$ and observe that $\gamma_i:=\frac{1+pd^{1-c_i}f_i}{1-pd^{1-c_i}f_i}$ is 
a $\O_L$-linear $p'$-quasi-isogeny. According to the preceding remark we are allowed to use the assertion for the family 
$\gamma_i$, by using appropriate lifts $\gamma'_i$ of $\gamma_i$ between two possibly different lifts  $Y_{i,\kappa_1}$ and $Y_{i,\kappa_2}$ of $Y_{i,\xi}$. 
We may deduce the assertion for the family $f_i=\frac{\gamma_i-1}{pd^{1-c_i}(\gamma_i+1)}$, and the proof of theorem \ref{bekanntIII} is complete.\\
Before we embark in the proof of theorem \ref{bekanntIV} we note that it holds at least under the following two additional assumptions: 
\begin{itemize}
\item[(i)]
$\xi_1$ and $\xi_2$ factor through $_{\tilde U}\Sb$ and,
\item[(ii)]
$Y_{\xi_1}[p^\infty]\cong Y_{\xi_2}[p^\infty]$ (as homogeneously $p$-principally polarized $p$-divisible groups with $\O_L^\Lambda$-operation)
\end{itemize}
The reason being: If some family of isomorphisms $u_i:Y_{i,\xi_1}[p^\infty]\stackrel{\cong}{\rightarrow}Y_{i,\xi_2}[p^\infty]$ preserves the $\O_L$-actions 
together with the polarization, then $h_i:=u_i^{-1}\circ\gamma_i[p^\infty]$ constitutes a $\q_p$-valued point in the group scheme $I_{i,\xi}/\z_{(p)}$ 
that represents the functor 
$$\alg_{\Spec\z_{(p)}}\ni R\mapsto\{f\in R\otimes\End_L(Y_{i,\xi})|f^*f\in R^\times\},$$ 
simply because $\End_L(\d^0(Y_{i,\xi}))=\z_p\otimes\End_L(Y_{i,\xi})$ holds for all $k$-valued points of $_{\tilde U}\Sb$ as a mediate consequence of \cite{tate}. 
Whence it follows that there exist factorizations $h_i=v_i\circ\beta_i$ with $v_i\in I_{i,\xi}(\z_p)$ and $\beta_i\in I_{i,\xi}(\q)$, as $I_{i,\xi}(\z_p)$ is open while 
$I_{i,\xi}(\q)$ is dense in $I_{i,\xi}(\q_p)$. The assertion follows by applying the said remark to the $p'$-quasi-isogenies $\gamma_i\circ\beta_i^{-1}=u_i\circ v_i$ 
and applying theorem \ref{bekanntIII} to $\beta_i\in\End_L^0(Y_{i,\xi_1})$.\\
Our proof of the general case of theorem \ref{bekanntIV} does not follow \cite[proof of Theorem 4.10]{habil}, instead we appeal to the two 
lemmas \ref{Delignetrick} and \ref{Raynaudtrick} below, so that it only remains to check the special case of quasi-isogenies satisfying (i) alone. 
We argue as follows: Methods of \cite[page 321, line 16- page 323, line 13]{viehmann} provide us with a $k$-valued point $\xi_0$ with:
\begin{itemize}
\item
$\xi_2$ and $\xi_0$ are lying in the same connected component of $_{\tilde U}\Sb$ and,
\item
$Y_{\xi_1}[p^\infty]\cong Y_{\xi_0}[p^\infty]$ (as homogeneously $p$-principally polarized $p$-divisible groups with $\O_L^\Lambda$-operation)
\end{itemize}
Pick a sequence of say $n$ points $\xi_2$, $\xi_3$,...,$\xi_{n+1}=\xi_0$, such that the points in each of the consecutive pairs $\{\xi_2,\xi_3\}$,...,
$\{\xi_n,\xi_{n+1}\}$ factor through common irreducible normal subvarieties say $S_2$,...,$S_n$ of ${_{\tilde U}\Sb}$. By induction one obtains a sequence of 
$n$ quasi-isogenies $Y_{\xi_1}\dashrightarrow Y_{\xi_2},\dots,Y_{\xi_{n+1}}$ of which each consecutive pair possesses a common $\O_L^\Lambda$-linear generization 
(note that the abelian scheme $Y$ is isotrivial over each of $S_2$,...,$S_n$). We know already that $Y_{\xi_1}\dashrightarrow Y_{\xi_0}$ satisfies the theorem, so we get 
it for $Y_{\xi_1}\dashrightarrow Y_{\xi_2}$ from a $n-1$-fold application of lemma \ref{Delignetrick}.
\end{proof}

The following can be regarded as an analog of a special case of Deligne's principle B \cite[Theorem 2.12]{deligne5}:

\begin{lem}[principle B]
\label{Delignetrick}
Let $S$ be a irreducible normal algebraic variety over $\O_R/\gr$, and let $\xi_1$ and $\xi_2$ be two morphisms from $S$ to $_{\tilde U}\S$, and let 
$\gamma:Y_{\xi_1}\dashrightarrow Y_{\xi_2}$ be a quasi-isogeny over $S$. If $\gamma\times_{S,\xi}k:Y_{\xi_1\circ\xi}\dashrightarrow Y_{\xi_2\circ\xi}$ 
satisfies the claim in theorem \ref{bekanntIV} for some geometric point $\xi$, then it does so for all other ones.
\end{lem}
\begin{proof}
From the Galois equivariance of the comparison isomorphisms $m_{\mathaccent 19 et}^\pi$ and the theorem of Mori-Zarhin \cite[Chapitre XII, Th. 2.5]{mori} we obtain 
an adelic $\a^{\infty,p}\otimes L$-linear map $\gamma_{\mathaccent 19 et}^\pi\in\a^{\infty,p}\otimes\Hom_L^0(Y_{\xi_1}^\pi,Y_{\xi_2}^\pi)$, with the desired properties. The 
lemma follows, as any $\xi:\Spec k\rightarrow S$ satisfies 
$$\Hom_L^0(Y_{\xi_1}^\pi,Y_{\xi_2}^\pi)=\Hom_L^0(Y_{\xi_1\circ\xi}^\pi,Y_{\xi_2\circ\xi}^\pi)\cap(\a^{\infty,p}\otimes\Hom_L^0(Y_{\xi_1}^\pi,Y_{\xi_2}^\pi)),$$
the intersection taking place in the group $\a^{\infty,p}\otimes\Hom_L^0(Y_{\xi_1\circ\xi}^\pi,Y_{\xi_2\circ\xi}^\pi)$. Note that the compatibility 
of $\gamma\mapsto g^\pi(\gamma)$ with $\d^0$ follows from the rigidity result \cite[Proposition 40]{zink2}, as
$$\begin{CD}
\dot\bigotimes_{i\in\pi}\S_{i,\gq}\times_{\xi_1}S@>{\bigotimes_{i\in\pi}g_i(\gamma)[\gq^\infty]}>>\dot\bigotimes_{i\in\pi}\S_{i,\gq}\times_{\xi_2}S\\ 
@V{m_{\xi_1,\gq}^\pi}VV@V{m_{\xi_2,\gq}^\pi}VV\\
\S_\gq^\pi\times_{\xi_1}S@>{g^\pi(\gamma)[\gq^\infty]}>>\S_\gq^\pi\times_{\xi_2}S
\end{CD}$$
commutes only if and if any of its specializations does that.
\end{proof}

The following lemma is well-known:

\begin{lem}
\label{Raynaudtrick}
Let $\xi_1$ and $\xi_2$ be $k$-valued points of $_{\tilde U}\essbar\backslash{_{\tilde U}\Sb}$ and consider a quasi-isogeny 
$\gamma:Y_{\xi_1}\dashrightarrow Y_{\xi_2}$. Then there exists another quasi-isogeny $\gamma':Y_{\xi'_1}\dashrightarrow Y_{\xi'_2}$ enjoying the following properties:
\begin{itemize}
\item
The Newton-polygon of $Y_{\xi'_1}$ lies strictly above the Newton-polygon of $Y_{\xi_1}$
\item
$\gamma$ and $\gamma'$ possess a common $\O_L^\Lambda$-linear generization.
\end{itemize}
\end{lem} 
\begin{proof}
The lemma \ref{isogdefo} gives us some non-isotrivial deformation $\tilde\gamma:Y_{\tilde\xi_1}\dashrightarrow Y_{\tilde\xi_2}$ over $k[[t]]$. The generic fiber of the 
triple $(\tilde\xi_1,\tilde\xi_2,\tilde\gamma)$ is certainly definable over some subfield $N\subset k((t))$ finitely generated over, and containing $k$. Let $S$ be the 
normalization of $_{\tilde U}\essbar_k^2$ in $N$, we will write $(\xi''_1,\xi''_2,\gamma'')$ for the extension to $S$ of our chosen model of $(\tilde\xi_1,\tilde\xi_2,\tilde\gamma)$ 
over $N$. Recall that the properness of $_{\tilde U}\essbar$ is conveniently implied by the overall assumptions on the Shimura datum $(\tilde G,\tilde X)$, as introduced 
in subsubsection \ref{ready}. It follows that $S$ is proper over $k$, because it is finite over $_{\tilde U}\essbar_k^2$. From the construction of $S$ it is clear, that the 
Newton-polygons of $Y_{\xi_1}$ and $Y_{\xi_2}$ agree with the generic one. The existence of some $\xi'\in S(k)$ with strictly larger Newton-polygon follows from the 
well-known fact, due to Raynaud, that over a proper, normal and irreducible basis the constancy of the Newton-polygon implies \'etale locally the isotriviality of the abelian scheme.  
\end{proof}

We fix a homogeneously polarized $L^\Lambda$-abelian $k$-variety $(Y,\lambda,\iota)$ up to $\q$-isogeny, say in the sense of 
\cite[chapter 9]{kottwitz1}, and we write $Y[\ge_i]$ for its homogeneously polarized $L$-abelian factor corresponding to $i\in\Lambda$ . We will 
say that $(Y,\lambda,\iota)$ is of type $\tilde V$ if the homogeneously polarized $\q$-isogeny class of $Y$ contains at least one member for which
\begin{itemize}
\item[(K1)]
the polarization $\lambda:Y\rightarrow\check Y$ is $p$-principal, and
\item[(K2)] 
the homomorphism $L^\Lambda\rightarrow\End^0(Y)$ is $p$-integral and satisfies the determinant condition with respect to the skew-Hermitian $L^\Lambda$-module $\tilde V$.
\end{itemize}
Verifying that $\iota$ be of type $\tilde V$ is certainly equivalent to finding a choice of $\O_L^\Lambda$-invariant selfdual Dieudonn\'e lattice in $\d^0(Y)$, with the correct 
determinant condition. Furthermore, any choice of a full $\tilde U^p$-level structure $\tilde\eta^p$ completes this to a PEL-quadruple, say $(Y,\lambda,\iota,\tilde\eta)$ constituting 
a $k$-valued point $\xi\in{_{\tilde U}\essbar}$. From now onwards we will write $\dot\bigotimes_{i\in\pi}Y[\ge_i]$ to denote the homogeneously $p$-principally polarized 
$\z_{(p)}$-isogeny class of abelian $k$-varieties derived by discarding the level structure on the homogeneously $p$-principally polarized $\z_{(p)}$-isogeny class of $Y_\xi^\pi$, 
this is meaningful, since the formation of $\dot\bigotimes_{i\in\pi}Y[\ge_i]$ is independent of the choice of $\tilde\eta^p$. Notice that the canonical $*$-invariant action 
$\iota^\pi:\O_L\rightarrow\z_{(p)}\otimes\End(\dot\bigotimes_{i\in\pi}Y[\ge_i])$ satisfies the determinant condition with respect to the skew-Hermitian $L$-module $\tilde V^\pi$.\\
However, the theorem \ref{bekanntIV} tells us that the $\q$-isogeny class of $\dot\bigotimes_{i\in\pi}Y[\ge_i]$ depends merely on the homogeneously 
polarized $L$-abelian $\q$-isogeny class of the family $\{Y[\ge_i]\}_{i\in\pi}$, it has an action $\iota^\pi:L\rightarrow\End^0(\dot\bigotimes_{i\in\pi}Y[\ge_i])$ 
(of type $\tilde V^{\pi}$), and it is canonically homogeneously polarized too. These observations give sense to the following definition:

\begin{defn}
\label{LEVEL}
Recall, that we have fixed a $L$-metaunitary datum $\bT=(\bd^+,\{(V_i,\rho_i,\Psi_i,\bj_i)\}_{i\in\Lambda},\{(R_\pi,\iota_\pi)\}_{\pi\in\Pi})$ for 
a Shimura datum with $L^+$-coefficients $(G,h)$, and let $k$ be an algebraically closed extension of $\f_{p^f}$: A homogeneously polarized 
$\q$-isogeny class of $L^\Lambda$-abelian $k$-varieties of type $\bT$ is a quadruple $(Y,\lambda,\iota,\{y_\pi\}_{\pi\in\Pi})$, where
\begin{itemize}
\item[(T1)]
$(Y,\lambda,\iota)$ is a homogeneously polarized $\q$-isogeny class of $L^\Lambda$-abelian $k$-varieties $(Y,\lambda,\iota)$ of type $\bigoplus_{i\in\Lambda}\tilde V_i$, and
\item[(T2)]
$y_\pi:R_\pi\rightarrow\End_L^0(\dot\bigotimes_{i\in\pi}Y[\ge_i])$ is a homomorphism which commutes with $L$ and preserves $*$, for every $\pi\in\Pi$.
\end{itemize}
We fix such a quadruple $y=(Y,\lambda,\iota,\{y_\pi\}_{\pi\in\Pi})$. A full level structure $\eta$ is a $L^\Lambda$-linear similitude 
$$\sum_{i\in\Lambda}\eta_i=\eta:\a^{\infty,p}\otimes\bigoplus_{i\in\Lambda}V_i\stackrel{\cong}{\rightarrow}H_1^{\mathaccent19 et}(Y,\a^{\infty,p}),$$ 
such that the diagrams
$$\begin{CD}
\a^{\infty,p}\otimes_{\q}\bigotimes_{i\in\pi}\End_L(V_i)@<<<R_\pi\\
@V{\bigotimes_{i\in\pi}\eta_i}VV@V{y_\pi}VV\\
\End(\bigotimes_{i\in\pi}H_1^{\mathaccent19 et}(Y[\ge_i],\a^{\infty,p})(\frac{1-\Card(\pi)}2))@<{H_1^{\mathaccent19 et}}<<\End_L^0(\dot\bigotimes_{i\in\pi}Y[\ge_i])
\end{CD}$$
are commutative for all $\pi\in\Pi$. 
\end{defn}

From now on we assume that $\bT$ is unramified at $p$, so that $G$ possesses a reductive $\z_p$-model $\gG_p$ which is compatible with certain self-dual 
$\z_{(p)}\otimes\O_L$-lattices $\gV_{i,p}\subset V_i$, in fact we may fix a $W(\f_{p^f})$-rational $\Phi$-datum $(\gG_p,\mu_p)$ which is compatible with the Shimura datum 
$(G,h)$, as done in subsubsection \ref{go}. Let us write $\B_k(\gG)$ for the (natural groupoid structure on the) class of $F$-crystals with $\gG$-structure over $k$, which are 
classically defined to be $\z_p$-linear exact and faithful rigid $\otimes$-functors from $\bRep_0(\gG_p)$ to $\cris_{W(k),F}$. We need to associate to any $(\gG_p,\mu_p)$-display 
$\P$ over $k$ an $F$-crystal with $\gG$-structure $M_\P$ over $k$. We do this by specializing subsection \ref{realII} to the case at hand and thus obtain:
$$M_\P:\bRep_0(\gG_p)\rightarrow\cris_{W(k),F};\,\varrho\mapsto\sy_{W(k),F}(\varrho,\P)$$
(N.B.: This is meaningful because $\BI_k(\gG_p,\mu_p)$ is canonically equivalent to the groupoid $\hat\CAS_{W(k),I(k)}(\gG,\mu_p)$, which shows up in \eqref{realIII})

\begin{defn}
\label{LEVELP}
Let $\bT=(\bd^+,\{(V_i,\rho_i,\Psi_i,\bj_i)\}_{i\in\Lambda},\{(R_\pi,\iota_\pi)\}_{\pi\in\Pi})$ be an unramified $L$-metaunitary Shimura datum for a 
Shimura datum with $L^+$-coefficients $(G,h)$. We fix an algebraically closed extension $k$ of $\f_{p^f}$ and a choice of integral structures for 
$\bT$ as above. By a $\mu_p$-fake integral structure on a homogeneously polarized $\q$-isogeny class of $L^\Lambda$-abelian $k$-varieties 
$y=(Y,\lambda,\iota,\{y_\pi\}_{\pi\in\Pi})$ of type $\bT$ we mean a $(\gG_p,\mu_p)$-display $\P$ together with a $L^\Lambda$-linear similitude
$$\sum_{i\in\Lambda}\zeta_i=\zeta:\q\otimes\bigoplus_{i\in\Lambda}M_\P(\varrho_i)\stackrel{\cong}{\rightarrow}\d^0(Y),$$ 
of which the induced similitudes from $\q\otimes M_\P(\varrho^\pi)$ to $\d^0(\dot\bigotimes_{i\in\pi}Y[\ge_i])$ are $R_\pi$-linear for every $\pi\in\Pi$. 
\end{defn}

We write $\Mot_{\mu_p}^\bT(k)$ for the (discrete!) groupoid consisting of homogeneously polarized $\q$-isogeny classes of $L^\Lambda$-abelian 
$k$-varieties $y=(Y,\lambda,\iota,\{y_\pi\}_{\pi\in\Pi})$ of type $\bT$, together a full level structure $\eta$ and a $\mu_p$-fake integral structure 
$(\P,\zeta)$. Observe that there is a canonical functor $\Mot_{\mu_p}^\bT(k)\rightarrow\BI_k(\gG_p,\mu_p)$, as $\P$ is banal.

\begin{rem}
\label{verlaengert}
Another way of looking at $\mu_p$-fake integral structures on some homogeneously polarized $\q$-isogeny class $y$ of $L^\Lambda$-abelian 
$k$-varieties of type $\bT$ is the following: Let us write $\Iso(y)$ for the set of $K(k)\otimes L^\Lambda$-linear similitudes 
$$\sum_{i\in\Lambda}\zeta_i=\zeta:K(k)\otimes\bigoplus_{i\in\Lambda}V_i\stackrel{\cong}{\rightarrow}\d^0(Y)$$
which induce further $R_\pi$-linear ones from $K(k)\otimes V^\pi$ to $\d^0(\dot\bigotimes_{i\in\pi}Y[\ge_i])$. The group $G(K(k))$ acts 
on $\Iso(y)$ simply transitively, from the right. It follows that every $\zeta\in\Iso(y)$ gives rise to an element $b_{y,\zeta}\in G(K(k))$ satisfying 
$^F\zeta=\zeta\circ b_{y,\zeta}$. Let us write $\Iso^{\mu_p}(y)$ for the set of pairs $(O,\zeta)\in\gG(W(k))\times\Iso(y)$ rendering the diagram
\begin{equation}
\label{Uuh}
\begin{CD}
K(k)\otimes V
@<{\varrho(O{^F\mu_p(\frac1p)}))\circ(F\otimes\id_V)}<<K(k)\otimes V\\
@V{\zeta}VV@V{\zeta}VV\\
\d^0(Y)@<F<<\d^0(Y)\\
\end{CD}
\end{equation}
commutative. The subgroup $\I^{\mu_p}(k)\subset G(K(k))$ acts on $\Iso^{\mu_p}(y)$ from the right according to $(O,\zeta)\circ h:=(O',\zeta')$ where:
\begin{eqnarray}
\label{Iih}
&&O'=h^{-1}O\Phi^{\mu_p}(h)\\
\label{Eeh}
&&\zeta'=\zeta\circ\varrho(h)
\end{eqnarray}
for any $h\in\I^{\mu_p}(k)$. Since $O$ is determined by $\zeta$ one could think of this as the restriction of the $G(K(k))$-action on $\Iso(y)$ to the $\I^{\mu_p}(k)$-invariant 
subset $\{\zeta\in\Iso(y)\mid\,b_{y,\zeta}{^F\mu_p(p)}\in\gG(W(k))\}$. Now a $\mu_p$-fake integral structure on $y$ is just an element of $\Iso^{\mu_p}(y)/\I^{\mu_p}(k)$.
\end{rem}

Our next aim is the study of the set $\Marb_{\bT,\gp}(k)$:

\begin{thm}
\label{uniformizeII}
There is a $G(\a^{\infty,p})$-equivariant bijection between the $G(\a^{\infty,p})$-set of sextuples $(Y,\lambda,\iota,\{y_\pi\}_{\pi\in\Pi},\eta,\zeba)$ comprising of 
\begin{itemize}
\item
a $\mu_p$-fake integral structure $\zeba$ on 
\item
a homogeneously polarized $\q$-isogeny class of $L^\Lambda$-abelian $k$-varieties $(Y,\lambda,\iota,\{y_\pi\}_{\pi\in\Pi})$ of type $\bT$ with
\item
a full level structure $\eta$,
\end{itemize}
and the $G(\a^{\infty,p})$-set $_{U^p}\Marb_{\bT,\gp}(k)$. Moreover, the bijection is functorial in $k$.
\end{thm}

Let $y=(Y,\iota,\lambda,\{y_\pi\}_{\pi\in\Pi})$ be a homogeneously polarized $\q$-isogeny class of $L^\Lambda$-abelian $k$-varieties of type 
$\bT$. Let us write $I_y$ for the $\q$-algebraic group $(\g_m\times\Res_{L^+/\q}I_y^1)/\{\pm1\}$ where $I_y^1$ represents the $\Spec L^+$-functor:
\begin{eqnarray*}
R\mapsto&&\{(\dots,f_i,\dots)\in R^\times\times\prod_{i\in\Lambda}R\otimes_{L^+}\End_L^0(Y[\ge_i])^\times|\\
&&\forall\pi\in\Pi:\,\dot\bigotimes_{i\in\pi}f_i\in R\otimes_{L^+}\End_{R_\pi}^0(\bigotimes_{i\in\pi}Y[\ge_i])^\times\\
&&\forall i\in\Lambda:\,f_i^*f_i=1\}
\end{eqnarray*}
Notice that every choice of full level structure $\eta$ as in definition \ref{LEVEL} furnishes $I_{y,\a^{\infty,p}\otimes L^+}^1$ with a $\a^{\infty,p}\otimes L^+$-rational group 
homomorphism, say $i_{y,\eta}$, to $G_{\a^{\infty,p}\otimes L^+}^1$ and that one has $\Int^{G^1}(\gamma^p/\a^{\infty,p}\otimes L^+)^{-1}\circ i_{y,\eta}=i_{y,\eta\circ\gamma^p}$ for all 
$\gamma^p\in G^1(\a^{\infty,p}\otimes L^+)$. Note also that every choice of $\zeta\in\Iso(y)$ furnishes $I_{y,\q_p}$ with a $\q_p$-rational group homomorphism, say $j_{y,\zeta}$, to 
the twisted centralizer $J_{b_{y,\zeta}}$ in the sense of \eqref{thejack}, and any $\gamma_p\in G(K(k))$ satisfies $\Int^G(\gamma_p/K(k))^{-1}\circ j_{y,\zeta}=j_{y,\zeta\circ\gamma_p}$, 
here observe that $J_{b_{y,\zeta},K(k)}$ is a subgroup of $G_{K(k)}$. The stabilizer $\gI_{y,\zeba}$ in $I_y(\q)$ of some $\zeba\in\Iso^{\mu_p}(y)/\I^{\mu_p}(k)$ 
is a congruence subgroup therein, because it arises from intersecting with the $p$-adically bounded open group $j_{y,\zeta}^{-1}(\I^{\mu_p}(k))\subset I_y(K(k))$. Note also, that 
$$J_{b_{y,\zeta}}(\q_p)\cap\I^{\mu_p}(k)=J_{b_{y,\zeta}}(\q_p)\cap\gG_p(W(k))=:\gJ_{y,\zeba}$$
holds, and this is just the automorphism group of the crystalline realization of $(y,\zeba)$, when regarded as $\B_k(\gG_p)$-object as
$$J_{b_{y,\zeta}}(\q_p)=\{h\in G(K(k))|h^{-1}b_{y,\zeta}{^Fh}=b_{y,\zeta}\}.$$ 
Also, we have $\gI_{y,\zeba}=I_y(\q)\cap j_{y,\zeta}^{-1}(\gJ_{y,\zeba})$ and what we found is:

\begin{cor}
\label{gekuerzt}
If $\xi\in\Marb_{\bT,\gp}(k)$ corresponds to $(y,\eta,\zeba)$ by means of the bijection of theorem 
\ref{uniformizeII}, then the group $i_{y,\eta}(\gI_{y,\zeba})$ agrees with the stabilizer of $\xi$ in $G(\a^{\infty,p})$. 
\end{cor}

For any $\gq\in S_p$ we write 
$$\sum_{i\in\Lambda}\zeta_{i,\gq}=\zeta_\gq:K(k)\otimes_{\iota_\gq,L}\bigoplus_{i\in\Lambda}V_i\stackrel{\cong}{\rightarrow}\d^0(Y[\gq^\infty])[\ge_{\iota_\gq}],$$
for the restriction of some $\zeta\in\Iso(y)$ to the $\iota_\gq$-eigenspace, and we write $j_{y,\zeta,\gq}:I_{y,K(k)}^1\rightarrow G_{\gq,K(k)}^1$ 
for the restriction of the previously introduced $j_{y,\zeta,K(k)}$ to the $\iota_\gq$-eigenspace, i.e. the map induced by $\zeta_\gq$.

\begin{prop}
\label{FOF}
Let $y=(Y,\lambda,\iota,\{y_\pi\}_{\pi\in\Pi})$ be a homogeneously polarized $\q$-isogeny class of 
$L^\Lambda$-abelian $k$-varieties of type $\bT$, and assume that there exists at least one full level structure $\eta$. 
\begin{itemize}
\item[(i)]
The image in the abelianization $G^{1ab}$, of the restriction of $i_{y,\eta}$, to the neutral component $I_y^{1\circ}$ is $L^+$-rational 
and independent of $\eta$ (and thus gives rise to a canonical homomorphism $i_y^{ab}:I_y^{1\circ}\rightarrow G^{1ab}$).
\item[(ii)]
For each $i\in\Lambda$, the composition of the character $\chi_{\rho_i}$ with $i_{y,\eta}$ is $L$-rational and independent of $\eta$ (and thus gives rise to a canonical 
class function $\chi_i$ on $I_{y,L}^1$).
\item[(iii)]
For any $\gq\in S_p$ and $\zeta\in\Iso(y)$ we have $\chi_{\rho_i}\circ j_{y,\zeta,\gq}=\chi_i$.
\end{itemize}
\end{prop}
\begin{proof}
Parts (ii) and (iii) are folklore, cf. \cite[chapter V]{demazure}. Towards proving (i) we let $Z$ be the center of $\prod_{i\in\Lambda}R_{\{i\}}$, 
and we let $\omega:G^1\rightarrow\Res_{Z/L^+}\g_{m,Z}$ be the homomorphism induced by the natural action of $G^1$ on the determinant 
$\det_Z(V)\in\Pic(\Spec Z)$ of $V$ (as a $\ect(\Spec Z)$-object). It is easy to see that $G^{1der}$ is the neutral component of $\ker(\omega)$, so 
that it suffices to check (i) for the composition $\omega\circ i_{y,\zeta}|_{I_y^{1\circ}}$, which follows again from the aforementioned methods (of Mumford).
\end{proof}

\subsection{Existence of elliptic points}

Let $K$ be a finite Galois extension of $\q$, and let $A$ be a finitely generated torsionfree abelian group with a left $\Gal(K/\q)$-action. Recall the global Tate-Nakayama isomorphism:
$$\hat H^{-1}(\Gal(K/\q),A)\rightarrow H^1(\Gal(K/\q),A\otimes C_K),$$
where $C_K$ is the id\'ele group of $K$. Now, let $r$ be a place of $\q$, and let $v|r$ be a place of $K$. Then we also have a local Tate-Nakayama isomorphism:
$$\hat H^{-1}(\Gal(K_v/\q_r),A)\rightarrow H^1(\Gal(K_v/\q_r),A\otimes K_v^\times),$$
and in either case one easily computes the left hand side Tate cohomology groups as the torsion subgroups of the groups of Galois coinvariants of 
$A$, because $A$ is torsionfree. In order to describe the relation between the global and local isomorphisms we must choose one single place of 
$K$ over each of the places of $\q$, and we write $S$ for the set of places obtained in that manner: Then there is a natural commutative diagram:
$$\begin{CD}
\bigoplus_{v\in S}(A_{\Gal(K_v/\q_r)})_{tors}@>{TN}>>\bigoplus_{v\in S}H^1(\Gal(K_v/\q_r),A\otimes K_v^\times)\\
@VVV@VVV\\
(A_{\Gal(K/\q)})_{tors}@>{TN}>>H^1(\Gal(K/\q),A\otimes C_K)
\end{CD},$$
where the vertical arrow on the left is defined by the summation over $S$, while the $v$-component of the vertical arrow on the right is 
defined by the Shapiro isomorphism between $H^1(\Gal(K_v/\q_r),A\otimes K_v^\times)$ and $H^1(\Gal(K/\q),A\otimes(K\otimes\q_r)^\times)$, 
followed by a map induced from the inclusions $(K\otimes\q_r)^\times\subset C_K$. Furthermore the long exact cohomology sequence to 
$$1\rightarrow A\otimes K^\times\rightarrow A\otimes(K\otimes\a)^\times\rightarrow A\otimes C_K\rightarrow1,$$ 
implies that every element in the kernel of the left hand-side vertical arrow maps to an element coming from $H^1(\Gal(K/\q),A\otimes K^\times)$. We are going to use this 
in the following way: Fix a finite prime $p$, and suppose $\mu$ is an element of $A$ for which $\n_{K_v/\q_p}(\mu)$ and $\n_{K_v/\r}(\mu)$ are both trivial (where $v$ 
indicates the divisor in $S$ of $p$ or $\infty$). Now consider the element $\mub\in\bigoplus_{v\in S}(A_{\Gal(K_v/\q_r)})_{tors}$ of which the $v$-component is given by: 
$$\mub_v=\begin{cases}[\mu]_\infty&v|\infty\\-[\mu]_p&v|p\\0&\mbox{ otherwise}\end{cases},$$
where $[\mu]_r$ denotes the class of $\mu$ in the group of coinvariants $(A_{\Gal(K_v/\q_r)})_{tors}$. Let us say that some 
$c\in H^1(\Gal(K/\q),A\otimes K^\times)$ is a Tate-Nakayama class for $\mu$ if the product of the local Tate-Nakayama isomorphisms send $\mub$ to the 
image of $c$ in $\bigoplus_{v\in S}H^1(\Gal(K_v/\q_r),A\otimes K_v^\times)$. Finally, let us point out two consequences from Galois descent: First giving a 
$\Gal(K/\q)$-module $A$ is equivalent to giving a $\q$-torus $T$ splitting over $K$, second $H^1(\Gal(K/\q),A\otimes K^\times)$ classifies isomorphism classes 
of principal homogeneous spaces for $T$ over $\Spec\q$ becoming trivial over $\Spec K$. More specifically, when fixing a $K$-valued point $\eta$ on some 
$P\in\ton(T)(\q)$, then $d_\eta(s)=\eta^{-1}{^s}\eta$ is a derivation on $\Gal(K/\q)$ with coefficients in $T(K)=A\otimes K^\times$, and vice versa. We say that 
$P\in\ton(T)(\q)$ is a principal homogeneous space of type $\mu$ if its cohomology class with respect to some choice of $K$ is a Tate-Nakayama class for $\mu$. 
For any CM-algebra $H$, we write $H^1$ for the $\q$-torus whose group of $K$-valued points are given by $\{\gamma\in H\otimes K|\gamma\gamma^*=1\}$ 
and $H^+=\{x\in H\mid\,x^*=x\}$, where $*$ is complex conjugation. We are now going to describe a certain class of tori which satisfy the Hasse-principle: 

\begin{lem}
\label{parityII}
Consider a CM algebra $H=\prod_{i=1}^nH_i$, where $H_1$, ..., $H_n$ are CM fields. Assume that $L\subset H$ is a CM subfield, 
such that $L^+$ possesses a place which is inert in each $H_1^+$, ..., $H_n^+$ and in $L$. Then $\sha(H^1/L^1)$ is trivial.
\end{lem}
\begin{proof}
Observe that the exact sequence $1\rightarrow L^1\rightarrow H^1\rightarrow H^1/L^1\rightarrow1$ induces a long exact cohomology sequence:
\begin{equation}
\label{parityIII}
H^1(\q,L^1)\rightarrow H^1(\q,H^1)\rightarrow H^1(\q,H^1/L^1)\rightarrow H^2(\q,L^1)
\end{equation}
Let us write $\l$ (resp. $\l^1$) for the $L^+$-torus $\Res_{L/L^+}\g_{m,L}$ (resp. for the kernel of $\l\rightarrow\g_{m,L^+}$). Moreover, consider the short 
exact sequence $1\rightarrow\l^1\rightarrow\l\rightarrow\g_{m,L^+}\rightarrow1$ (resp. $1\rightarrow\g_{m,L^+}\rightarrow\l\rightarrow\l^1\rightarrow1$). The 
associated long exact cohomology sequence yields an isomorphism $H^2(\q,L^1)\cong H^2(L^+,\l^1)\cong\ker(\Br(L)\stackrel{\core}{\rightarrow}\Br(L^+))$ (resp. 
$H^1(\q,L^1)\cong H^1(L^+,\l^1)\cong\ker(\Br(L^+)\stackrel{\rest}{\rightarrow}\Br(L))$). It follows that $H^2(\q,L^1)$ is isomorphic to a direct sum of copies of $\q/\z$ indexed 
by the set of places of $L^+$ which are split in $L$. Moreover, $H^1(\q,L^1)$ can be described by families $x_r\in\z/2\z$, indexed by the set of $L^+$-places $r$ which are 
inert in $L$, subject to $\infty>\Card(\{r\mid\,x_r\neq0\})\equiv0\pmod2$ while the canonical map from $H^1(\q,L^1)$ to $H^1(\q,H^1)$ can be recovered from the assignment
\begin{equation}
\label{parityI}
y_{i,v}=[H_{i,v}^+:L_r^+]x_r,
\end{equation}
where $i$ runs through $\{1,\dots,n\}$ and $v\mid r$ runs through the places of $H_i^+$ which are inert in $H_i$. We are now in the position to prove the lemma. The 
exactness of \eqref{parityIII} together with the aforementioned description of $H^2(\q,L^1)$ reveals that some locally trivial element of $H^1(\q,H^1/L^1)$ can be lifted 
to an element $c\in H^1(\q,H^1)$, which in turn is described by a family $y_{i,v}\in\z/2\z$ satisfying $\infty>\Card(\{v\mid\,y_{i,v}\neq0\})\equiv0\pmod2$ for each $i$. Since 
$c$ maps to a locally trivial element in $H^1(\q,H^1/L^1)$ we deduce the existence of a family $x_r\in\z/2\z$ satisfying \eqref{parityI} and $\infty>\Card(\{r\mid\,x_r\neq0\})$. 
If one of the degrees $[H_i^+:L^+]$ is odd we have $\sum_rx_r=\sum_vy_{i,v}=0$ because of $[H_i^+:L^+]\equiv\sum_{v\mid\,r}[H_{i,v}^+:L_r^+]\pmod2$, here notice 
that $[H_{i,v}^+:L_r^+]\equiv0\pmod2$ holds whenever a split place $v$ of $H_i^+$ lies over an inert place $r$ of $L^+$. It follows that $c$ is contained in the image of 
$H^1(\q,L^1)$ and we are done. If all of the degrees $[H_i^+:L^+]$ are even we do have to use our assumption on the existence of an inert place $r_0$ of $L^+$ allowing 
one and only one place $v_i$ of $H_i^+$ lying over it. In this case we may replace $x_{r_0}$ by $\sum_{r\neq r_0}x_r$ without changing the family $y_{i,v}$, due to 
$[H_{v_i}^+:L_{r_0}^+]=[H_i^+:L^+]\equiv0\pmod2$. Again we found a family satisfying \eqref{parityI} and $\infty>\Card(\{r\mid\,x_r\neq0\})\equiv0\pmod2$ and we are done.
\end{proof}

\begin{lem}
\label{parityIV}
Let $H$ be a CM algebra and let $L\subset H$ be a CM field:
$$(H^1/L^1)(\a^\infty)=(H^1/L^1)(\q)H^1(\a^\infty)/L^1(\a^\infty)$$
\end{lem}
\begin{proof}
Let $x$ be an $\a^\infty$-valued element of $H^1/L^1$. Its image in the larger group $H^\times/L^\times$ 
allows a lift $\tilde x\in(H\otimes\a^\infty)^\times$, according to Hilbert 90. Observe that we have
$$(L^+\otimes\a^\infty)^\times=\n_{L/L^+}(L\otimes\a^\infty)L^{+\;\times},$$
which is due to $[(L^+\otimes\a)^\times/\n_{L/L^+}(L\otimes\a)^\times:L^{+\;\times}/\n_{L/L^+}L^\times]=2$, while $(L^+\otimes\r)^\times$ is not contained in 
$\n_{L/L^+}(L\otimes\r)^\times$. Therefore we may assume that $\tilde x\tilde x^*\in L^{+\;\times}$ upon multiplying $\tilde x$ with a suitable element in
$(L\otimes\a^\infty)^\times$. However, this implies that the totally positive element $\tilde x\tilde x^*$ lies in the image of $\n_{H/H^+}:H^\times\rightarrow H^{+\;\times}$, 
as this does hold locally. This results in a $\q$-valued element $y$ of $H^1/L^1$ such that $xy^{-1}$ allows a lift to $H^1(\a^\infty)$
\end{proof}

\subsubsection{Isocrystals of CM-type}
Let us introduce $F$-isocrystals in the generality we need: An $F$-isocrystal is a pair $(M,F)$ where $M\in\Ob_{\ect(\Spec K(\fc))}$ and $F:M\rightarrow M$ 
is a $F$-linear bijection. The class of $F$-isocrystals forms a $\q_p$-linear tannakian category which we denote by $\cry$. It possesses a natural 
$K(\fc)$-valued fiber functor $\omega^{K(\fc)}:\cry\rightarrow\ect(\Spec K(\fc));(M,F)\mapsto M$. Notice the fully faithful $\otimes$-functor:
\begin{equation}
\label{reizIV}
\ect(\Spec\q_p)\rightarrow\cry;V\mapsto K(\fc)\otimes_{\q_p}V
\end{equation}
of which the essential image consists of the $F$-isocrystals of slope $0$. As a consequence of Steinberg's theorem Rapoport and Richartz identify $B(G)$ 
with the set of isomorphism classes of $F$-isocrystal with $G$-structure, i.e. faithful exact $\q_p$-linear $\otimes$-functors from the representation category 
$\bRep_0(G)$ to $\cry$ (cf. \cite[Definition 3.3/Remark 3.4(i)]{rapoport}). In particular the twisted fiber functor $\omega_P$ of any locally trivial principal homogeneous 
space $P$ under $G$ over $\q_p$ determines an element of $B(G)$, moreover there is a commutative diagram (cf. \cite[Theorem 1.15]{rapoport}):
$$\begin{CD}
\pi_1(G)_{\Gal(\q_p^{ac}/\q_p)}@<{\gamma_G}<<B(G)\\
@AAA@AAA\\
(\pi_1(G)_{\Gal(\q_p^{ac}/\q_p)})_{tors}@>>>H^1(\Gal(\q_p^{ac}/\q_p),G(\q_p^{ac}))\\
\end{CD}$$
(here we think of $H^1(\Gal(\q_p^{ac}/\q_p),G(\q_p^{ac}))$ as the set of isomorphism classes of the groupoid $\ton(G)(\q)$). In the special 
case of a $\q_p$-torus $T$ the algebraic fundamental group $\pi_1(T)$ is simply the group of cocharacters while the lower horizontal arrow 
is simply the local Tate-Nakayama isomorphism, furthermore: In the toric case the map $\gamma_T:B(T)\rightarrow X_*(T)_{\Gal(\q_p^{ac}/\q_p)}$ is a bijection.

\subsubsection{Endomorphism algebras of $\mod p$ reductions}
\label{gausssum}
Let $H$ be a CM algebra acting faithfully on an isogeny class of complex abelian $\frac{[H:\q]}2$-folds, and let $\lambda:Y\rightarrow\check Y$ be a polarization whose Rosati 
involution stabilizes $H$. One associates a canonical cocharacter $$\mu\in X_*((\g_m\times H^1)/\{\pm1\})$$ to this situation. Let us write $\nu$ for the projection of $\mu$ onto 
the space of $\Gal(\q_p^{ac}/\q_p)$-invariants of $\q\otimes X_*((\g_m\times H^1)/\{\pm1\})$. We are interested in the $\mod p$-reduction of $Y$ together with these additional 
structures, and for that purpose we fix an embedding $K(\fc)\hookrightarrow\c$. It is known that $Y$ possesses good models over sufficiently large rings of algebraic integers, 
and hence its reduction $\Ybar/\fc$ is a well-defined isogeny class of abelian $\frac{[H:\q]}2$-folds over $\fc$. As $\End^0(Y)\subset\End^0(\Ybar)$, the reduction is of CM-type 
as well, and it canonically inherits a polarization $\lambda:\Ybar\rightarrow\check\Ybar$ too. From now on we assume that there exists a CM-field $L\subset H$, such that:
\begin{itemize}
\item[(i)]
$\nu$ is contained in $\q\otimes X_*((\g_m\times L^1)/\{\pm1\})$
\item[(ii)]
$H$ is unramified at $p$ and $\sha(H^1/L^1)$ is trivial.
\end{itemize}
The central theme of this paragraph is the study of the involutive $*$-algebra $\End_L^0(\Ybar)$, under the assumptions (i) and (ii). 
Let $\mu'$ be the image of $\mu$ in the quotient $H^1/L^1$. To prepare the setting of the result we need to fix more choices:
\begin{itemize}
\item
Let $P\in\ton(H^1/L^1)(\q)$ be of type $\mu'$.
\item
Let $M:\bRep_0((\g_m\times H^1)/\{\pm1\})\rightarrow\cry$ be a $F$-isocrystal with $(\g_m\times H^1)/\{\pm1\}$-structure 
whose invariant $\gamma(M)$ agrees with $\mu^{-1}$ as an element of $X_*((\g_m\times H^1)/\{\pm1\})_{\Gal(\q_p^{ac}/\q_p)}$. 
\item
Let $\sigma_p:M|_{\bRep_0(H^1/L^1)}\rightarrow K(\fc)\otimes\omega_P$ be a $\q_p$-linear $\otimes$-isomorphism, 
where the target on the right is the functor defined by composing $\q_p\otimes\omega_P$ with \eqref{reizIV}.
\item
Let $\sigma^{\infty,p}$ be a trivialization of $P$ over $\a^{\infty,p}$.
\end{itemize}
The existence of $\sigma_p$ is due to the assumption (i).

\begin{rem}
\label{weak} 
The assumptions above imply that the group of $\q_p$-valued points of $H^1/L^1$ agrees with $H^1(\q_p)/L^1(\q_p)$. Using the 
reductive $\z_{(p)}$-models $\gH^1$ and $\gL^1$ for $H^1$ and $L^1$ this can be seen as follows: Lang's trick gives the exactness of 
$$1\rightarrow\gL^1(\f_p)\rightarrow\gH^1(\f_p)\rightarrow\gH^1/\gL^1(\f_p)\rightarrow1,$$ 
given that the special fiber of $\gL^1$ is connected. We deduce the surjectivity of $\gH^1(\z_p)\rightarrow\gH^1/\gL^1(\z_p)$, since all involved groups are smooth. At last 
observe that $\gH^1/\gL^1(\z_p)$ agrees with the $\q_p$-valued points of the $\q_p$-anisotropic group $H^1/L^1$, which one might deduce from the Satake isomorphism.
\end{rem}

\begin{lem}
\label{reizV}
Let $L\subset H\subset\End^0(Y)$ satisfy the above properties (i) and (ii), where $Y$ is an isogeny class of complex abelian $\frac{[H:\q]}2$-folds, 
equipped with a polarization $\lambda$, and let $(P,M,\sigma_p,\sigma^{\infty,p})$ be as above. Then there exists a triple of isomorphisms:
\begin{eqnarray*}
&&\eta^{(0)}:\omega_P(\End_L^0(Y(\c)))\rightarrow\End_L^0(\Ybar)\\
&&\eta_p:M(H_1(Y(\c),\q))\rightarrow\d^0(\Ybar)\\
&&\eta^{\infty,p}:H_1(Y(\c),\a^{\infty,p})\rightarrow H_1^{\mathaccent19 et}(\Ybar,\a^{\infty,p})
\end{eqnarray*}
such that the following properties are fullfilled:
\begin{itemize}
\item[(i)]
The map $\eta_p$ (resp. $\eta^{\infty,p}$) preserves the $H$-action and the Weil-pairing up to a factor in $\q_p\backslash\{0\}$ (resp. $\a^{\infty,p\;\times}$)
\item[(ii)]
The diagram
$$\begin{CD}
\omega_P(\End_L^0(Y(\c)))@>{\eta^{(0)}}>>\End_L^0(\Ybar)\\
@AAA@AAA\\
H@>>>\End_L^0(Y)
\end{CD}$$
is commutative, and $\eta^{(0)}$ preserves $*$. 
\item[(iii)]
The diagram
$$\begin{CD}
\End_L^0(\Ybar)\otimes\d^0(\Ybar)@>>>\d^0(\Ybar)\\
@A{\eta^{(0)}\otimes\eta_p}AA@A{\eta_p}AA\\
\omega_P(\End_L^0(Y(\c)))\otimes M(H_1(Y(\c),\q))@>>>M(H_1(Y(\c),\q))
\end{CD}$$
is commutative, where the lower horizontal arrow is defined by using the canonical isomorphism $K(\fc)\otimes\omega_P(\End_L^0(Y(\c)))\cong M(\End_L^0(Y(\c)))$ 
arising from the evaluation of $\sigma_p$ on the $\bRep_0(H^1/L^1)$-object $\End_L^0(Y(\c))$.
\item[(iv)]
The diagram
$$\begin{CD}
\End_L^0(\Ybar)\otimes H_1^{\mathaccent19 et}(\Ybar,\a^{\infty,p})@>>>H_1^{\mathaccent19 et}(\Ybar,\a^{\infty,p})\\
@A{\eta^{(0)}\otimes\eta^{\infty,p}}AA@A{\eta^{\infty,p}}AA\\
\omega_P(\End_L^0(Y(\c)))\otimes H_1(Y(\c),\a^{\infty,p})@>>>H_1(Y(\c),\a^{\infty,p})
\end{CD}$$
is commutative, where the lower horizontal arrow is defined by using the canonical isomorphism 
$\a^{\infty,p}\otimes\omega_P(\End_L^0(Y(\c)))\cong\a^{\infty,p}\otimes\End_L^0(Y(\c))$ arising from the trivialization $\sigma^{\infty,p}$.
\end{itemize}
\end{lem}
\begin{proof}
The existence of an isomorphism $\eta^{\infty,p}$ satisfying property (i) only is clear, in fact the same is true for the existence $\eta_p$, which follows readily from 
the result \cite[Satz(1.6)]{reimann}, which describes $\d^0(\Ybar)$ as an element of $B((\g_m\times H^1)/\{\pm1\})$. Once $\eta^{\infty,p}$ and an isomorphism $\eta^{(0)}$ 
satisfying property (ii) alone, have been found, we can use the lemma \ref{parityIV} to adjust the pair $(\eta^{(0)},\eta^{\infty,p})$ in order to achieve the property (ii) together 
with the compatibility (iv). In view of $(\g_m\times H^1)/\{\pm1\}(\q_p)=\Aut(M)$ we can use remark \ref{weak} for an adjustment of $\eta_p$ in order to achieve the 
compatibility (iii). It remains to demonstrate the mere existence of $\eta^{(0)}$: Let us denote $\End_L^0(Y(\c))$ by $\tilde A$ and $\omega_P(\tilde A)$ by $\Abar$. Notice that 
$\Abar$ is an involutive $L$-algebra in a natural way, because $\omega_P$ is a $\otimes$-functor, and both the involution and the $L$-algebra structure of $\tilde A$ can 
be expressed by certain morphisms in $\bRep_0(H^1/L^1)$ between $\tilde A$ and its tensor powers, moreover $H=\omega_P(H)$ is naturally contained in $\Abar$, and it 
is preserved under $*$. Let $K$ be a $\q$-algebra. In the category of $K$-algebras with $K$-linear involution we want to consider the set $C(K)$ of commutative diagrams: 
$$\begin{CD}
K\otimes\Abar@>{\cong}>>K\otimes\End_L^0(\Ybar)\\
@AAA@AAA\\
K\otimes H@>>>K\otimes\End_L^0(Y)\\
\end{CD}$$
where all arrows, except the upper horizontal one, are the canonical ones. Notice that the natural $H^1/L^1$-action on $\Abar$ fixes $H$ elementwise, 
and hence it is clear that $C$ is a torsor under $H^1/L^1$, which satisfies the Hasse principle. The set of $p$-adic points on $C$ are the diagrams:
$$\begin{CD}
\q_p\otimes\Abar@>{\cong}>>\End_L(\d^0(\Ybar))\\
@AAA@AAA\\
\q_p\otimes H@>>>\q_p\otimes\End_L^0(Y)\\
\end{CD},$$
where we have used the Tate-conjecture to rewrite the upper right entry. From the existence of $\eta_p$ we can now deduce $C(\q_p)\neq\emptyset$. 
If we use the $\ell$-adic Tate-conjectures, we easily get points for the completions at all other finite primes. The triviality of $C$ follows 
once $*$ acts positively on $\r\otimes\Abar$. At the infinite place the result of the inner twist by our Tate-Nakayama class is simply the 
endomorphism ring of $H_1(Y(\c),\r)$, together with the involution that accompanies the positive definite pairing given by $\psi(h(\sqrt{-1})x,y)$.
\end{proof}

\subsubsection{Enriched metaunitary Shimura data}
\label{rich}
Assume that $(G,h)$ is a Shimura datum with coefficients in $L^+$ and that $h$ factors through a $\q_p\&\r$-elliptic, $\q_p$-unramified maximal $\q$-torus 
of the form $G\supset T=(\g_m\times\Res_{L^+/\q}T^1)/\{\pm1\}$. We fix an embedding $\iota_p:K(\fc)\rightarrow\c$, and we let $\nu\in\q\otimes X_*(\zen^G)$ 
be the projection of the cocharacter $\mu_h\in X_*(T)$ onto the space of $\Gal(\q_p^{ac}/\q_p)$-invariants of $\q\otimes X_*(T)$ (N.B.: $T_{\q_p}/\zen_{\q_p}^G$ 
contains no copy of $\g_{m,\q_p}$). Let $\mu'$ be the image of $\mu_h$ in $X_*(T/\zen^G)$. Eventually, we want to make the following choices:
\begin{itemize}
\item[(i)]
Let $P\in\ton(T/\zen^G)(\q)$ be of type $\mu'$.
\item[(ii)]
Let $M:\bRep_0(T)\rightarrow\cry$ be the $F$-isocrystal whose invariant $\gamma_T(M)$ agrees with $\mu^{-1}$ as an element of $X_*(T)_{\Gal(\q_p^{ac}/\q_p)}$.
\item[(iii)]
Let $\sigma_p:M|_{\bRep_0(T/\zen^G)}\rightarrow K(\fc)\otimes\omega_P$ be a $\q_p$-linear $\otimes$-isomorphism.
\item[(iv)]
Let $\sigma^{\infty,p}$ be a trivialization of $P$ over $\a^{\infty,p}$.
\end{itemize}
The existence of $P$ is due to the above assumptions on $T$. We have to introduce the important $\q$-group $I:=P\times^{T/\zen^G}G$. The 
isomorphism $\sigma_p$ induces a further $I(\q_p)\cong\Aut^\otimes(M|_{\bRep_0(G)})$ one. Let us come back to a $L$-metaunitary Shimura datum, 
$$\bT=(\bd^+,\{(V_i,\Psi_i,\rho_i,\bj_i)\}_{i\in\Lambda},\{(R_\pi,\iota_\pi)\}_{\pi\in\Pi})$$ 
for $(G,h)$. By a $T$-enrichment we mean further choices of maximal commutative $*$-invariant $L$-algebras $H_i$ in each $\End_L(V_i)$ such that:
\begin{itemize}
\item
The image of $T$ under $\varrho_i$ is contained in $(\g_m\times H_i^1)/\{\pm1\}$.
\item
Each $H_i$ is unramified at $p$ and $\sha(H_i^1/L^1)$ is trivial.
\end{itemize}
In the presence of a $T$-enrichment $\{H_i\}_{i\in\Lambda}$ we show that (i), (ii), (iii) and (iv) lead to a (homogeneously) polarized $\q$-isogeny class 
of abelian $L^\Lambda$-varieties of type $\bT$. In fact we will be able to describe its $\q$-group of self-quasi-isogenies, the isogeny class of its crystalline 
realization, and the canonical action of the former on the latter. Let $\Rbar\subset\c$ be the splitting field of the CM-algebra $H=\bigoplus_{i\in\Lambda}H_i$, 
so that the set of ring homomorphisms $H_{an}=\dot\bigcup_{i\in\Lambda}\Hom(H_i,\c)$ carries a natural left $\Gal(\Rbar/\q)$-action. We write 
$\thevar\in\Gal(\Rbar/\q)$ for the restriction of the absolute Frobenius on $K(\fc)$ to $\Rbar$, notice $\vartheta=\thevar|_R$. Finally, let $\bard^+$ 
(resp. $\jar_i$) denote the composition of the function $\bd^+$ (resp. $\bj_i$) with the natural one from $H_{an}$ (resp. $H_{i,an}$) to $L_{an}$.\\
Writing the $L$-valued skew-Hermitian pairing on $V_i$ in the usual form $\Psi_i=\tr_{H_i/L}\Psi_i'$ gives rise to skew-Hermitian $H_i$-modules 
$(V_i,\Psi_i')$ and thus to $H_i$-unitary representations $\rho'_i:T_{H_i^+}^1\rightarrow\UL(V_i/H_i,\Psi'_i)$, and it is easy to see that $(V_i,\Psi'_i,\rho'_i,\jar_i)$ 
is a $\thevar$-gauged $H_i$-unitary representation (in the sense of remark \ref{tricky}). Thus, we are allowed to apply the lemma \ref{twistingVII} to the skew-Hermitian Hodge 
$H_i$-modules $(V_i,\Psi_i',\varrho_{i,\r}\circ h)$, the outcome being some family of skew-Hermitian Hodge structures $(\tilde V_i,\tilde\Psi_i',\tilde h_i)$, again with coefficients in 
$H_i$. For each $i$ we fix a complex abelian variety $Y_i/\c$ with polarization $\lambda_i$ and Rosati-invariant $H_i$-action $\iota_i$, such that $H_1(Y_i(\c),\q)$ is isometric to 
$(\tilde V_i,\tr_{H_i/\q}\tilde\Psi_i',\tilde h_i)$. Let $\tilde\upsilon_i\in X_*((\g_m\times H^1)/\{\pm1\})$ be the associated cocharacters, as in subsubsection \ref{gausssum}. Let 
us fix a $\a^\infty$-valued point $\varepsilon_i$ on the $\ton(H_i^1)(\q)$-object $Q_i$ with $\omega_{Q_i}(V_i)=\tilde V_i$ and notice that
\begin{equation}
\label{contracted}
\Ob_{\ton(H_i^1/L^1)(\q)}\ni\tilde P_i:=\Homo_{T/\zen^G}(Q_i,P)
\end{equation} 
is a torsor of type $\tilde\upsilon_i'\in X_*(H_i^1/L^1)$. Consider compatible isomorphisms,
\begin{eqnarray*}
&&\tilde\eta_i^{(0)}:\omega_{\tilde P_i}(\End_L(\tilde V_i))\stackrel{\cong}{\rightarrow}\End_L^0(\Ybar_i)\\
&&\tilde\eta_{i,p}:M(\tilde V_i)\stackrel{\cong}{\rightarrow}\d^0(\Ybar_i)
\end{eqnarray*}
of which the existence is granted by applying lemma \ref{reizV} to the orthogonal direct sum of the $(Y_i,\lambda_i,\iota_i)$'s, which we will 
denote by $(Y,\lambda,\iota)$. We proceed by writing $\eta_i^{(0)}$ for the composition of $\tilde\eta_i^{(0)}$ with the canonical isomorphism from $\omega_P(\End_L(V_i))$ to 
$\omega_{\tilde P_i}(\End_L(\tilde V_i))$ while letting $\eta_i^{(1)}$ be the composition of $\tilde\eta_i^{(1)}$ with $M(\varepsilon_{i,p})$. In view of the theorem \ref{bekanntIII} 
it is now clear that one obtains a (homogeneously) polarized $\q$-isogeny class of $L^\Lambda$-abelian $\fc$-varieties of type $\bT$ from $(\Ybar,\lambda,\iota)$ if one puts 
$$y_\pi:=(\bigotimes_{i\in\pi}\eta_i^{(0)})\circ\omega_P(\iota_\pi),$$
and that $\varepsilon_i^{\infty,p}$ is a full level structure for it.

\begin{thm}
\label{aechz}
Fix $\vartheta$ and $p$ and let $(\bd^+,\{(V_i,\Psi_i,\rho_i,\bj_i)\}_{i\in\Lambda},\{(R_\pi,\iota_\pi)\}_{\pi\in\Pi}=\bT$ and $\{H_i\}_{i\in\Lambda}$ be a $T$-enriched 
$L$-metaunitary datum for a Shimura datum $(G,h)$ with coefficients in $L^+$. Let $(P,M,\sigma_p,\sigma^{\infty,p})$ be a quadruple as introduced at the 
beginning of this subsubsection. Then there exists a fully leveled and homogeneously polarized $\q$-isogeny class of $L^\Lambda$-abelian $\fc$-varieties 
$(\Ybar,\lambda,\iota,\{y_\pi\}_{\pi\in\Pi},\varepsilon^{\infty,p})$ of type $\bT$ satisfying the following:
\begin{itemize}
\item
There exists an isomorphism $\zeta^{(0)}$ from the twisted $\q$-group $I=P\times^{T/\zen^G}G$ to the group of self-quasi-isogenies $I_{\Ybar,\lambda,\iota,\{y_\pi\}_{\pi\in\Pi}}$
\item
There exists a $L^\Lambda$-linear similitude of isocrystals with $L^\Lambda$-operation:
$$\sum_{i\in\Lambda}\zeta_i^{(1)}=\zeta^{(1)}:M(V)\stackrel{\cong}{\rightarrow}\d^0(\Ybar),$$
such that
$$\begin{CD}
\End_L(M(V^\pi))@<<<R_\pi\\
@V{\bigotimes_{i\in\pi}\zeta_i^{(1)}}VV@V{y_\pi}VV\\
\End_L(\bigotimes_{i\in\pi}\d^0(\Ybar[\ge_i])(\frac{1-\Card(\pi)}2))@<{\d^0}<<\End_L^0(\dot\bigotimes_{i\in\pi}\Ybar[\ge_i])
\end{CD}$$
is commutative for every $\pi\in\Pi$.
\item
$\zeta^{(1)}$ carries the canonical $I$-action on $M$, which results from $\sigma_p$, to 
the $I$-action on $(\Ybar,\lambda,\iota,\{y_\pi\}_{\pi\in\Pi})$ that results from $\zeta^{(0)}$.
\end{itemize}
\end{thm}

\subsubsection{Automorphisms of lifts} 
\label{Pepsi}
Recall that $$\bT=(\bd^+,\{(V_i,\rho_i,\Psi_i,\bj_i)\}_{i\in\Lambda},\{(R_\pi,\iota_\pi)\}_{\pi\in\Pi})$$ 
is an unramified $L$-metaunitary Shimura datum for our Shimura datum with $L^+$-coefficients $(G,h)$. Throughout this subsubsection we 
fix a complete discretely valued $K(\f_{p^f})$-extension $N$ with algebraically closed residue field $l$. In the following result $\Gamma_{\tilde x}$ 
denotes the stabilizer of any $\tilde x\in\F_{\bT,\gp}(N)$ in the group $\I_0^{\mu_p}(N)\times G(\a^{\infty,p})$, which acts on $\F_{\bT,\gp}(N)$:

\begin{lem}
\label{Fanta}
For any $x\in\M_{\bT,\gp}(N)$ there exists a canonical reductive $\q$-group $K_x$ equipped with an embedding $i_x:K_{x,\a^{\infty,p}}\hookrightarrow G_{\a^{\infty,p}}$
and a compact open subgroup $\gK_x\subset K_x(\q_p)$ such that for every lift $\tilde x\in\F_{\bT,\gp}(N)$ of $x$ one has
\begin{equation}
\label{Cola}
\Gamma_{\tilde x}=(j_{\tilde x}\times i_x)(K_x(\q)\cap\gK_x),
\end{equation} 
where $j_{\tilde x}:K_{x,N}\rightarrow\I_{0,N}^{\mu_p}$ is another embedding which is canonically associated to $\tilde x$. Moreover, the assignments $x\mapsto(\gK_x,K_x,i_x)$ 
and $\tilde x\mapsto j_{\tilde x}$ are functorial in $N$, there exists a canonical inclusion $\zen^G\hookrightarrow\zen^{K_x}$ and $(K_x/\zen^G)(\r)$ is compact. 
\end{lem}
\begin{proof}
Notice that $\M_{\bT,\gp}(N)=\M_{\bT,\gp}(\O_N)$, in fact we may assume $\tilde x\in\F_{\bT,\gp}(\O_N)$ too, because the content of the lemma 
depends only on the $\I_0^{\mu_p}(N)\times G(\a^{\infty,p})$-orbit of $\tilde x$ and $\F_{\bT,\gp}\rightarrow\M_{\bT,\gp}$ is smooth and surjective. 
A straightforward modification of remark \ref{verlaengert} and theorem \ref{uniformizeII} yields a $\I_0^{\mu_p}(\O_N)\times G(\a^{\infty,p})$-equivariant 
description of $\F_{\bT,\gp}(\O_N)$: We look at the set of data consisting of a fully leveled, homogeneously polarized $\q$-isogeny class of 
$L^\Lambda$-abelian $l$-varieties $(y,\eta)$ of type $\bT$ together with pairs $(U,\zeta)\in\gG_p(W(\O_N))\times\Iso(y)$ rendering the diagram
$$\begin{CD}
K(k)\otimes V
@<{\varrho(O{^F\mu_p(\frac1p)})\circ(F\otimes\id_V)}<<K(k)\otimes V\\
@V{\zeta}VV@V{\zeta}VV\\
\d^0(Y)@<F<<\d^0(Y)\\
\end{CD}$$
commutative, where $O\in\gG_p(W(l))$ is the $\mod\gm_N$-reduction of $U$. We look at the equivalence 
relation defined by $(U',\zeta')\thicksim(U,\zeta)$ if and only if there exists $h\in\gG_p(I(\O_N))$ such that
\begin{eqnarray}
\label{Aah}
&&U'=h^{-1}U\Phi^{\mu_p}(h)\\
\label{Ooh}
&&\zeta'=\zeta\circ\varrho(h_0)
\end{eqnarray}
holds, where $h_0\in\gG_p(I(l))$ denotes the $\mod\gm_N$-reduction of $h$. In fact, the former is a $\I^{\mu_p}(\O_N)\times G(\a^{\infty,p})$-set, 
so that the set of its $\thicksim$-classes acquires a natural $\I_0^{\mu_p}(\O_N)\times G(\a^{\infty,p})$-action, as $\gG_p(I(\O_N))$ is a normal 
subgroup of $\I^{\mu_p}(\O_N)$. Now recall $b_{y,\zeta}=O{^F\mu_p(\frac1p)}$ and that lemma \ref{obstX} constructs a $\q_p$-algebraic subgroup 
$$\hat\Gamma_U\subset J_{b_{y,\zeta}},$$ 
of which some compact open subgroup is canonically isomorphic to the subgroup $\Aut(U)$, which is in turn a subgroup of $\Aut(O)$. We define $K_x$ 
to be the largest $\q$-algebraic subgroup of $I_y$ of which the scalar extension to $\q_p$ is contained in $j_{y,\zeta}^{-1}(\hat\Gamma_U)$, please see 
to subsection \ref{overFq} for the definition of $j_{y,\zeta}:I_{y,\q_p}\rightarrow J_{b_{y,\zeta}}$. We define the compact open subgroup $\gK_x\subset K_x(\q_p)$ 
to be the inverse image of $\Aut(U)\subset\hat\Gamma_U(\q_p)$ under $j_{y,\zeta}$. The $N$-rational inclusion $j_{\tilde x}:K_{x,N}\rightarrow\I_{0,N}^{\mu_p}$ 
arises naturally from lemma \ref{obstX} and the proof of \eqref{Cola} is a straightforward modification of corollary \ref{gekuerzt}.
\end{proof}

Let us say that a $T$-enrichment $\{H_i\}_{i\in\Lambda}$ for our $L$-metaunitary datum $(\bd^+,\{(V_i,\Psi_i,\rho_i,\bj_i)\}_{i\in\Lambda},\{(R_\pi,\iota_\pi)\}_{\pi\in\Pi}=\bT$ 
is unramified at $p$ if there exist self-dual $\z_{(p)}\otimes\O_L$-lattices $\gV_{i,p}\subset V_i$ which enjoy the following properties:
\begin{itemize}
\item
The subgroup $U_p\subset G(\q_p)$, as defined in \eqref{twistingIII}, is hyperspecial.
\item
The subgroup $U_p\cap T(\q_p)\subset T(\q_p)$ is hyperspecial too.
\end{itemize}
Notice that unramified enrichments for $\bT$ exist if and only if $\bT$ is unramified, essentially because unramified reductive $\z_p$-groups are known to possess elliptic maximal 
$\z_p$-tori. So we may fix one and write $\gT_p=(\g_m\times\prod_{\gq\in S_p}\Res_{W(\f_{p^{r_\gq}})/\z_p}\T_\gq^1)/\{\pm1\}$ (resp. $\gG_p$) for the reductive $\z_p$-model 
of $T=(\g_m\times\Res_{L^+/\q}T^1)/\{\pm1\}$ (resp. $G$) corresponding to the hyperspecial group $U_p\cap T(\q_p)$ (resp. $U_p$), following the grammar of subsubsection 
\ref{steady}. Moreover, let us fix a quadruple $(P,M,\sigma_p,\sigma^{\infty,p})$ as introduced at the beginning of subsubsection \ref{rich}. According to theorem \ref{aechz} 
it gives rise to a homogeneously polarized $\q$-isogeny class of $L^\Lambda$-abelian $\fc$-varieties $y=(\Ybar,\lambda,\iota,\{y_\pi\}_{\pi\in\Pi})$ of type $\bT$ with a 
certain full level structure $\varepsilon^{\infty,p}$. Choose any banal $(\gT_p,\mu_p)$-display $\E$ over $W(\f_{p^f})$, for example the object represented by the neutral 
element $1_{\gT}\in\gT_p(W(\f_{p^f}))$ (N.B.: From the start our $W(\fc)$-rational cocharacter $\mu_p$ was required to factor through $\gT_p\times_{\z_p}W(\f_{p^f})$). 
Choose an isomorphism $\zeta$ from $\q\otimes M_{\E_{\fc}}$ to $M$, pretending that $\E$ was a $(\gG_p,\mu_p)$-display yields a $\mu_p$-fake integral structure 
$\zeba$ on $y$, in the sense of definition \ref{LEVEL}. According to theorem \ref{uniformizeII} the sextuple $(Y,\lambda,\iota,\{y_\pi\}_{\pi\in\Pi},\eta,\zeba)$ determines 
a point $\xi\in\Marb_{\bT,\gp}(\fc)$, which is clearly equipped with a $W(\fc)$-lift $x$ on $\M_{\bT,\gp}$, due to the presence of $\E_{W(\fc)}$, and a little bit of 
extra work shows that $x$ is already defined over the subring $\z_p^{sh}:=\bigcup_{s=1}^\infty W(\f_{p^s})$. We have a $2$-commutative diagram
$$\begin{CD}
\B(\gT_p,\mu_p)@<<<\Spf W(\fc)\\
@VVV@V{x}VV\\
\B(\gG_p,\mu_p)@<{\P^{(\infty)}}<<\M_{\bT,\gp}^{(\infty)}\\
\end{CD},$$
The following properties are fulfilled:
\begin{itemize}
\item
The bottom row is $G(\a^{\infty,p})$-equivariant.
\item
The top row is $\gT_p(\z_p)$-equivariant with respect to the trivial action on $\Spf W(\fc)$ and the natural action of $\gT_p(\z_p)$ on any $\B(\gT_p,\mu_p)$-object.
\item
The vertical map $x$ preserves the $T(\q)\cap\gT_p(\z_p)$-equivariance, with respect to the map to $G(\a^{\infty,p})$ given by $i_{y,\eta}$.
\end{itemize}

The following result combines lemma \ref{Fanta} with the scenario at hand:

\begin{cor}
\label{7up}
Under the assumptions and with the notations of lemma \ref{Fanta}, there exists $x\in\M_{\bT,\gp}(W(\fc))$ such that $K_x$ and $G$ have the same rank.
\end{cor}

\section{Conclusions}
\label{wrap}

Let us choose a $\O_{ER}$-linear embedding $\iota_p:W(\fc)\hookrightarrow\c$ making $\c$ into a $W(\fc)$-algebra. We would like to describe the 
$G(\a^{\infty,p})$-space $\M_{\bT,\gp}(\c_{[\iota_p]})$ in the analytic category, together with the complexification of the pair $(\F_{\bT,\gp},\nabla_{\bT,\gp})$, 
again taking into account its $G(\a^{\infty,p})$-equivariance. Let us choose a base point $s_0\in\M_{\bT,\gp}(\c_{[\iota_p]})$, and let $M_{s_0}$ be the set of pairs 
$(s_1,H)$ where $H$ is a homotopy class of paths from $s_0$ to some $s_1\in\M_{\bT,\gp}(\c_{[\iota_p]})$. Let $\Delta_{s_0}$ be the set of pairs $(g,H)$ where 
$H$ is a homotopy class of paths from $s_0g$ to $s_0$ for some $g\in G(\a^{\infty,p})$. It is clear that $M_{s_0}$ is a simply connected complex manifold, while
$$(g_0,H_0)*(g_1,H_1):=(g_0g_1,(H_0g_1)*H_1)$$
defines a group law on $\Delta_{s_0}$, as the homotopy classes $H_0g_1$ of paths from $s_0(g_0g_1)$ to $s_0g_1$ and $H_1$ of paths 
from $s_0g_1$ to $s_0$ can be concatenated. The isotropy group $\Gamma_{s_0}$ of $s_0$ in $G(\a^{\infty,p})$ can be identified with the 
subgroup of elements $(g,H)\in\Delta_{s_0}$ for which $H$ is constant. The group $\Delta_{s_0}$ acts on $M_{s_0}$ from the left by means of 
$$(g_0,H_0)*(s_1,H):=(s_1g_0^{-1},(H_0*H)g_1^{-1}),$$
and $\Delta_{s_0}\backslash(M_{s_0}\times G(\a^{\infty,p}))$ can be identified with an open $G(\a^{\infty,p})$-invariant union of path-components 
of $\M_{\bT,\gp}(\c_{[\iota_p]})$. Let us write $G_\infty\subset G_\r$ for the algebraic subgroup which is generated by all non-compact almost simple factors of $G_\r$. Let $P$ be 
the image of $U_{\mu_p^{-1}}^0\times_{W(\f_{p^f}),\iota_p}\c$ in $G_{\infty,\c}^{ad}$. Upon having chosen a horizontal section $\tilde s$ of $\F_{\bT,\gp}\times^{\I_0^{\mu_p}}\gG_p$ 
over $M_{s_0}$, we obtain a canonical monodromy homomorphism 
$$\varphi_{\tilde s}:\Delta_{s_0}\rightarrow G_\infty^{ad}(\c)$$ 
and the locally biholomorphic $\Delta_{s_0}$-equivariant period map 
$$I_{\tilde s}:M_{s_0}\rightarrow G_\infty^{ad}(\c)/P(\c)=G(\c)/U_{\mu_p^{-1}}^0(\c_{[\iota_p]}),$$
in fact $\varphi_{\tilde s}$ factors through $G(\c)\rightarrow G_\infty^{ad}(\c)$, but we do not need that. Before we proceed we give the provisional outcome of these methods, 
$U_p^{ab}$ denotes the maximal compact subgroup of $G^{ab}(\q_p)$:

\begin{thm}
\label{uniformizeIV}
Let $\bT=(\{(V_i,\Psi_i,\rho_i,\bj_i)\}_{i\in\Lambda},\{(R_\pi,\iota_\pi)\}_{\pi\in\Pi})$ be an unramified $L$-metaunitary datum for a Shimura datum $(G,h)$ and fix an integral 
structure $\{\gV_{i,p}\}_{i\in\Lambda}$. Suppose that $x$ is a $W(\fc)$-rational point of $\M_{\bT,\gp}$ such that $K_x$ and $G$ have the same rank. The smallest open and 
closed $G(\a^{\infty,p})$-invariant subscheme of $\M_{\bT,\gp}\times_{W(\f_{p^f})}W(\fc)$ containing $x$, shall be denoted by $\M_{\bT,\gp}^*$.
\begin{itemize}
\item
The complexification $\M_{\bT,\gp}^*(\c_{[\iota_p]})$ is $G(\a^{\infty,p})$-equivariantly biholomorphically equivalent to the space 
$$\Delta^*\backslash(X^+\times G(\a^{\infty,p})),$$
for a discrete and cocompact subgroup $\Delta^*\subset\Aut(X^+)^+\times G(\a^{\infty,p})$.
\item
The image of $\Delta^*$ in $G^{ab}(\a^{\infty,p})$ is contained in
$$G^{ab}(\z_{(p)})^+:=G^{ab}(\q)\cap(G^{ab}(\r)^+\times U_p^{ab}\times G^{ab}(\a^{\infty,p})),$$ 
the intersection taking place in $G^{ab}(\a)$, and it has finite index in that group.
\end{itemize}
\end{thm}
\begin{proof}
Let $s_0$ be the generic fiber of $x$ and let $\tilde s$ be as above. We show that one obtains a bounded period morphism 
$(\Delta_{s_0},M_{s_0},G_{\infty,\c}^{ad},P,\varphi_{\tilde s},I_{\tilde s})$ in the sense that all properties (M1), (M2), (M3), and (M4) of the part \ref{characterization} of the appendix 
are valid, (M3) follows immediately from corollary \ref{7up} and (M2) follows immediately from lemma \ref{Fanta}. Property (M1) holds because the monodromy group is already 
maximal in the formal category by lemma \ref{maximalholonomy}, and property (M4) holds because $\L$ is ample. Now recall our preferred choice of 
$h\in X^+$, and that part \ref{characterization} allows to choose $\tilde s$ in such a way that 
\begin{eqnarray*}
&&\varphi_{\tilde s}(\Delta_{s_0})\subset G_\infty^{ad}(\r)^+\\
&&I_{\tilde s}(s_0)=h
\end{eqnarray*}
hold, giving the requested bijection from the complexification of $\M_{\bT,\gp}^*$ to $\Delta_{s_0}\backslash(X^+\times G(\a^{\infty,p}))$, which sends 
the base point $s_0$ to $h$. From now on we will treat $\Delta_{s_0}$ as a subgroup of $G(\a^{\infty,p})$. Notice that every element of $\Delta_{s_0}$ 
can be written as $\gamma=\gamma_1\dots\gamma_n$ with elliptic regular semisimple elements $\varphi_{\tilde s}(\gamma_i)$, according to 
\cite[Chapter IX, Lemma 3.3]{margulis}. Following the method of Ihara we let $s_1,\dots,s_n$ be their fixed points on $M_{s_0}$, and consider corresponding 
maximal $\q$-tori $T_i\subset K_{s_i}$ containing $\gamma_i$. Now notice that we can apply part (i) of proposition \ref{FOF}, because tori are connected.
\end{proof}

We choose an auxiliary compact open subgroup of the form $K^p=\prod_{\ell\neq p}K_\ell\subset G(\a^{\infty,p})$. Spurred by the above result we continue with the study of the groups: 
\begin{eqnarray*}
&&\Delta_{s_0}^{der}:=\Delta_{s_0}\cap G^{der}(\a^{\infty,p})\\
&&\Delta_{s_0,n}^{der}:=\Delta_{s_0}^{der}\cap(\prod_{p\neq\ell\leq n}G(\q_\ell)\times\prod_{p\neq\ell>n}K_\ell)
\end{eqnarray*}
We consider the maps $\phi:\Delta_{s_0}\rightarrow G_\infty^{ad}(\r)\times G(\a^{\infty,p})$ (resp. 
$\phi_n:\Delta_{s_0}\rightarrow G_\infty^{ad}(\r)\times\prod_{p\neq\ell\leq n}G(\q_\ell)$) defined by the cartesian product of $\varphi_{\tilde s_0}$ with the natural inclusion 
into $G(\a^{\infty,p})$ (resp. projection to $\prod_{p\neq\ell\leq n}G(\q_\ell)$). For any place $v\in L_{an}^+\dot\cup(\Spec\O_{L^+}\backslash\{0\})$ we let $G_v^1$ be the algebraic 
group $G_{L_v^+}^1$ where $L_\iota^+=\r_{[\iota]}$ for any $\iota\in L_{an}^+$ while $L_\gm^+$ is the $\gm$-adic  completion for any maximal ideal $\gm\subset\O_{L^+}$. We are particularly interested in the sets of places:
\begin{eqnarray*}
&&S_n^p=\{\gm\in\Spec\O_{L^+}|p\notin\gm\ni n!\}\\
&&S^p=\{\gm\in\Spec\O_{L^+}|p\notin\gm\neq0\}
\end{eqnarray*}
For the discussion in the remainder of this section we need the following additional assumptions on $(G,h)$ and $\bT$:
\begin{itemize}
\item[(i)]
$p$ is inert in $L^+$ and $G^{1der}$ is a simply connected and absolutely almost simple group of type $B_l$, $C_l$ or $E_7$.
\item[(ii)]
There exist $i_1,i_2\in\Lambda$ such that at least one of the two equivalences $\rho_{i_1}\otimes(L\oplus\Ad^{G_L^{1ad}})\simeq\rho_{i_2}(\mbox{or }\bohr_{i_2})$ hold.
\end{itemize}
Here notice, that part (ii) of our proposition \ref{FOF} together with the above assumption (ii) implies a $\a^{\infty,p}$-version of the "property $f_1$" of 
\cite[Remark 3.8]{varshavskyII}. Let $\Sigma$ be the set of real embeddings of $L^+$ such that $G^1(\r_{[\iota]})$ is not compact. Following \cite[Step 2V]{varshavskyII} there exists a simply connected algebraic group $D$ over a number field $F$ together with 
\begin{itemize}
\item
embeddings $\iota_v:F\rightarrow L_v^+$ and 
\item
isomorphisms $\psi_v:D\times_{F,\iota_v}L_v^+\stackrel{\cong}{\rightarrow}G_v^{1der}$
\end{itemize}
for any $v\in\Sigma\dot\cup S^p$, such that the following holds for every $n$: The inverse image of the irreducible lattice
$$\phi_n(\Delta_{s_0,n}^{der})\subset G_\infty^{ad}(\r)\times\prod_{p\neq\ell\leq n}G^{der}(\q_\ell)$$
under the topological isomorphism
\begin{eqnarray*}
&&\psi_n:\prod_{v\in\Sigma}D^{ad}({L_v^+}_{[\iota_v]})\times\prod_{v\in S_n^p}D({L_v^+}_{[\iota_v]})\\
&&\rightarrow G_\infty^{ad}(\r)\times\prod_{v\in S_n^p}G_v^{1der}(L_v^+),
\end{eqnarray*}
which is induced by $\{\psi_v^{ad}\}_{v\in\Sigma}$ and $\{\psi_v\}_{v\in S_n^p}$ is a $S_n^p$-arithmetic lattice
$$\psi_n^{-1}(\phi_n(\Delta_{s_0,n}^{der}))\subset\prod_{v\in\Sigma}D^{ad}({L_v^+}_{[\iota_v]})\times\prod_{v\in S_n^p}D({L_v^+}_{[\iota_v]})$$
i.e. commensurable with the $S_n^p$-congruence subgroup of $\O_F[\prod_{p\neq\ell\leq n}\ell^{-1}]$-valued points of some $\O_F[\prod_{p\neq\ell\leq n}\ell^{-1}]$-form 
of $D$ (N.B.: Since $A_l$ is excluded, the map of places $v\mapsto\iota_v^{-1}(v)$ is an injection, but this is not essential, cf. \cite[step 2IV]{varshavskyII}). 
According to the method described in \cite[Step 2VI]{varshavskyII}, we can use the aforementioned property $f_1$ again to show that:
\begin{itemize}
\item
$F=L^+$
\item
$\iota_v$ is equal to the canonical embedding of $L^+$ into $L_v^+$ for each $v\in\Sigma\dot\cup S^p$
\end{itemize}

The observation \cite[Claim 2.5.5]{varshavskyI} and strong approximation for the simply connected 
algebraic group $D$ (e.g. \cite[Chapter II, Theorem(6.8)]{margulis}) yields a topological map: 
$$\psi:\prod_{\iota\in\Sigma}D^{ad}(\r_{[\iota]})\times D(\a^{\infty,p}\otimes L^+)\stackrel{\cong}{\rightarrow}G_\infty^{ad}(\r)\times G^{1der}(\a^{\infty,p}\otimes L^+)$$ 
One proves this along the lines of \cite[Step 2VII]{varshavskyII}. Using $\zen^D\cong\{\pm1\}$, Chebotarev and the exactness of 
$$1\rightarrow\zen^D(L^+)\rightarrow D(L^+)\rightarrow D^{ad}(L^+)\rightarrow H^1(L^+,\zen^D)\rightarrow H^1(L^+,D)$$
one infers $D(L^+)\zen^D(\a^{\infty,p})\supset\psi^{-1}(\phi(\Delta_{s_0}^{der}))$ and the latter is a $S^p$-arithmetic lattice. Lemma \ref{Fanta} reveals 
$D(L^+)\supset\psi^{-1}(\phi(\Delta_{s_0}^{der}))$. Under the above assumption (i) there are no outer forms of $G^{1der}$ and the inertness of $p$ in $L^+$ allows to 
deduce $D\cong G^{1der}$ from the Hasse principle for $G^{1ad}$, just as in \cite[Step 2X]{varshavskyII}. The proof of theorem \ref{uniformizationIII} is completed by the 
usual methods of \cite[Paragraph 6]{deligne4}, together with the known local cases of the congruence subgroup property, cf. \cite[Theorem 9.1]{rapinchuk}, \cite{sury}.

\begin{appendix}
\section{Some facts on group schemes}
\label{prov}

Recall, that every element $g$ in an abstract group $G$ gives rise to an inner automorphism, which is defined by sending any element $x\in G$ to its conjugate 
$gxg^{-1}$, and denoted by $\int^G(g)\in\Aut(G)$. Now let $X$ be a scheme and let $\G$ be an $X$-functor in groups. We write $\End(\G)$ for the class of 
functorial transformations $\G\rightarrow\G$ which preserve the group structure, and the subclass of invertible elements therein is denoted by $\Aut(\G)$. Again, the 
elements $g\in\G(R)$ give rise to natural elements $\Int^\G(g/R)\in\Aut(\G_R)$, because for varying $X$-morphisms $\phi:R\rightarrow S$ the families of inner 
automorphisms $\int^{\G(S)}(\G(\phi)(g))$ form invertible transformations from the group functor $\G_R$ to itself. By $\Endo(\G)$ we mean the monoidal $X$-functor 
$R\mapsto\End(\G_R)$, whenever it exists i.e. provided that each of the classes on the right hand-side is a set, and similarly for the group functor $\Auto(\G)$. 
Moreover, in case $\Auto(\G)$ is well-defined, there is an associated group structure-preserving functorial transformation $\Int^\G:\G\rightarrow\Auto(\G)$, which is 
given by $\Int_R^\G:\G(R)\rightarrow\Aut(\G_R);g\mapsto\Int^\G(g/R)$ on the set of points over $R\in\Ob_{\alg_X}$. We let $\zen^\G$ be the kernel of $\Int^\G$. 

\section{Some facts on Connections}
\label{vitamins}
Let $Z$ be a $B$-algebra, let $A$ be a $Z$-algebra, let $\pr_1,\pr_2:A\rightarrow A\otimes_ZA$ and 
$\pr_{12},\pr_{23},\pr_{13}:A\otimes_ZA\rightarrow A\otimes_ZA\otimes_ZA$ be the $Z$-algebra homomorphisms which are defined by:
\begin{eqnarray*}
&&\pr_1(x)=x\otimes1\\
&&\pr_2(x)=1\otimes x\\
&&\pr_{12}(x\otimes y)=x\otimes y\otimes1\\
&&\pr_{23}(x\otimes y)=1\otimes x\otimes y\\
&&\pr_{13}(x\otimes y)=x\otimes1\otimes y
\end{eqnarray*}
We need to introduce two related $Z$-algebras, namely $A\oplus\Omega_{A/Z}^1=:A'$, on which the multiplication law is given by
$$(x,\phi)(x',\phi')=(xx',x\phi'+x'\phi),$$
and $A\oplus\Omega_{A/Z}^1\oplus\Omega_{A/Z}^1\oplus\Omega_{A/Z}^2=:A''$ with multiplication law given by:
$$(x,\epsilon,\delta,\eta)(x',\epsilon',\delta',\eta')=(xx',x\epsilon'+x'\epsilon,x\delta'+x'\delta,x\eta'+x'\eta+\epsilon\wedge\delta'+\epsilon'\wedge\delta)$$
Their relevance stems from the surjections:
\begin{eqnarray*}
&&A^{\otimes_Z2}\twoheadrightarrow A';\,x\otimes y\mapsto(xy,xd_{A/Z}(y))\\
&&A^{\otimes_Z3}\twoheadrightarrow A'';\,x\otimes y\otimes z\mapsto(xyz,xd_{A/Z}(yz),xyd_{A/Z}(z),xd_{A/Z}(y)\wedge d_{A/Z}(z))
\end{eqnarray*}
For each $i\in\{1,2\}$ we denote the composition of the above surjection $A^{\otimes_Z2}\rightarrow A'$ with $\pr_i$ by $\peebar_i:A\rightarrow A'$, in fact 
$\peebar_2$ (resp. $\peebar_1$) is given by the well-known formula $x\mapsto x+d_{A/Z}(x)$ (resp. by $x\mapsto x$). %Let us spell out the analog for $A''$:

\begin{prop}
\label{obstV}
Consider the $Z$-algebra homomorphism $\peebar_{12}$ (resp. $\peebar_{23}$ or $\peebar_{13}$) from $A'$ to $A''$ which is given 
by $(x,\phi)\mapsto(x,\phi,0,0)$ (resp. $(x,\phi)\mapsto(x,d_{A/Z}(x),\phi,d_{A/Z}(\phi))$ or $(x,\phi)\mapsto(x,\phi,\phi,0)$. Then
$$\begin{CD}
A'@>{\peebar_{ij}}>>A''\\
@AAA@AAA\\
A^{\otimes_Z2}@>{\pr_{ij}}>>A^{\otimes_Z3}\\
\end{CD}$$
is a commutative diagram, for any $1\leq i<j\leq3$.
\end{prop}

In this optic the diagonal $A^{\otimes_Z2}\rightarrow A;\,x\otimes y\mapsto xy$ (resp. the two degeneracy homomorphisms 
$A^{\otimes_Z3}\rightarrow A^{\otimes_Z2};\,x\otimes y\otimes z\mapsto xy\otimes z$ or $x\otimes yz$) give rise to the homomorphisms 
$A'\rightarrow A;\,(x,\phi)\mapsto x$ (resp. $A''\rightarrow A';\,(x,\epsilon,\delta,\eta)\mapsto(x,\epsilon\mbox{ or }\delta)$). Observe that the intersection of the kernels 
of the two degeneracy maps is the ideal $\Omega_{A/Z}^2$. Now, let $P$ be a locally trivial principal homogeneous space for $\G$ over $A$. Let $\tilde\G=\Auto_\G(P)$ 
be the twist of $\G$, that is determined by $P$. Notice that $\tilde\G$ is an affine and flat group scheme over $A$, simply because it is fpqc locally isomorphic to $\G$.

\begin{defn}
\label{horizon}
\begin{itemize}
\item[(i)]
Write $P_i$ for the locally trivial principal homogeneous spaces for $\G$ over $A'$ which are obtained by 
pull-back along $\peebar_i$, for $1\leq i\leq2$. By a connection (relative to $Z$) one means an isomorphism 
\begin{equation*}
\alpha:P_2\rightarrow P_1
\end{equation*}
whose pull-back along the diagonal $A'\rightarrow A$ agrees with the identity section. We denote the set of all connections on $P$ by 
$\con_Z(P/A)$. If $P$ is equipped with a descent to $Z$, then we call the natural isomorphism from $P_2$ to $P_1$ the trivial connection.
\item[(ii)]
To any two connections $\alpha$ and $\alpha'$ one may consider the map
\begin{equation*}
\alpha'\circ\alpha^{-1}:P_1\rightarrow P_1,
\end{equation*}
which is seen to be a $A'$-valued automorphism of our locally trivial principal homogeneous space $P$, whose pull-back along the diagonal 
map $(x,\phi)\mapsto x$ agrees with identity section. The canonical element $\beta\in\tilde\G(\Omega_{A/Z}^1)$, which is thus obtained 
is called the difference of $\alpha'$ and $\alpha$ and it will be denoted by $\alpha'-\alpha$, and one also writes $\beta+\alpha:=\alpha'$.
\item[(iii)]
Write $\alpha_{ij}$ for the isomorphism between locally trivial principal homogeneous spaces for $\G$ over $A''$ which are 
obtained by pull-back of some $\alpha\in\con_Z(P/A)$ along the maps $\peebar_{ij}$ for $1\leq i<j\leq3$. The identities 
$\peebar_{13}\circ\peebar_2=\peebar_{23}\circ\peebar_2$ and $\peebar_{23}\circ\peebar_1=\peebar_{12}\circ\peebar_2$ allow us to consider the composition 
\begin{equation*}
\alpha_{12}\circ\alpha_{23}\circ\alpha_{13}^{-1},
\end{equation*}
which we are allowed to regard as a $A''$-valued automorphism of our locally trivial principal homogeneous space $P$, given that 
$\peebar_{13}\circ\peebar_1=\peebar_{12}\circ\peebar_1$, furthermore, since its pull-backs under the two degeneracy maps $(x,\epsilon,\delta,\eta)\mapsto(x,\epsilon\mbox{ or }\delta)$ agree with identity sections, we obtain a canonical element $\cur_Z(\alpha)\in\tilde\G(\Omega_{A/Z}^2)$, which is called the curvature of $\alpha$.
\end{itemize}
\end{defn}

\begin{rem}
The curvature map from $\con_Z(P/A)$ to $\tilde\G(\Omega_{A/Z}^2)$ as well as the summation from $\tilde\G(\Omega_{A/Z}^1)\times\con_Z(P/A)$ to $\con_Z(P/A)$ 
are $\tilde\G(A)$-equivariant with respect to the natural left actions defined on each of $\tilde\G(\Omega_{A/Z}^1)$, $\tilde\G(\Omega_{A/Z}^2)$ and $\con_Z(P/A)$. 
\end{rem}

Let $S$ be a $B$-scheme, let $X$ be a $S$-scheme, and let $P$ be a locally trivial principal homogeneous space for $\G$ over $X$. It is completely 
clear that the concepts of \ref{horizon} generalize immediately to the scheme theoretic setting: The set $\con_S(P/X)$ of $S$-connections still 
forms a $\tilde\G(X)$-set, which is empty or a $\tilde\G(X)$-equivariant principal homogeneous space under $\tilde\G(\Omega_{X/S}^1)$, 
and the curvature map reads $\cur_S:\con_S(P/X)\rightarrow\tilde\G(\Omega_{X/S}^2)$, and it is $\tilde\G(X)$-equivariant too.

\subsection{A remark on the Cartan-Maurer form}

Recall that a vector bundle on a scheme $S$ is an $\O_S$-module which is locally free of finite rank and we write $\ect(S)$ for the additive 
rigid $\otimes$-category of vector bundles on $S$. Recall that a $S$-connection for a vector bundle $\F$ on some smooth $S$-scheme $X$ is a 
$\O_S$-linear map $\nabla:\F\rightarrow\Omega_{X/S}^1\otimes_{\O_X}\F$ satisfying an analog of the Leibniz rule, namely:
\begin{equation}
\label{Keks}
\nabla(ax)=d_{X/S}(a)\otimes x+a\otimes\nabla(x)
\end{equation}
for any sections $a$ (resp. $x$) of $\O_X$ (resp. $\F$) over any open subset of $X$. Notice that the set of all $S$-connections for $\F$ forms a principal homogeneous space 
for the group of global sections of $\F\otimes_{\O_X}\check\F$-valued $1$-forms, unless it is empty. Every $S$-connection $\nabla$ gives rise to a further $\O_S$-linear map
$$\nabla_1:\Omega_{X/S}^1\otimes_{\O_X}\F\rightarrow\Omega_{X/S}^2\otimes_{\O_X}\F;\,\eta\otimes x\mapsto d_{X/S}(\eta)\otimes x-\eta\otimes\nabla(x)$$
and $\nabla$ is called integrable if $\nabla_1\circ\nabla=0$. We write $\yst_{X/S}$ (resp. $\mic_{X/S}$) for the additive rigid $\otimes$-category 
consisting of vector bundles on $X$ together with a connection (resp. integrable connection), please see to \cite{deligne2} for further details.\\
Suppose that $S$ is a $B$-scheme and let $\G$ be a smooth affine group over $B$. Let $P$ be a locally trivial principal homogeneous space for $\G$ 
over $X$ and recall that it gives rise to a twisted fiber functor $\omega_P$. Furthermore applying $\omega_P$ to the adjoint representation $\Ad^\G$ 
(together with the Lie-bracket from $\Ad^\G\otimes_B\Ad^\G$ to $\Ad^\G$) yields a Lie-$\O_X$-algebra $\tilde\gg=\Lie\tilde\G$. In the same vein every 
$\alpha\in\con_S(P/X)$ gives rise to $S$-connections $\nabla_\alpha(\rho)$ on each $\omega_P(\rho)\in\Ob_{\ect(X)}$, which are compatible in the sense that 
$$\bRep_0(\G)\rightarrow\yst_{X/S};\rho\mapsto(\omega_\alpha(\rho),\nabla_\alpha(\rho))$$ 
sets up a $\otimes$-functor (see e.g. \cite[VI.1.2.3.1]{rivano}). Moreover, applying $\omega_P$ to the $\bRep_0(\G)$-morphism 
$\rho^{der}:\gg\rightarrow\End_B(\omega^\G(\rho))$ yields a map, say $\omega_P(\rho^{der})$ from $\tilde\gg$ to 
$\End_B(\omega_P(\rho))$ and the formula $\nabla_{E+\alpha}(\rho)=\omega_P(\rho^{der})E+\nabla_\alpha(\rho)$ 
holds for all $E\in\Gamma(X,\Omega_{X/S}^1\otimes_{\O_X}\tilde\gg)$. The following formulae are well-known:
\begin{itemize}
\item
$\nabla_\alpha(\Ad^\G)\cur_S(\alpha)=0$
\item
$\cur_S(E+\alpha)=\frac{[E,E]}2+\nabla_\alpha(\Ad^\G)E+\cur_S(\alpha)$
\end{itemize}
At last we want to make these things more explicit for the case of the globally trivial principal homogeneous space $P=\G\times_BX$, in which case $\tilde\G$ and $\G\times_BX$ 
agree canonically, so that $\con_S(P/X)=\Omega_{X/S}^1\otimes_B\gg$ (by decreeing the neutral element to be the trivial connection), and the $\G(X)$-equivariant map 
$\cur_S:\Omega_{X/S}^1\otimes_B\gg\rightarrow\Omega_{X/S}^2\otimes_B\gg$ can be described by $E\mapsto\frac{[E,E]}2+ d_{X/S}(E)$. However, the left actions of 
some $g\in\G(X)$ on $\Omega_{X/S}^2\otimes_B\gg$ and $\Omega_{X/S}^1\otimes_B\gg$ are a little subtle: While the $g$-action on $\Omega_{X/S}^2\otimes_B\gg$ 
is simply given by $\id_{\Omega_{X/S}^2}\otimes\Ad^\G(g)$, it acts by means of $E\mapsto(\id_{\Omega_{X/S}^1}\otimes\Ad^\G(g))E-\eta^\G\circ g$ on 
$\Omega_{X/S}^1\otimes_B\gg$ where $\eta^\G\in\Omega_{\G/B}^1\otimes_B\gg$ is the canonical right-invariant Cartan-Maurer form, and $\eta^\G\circ g$ 
denotes the pull-back of the $1$-form $\eta^\G$ by means of $g:X\rightarrow\G$. We remark that the Cartan-Maurer form of $\GL(n)$ is $(dA)A^{-1}$.

\section{Gradations and Filtrations}
\label{cocha}
By a gradation of type $T$ of some torsion-free finitely generated $B$-module $\V$, we mean a direct sum $\V=\bigoplus_{l\in T}\V_l$ (N.B. we do not require that $\V_l\neq0$ for 
$l\in T$). Let us write $\GL(\V/B)$ (resp. $\underline\V$) for the  multiplicative (resp. additive) group functor, of which the $R$-valued points are the units in $\End_B(\V)\otimes_BR$ 
(resp. $\V\otimes_BR$). The group homomorphisms from the multiplicative group $\g_{m,B}$ to $\GL(\V/B)$, i.e. the cocharacters, are given by their respective weight-gradations 
of $\V$, which are of type $\z$, as the group $\z$-scheme $\g_m$ is given by the spectrum of $\z[x^{\pm1}]=\bigoplus_{l\in\z}\z x^l$. By slight abuse of language we will say 
that some cocharacter $\nu$ was of type $T\subset\z$ if and only if its weight-gradation would have that property and the unique cocharacter of type $\{1\}$ will be denoted 
by $\delta_{\V}$. Let us write $\a^1:=\Spec\z[x]$ for the affine line, and notice in particular, that a cocharacter is of type $\n_0$ if and only if it extends to a homomorphism 
\begin{equation*}
\a_B^1\rightarrow\underline{\V\otimes_B\check\V}, 
\end{equation*}
respecting the natural monoidal structures on either side. A cocharacter thus extended to $\a^1$ will be referred to as an effective one. We need to introduce a couple 
of operations for a cocharacter $\nu$ of the group $\GL(\V/B)$, say defined by a weight gradation $\V=\bigoplus_{l\in\z}\V_l$. There exists a canonical isomorphism
$$\GL(\V/B)\rightarrow\GL(\check\V/B);\,g\mapsto\check g,$$ 
where $\check\V=\Hom_B(\V,B)$. It follows that there is a canonically determined dual cocharacter $\check\nu:\g_{m,B}\rightarrow\GL(\check\V/B)$, and indeed one 
finds easily that the $\check\nu$-weight of $\check\V_l$ (regarded as a subspace of $\check\V$) is equal to $-l$. Let $H_0:\z\rightarrow\z$ be the Heaviside step function
\begin{equation*}
H_0(l):=\begin{cases}0&l\leq0\\1&l\geq1\end{cases},
\end{equation*}
consider the $B$-modules $\tilde\V_l:=\bigoplus_{H_0(m)=l}\V_m$, and let $H_0(\nu)$ be the cocharacter (of type $\{0,1\}$) which is defined 
by the weight-gradation $\V=\tilde\V_0\oplus\tilde\V_1$, i.e. $H_0(\nu)$ acts with weight $H_0(l)$ on each direct summand $\V_l$. Observe that 
the dual of $H_0(\nu\delta_\V)$ is equal to $\frac{H_0(\check\nu)}{\delta_{\check\V}}$, because $H_0(1-l)=1-H_0(l)$ holds for every integer $l$, 
the dual of $\delta_\V$ is $\delta_{\check\V}^{-1}$. By a filtration on a $B$-algebra $R$ one means a family of ideals $\{\ga_n\}_{n\in\z}$ with
\begin{itemize}
\item
$\ga_n\ga_m\subset\ga_{n+m}$
\item
$R=\ga_0=\ga_{-1}$
\end{itemize}
Equivalently, a filtration is given by a function 
$v:R\rightarrow\n_0\cup\{\infty\}$ with
\begin{itemize}
\item
$v(a+b)\geq\min\{v(a),v(b)\}$
\item
$v(ab)\geq v(a)+v(b)$
\end{itemize}
by means of $\ga_n=\{x\in R|v(x)\geq n\}$. We call $v$ the characteristic valuation of the filtration.

\subsection{Fiber Functors}
\label{ballast}
If $X$ is a scheme, we write $\ect(X)$ for the additive rigid $\otimes$-category of vector bundles on $X$, i.e. locally free $\O_X$-modules which are locally of finite type. 
We will also utilize the $B$-linear, additive, rigid $\otimes$-category $\bRep_0(\G)$ of 'representations', i.e. $\G-B$-modules which are finitely generated and projective 
over $B$, this has been introduced and studied in \cite[II.4.1.2.1]{rivano}, along with various functors defined thereon: There is a natural forgetful faithful fiber functor 
$$\omega^\G:\bRep_0(\G)\rightarrow\ect(\Spec B),$$ 
which takes a $\G-B$-module to its underlying $B$-module. Furthermore, for every locally trivial principal 
homogeneous space $P$ for $\G$ over some $B$-scheme $X$, there also exists a certain twisted fiber functor 
$$\omega_P:\bRep_0(\G)\rightarrow\ect(X).$$ 
This takes a representation $(\V,\rho)$ to $P\times^\G\V$, which is the locally free $\O_X$-module obtained by the usual extension of structure group (the map $\rho$ takes $P$ naturally 
to a locally trivial principal homogeneous space for $\GL(\V/B)$ over $X$, and in turn this defines a vector bundle of constant rank equal to $\rk_B\V$, cf. \cite[II.3.2.3.4]{rivano}). The 
functors $\omega^\G$ and $\omega_P$ are exact and respect the rigid $\otimes$-structure defined on their source and target. Moreover, the group $\G/B$ (resp. the locally trivial principal 
homogeneous space $P/X$) can be recovered from $\omega^\G$ (resp. $\omega_P$) by means of the natural isomorphisms $\G(R)\cong\Aut^\otimes(R\otimes_B\omega^\G)$ (resp. 
$P(R)\cong\Iso^\otimes(R\otimes_B\omega^\G,R\otimes_{\O_X}\omega_P)$), where $\Aut^\otimes$ (resp. $\Iso^\otimes$) denotes the $\otimes$-preserving subset of functorial isomorphisms 
from the $\ect(\Spec R)$-valued $\otimes$-functor $R\otimes_B\omega^\G$ to itself (resp. to $R\otimes_{\O_X}\omega_P$). Notice that every cocharacter $\mu:\g_{m,B}\rightarrow\G$ 
defines a gradation of type $\z$ on $\omega^\G$, \cite[Corollaire IV.1.2.2.2]{rivano}, i.e. compatible gradations $\omega^\G(\rho)=\bigoplus_{l\in\z}\omega_\mu^\G(l,\rho)$ 
as $\rho$ runs through $\bRep_0(\G)$. Following \cite[IV.2.1.3]{rivano}, we consider the associated closed subgroup functors
\begin{eqnarray}
&&U_\mu^m(R)=\{g\in\G(R)|\forall\rho,l_0:\\
\label{realIV}
&&(\rho(g)-1)(\omega_\mu^\G(l_0,\rho))\subset R\otimes_B\bigoplus_{l=m+l_0}^\infty\omega_\mu^\G(l,\rho)\},
\end{eqnarray}
for any $m\in\n_0$. If $\G$ is smooth, then $U_\mu^m$ is representable by (and will be identified with) a smooth subscheme, by \cite[Proposition IV.2.1.4.1]{rivano}. Since $U_\mu^0$ preserves the filtration $\{\bigoplus_{l=l_0}^\infty\omega_\mu^\G(l,\rho)|l_0\in\z\}$ it acts on each subquotient $\omega_\mu^\G(l_0,\rho)$, and it is useful to introduce the character: 
$$\chi_\mu(l_0,\rho):U_\mu^0\stackrel{\rho}{\rightarrow}\GL(\omega_\mu^\G(l_0,\rho)/B)\stackrel{\det}{\rightarrow}\g_{m,B}$$ 
Note that $\chi_\mu(l_0,\rho)|_{U_\mu^1}$ is trivial. It also follows that there is a canonical homomorphism
$$L_\mu:\a_B^1\rightarrow\Endo(U_\mu^0)$$
of multiplicative monoids which extends the interior $\g_{m,B}$-action $\Int^{U_\mu^0}\circ\mu$ on $U_\mu^0$. We would like to point out two variants of Saavedra's 
parabolic and unipotent subgroup constructions: First, if $I$ is any ideal in any $B$-algebra $R$, then we denote the inverse image $U_\mu^0(R/I)$ in $\G(R)$ by:
\begin{equation}
\label{realI}
\Gamma_\mu^{R,I}:=\G(R)\times_{\G(R/I)}U_\mu^0(R/I)
\end{equation}
Second, if $v:R\rightarrow\n_0\cup\{\infty\}$ is the valuation corresponding to a filtration $\ga_n\subset R$, then we introduce a family of subgroups of $\G(R)$ via
\begin{eqnarray*}
&&m\mapsto\hat U_\mu^m(R,v)=\{g\in\G(R)|\forall\rho,l_0:\\
&&(\rho(g)-1)(\omega_\mu^\G(l_0,\rho))\subset
\bigoplus_{l\in\z}\ga_{m+l_0-l}\otimes_B\omega_\mu^\G(l,\rho)\},
\end{eqnarray*}
which are going to serve as ``topological'' analogs to the groups $U_\mu^m$. Notice that $R=\ga_0=\ga_{-1}=\dots$, so that $U_\mu^m(R)$ is indeed 
contained in $\hat U_\mu^m(R,v)$. Note that $\G(Q)\cap\hat U_\mu^m(R,v)=\hat U_\mu^m(Q,v|_Q)$ holds for $B$-subalgebras $Q\subset R$.
%, and that $\Gamma_\mu^{R,I}=\hat U_\mu^0(R,v_I)$, provided that $v_I$ is the valuation:$$v_I:R\rightarrow\n_0\cup\{\infty\};\,x\mapsto\begin{cases}\infty&x\in I\\0&\mbox {otherwise}\end{cases}$$

\begin{lem}
\label{immersionI}
Let $i:\G\rightarrow\H$ be a closed immersion of affine and flat $B$-groups. Then $\hat U_\mu^m(R,v)=\{g\in\G(R)|i(g)\in\hat U_{i\circ\mu}^m(R,v)\}$ holds.
\end{lem}

\begin{thm}
\label{triangleI}
Let $v$ be the characteristic valuation of the filtration $\ga_n$ on a $B$-algebra $R$, where $\ga_1\subset\rad(R)$. Let $\T$ be a split maximal torus in a split reductive group scheme 
$\G$ over $B$ and let $\mu:\g_{m,B}\rightarrow\T$ be a cocharacter. Let us write $\{\alpha_1,\dots,\alpha_d\}$ for the subset of $\T$-roots of $\G$ which have positive $\mu$-weights, 
ordered such that the sequence of their $\mu$-weights $h_i:=<\alpha_i,\mu>$ is increasing. If $\U_i\subset\G$ denotes the root space corresponding to $\alpha_i$ one has
$$\hat U_{\mu^{-1}}^0(R,v)=U_{\mu^{-1}}^0(R)\prod_{i=1}^d\U_i(\ga_{h_i}),$$  
where the order of multiplication is according to increasing $i$.
\end{thm}
\begin{proof}
Consider some $g\in\hat U_{\mu^{-1}}^0(R,v)$. Notice that the group multiplication identifies $U_{\mu^{-1}}^0\times_BU_\mu^1$ with an open subscheme of 
$\G$ (use the implication "b)$\Rightarrow$a)" in \cite[Th\'eor\`eme (17.9.1)]{egaiv} followed by the "d)$\Rightarrow$b)"-one in the result \cite[Th\'eor\`eme (17.11.1)]{egaiv}). 
According to $\ga_1\subset\rad(R)$, we are allowed to deduce that $\hat U_{\mu^{-1}}^0(R,v)$ is contained in (the $R$-valued points of) the image of 
$U_{\mu^{-1}}^0\times_BU_\mu^1$ in $\G$. This yields a unique decomposition $g=g_0g_1$, with $g_0\in U_{\mu^{-1}}^0(R)$ and $g_1\in U_\mu^1(R)$, 
in fact we have $g_1\in U_\mu^1(R)\cap\hat U_{\mu^{-1}}^0(R,v)$, as $U_{\mu^{-1}}^0(R)\subset\hat U_{\mu^{-1}}^0(R,v)$. It is well known that the map
$$\prod_{i=1}^d\U_i\rightarrow U_\mu^1;(u_1,\dots,u_d)\mapsto u_1\cdot\dots\cdot u_d$$ 
is an isomorphism. Again we find a decomposition $g_1=\prod_{i=1}^du_i$, where the order of multiplication is according to increasing $i$ and 
$u_i\in\U_i(R)$. It remains to check $u_i\in\U_i(\ga_{h_i})$, which we do by induction on $i$: To this end notice that we have the inclusions
$U_\mu^{h_i}(R)\cap\hat U_{\mu^{-1}}^0(R,v)\subset\G(\ga_{h_i})$, $\U_i(\ga_{h_i})\subset\hat U_{\mu^{-1}}^0(R,v)$, and $\U_i\subset U_\mu^{h_i}$, so that 
$u_i\cdot\dots\cdot u_d=u_{i-1}^{-1}\cdot\dots\cdot u_1^{-1}g_1$ implies $u_i\equiv\dots\equiv u_d\equiv1\mod\ga_{h_i}$.
\end{proof}

\section{Existence of metaunitary Shimura data}
\label{hilfssaetze}

Here is our raw material for sample poly-unitary Shimura data:

\begin{prop}
\label{triple}
Let $K$ be a field of characteristic $0$ and let $G_0$ be a connected reductive algebraic group over $K$ whose commutator subgroup $G_0^{der}$ is 
simply connected and whose adjoint group $G_0^{ad}$ is absolutely simple. Let $\rho_0$ be a linear representation of $G_0$ on a vector space $U_0$ 
of finite dimension over $K$. Assume that the restriction of $\rho_0$ to the center $\zen^{G_0}$ is faithful. Let us write $\gg_0^{der}$ (resp. $\gg_0$) for the 
Lie algebra of $G_0^{der}$ (resp. $G_0$), let $\rho'_0$ be the representation of $G_0$ on the space $U'_0=U_0\oplus\gg_0^{der}$ and let ${\rho'_0}^{der}$ 
be the corresponding representation of $\gg_0$ on $U'_0$. Let us write $U'_1$ for $U_1\oplus\gg_0^{der}$, where $U_1$ is the trivial one-dimensional 
representation of $G_0$, and fix an element $e\in U_1\backslash\{0\}$. Consider the $K$-subalgebra $K[a]\subset\End_{G_0}(U'_1\otimes_KU'_0)$ that is generated by 
$$(x_1+x'_1)\otimes(x_0+x'_0)\stackrel{a}{\mapsto}(x'_0\otimes x'_1+e\otimes{\rho'_0}^{der}(x'_1)(x_0+x'_0)),$$ 
where $x_i\in U_i$ and $x'_i\in\gg_0^{der}$ for $i\in\{0,1\}$. Write $G^0$ for the stabilizer in 
$\GL(U'_1/K)\times_K\GL(U'_0/K)$ of $\End_{G_0}(U'_1)$, $\End_{G_0}(U'_0)$, and $K[a]$. Then
\begin{equation}
\label{tripleI}
G^0=G_0\zen^{G^0}
\end{equation}
holds. More specifically, write $Z'_0$ for the center of $\End_{G_0}(U'_0)^\times$ and write $Z'\subset Z'_0$ for the subtorus consisting of elements whose 
action on $\gg_0^{der}$ is a scalar. Then $Z'$ is naturally contained in $\zen^{G^0}$, and indeed one has $\zen^{G^0}=\g_{m,K}\times_KZ'$, where the 
copy of $\g_{m,K}$ acts as scalars on $U'_1$. The rank of $\zen^{G^0}$ is equal to $1+\Card\{\text{isotypic components of }U'_0\}$ and it is connected.
\end{prop}
\begin{proof}
The proposition and its proof are both similar to \cite[Lemma 7.3]{habil} and it is clear that we may assume $K$ to be algebraically closed. Fix an element of 
$G^0(K)$, according to the presence of $\End_{G_0}(U'_0)$ and $\End_{G_0}(U'_1)$, we can write $\left(\begin{matrix}g_0&0\\0&g'_0\end{matrix}\right)$ 
and $\left(\begin{matrix}g_1&0\\0&g'_1\end{matrix}\right)$ for the induced maps on $U'_0$ and $U'_1$. Notice that $g'_0$ and $g'_1$ are proportional, 
due to $g'_1\otimes g'_0=g'_0\otimes g'_1$, according to the presence of $a$. For the same reason we have
\begin{eqnarray*}
&&g_0\rho_0^{der}(x'_1)x_0=\rho_0^{der}(\frac{g'_1}{g_1}x'_1)(g_0x_0)\\
&&\frac{g'_0}{g_1}[x'_1,x'_0]=[\frac{g'_1}{g_1}x'_1,\frac{g'_0}{g_1}x'_0]
\end{eqnarray*}
for all $x_0\in U_0$ and $x'_0$, $x'_1\in\gg_0^{der}$. This implies the existence of an automorphism $\alpha$ of $G_0^{der}$ which satisfies $\alpha^{der}=\frac{g'_1}{g_1}$ and
\begin{equation}
\label{tripleII}
g_0\rho_0(g)g_0^{-1}=\rho_0(\alpha(g))
\end{equation} 
for all $g\in G_0^{der}(K)$. Using the presence of $\End_{G_0}(U'_0)$ again, we see that $\rho_0(\alpha(z))=\rho_0(z)$ holds for any $z\in\zen^{G_0}(K)\cap G_0^{der}(K)$. 
Using the faithfulness of $\rho_0\mid_{\zen^{G_0}}$ we deduce that $\alpha$ must be an inner automorphism, here notice that the group of outer automorphisms 
of a simply connected absolutely almost simple group acts faithfully on its center (cf. \cite[Chapter III, Section 13, Exercise 5]{humphreys}). Upon an adjustment 
we are allowed to assume $\alpha=\id_{G_0^{der}}$, so that $g'_1$ is equal to the multiplication by the scalar $g_1$ while $g_0$ commutes with the 
$G_0^{der}$-action according to \eqref{tripleII}. The presence of $\End_{G_0}(U'_0)$ allows to deduce that $g_0$ commutes with the $G_0$-action 
too. Consequently $g_0$ lies in the center of $\End_{G_0}(U_0)^\times$. This proves \eqref{tripleI} and the remaining assertions are left to the reader.
\end{proof}

Just as in subsection \ref{special} we let $L$ be a CM field and $(G,h)$ a Shimura datum with coefficients in its maximal totally 
real subfield $L^+$. The next two subsections address the existence of $L$-metaunitary Shimura data: In the sequel we say that 
$(\{(V_i,\Psi_i,\rho_i)\}_{i\in\Lambda},\{(R_\pi,\iota_\pi)\}_{\pi\in\Pi})$ is a $L$-multi-unitary collection for $(G,h)$ if and only if:

\begin{itemize}
\item
Each $(V_i,\Psi_i,\rho_i)$ satisfies the conditions (U1) and (U2), and
\item
$(\{(V_i,\Psi_i,\rho_i)\}_{i\in\Lambda},\{(R_\pi,\iota_\pi)\}_{\pi\in\Pi})$ is a $L$-multi-unitary collection for 
$(\g_{m,L^+}^\Lambda\times_{L^+}G^1)/\{\pm1\}^\Lambda$, in the sense of definition \ref{exampleIX}.
\end{itemize}
Without further notice we write $h_i:=\varrho_{i,\r}\circ h$ for the skew-Hermitian Hodge structure of weight $-1$ on $(V_i,\Psi_i)$ which is induced by the first of the two 
conditions. Moreover, we call $(\{(V_i,\Psi_i,\rho_i)\}_{i\in\Lambda},\{(R_\pi,\iota_\pi)\}_{\pi\in\Pi})$ unramified at some prime $p$ if there exist self-dual $\z_{(p)}\otimes\O_L$-lattices 
$$V_i\supset\gV_i=\gV_i^\perp:=\{x\in V_i|\Psi_i(x,\gV_i)\subset\z_{(p)}\otimes\O_L\}$$
such that the compact open subgroups $\UL(\gV_i/\z_{(p)}\otimes\O_L,\Psi_i)(\z_p\otimes\O_{L^+})=:U_i^1\subset\UL(V_i/L,\Psi_i)(\q_p\otimes\O_{L^+})$ and
\begin{equation}
\label{twistingII}
\bigcap_{i\in\Lambda}\rho_i^{-1}(U_i^1)=:U^1\subset G^1(\q_p\otimes L)
\end{equation}
are hyperspecial. The standard euclidean pairing on the $n$-dimensional $\q$-vector space $\q^n$ shall be denoted by $E_n(x,y):=x^ty$.

\begin{lem}
\label{polyV}
If $(\{(V_i,\Psi_i,\rho_i)\}_{i\in\Lambda},\{(R_\pi,\iota_\pi)\}_{\pi\in\Pi})$ is a $L$-multi-unitary collection for $(G,h)$, then the same is true for:
$$(\{(V_i^{\oplus8},E_8\otimes\Psi_i,\rho_i^{\oplus8})\}_{i\in\Lambda},\{(\Mat(8^{\Card(\pi)},R_\pi),\iota_\pi)\}_{\pi\in\Pi})$$
Moreover, the latter is unramified at some prime $p$ if and only if $L$ and $\Res_{L^+/\q}G_0$ are unramified at $p$.
\end{lem}
\begin{proof}
We pick a hyperspecial subgroup $U^1\subset\Res_{L^+/\q}G^1(\q_p)$. Sufficiently small $U^1$-invariant $\z_{(p)}\otimes\O_L$-lattices $\gV_i\subset V_i$ for $i\in\Lambda$, 
and a sufficiently large integer $m$ would satisfy $\gV_i\subset\gV_i^\perp\subset p^{-m}\gV_i$. Now recall Zarhin's trick: Choose an even self-dual lattice $Z\subset\q^8$ 
containing a direct factor $M$ of rank $4$, such that $E_8(x,y)\equiv0\pmod{p^m}$ for any $x,y\in M$. Replacing $\gV_i$ by the $\z_{(p)}\otimes\O_L$-lattices
$$Z\otimes\gV_i+M\otimes\gV_i^\perp\subset V_i^{\oplus8}$$
proves the result.
\end{proof}

Let us fix an element $\vartheta\in\Gal(R/\q)$, where $R\subset\c$ denotes the splitting field of $L$. We have the following further criterion:

\begin{lem}
\label{polyIII}
Fix $(G,h)$ together with an $L$-multi-unitary collection $(\{(V_i,\Psi_i,\rho_i)\}_{i\in\Lambda},\{(R_\pi,\iota_\pi)\}_{\pi\in\Pi})$. 
We let $[a_{i,\iota},b_{i,\iota}]$ be the smallest intervals, such that the Hodge decomposition of $(V_i,\Psi_i,h_i)$ is of type 
\begin{equation}
\label{fourtyeight}
\{(-b_{i,\iota},b_{i,\iota}-1),\dots,(-a_{i,\iota},a_{i,\iota}-1)\},
\end{equation}
and we let $w_{i,\iota}:=b_{i,\iota}-a_{i,\iota}$ be their lengths. We assume $1\equiv\Card(\Lambda)\pmod 2$ and $\Pi=\{\Lambda\}\cup\{\{i\}|i\in\Lambda\}$ and that
\begin{equation}
\label{critical}
\sum_{i\in\Lambda}\lim_{N\to\infty}\frac{\sum_{k=1}^Nw_{i,\vartheta^k\circ\iota}}N<1,
\end{equation}
holds for every $\iota\in L_{an}$. Then there exist elements $n_i\in L^+\backslash\{0\}$ for each $i\in\Lambda$ and a $L$-metaunitary Shimura datum of the form 
$$(\bd^+,\{(V_i,n_i\Psi_i,\rho_i,\bj_i)\}_{i\in\Lambda},\{(R_\pi,\iota_\pi)\}_{\pi\in\Pi})$$ 
for $(G,\frac h{h_1})$, where $\s^1\cong\s/\g_{m,\r}\stackrel{h^1}{\rightarrow}(\Res_{L^+/\q}G^1)_\r$ is a central cocharacter.
\end{lem}
\begin{proof}
It does no harm to assume $\Lambda=\{1,\dots,n\}$, consider the numbers:
\begin{itemize}
\item
$a_\iota:=\frac{1-n}2+\sum_{i=1}^na_{i,\iota}$ 
\item
$b_\iota:=\frac{1-n}2+\sum_{i=1}^nb_{i,\iota}=1-a_{\iota\circ*}$
\item
$w_\iota:=\sum_{i=1}^nw_{i,\iota}=b_\iota-a_\iota$
\end{itemize}
Due to $\lim_{N\to\infty}\frac{\sum_{k=1}^Nw_{\vartheta^k\circ\iota}}N<1$ for each $\iota$, there does exist a $\vartheta$-gauge $(\bd^+,\bj)$ for the type 
$\{(-b_\iota,b_\iota-1),\dots,(-a_\iota,a_\iota-1)\}$, such that each cycle of $L_{an}$ contains at least one element $\kappa$ satisfying $\bj(\kappa)\notin[a_{\bd(\kappa)},b_{\bd(\kappa)}-1]$. We continue 
to ignore the conditions (N1), (N2) and (N3), but we retain $\bd^+$ and proceed with the construction of a family of $\vartheta$-gauges $(\bd^+,\bj_i)$ of the types \eqref{fourtyeight} satisfying:
\begin{itemize}
\item[(i)]
For every $\kappa\in L_{an}$, the cardinality of $\{i\in\Lambda|\bj_i(\kappa)<a_{i,\bd(\kappa)}\}$, and of $\{i\in\Lambda|\bj_i(\kappa)\geq b_{i,\bd(\kappa)}\}$ 
is at least $\frac{\Card(\Lambda)-1}2$.
\item[(ii)]
Each cycle of $L_{an}$ contains at least one element $\kappa$ which satisfies 
$\bj_i(\kappa)\notin[a_{i,\bd(\kappa)},b_{i,\bd(\kappa)}-1]$ for all $i\in\Lambda$.
\end{itemize}
Choose a disjoint union $L_{an}=S\circ*\cup S$, where we may require that each split cycle is entirely contained in one of $S$ or $S\circ*$. 
Consider some $\kappa$ with $\iota:=\bd(\kappa)\in S$. Thinking of $[a_\iota,b_\iota-1]$ as the concatenation of suitable translates of 
the intervals $[a_{1,\iota},b_{1,\iota}-1],\dots,[a_{n,\iota},b_{n,\iota}-1]$ leads to unique integers $l$ and $m$ with the properties: 
\begin{eqnarray*}
&&\frac{n-1}2+\bj(\kappa)=\sum_{i<m}b_{i,\iota}+l+\sum_{i>m}a_{i,\iota}\\
&&m\in[0,n+1]\\
&&l\in\begin{cases}{]-\infty,-1]}&m=0\\
{[a_{m,\iota},b_{m,\iota}-1]}&m\in[1,n]\\
{[0,\infty[}&m=n+1\end{cases}
\end{eqnarray*}
So that we may define, for every $i\in\Lambda$:
$$-\bj_i(\kappa\circ*)=\bj_i(\kappa):=
\begin{cases}a_{i,\iota}-1&(-1)^i(i-m)<0\\
l&i=m\\b_{i,\iota}&(-1)^i(i-m)>0\end{cases}.$$
Observe that we have $b_{i,\iota}-a_{i,\iota}\leq\Card(\bd^{-1}(\{\iota\}))$ for every $i$ and $\bj_i(\kappa)\in[a_{i,\bd(\kappa)}-1,b_{i,\bd(\kappa)}]$ 
for every $\kappa$ and in order to enforce the condition (N1) we simply modify the types according to
\begin{eqnarray}
&&\tilde a_{i,\iota}:=a_{i,\iota}-\Card(\{\kappa\mid\,\bd(\kappa)=\iota,\;\bj_i(\kappa)=a_{i,\iota}-1\})\\
&&\tilde b_{i,\iota}:=b_{i,\iota}+\Card(\{\kappa\mid\,\bd(\kappa)=\iota,\;\bj_i(\kappa)=b_{i,\iota}\})
\end{eqnarray}
An obvious change of the of $\bj_i$'s yields $\vartheta$-gauges $(\bd^+,\tilde\bj_i)$ of the types 
$\{(-\tilde b_{i,\iota},\tilde b_{i,\iota}-1),\dots,(-\tilde a_{i,\iota},\tilde a_{i,\iota}-1)\}$ still satisfying (i) and (ii). We enforce the conditions (N2) and (N3) by a translation, consider
$$a'_\iota:=\begin{cases}0&\bd(\iota)\in S\\1&\bd(\iota)\notin S\end{cases}$$
for all $\iota\in S$, which happens to satisfy not only (N1) but also (N2) and (N3). Let $c_i$ be the function $\iota\mapsto c_{i,\iota}:=\tilde a_{i,\iota}-a'_\iota$, 
and let $n_i\in L^+\backslash\{0\}$ satisfy $(-1)^{c_{i,\iota}}\iota(n_i)>0$. There exists a cocharacter $\gamma:\g_{m,\c}\rightarrow(\Res_{L^+/\q}G^1)_\c$ 
such that the weights of $\rho_i\circ\gamma$ are given by the family $c_i$. Since $c_{i,\iota}=-c_{i,\iota\circ*}$ holds, it is easy to see that $\gamma=\gamba^{-1}$ 
descends to a cocharacter $h^1:\s^1\rightarrow(\Res_{L^+/\q}G^1)_\r$ (N.B.: $\underbrace{L^1\times\dots\times L^1}_\Lambda$ 
is canonically contained in the center of $\Res_{L^+/\q}G^1$ and $\s$ maps canonically onto $\s^1$ by $z\mapsto z\zbar^{-1}$).
\end{proof}

In the next proposition we consider an adjoint Shimura datum of the form $(\Res_{L^+/\q}G_0,h_0)$, where we require $G_0=G_0^{ad}$ to be an absolutely simple reductive group 
over a totally real field $L^+$. Let $K$ be the smallest $L^+$-extension over which $G_0$ is an inner form of a split form. Observe that $K$ is a totally imaginary quadratic (resp. totally 
real) extension of $L^+$ if $G_0$ is of type $A_l$ with $l>1$, $D_l$ with $l$ odd or $E_6$ (resp. of type $A_1$, $B_l$, $C_l$, $D_l$ with $l$ even or $E_7$). We also choose a totally imaginary quadratic extension $L$ of our totally real field $L^+$, but we require that the CM field $L$ agrees with $K$ whenever the latter happens to be a CM field i.e. in the $A_2$, $A_3,\dots$, $D_3$, $D_5,\dots$ and $E_6$ cases. We have to introduce a provisional central extension $\Gbar\twoheadrightarrow G_0$, defined by a specific push-out diagram:
$$\begin{CD}
1@>>>\Cbar@>>>\Gbar@>>>\Gbar^{ad}@>>>1\\
@AAA@A{\zeta}AA@AAA@A{\cong}AA@AAA\\
1@>>>\zen^{\tilde G_0}@>>>\tilde G_0@>>>G_0@>>>1
\end{CD},$$
where $\tilde G_0$ is the simply-connected and semisimple covering group of $G_0$. In the 
$D_l$-case with $l$ even, we start with the description of $\Cbar$ being the qudratic twist:
\begin{equation*}
\Cbar:=\ker(\Res_{L/L^+}T_L\stackrel{\n_{L/L^+}}{\rightarrow} T),
\end{equation*}
where $T$ is the $L^+$-torus whose group of characters agrees with the lattice 
$$X^*(T)=\{\chi:\zen^{\tilde G_0}(\c)\rightarrow\z|0=\chi(1)=\sum_{g\neq1}\chi(g)\},$$
which is of rank $2$. Identifying the Klein four-group scheme $\zen^{\tilde G_0}$ with $T[2]\cong\Cbar[2]$ by decreeing the character $\chi$ to attain the value $(-1)^{\chi(g)}$ 
on an element $g\in\zen^{\tilde G_0}(\c)$ defines an inclusion $\zeta$ of $\zen^{\tilde G_0}$ into $\Cbar$. In the remaining cases we define $\Cbar$ to be algebraic torus:
\begin{equation*}
\l^1:=\ker(\Res_{L/L^+}\g_{m,L}\stackrel{\n_{L/L^+}}{\rightarrow}\g_{m,L^+})
\end{equation*}
If $G_0$ is of type $E_7$, $C_l$, $B_l$ or $A_1$ we have $\zen^{\tilde G_0}=\{\pm1\}$, and we let $\zeta$ be the inclusion of 
$\zen^{\tilde G_0}$ into $\l^1$, and otherwise we introduce $\zeta$ by observing that the theory of root data produces an isomorphism
$$Z^{\tilde G_0}\cong\begin{cases}\l^1[3]&E_6\\\l^1[4]&D_3,D_5,\dots\\\l^1[l+1]&A_l\end{cases}$$
which is canonical up to composition with the reciprocal function. Given that $(\Res_{L^+/\q}\Cbar)(\r)$ is connected, 
and that $h_0$ factors through the map $\s\rightarrow\s^1;z\mapsto z\zbar^{-1}$, there exists a diagram
$$\begin{CD}
\s@>{h_0}>>(\Res_{L^+/\q}G_0)_\r\\
@VVV@AAA\\
\s^1@>{\hbar}>>(\Res_{L^+/\q}\Gbar)_\r
\end{CD},$$
so that $(\Res_{L^+/\q}\Gbar,\hbar)$ is a Shimura datum too. Recall that $(G_0,h_0)$ is said to be of type $D_l^\r$ if all simple factors of $G_{0,\r}$ are 
of the form $\SO(2l-2,2)/\{\pm1\}$ or $\SO(2l)/\{\pm1\}$. In the following proposition the type $D_l^{\r\&\h}$ stands for data $(G_0,h_0)$ with $G_{0,\c}$ 
of type $D_l$, but $(G_0,h_0)$ not of type $D_l^\r$ (N.B.: $D_4^{\r\&\h}$ does not exist, because $\SO^*(8)/\{\pm1\}\cong\SO(6,2)/\{\pm1\}$). 

\begin{prop}
\label{polyII}
Let $(\Res_{L^+/\q}G_0,h_0)$ and $L$ be as above. If
$$\lim_{N\to\infty}\frac{\Card\{k\in\{1,\dots,N\}|G_0(\r_{[\vartheta^k\circ\iota]})\text{ is compact }\}}N>
\begin{cases}0&A_l\\\frac34&E_6,B_l,C_l,D_l^\r\\\frac45&E_7\\\frac{l+1}{l+3}&D_l^{\r\&\h},\,2\nmid l\\\frac{l+2}{l+4}&D_l^{\r\&\h},\,2\mid l\end{cases}$$
holds for all embeddings $\iota:L^+\rightarrow\r$, then there exists and a $L$-metaunitary Shimura datum for a Shimura datum of the form
$$((\g_m\times\Res_{L^+/\q}G^1)/\{\pm1\},h),$$ 
where $G^1\twoheadrightarrow G_0$ is a central extension such that $h$ maps to $h_0$, $G^{1der}$ is simply connected and $\zen^{G^1}$ 
is a connected $L^+$-torus which splits over the composite of $L$ with the smallest $L^+$-extension over which $G_0$ is an inner form of a split form. 
\end{prop}
\begin{proof}
In the $A_l$-case there is nothing to prove. Excluding the case $D_l$ with $l$ even for a while we let $\rho_0:\Gbar_L\rightarrow\GL(W_0/L)$ be a isotypic $L$-rational 
representation which is a direct sum of copies of minuscule irreducible representations each of whose restrictions to the center $\Cbar_L=\g_{m,L}$ agree with scalar multiplication. 
The sets of $\mu_{\hbar}$-weights of each eigenspace (to the embeddings $L\hookrightarrow\c$) can be read off from \cite[Table 1.3.9]{deligne3}, so these are translates of 
\begin{eqnarray*}
\{-\frac23,\frac13,\frac43\}&&E_6\\
\{-\frac 32,-\frac12,\frac12,\frac32\}&&E_7\\
\{-\frac12,\frac12\}&&B_l,C_l,D_l^\r\\
\{\frac{2-l}4,\dots,\frac l4\}&&D_l^{\r\&\h}
\end{eqnarray*}
According to \cite[Rapel 1.1.15(ii)]{deligne3} we may pick $\Gbar$-invariant positive $L$-Hermitian pairings $\Psi_0$ on $W_0$ and $\Psi_1$ on 
the trivial one-dimensional $\Gbar$-representation, which we wish to denote by $W_1$. Let $(W,\Psi,h)$ be an arbitrary skew-Hermitian Hodge 
structure of weight $-1$ and rank $1$, regarded as a unitary $\l^1$-representation of type $\{(-1,0),(0,-1)\}$ (N.B.: Every CM-type for $L$ endows 
$(\g_{m,\q}\times L^1)/\{\pm1\}$ with the structure of a Shimura datum). We continue with the following triple of unitary $\l^1\times_{L^+}\Gbar$-representations:
\begin{eqnarray*}
&&V_1=W\\
&&V_2=(W_1\oplus L\otimes_{L^+}\gg_0)\otimes_LW\\
&&V_3=(W_0\oplus L\otimes_{L^+}\gg_0)\otimes_LW
\end{eqnarray*}
where $\gg_0$ is the Lie algebra of $G_0$ and $L\otimes_{L^+}\gg_0$ is endowed with the negative Killing form. Finally, the $D_l^\r$-cases are similar to the 
$B_l$-one, and left to the reader, but it remains to do the $D_l^{\r\&\h}$-case with $l$ even. If $K=L^+$, the theory of root data produces an isomorphism
$$\Cbar\cong\l^1\times_{L^+}\l^1,$$ 
which is canonical up to a swap of the factors. So let us set $\rho_0$ to be the direct sum of (possibly several copies of) the two minuscule representations 
$\rho_0^+$ and $\rho_0^-$, with central characters being the two projections $\Cbar_L\rightarrow \g_{m,L}$. Again, the weight sets of each eigenspace (to the 
embeddings $L\hookrightarrow\c$) can be read off from \cite[Table 1.3.9]{deligne3}, so depending on $\mu_{\hbar}$ these are translates of one of the sets:
\begin{eqnarray*}
\{-\frac12,\frac12\}&&\\
\{\frac{2-l}4,\dots,\frac{l-2}4\}&&\\
\{-\frac l4,\dots,\frac l4\}&&
\end{eqnarray*}
In general $K$ is a totally real quadratic extension of $L^+$, $\Cbar$ is canonically isomorphic to $\Res_{K/L^+}\l^1$ and the argument is analogous.
\end{proof}

\section{Varshavsky's characterization method}
\label{characterization}

The paper \cite{varshavskyII} deals with characterizations of Shimura varieties, here are the objects under consideration:

\begin{itemize}
\item
Let $\Delta$ be a group and let $M$ be a separated, non-empty, connected, complex manifold with a holomorphic left $\Delta$-action.
\item
Let $G_1,\dots,G_z$ be connected, simple, non-abelian algebraic groups over $\c$, and let $P_i$ be a proper minuscule parabolic subgroup of $G_i$. Denote the associated 
irreducible symmetric Hermitian domains of compact type by $X_i:=G_i(\c)/P_i(\c)$, and write $G:=\prod_{i=1}^zG_i$ and $P:=\prod_{i=1}^zP_i$ and $X:=\prod_{i=1}^zX_i$.
\item
Let $\varphi:\Delta\rightarrow G(\c)$ be a group homomorphism, and let $I:M\rightarrow X$ be a 
locally biholomorphic $\Delta$-equivariant map (use $\varphi$ to define a left action of $\Delta$ on $X$)
\end{itemize}

A sextuple $(\Delta,M,G,P,\varphi,I)$ as above is called a period map if:

\begin{itemize}
\item[(M1)]
The image of $\Delta$ in $G(\c)$ is Zariski-dense.
\item[(M2)]
The image in $G(\c)$ of the stabilizer in $\Delta$ of any element in $M$ possesses a compact closure.
\item[(M3)]
There exists some $y_0\in M$ whose stabilizer in $\Delta$ contains some subgroup $\Delta_0$ the closure of whose image in $G(\c)$ is equal to 
$\prod_{i=1}^z\bT_i(\r)$, where each $\bT_i$ is a maximal compact torus in $\Res_{\c/\r}G_i$ (i.e. the $\r$-rank of $\bT_i$ is equal to the $\c$-rank of $G_i$).
\end{itemize}

Without any attempt of originality we wish to give a slight reformulation of \cite[p.89,p.92-94]{varshavsky} in the above axiomatic setup:

\begin{thm}[Varshavsky]
\label{uniformizationI}
Suppose that $(\Delta,M,G_i,P_i,\varphi,I)$ is a period map. Then there exist real forms $\bJ_i$ of $G_i$ for every $i\in\{1,\dots,z\}$, 
such that $\overline{\varphi(\Delta)}=\bJ(\r)^+$, where $\bJ:=\prod_{i=1}^z\bJ_i$. The locally biholomorphic map $I$ is actually 
an injection and the $\Delta$-action on $M$ can be extended to a continuous and transitive $\bJ(\r)^+$-action thereon.\\
Pick a base point $\tilde y\in M$, and assume for notational convenience that $I(\tilde y)$ is the canonical base point of $\prod_{i=1}^zX_i$, i.e. equal to 
$(1,\dots,1)$. Write $\bU_i:=P_i\cap\Pbar_i$, where $\Pbar_i$ denotes the complex conjugate of $P_i$ with respect to the real form $\bJ_i$. Consider the homogeneous 
spaces $M_i:=\bJ_i(\r)^+/\bU_i(\r)$, so that $M=\prod_{i=1}^zM_i$. Then for each $i\in\{1,\dots,z\}$ one and only one of the following alternatives hold:
\begin{itemize}
\item
$\bJ_i$ is compact, and $\bU_i$ is a maximal proper connected subgroup, 
\item
or $\bJ_i$ is not compact, and $\bU_i$ is a maximal compact subgroup,
\end{itemize}
and in any case $\bU_i$ has indiscrete center so that $M_i$ is a symmetric Hermitian domain of compact or of non-compact type.
\end{thm}

We give a synopsis of the proof: The solution to the 5th problem of Hilbert implies that $J:=\overline{\varphi(\Delta)}$ is a Lie-group. Note that $\c\otimes_\r\Lie J$ is semisimple, 
because $\Lie G$ is semisimple and both of $\Lie J\cap\sqrt{-1}\Lie J$ and $\Lie J+\sqrt{-1}\Lie J$, being $G$-invariant $\c$-subspaces of $\Lie G$ in view of (M1), are semisimple 
too. Moreover, there exists a semi-simple real algebraic group $\bJ\subset\Res_{\c/\r}G$ such that $\Lie\bJ=\Lie J$, as semisimple Lie-algebras are algebraic. Finally the existence 
of $\bT$ tells us that $\bJ=\prod_{i=1}^r\bJ_i$ where each single $\bJ_i$ contains $\bT_i$ and it is either a real form of $G_i$ or it is equal to $\Res_{\c/\r}G_i$. Let us write $\Aut(M)$ 
for the the homeomorphism group of $M$, and let us endow it with the compact-open topology. Let $\tilde J$ (resp. $\tilde T$) denote the closure of the image of $\Delta$ (resp. $\Delta_0$) 
in $\Aut(M)$. The group $\tilde T$ is compact because it fixes a point and preserves a suitable Riemannian metric, \cite[II, Theorem 1.2]{kobayashi}. It is straightforward to see that 
$\varphi$ extends to a continuous group homomorphism, say $\tilde\varphi:\tilde J\rightarrow J$, the kernel of which is a discrete subgroup of $\tilde J$, cf. \cite[Lemma 3.1]{varshavsky}. 
The following argument shows that $\tilde\varphi$ is surjective: We clearly have $\overline{\tilde\varphi(\tilde J)}=J$ and note also that $\tilde\varphi(\tilde T)=\bT(\r)$ because $\tilde T$ 
is compact. Now there exist elements $\gamma_1,\dots,\gamma_n\in\Delta$ such that $\Ad(\gamma_1)\Lie\bT+\dots+\Ad(\gamma_n)\Lie\bT=\Lie J$. It follows that the product 
$\tilde\varphi(\gamma_1\tilde T\gamma_1^{-1})\cdot\dots\cdot\tilde\varphi(\gamma_n\tilde T\gamma_n^{-1})$ contains an open neighborhood of the identity in $J$ and whence 
it follows that $\tilde\varphi(\tilde J)=J$. One shows the openness of $\tilde\varphi$ along the same lines.\\
If a Cartan subgroup $C\subset J$ fixes some point in the image of $M$ in $X$, then it must be compact. This can be shown as in \cite[Lemma 3.6]{varshavsky} using the property (M2) only,
 together with the fact that $\tilde J$ is a Lie-group, which is implied by the openness of $\tilde\varphi$. Finally two more facts follow easily from that observation: First, no 
$\bJ_i$ is equal to $\Res_{\c/\r}G_i$ so that $\bJ$ is actually a real form of $G$ and second, the stabilizer in $J$ of any point on $X$ has to contain some Cartan subgroup $C$ 
of $J$, this is because it is equal to the intersection of two complex mutually conjugate parabolics, namely the stabilizer in $G$ of that very point and its complex conjugate.\\
We note in passing that $(\tilde J,M,G,P,\tilde\varphi,I)$ is a period map too, by \cite[Proposition 3.5]{varshavsky}. One next establishes the transitivity of the $\tilde J$-action, 
which is accomplished by a simple dimension count (look at the stabilizers and notice that all of the maximal compact subgroups of $G$ act transitively on $G/P$). 

\begin{rem}
\label{boundedI}
Let us say that a period map $(\Delta,M,G,P,\varphi,I)$ is bounded if
\begin{itemize}
\item[(M4)]
there exist positive integers $p_j$ such that the pull-back of some line bundle of the form $\bigotimes_j\omega_j^{\otimes p_j}$ by 
means of $I$ is generated by its holomorphic global sections, where $\omega_j$ denotes the canonical line bundle on $X_i$.
\end{itemize}
It is clear that bounded period maps give rise to bounded symmetric Hermitian domains $M$.
\end{rem}

\end{appendix}

\end{document}